\documentclass[12pt, a4paper]{amsart}
\usepackage{amsmath,amssymb,amsthm,booktabs,longtable}
\usepackage{enumitem}
\usepackage{comment}
\usepackage{graphicx}
\usepackage[all]{xy}
\usepackage{caption}
\usepackage{color}
\usepackage[colorlinks=true, linkcolor=blue, citecolor=blue, urlcolor=blue]{hyperref}
\hypersetup{breaklinks=true}
\usepackage{mathdots}
\usepackage{url}
\usepackage[T1]{fontenc} 
\usepackage{here}
\usepackage{lscape}
\usepackage{tikz}

\theoremstyle{plain}
\makeatletter 

\def\l@section{\@tocline{1}{0pt}{1em}{}{}}
\def\l@subsection{\@tocline{2}{0pt}{2em}{5em}{}}
\def\l@subsubsection{\@tocline{3}{0pt}{3em}{5em}{}}

\@addtoreset{table}{section}

\@addtoreset{figure}{section}
\makeatother %
\newtheorem{theorem}{Theorem}[section]
\newtheorem{lemma}[theorem]{Lemma}
\newtheorem{corollary}[theorem]{Corollary}
\newtheorem{proposition}[theorem]{Proposition}
\newtheorem{fact}[theorem]{Fact}
\newtheorem{example}[theorem]{Example}
\newtheorem{setting}[theorem]{Setting}
\newtheorem{problem}[theorem]{Problem}
\newtheorem{question}[theorem]{Question}

\theoremstyle{definition}
\newtheorem{definition}[theorem]{Definition}
\newtheorem{deflem}[theorem]{Definition-Lemma}

\newtheorem{condition}[theorem]{Condition}
\newtheorem{remark}[theorem]{Remark}
\newtheorem{notation}[theorem]{Notation}
\newcommand{\ad}{\mathop{\mathrm{ad}}\nolimits}
\newcommand{\Ad}{\mathop{\mathrm{Ad}}\nolimits}
\newcommand{\Int}{\mathop{\mathrm{Int}}\nolimits}
\newcommand{\Aut}{\mathop{\mathrm{Aut}}\nolimits}
\newcommand{\End}{\mathop{\mathrm{End}}\nolimits}

\newcommand{\id}{\mathop{\mathrm{id}}\nolimits}
\newcommand{\Hom}{\mathop{\mathrm{Hom}}\nolimits}
\newcommand{\rank}{\mathop{\mathrm{rank}}\nolimits}
\newcommand{\diag}{\mathop{\mathrm{diag}}\nolimits}
\newcommand{\even}{\mathrm{even}}
\newcommand{\ord}{\operatorname{ord}}
\newcommand{\Q}{\mathbb{Q}}
\newcommand{\R}{\mathbb{R}}
\newcommand{\C}{\mathbb{C}}
\newcommand{\N}{\mathbb{N}}
\newcommand{\Z}{\mathbb{Z}}
\newcommand{\F}{\mathbb{F}}
\newcommand{\E}{\mathbb{E}}
\newcommand{\HA}{\mathbb{H}}
\newcommand{\T}{\mathbb{T}}
\newcommand{\Spin}{\mathbf{Spin}}
\makeatletter
    
    \@addtoreset{equation}{section}
  \makeatother
\makeatother

\title[]{Zariski-dense deformations of standard discontinuous groups for pseudo-Riemannian homogeneous spaces}
\author{Kazuki Kannaka and Toshiyuki Kobayashi}

\subjclass[2020]{
Primary:
57S30, 
58H15. 
Secondary:
22D50, 
22E40, 
22E46, 
53C30, 
58J50.  
}
\keywords{
discontinuous group; 
proper action; 
Zariski dense subgroup; 
rigidity; 
locally symmetric space; 
pseudo-Riemannian manifold;
Clifford–Klein form; 
spin group.
}

\address[Kazuki KANNAKA]{%
Faculty of Mathematics and Physics, Institute of Science and Engineering, Kanazawa University, Kakumamachi, Kanazawa, Ishikawa, 920-1192, JAPAN;
RIKEN Interdisciplinary Theoretical and Mathematical Sciences (\lowercase{i}THEMS), 
Wako, Saitama 351-0198, Japan.
}
\email{kannaka@se.kanazawa-u.ac.jp}

\address[Toshiyuki KOBAYASHI]{%
Graduate School of Mathematical Sciences,
The University of Tokyo, 3-8-1 Komaba, Meguro, Tokyo, 153-8914, Japan;
French-Japanese Laboratory in Mathematics and its Interactions,
FJ-LMI CNRS IRL2025, Tokyo, Japan
}
\email{toshi@ms.u-tokyo.ac.jp}

\date{\today}

\begin{document}

\begin{abstract}
Let $X=G/H$ be a homogeneous space of a Lie group $G$. When the isotropy subgroup $H$ is non-compact, a discrete subgroup $\Gamma$ may fail to act properly discontinuously on $X$.
In this article, we address the following question: in the setting where
$G$ and $H$ are reductive Lie groups
and $\Gamma \backslash X$ is a standard quotient, to what extent can one deform the discrete subgroup $\Gamma$ while preserving
the proper discontinuity of the action on $X$?

We provide several classification results, including conditions under which local rigidity holds for compact standard quotients $\Gamma\backslash X$, when a standard quotient can be deformed into a non-standard quotient, 
a characterization of the largest Zariski-closure of discontinuous groups under  small deformations, and conditions under which Zariski-dense deformations occur.

\end{abstract}

\maketitle

\tableofcontents

\section{Introduction}
Let $X$ be a homogeneous space $G/H$, where $G$ and $H$ are real reductive algebraic groups. 
In this article, we focus on the case where the isotropy subgroup $H$ is non-compact, and we consider the following two problems:

\begin{problem}
\label{mainquestion} 
To what extent can we deform cocompact discontinuous groups for the homogeneous space $X$?  
\end{problem}

\begin{problem}
\label{mainquestion'}
Does there exist a Zariski-dense subgroup $\Gamma$ 
with cohomological dimension greater than 2 such that $\Gamma$ acts properly discontinuously on $X$?
\end{problem}

In this context, the term ``discontinuous group for $X$'' is used distinctively from ``discrete subgroup of $G$''.
In fact, even when $\Gamma$ is a discrete subgroup of $G$, the quotient space $X_{\Gamma}:=\Gamma\backslash X$
is not always Hausdorff when $H$ is non-compact. 
We say that a discrete subgroup $\Gamma$ of a Lie group $G$ is a \emph{discontinuous group for} the homogeneous space $X$ if the $\Gamma$-action on $X$ is properly discontinuous and free. In this case, $X_{\Gamma}$ carries a unique $C^{\infty}$-manifold structure such that  
the quotient map $X\rightarrow X_{\Gamma}$ is a smooth covering of $X_\Gamma$, through which the quotient space $X_\Gamma$ inherits any local $G$-invariant geometric structure on $X$. The resulting quotient manifold $X_\Gamma$ is also referred to as a \textit{Clifford--Klein form} of $X$, which is a typical example of $(G,X)$-manifolds in the sense of Ehresmann and Thurston. See  e.g., \cite{Kobayashi-unlimit} for a detailed survey.  

We now briefly explain our motivations for the aforementioned problems,
the existing knowledge, and the contributions presented in this article.
The necessary notations and basic concepts will be reviewed in Section~\ref{section:deformation_compact_CK}.

The classical Selberg--Weil local rigidity theorem asserts that the compact quotient $X_\Gamma$ of the irreducible \emph{Riemannian} symmetric space $X=G/K$, where $K$ is a compact subgroup, cannot be continuously deformed unless $\dim X=2$ (\cite{Weil_discrete_subgroups}). 
In contrast, a notable feature when $H$ is non-compact, observed by Kobayashi in the early 1990s \cite{Kobayashi93},
is that cocompact discontinuous groups for \emph{pseudo-Riemannian} symmetric spaces exhibit
greater \lq\lq flexibility\rq\rq. 
Specifically, there exist arbitrarily high-dimensional compact quotients $X_\Gamma$ of irreducible symmetric spaces $X=G/H$ that do allow continuous deformation (\cite[Thms.~A and B]{Kobayashi98}).
On the other hand, it has also been proven that there exist compact quotients $X_\Gamma$ with local rigidity in pseudo-Riemannian symmetric spaces of the form $X=G/H$ (\cite[Prop.\ 1.8]{Kobayashi98}).

The difficulty regarding Problem~\ref{mainquestion} 
 is that when $H$ is not compact, small deformations of a discrete subgroup can easily destroy the proper discontinuity.
Goldman~\cite{Goldmannonstandard} conjectured, in the context of the 3-dimensional compact anti-de Sitter space, that any small deformation of any standard cocompact discontinuous group preserves the proper discontinuity.
This conjecture was proved by Kobayashi~\cite{Kobayashi98}, 
by extending the properness criterion~\cite{Kobayashi89} to the general setting, as shown in \cite{Benoist96, Kobayashi96}. 
He also demonstrated the existence of compact ``standard'' quotients that allow ``non-standard'' deformation when $X$ is locally isomorphic to the group manifold $SO(n,1)$ and $SU(n,1)$ for all $n \ge 2$, which are regarded as symmetric spaces of the form $(G\times G)/\diag G$.
Deformation in the context of a discontinuous group involves the study of deformation spaces of the corresponding geometric structures modelled on $(G, X)$. See also 
Ghys~\cite{Ghys95} for geometric interpretations in the case of $SL(2,\C)$. These results have been extended to the symmetric space $G/H=SO(2,2n)/U(1,n)$ by Kassel~\cite{Kassel12}, where she constructed a small deformation into a Zariski-dense subgroup (Definition~\ref{def:Zariski-dense} in Appendix~\ref{section:algebraic_group})
while preserving the proper discontinuity of the action on $G/H$.

The deformation of the discrete subgroup described above has the following properties: \begin{enumerate}[label=(\arabic*)]
    \item (standard quotient) it starts with a discrete subgroup whose Zariski-closure acts properly on $X=G/H$, or more precisely, it corresponds to a ``standard'' discontinuous group in the sense of \cite[Def.~1.4]{KasselKobayashi16}, see Definition~\ref{def:standard};
    \item (deformation as a discontinuous group) the deformation of the discrete subgroup preserves its proper discontinuity;
    \item (non-standard deformation) after the deformation, the Zariski-closure of the new discrete subgroup no longer acts properly on $X$.
\end{enumerate}

Keeping the assumption that a cocompact discontinuous group is standard, we examine Problem~\ref{mainquestion} in the general setting where $G/H$ is a homogeneous space of \emph{reductive type}, i.e., where $G \supset H$ are real reductive algebraic groups. We begin by providing a rigorous formulation in Question~\ref{question:deform-ck}, which includes conditions for when local rigidity holds, when non-standard deformations are possible, and when Zariski-dense deformations occur. 

We provide answers to these questions by dividing them into the following cases of compact standard quotients, denoted symbolically as $\Gamma \backslash G/H$ and $\Gamma_L\backslash G/\Gamma_H$, 
 which are modeled on the homogeneous space $G/H$ and the group manifold $(G\times G)/\diag G$,
 respectively.
In the first case, $\Gamma\backslash G/H$, we provide answers to Question~\ref{question:deform-ck} in Table~\ref{tab:cpt-CK-sym} for a classification, where $G$ is a simple Lie group 
(Theorem~\ref{theorem:local-nonstd-zariski}). In the second case, $\Gamma_{L}\backslash G/\Gamma_{H}$, the answer is given in Theorem~\ref{theorem:group-manifold-case}. 

We observe from our classification that,
in roughly half of the cases of $\Gamma\backslash G/H$, local rigidity holds, while in the case $\Gamma_{L}\backslash G/\Gamma_{H}$, more deformable cocompact discontinuous groups exist.
We also observe that in the case $\Gamma_{L}\backslash G/\Gamma_{H}$ with $G$ simple, 
there are no Zariski-dense deformations.
In contrast, in the $\Gamma\backslash G/H$ case with $G$ simple, we show that
when deformation to a non-standard form is possible, the same homogeneous space $G/H$ also admits a cocompact discontinuous group that can be deformed into a Zariski-dense subgroup.
Such homogeneous spaces $G/H$ include the indefinite-K\"ahler symmetric space $G/H=SO(2,2n)/U(1,n)$ (\cite{Kassel12})
and the following $7$-dimensional space form, with pseudo-Riemannian metric of signature $(4,3)$ and constant sectional curvature $-1$:

\begin{theorem}[Answer to (Q3) in Table~\ref{tab:cpt-CK-sym}]
\label{thm:7_dim_compact_space_form}
    There exists a standard cocompact discontinuous group
    for $SO(4,4)/SO(3,4)$ that
    admits a small deformation which is Zariski-dense in $SO(4,4)$. 
\end{theorem}

So far, we have outlined our results for the case of \emph{cocompact} discontinuous groups in Problem~\ref{mainquestion}. 
Now, in Problem~\ref{mainquestion'}, we investigate the possibility of Zariski-dense deformations. Keeping in mind that such deformations in $G$ are more likely to occur for discrete groups with lower cohomological dimension, we allow discontinuous groups that are not necessarily cocompact for $X$ in addressing Problem~\ref{mainquestion'}.

As one of numerical invariants for an abstract group $\Gamma$,
we recall the projective dimension of the group ring $\R[\Gamma]$ is called the \emph{cohomological dimension} of $\Gamma$, denoted as  $\operatorname{cd}_{\R}(\Gamma)$. Free groups and surface groups (the fundamental groups of closed hyperbolic surfaces) have cohomological dimensions of 1 and 2, respectively.
  
The general upper bound on the cohomological dimension of a subgroup $\Gamma$ acting properly discontinuously on a homogeneous space $X=G/H$ was established in \cite[Cor.~5.5]{Kobayashi89} as the inequality  
\begin{align}
\label{ineq:cd-upper-bound}
\operatorname{cd}_{\R}(\Gamma)\leq d(X),
\end{align}
where
we define the ``non-compact dimension''  of $X$ by
\begin{align}
\label{eqn:dX}
    d(X):= d(G)-d(H),
\end{align}
 and define $d(G)$ for the group $G$ to be the dimension of the Riemannian symmetric space $G/K$. 
The homogeneous space $X$ is diffeomorphic to a vector bundle over a compact manifold
with fiber $\R^{d(G)}$.
The equality in \eqref{ineq:cd-upper-bound} holds if and only if the $\Gamma$-action on $X$ is cocompact (\cite[Cor.~5.5]{Kobayashi89}). 

We have already discussed Problem~\ref{mainquestion'} in the case where $X_\Gamma$ is compact as a part of Problems~\ref{mainquestion}. This is the case where the equality $\operatorname{cd}_{\R}(\Gamma) = d(X)$ holds.

In contrast, for cases where the cohomological dimension is lower, some solutions to Problem~\ref{mainquestion'} are already known (see also Section~\ref{section:deform-noncptCK} for some open questions): 
\begin{itemize}
\item 
($\operatorname{cd}_{\R}(\Gamma)=1$)
when $\Gamma$ is isomorphic to a free group, see Benoist~\cite{Benoist96};
\item
($\operatorname{cd}_{\R}(\Gamma)=2$)
when $\Gamma$ is isomorphic to a surface group  and $G/H$ is a symmetric space, see a recent paper \cite{KannakaOkudaTojo24}. 
\end{itemize}

This article highlights Problem~\ref{mainquestion'} in the setting where
\[
2<\operatorname{cd}_{\R}(\Gamma)<d(X).
\]
by considering deformations of a discrete subgroup that is isomorphic to a cocompact discrete subgroup $\Gamma$ of $Spin(n,1)$, in particular, where $\operatorname{cd}_{\R}(\Gamma)=n$.
Our answers to Problem~\ref{mainquestion'} include the following examples, see Theorem~\ref{theorem:zariski-dense-noncpt-CK}:
there exist Zariski-dense discontinuous groups with cohomological dimension 6 in the cases where $d(X) = 7$, $8$, or $16$ such as
\begin{equation*}
X =
\left\{
\begin{array}{lll}
SO(8,\C)/SO(7,\C) & d(X)=7 & \textrm{(complex sphere)}, \\
SO(8,8)/SO(7,8) & d(X)=8 & \textrm{(space form)}, \\
SU(8,8)/U(7,8) & d(X)=16 & \textrm{(indefinite K\"ahler)}, \\
SL(16,\R)/SL(15,\R) & d(X)=16. &
\end{array}
\right.
\end{equation*}

\medskip

In proving the aforementioned results, particularly Theorems~\ref{theorem:local-nonstd-zariski}, \ref{theorem:group-manifold-case}, and \ref{theorem:zariski-dense-noncpt-CK}, we focus on a special case of the following problem regarding deformations of cocompact discrete subgroups of $Spin(n,1)$.

\begin{problem}
    \label{problem:hyperbolic-lattice-deform}
    Let $G$ be a real algebraic group, 
    $\Gamma$ a torsion-free cocompact discrete subgroup of $Spin(n,1)$,
    and $\varphi\colon \Gamma\rightarrow G$ a group homomorphism. 
    Find a small deformation $\varphi' \in \Hom(\Gamma,G)$ of $\varphi$ that maximizes the Zariski-closure of $\varphi'(\Gamma)$. 
\end{problem}

The reader may wonder why Problem~\ref{problem:hyperbolic-lattice-deform}
and Theorem~\ref{introtheorem:bending} below are formulated in terms of $Spin(n,1)$ rather than $SO(n,1)$, despite the technical difficulties involved. 
The main reasons for this choice are summarized in the following remark.
\begin{remark}
\label{rem:lifting_to_spin}
\begin{enumerate}[label=(\arabic*)]
    \item 
    \label{item:proper-action}
    (Proper actions). Some homogeneous spaces $X$, such as the 15-dimensional space form $SO(8,8)/SO(7,8)$ mentioned above, admit a proper action by reductive subgroups only if they are globally isomorphic to $Spin(n,1)$. In particular, $SO(n,1)$ cannot act properly on such space $X$, as stated in Theorem~\ref{theorem:Kannaka_Tojo_Spin_very-even}
    \item  (Direct implications). The deformation theory of discrete subgroups of $Spin(n,1)$ will imply similar results for quotient groups such as $SO(n,1)$, but not vice versa, as discussed below.
    \item 
    \label{item:lift-counterexample}
    (Counterexample to the lifting for $n \ge 4$). 
         We consider whether the following 
        claim $(P_{n})$ holds: 
        
        $(P_{n})$. \emph{Any torsion-free cocompact discrete subgroup of $SO_{0}(n,1)$ can be lifted to $Spin(n,1)$.} 

        \begin{itemize}
            \item $(P_{2})$ and $(P_{3})$ holds (Culler--Shalen~\cite{Culler-Shalen83} and Thurston~\cite{Thurston-3fold}).
            \item $(P_{n})$ fails when $n\geq 4$ (Martelli--Riolo--Slavich~\cite{nonspin-hyperbolic}). 
        \end{itemize}
\end{enumerate}
\end{remark}

Since a torsion-free cocompact discrete subgroup of $Spin(2,1)$ is a surface group, Problem~\ref{problem:hyperbolic-lattice-deform} for the case $n=2$ already presents non-trivial challenges, as it asks for the largest Zariski-closure of a surface subgroup in $G$ among its small defomations.
This problem has been studied in previous work, such as
 Burger--Iozzi--Wienhard~\cite{BurgerIozziWienhard10} and Kim--Pansu~\cite{KIMPAN15}, for example.

For general $n\geq 2$, and particularly when $n \ge 3$, 
we explore Problem~\ref{problem:hyperbolic-lattice-deform} in the special case 
where the deformation starts from a standard quotient $X_\Gamma$.
To be more explicit, we consider the setting
where
$\varphi\colon\Gamma\rightarrow G$ can be extended to a homomorphism $Spin(n,1) \rightarrow G$, that is, a homomorphism of real algebraic groups.
(There are several equivalent definitions  for \emph{homomorphisms} between real algebraic groups, as discussed in Lemma~\ref{lemma:morphism} in Appendix~\ref{section:algebraic_group}.)

In this case, we introduce a real algebraic subgroup of $G$, denoted by $G^{\varphi}$ (Definition~\ref{def:g-phi}), and obtain the following theorem.
\begin{theorem}[see Theorem~\ref{theorem:bending}]
    \label{introtheorem:bending}
Let $n\ge 2$, $G$ a Zariski-connected real algebraic group, and $\varphi\colon Spin(n,1) \to G$ a homomorphism.
    \begin{enumerate}[label=(\arabic*)]
        \item 
        \label{introitem:realization-gphi}
        There exists a cocompact arithmetic subgroup $\Gamma$ of $Spin(n,1)$ such that $\varphi(\Gamma)$
    can be deformed so that its Zariski-closure is  
    $G^{\varphi}$.
    \item 
    \label{introitem:upper-bound}
    Conversely, the group $G^{\varphi}$ gives the upper bound for the Zariski-closure of $\varphi'(\Gamma)$, up to $G$-conjugacy, for any small deformation $\varphi'$ of $\varphi|_\Gamma$, provided that $n \geq 3$.
    \end{enumerate}
\end{theorem}

We provide a description of the largest Zariski-closure $G^{\varphi}$ purely in terms of finite-dimensional representations of orthogonal Lie algebras.
We then carry out explicit computations
of $G^{\varphi}$ in certain cases, including those that arise naturally from
the spin representations associated with indefinite quadratic forms.
See Theorem~\ref{thm:Clifford-z.d} for a sufficient condition for $G^{\varphi}$ to coincide with $G$, for instance.

The geometric idea behind our proof of Theorem~\ref{introtheorem:bending}~\ref{introitem:realization-gphi} in Section~\ref{section:deform-spin-lattice} is largely parallel to that of Johnson--Millson~\cite{JoMi84}. This geometric approach is also used in Kassel~\cite{Kassel12} and Beyrer--Kassel~\cite{Beyrer-Kassel23}.

However, these results highlight the case where $\varphi$ is a natural embedding from $SO(n,1)$ into a classical Lie group $G$ such as $SO(p,q)$ or $PSL(n+1, \mathbb R)$ (see also Remark~\ref{remark:difference-Kassel}).
In contrast, Theorem~\ref{introtheorem:bending} addresses the situation where $\varphi\colon Spin(n,1) \to G$ is a homomorphism to an arbitrary real algebraic group $G$, thus working in full generality. 

This requires considerable preparation for the proof, not only in extending the bending construction \cite{JoMi84} for $SO(n,1)$ to $Spin(n,1)$, 
but also in providing a detailed framework, within the context of Clifford algebras, for iterating the appropriate bending constructions to reach the largest Zariski-closure.

\medskip
In Sections~\ref{section:deformation_compact_CK}-\ref{section:deformation-cpt-non-cpt}, we have discussed the geometric questions of the quotient space $X_{\Gamma}$, including the Zariski-dense deformation addressed in Theorems~\ref{theorem:local-nonstd-zariski} and  \ref{theorem:group-manifold-case}. 
 Now, turning our attention to the analytic aspects,
in Section~\ref{section:spectral-analysis}, we address analytic problems related to $L^2(X_\Gamma)$ for the quotient space $X_\Gamma$. This is a new area of research
\cite{KasselKobayashi16,KasselKobayashi2019Invariant,KasselKobayashi2020Spectral,KasselKobayashi2019standard,Kobayashi16intrinsic}, 
with a list of basic open problems and related topics discussed in a recent article \cite{Kobayashi22conjecture}.

 We observe that the quotient space
$X_{\Gamma}$ inherits any $G$-invariant differential operators from $X$, such as the pseudo-Riemannian Laplacian, which are referred to as \emph{intrinsic differential operators} in \cite{KasselKobayashi16}.
In contrast to the classical Riemannian setting, it is noteworthy that, when $H$ is non-compact, there may be discrete spectra for such intrinsic differential operators, some of which are stable under deformations of  $X_\Gamma$, while others are not.

In Section~\ref{section:spectral-analysis}, we highlight the analytic aspects of this study from the perspective of deformations of discontinuous groups for $X=G/H$ with $H$ non-compact. We briefly review the current state of the art regarding the stable discrete spectrum and its multiplicities, and pose several problems concerning the distribution of the stable discrete spectrum and its multiplicities.

\medskip

At the end of this article, for the convenience of a broader audience, several technical details and concepts that may not be immediately obvious to the reader are collected in Appendices 
\ref{section:algebraic_group}-\ref{section:appendix-Zariski-dense-subgroups},
on the following topics: 
\begin{itemize}
    \item 
    Appendix~\ref{section:algebraic_group}: 
    Basic concepts regarding real algebraic groups, particularly topologies;
    \item Appendix~\ref{section:clifford-spin}: Basic notations for Clifford algebras and spin groups over a field of characteristic $\neq 2$;
    \item Appendix~\ref{section:torsion-arithmetic}: A sufficient condition for arithmetic groups to be torsion-free;
    \item Appendix~\ref{section:hnn-extension}: Describing the fundamental group of a manifold endowed with a hypersurface in terms of an HNN extension;
    \item Appendix~\ref{section:deform-upper-bound}: An upper bound for small deformations of discrete subgroups in terms of the first cohomology;
    \item Appendix~\ref{section:appendix-Zariski-dense-subgroups}:
    An optimal bound on the number of generators for a continuous family of Zariski-dense subgroups.
\end{itemize}

\subsection*{Notation and Conventions}
\begin{itemize}
    \item $\N=\{0,1,2,\dots\}$, and $\N_{+}=\{1,2,\dots\}$.
    \item For a unital associative algebra $A$, 
    we denote by $A^{\times}$ the group of invertible elements in $A$.
    \item 
    We denote by $\id_{S}$ the identity map on a set $S$.
    \item 
    We denote by $\# I$
    the number of elements in a finite set $I$.
    \item 
    We assume that algebraic groups are linear.  
    \item 
    For an algebraic group $\mathbf{G}$ over a field $\mathbb{F}$, 
    and for an extension field $\mathbb{E}$ of $\mathbb{F}$,  
    $\mathbf{G}(\mathbb{E})$ denotes the group of $\mathbb{E}$-points of $\mathbf{G}$.  
    When $\mathbb{E} = \mathbb{R}$ or $\mathbb{C}$, we write $G = \mathbf{G}(\mathbb{R})$ and $G_{\mathbb{C}} = \mathbf{G}(\mathbb{C})$.
    \item 
    The Lie algebras of Lie groups $G,H,L,\ldots$ are denoted by 
    the corresponding lower-case German letters $\mathfrak{g},\mathfrak{h},\mathfrak{l},\ldots$, respectively.
    \item 
    For a Lie group homomorphism $\varphi\colon L\rightarrow G$, 
    we denote by $d\varphi\colon \mathfrak{l}\rightarrow \mathfrak{g}$
    the differential homomorphism of $\varphi$.
\end{itemize} 

\section{Deformation of locally homogeneous spaces}
\label{section:deformation_compact_CK}

In this section, we discuss Problem~\ref{mainquestion}, which addresses the extent to which we can deform cocompact discontinuous groups for homogeneous spaces.
To this end, we first review basic terminologies related to the deformation of discontinuous groups for homogeneous spaces
in Section~\ref{section:deformation_compact_CK-def}.
We then proceed to provide a rigorous formulation in Question~\ref{question:deform-ck} in Section~\ref{section:formulate-question}, which includes conditions for when local rigidity holds, when non-standard deformations are possible, and when Zariski-dense deformations occur. 

We will provide answers to this question by dividing it into the following cases of compact standard quotients, denoted symbolically as $\Gamma \backslash G/H$ and $\Gamma_L\backslash G/\Gamma_H$, where $G$ is a simple Lie group.

In the first case, $\Gamma\backslash G/H$, which is modeled on the homogeneous space $G/H$, we provide answers to Question~\ref{question:deform-ck} in Table~\ref{tab:cpt-CK-sym} for a classification 
(Theorem~\ref{theorem:local-nonstd-zariski}).
In the second case, $\Gamma_{L}\backslash G/\Gamma_{H}$, which is modeled on the group manifold $(G\times G)/\diag G$, 
the answer is given in Theorem~\ref{theorem:group-manifold-case}. 

\medskip 
We observe from our classification that,
in roughly half of the cases of $\Gamma\backslash G/H$, local rigidity holds, while in the case $\Gamma_{L}\backslash G/\Gamma_{H}$, more deformable cocompact discontinuous groups exist.
We also observe that in the case $\Gamma_{L}\backslash G/\Gamma_{H}$, there are no Zariski-dense deformations, whereas in the $\Gamma\backslash G/H$ case, some spaces of $G/H$ admit Zariski-dense deformations of cocompact discontinuous groups.

The proof of the main results of this section, Theorems~\ref{theorem:local-nonstd-zariski} and~\ref{theorem:group-manifold-case}, will be given in Sections~\ref{section:deform-cptCK} and~\ref{section:proof-group-manifold}, respectively.

\subsection{Deformation space \texorpdfstring{$\mathcal{R}(\Gamma,G;X)$}{R(Gamma,G;X)}}
\label{section:deformation_compact_CK-def}
As mentioned in Introduction, 
not every discrete subgroup $\Gamma$ of a Lie group $G$ acts properly discontinuously on the homogeneous space $X=G/H$ when $H$ is non-compact. 
Moreover, a small deformation of a discrete subgroup $\Gamma$ in $G$
may destroy the proper discontinuity
of the action on $X$.
In light of these considerations, 
we provide a precise definition and notation here to formalize Problems~\ref{mainquestion} 
and \ref{mainquestion'} 
in the next section
(see Question~\ref{question:deform-ck}).

\medskip
First, we set aside the space $X$, and consider the classical setting in which the $\Gamma$-action on $X=G/H$ is not relevant.

Let $\Gamma$ be a finitely-generated discrete subgroup of a Lie group $G$.
By deformation, we mean fixing an abstract group $\Gamma$ and varying the homomorphisms from $\Gamma$ into the group $G$. 
We set: 
\begin{align*}
    \mathcal{R}(\Gamma,G) := &\{\varphi \in \Hom(\Gamma,G) 
    \mid \varphi\text{ is faithful and discrete}\}
\end{align*}
We topologize $\Hom(\Gamma,G)$ by pointwise convergence and equip the subspace $\mathcal{R}(\Gamma,G)$ by the relative topology.

\begin{remark}[{Goldman--Millson~\cite[Thm.~1.1]{goldman-millson87}}]
When $G$ is a linear group, 
    $\mathcal{R}(\Gamma,G)$ is a closed subset of
    $\Hom(\Gamma,G)$
    for any torsion-free group $\Gamma$ that can be realized as a Zariski-dense subgroup of a real semisimple algebraic group.
\end{remark}

The group $G$ acts on itself through inner automorphisms, that is,
the subset $\mathcal{R}(\Gamma,G)$ invariant.  
We recall from Weil~\cite{Weil_remark} the following definition:
\begin{definition}
    A discrete and faithful representation 
    $\varphi\in \mathcal{R}(\Gamma,G)$ is \emph{locally rigid}
    if the $G$-orbit of $\varphi$
    is open in $\mathcal{R}(\Gamma,G)$.
\end{definition}

Next, we return to the main theme of this article: the deformation of \emph{discontinuous groups} for homogeneous spaces $X=G/H$.
To formulate this properly, it is important to note that when $H$ is non-compact, a discrete subgroup of $G$ does not necessarily act properly discontinuously on $G/H$.

We recall from \cite{Kobayashi93}  (see also \cite{Kobayashi98})
the definition of the following subset
of $\mathcal{R}(\Gamma,G)$:
\begin{align*}
    \mathcal{R}(\Gamma,G;X) := &\{\varphi \in \mathcal{R}(\Gamma,G) 
    \mid \text{$\varphi(\Gamma)$ is a discontinuous group for $G/H$}\},
\end{align*}
which plays a basic role in the deformation theory of discontinuous groups for $X$.

There are two natural actions: those of the automorphism group
$\Aut(\Gamma)$ and the inner automorphism group of $G$ on $\Hom(\Gamma,G)$.
These actions commute with each other and
both leave the subset $\mathcal{R}(\Gamma,G;X)$ invariant.

We recall further from \cite[Sect.~5.3]{Kobayashi-unlimit} the definitions of 
$\mathcal{M}(\Gamma,G;X)$ and $\mathcal{T}(\Gamma,G;X)$ as follows:
\begin{align*}
\mathcal{M}(\Gamma,G;X) := \Aut(\Gamma)\backslash \mathcal{R}(\Gamma,G;X)/G.
\end{align*}
In the case where $G$ acts faithfully on $X$, 
we can think of $\mathcal{M}(\Gamma,G;X)$ as
the \emph{moduli space} of
quotients of $X$ by discontinuous groups isomorphic to $\Gamma$. 
Geometrically, this means that, for $\varphi_{1},\varphi_{2}\in \mathcal{R}(\Gamma,G;X)$,  
the locally homogeneous spaces
$\varphi_{1}(\Gamma)\backslash X$ and $\varphi_{2}(\Gamma)\backslash X$ are isomorphic to each other 
if and only if $[\varphi_{1}] = [\varphi_{2}]$ in
$\mathcal{M}(\Gamma,G;X)$. 
The \lq\lq higher-dimensional Teichm\"uller space\rq\rq, which represents the local properties of the \lq\lq moduli space\rq\rq, is defined as follows. 
\[
\mathcal{T}(\Gamma,G;X) := \mathcal{R}(\Gamma,G;X)/G.
\]

\begin{definition}
\label{def:locally_rigid_for_X}
(local rigidity as discontinuous groups, \cite[Sect.\ 1]{Kobayashi98}).
We say that $\varphi \in \mathcal{R}(\Gamma, G; X)$ is \emph{locally rigid as a discontinuous group for $X$}, or that $\varphi(\Gamma) \backslash X$ is \emph{locally rigid}, if the $G$-orbit of $\varphi$ is open in $\mathcal{R}(\Gamma, G; X)$; equivalently, $[\varphi]$ is an isolated point in $\mathcal{T}(\Gamma, G; X)$.

If this is not the case, 
we say that $\varphi \in \mathcal{R}(\Gamma, G; X)$ is \emph{deformable as a discontinuous group for $X$}, or that $\varphi(\Gamma) \backslash X$ is \emph{deformable}.
\end{definition}

In the group case where $X=G/\{e\}$, we have $\mathcal{R}(\Gamma,G;X)=\mathcal{R}(\Gamma,G)$, and 
$\varphi\in \mathcal{R}(\Gamma,G)$ is locally rigid if and only if it is locally rigid
as a discontinuous group for $X=G/\{e\}$. 
We summarize in Section~\ref{section:(G,Gamma)-locally-rigid} some known results regarding local rigidity of discrete subgroups of $G$, without considering their actions on $X = G/H$.
In Section~\ref{section:preliminary-deform-discontinuous-group}, we discuss the difference between the concepts of local rigidity as \emph{discrete groups} and local rigidity as \emph{discontinuous groups}.

\subsection{Question about deformations of standard cocompact discontinuous groups}
\label{section:formulate-question}

In this section, we provide a rigorous formulation of Problem~\ref{mainquestion} concerning
deformations of standard cocompact \emph{discontinuous groups} for homogeneous spaces $G/H$ of reductive type. 
This is one of the central questions addressed in this article, and it includes conditions for when local rigidity holds (Q1), when non-standard deformations are possible (Q2), and when Zariski-dense deformations occur (Q3).

The basic setup involves standard quotients $X_\Gamma$, which we now recall.

\begin{definition}(\cite[Def.~1.4]{KasselKobayashi16}).
\label{def:standard}
Let $X=G/H$ be a homogeneous space of reductive type. 
A discontinuous group $\Gamma$ for $X$
is called \emph{standard} if there exists a closed reductive subgroup $L$ of $G$ which contains $\Gamma$ and acts properly on $X$. 
\end{definition}

We address the following question:

\begin{question}
\label{question:deform-ck}
Let $L \subset G \supset H$ be a triple of real reductive algebraic groups such that
\begin{equation}
\label{eq:LGH_cocompact_proper}
    L \ \text{acts properly and cocompactly on } X=G/H.
\end{equation}
Classify the triples $(G, H, L)$ for which $L$ admits a (torsion-free), cocompact discrete subgroup $\Gamma$ satisfying each of the following properties: (Q1), (Q2), or (Q3), where the conditions become progressively stronger, when $H$ is non-compact.
\begin{enumerate}[label=$(Q\arabic*).$]
    \item 
    $\Gamma$ is deformable, i.e., it is not locally rigid as a discontinuous group for $X$ 
    (Definition~\ref{def:locally_rigid_for_X});
    \item 
    $\Gamma$ can be deformed into a non-standard discontinuous group for $X$;
    \item
    $\Gamma$ can be deformed into a Zariski-dense discrete subgroup of $G$
(Definition~\ref{def:Zariski-dense} in Appendix~\ref{section:algebraic_group}), while preserving the proper discontinuity of the action on $X$.
\end{enumerate}
\end{question}

Clearly, the difference in the compact factors of the subgroups $H$ does not affect the answer to Question~\ref{question:deform-ck}.
We will provide the classification
in Theorem~\ref{theorem:local-nonstd-zariski} 
for Question~\ref{question:deform-ck},
based on the list of triples $(G, H, L)$ satisfying (\ref{eq:LGH_cocompact_proper}),
as presented in Tables~\ref{tab:kobayashi-yoshino} and~\ref{tab:kobayashi-yoshino-dual},
see Notation~\ref{notation:compact-factor}.

If a triple $(G, H, L)$ satisfies the condition (\ref{eq:LGH_cocompact_proper}), that is,
if $L$ acts properly and cocompactly on $X=G/H$,
then the double coset $\Gamma_H \backslash G/ \Gamma_L$
becomes a compact manifold, referred to as
an \emph{exotic} quotient in \cite{KasselKobayashi16}, for any discrete subgroups $\Gamma_H$ and $\Gamma_L$ that have no torsion of $H$ and $L$, respectively.
(The term \emph{exotic quotient} is also used to mean something different in other literature.)
This quotient manifold $\Gamma_H \backslash G/ \Gamma_L$ is a special case of the
  standard quotient of the group manifold $G$,
  considered as the symmetric space $(G\times G)/\diag G$.
We will examine the deformation of these
 (exotic) quotients  $\Gamma_H \backslash G/ \Gamma_L$
 and provide the classification results
in Theorem~\ref{theorem:group-manifold-case},
which answer Question~\ref{question:deform-ck},
based on the same list of triples $(G, H, L)$.

\subsection{List of triples \texorpdfstring{$(G,H,L)$}{(G,H,L)} 
with proper and cocompact actions}

In this section, we provide 
a list of the triples of reductive Lie groups $(G,H,L)$ 
for which Question~\ref{question:deform-ck} will be considered. 
The basic assumption is that the triple $(G, H, L)$ satisfies the condition (\ref{eq:LGH_cocompact_proper}), equivalently, one of the following equivalent conditions is satisfied:
\begin{enumerate}
    \item
      $L$ acts properly and cocompactly on $G/H$;
 \item
      $H$ acts properly and cocompactly on $G/L$;
   \item
      $H\times L$ acts properly and cocompactly on
      $(G \times G)/\diag G$.
\end{enumerate}
These conditions are obviously satisfied if $G$ is compact.
Furthermore, these conditions are unchanged under taking a finite covering or replacing the connected component of the groups.
The following list,
 Tables~\ref{tab:kobayashi-yoshino} and~\ref{tab:kobayashi-yoshino-dual},
 exhibits such triples when $G$ is a simple Lie group and when both $H$ and $L$ are non-compact.
 There are no new results in this section, 
 but we provide a few comments.

 A criterion for a triple $(G,H,L)$ of reductive Lie groups to ensure that $L$ acts properly on $X=G/H$ was established in \cite[Thm.~4.1]{Kobayashi89} in the late 1980s, along with a criterion for cocompactness in \cite[Thm.~4.7]{Kobayashi89}. 
 These criteria are computable. Table~\ref{tab:kobayashi-yoshino} is taken from  \cite{KobayashiYoshino05},  by Kobayashi and Yoshino, 
  and lists irreducible symmetric spaces $G/H$ that admit proper and cocompact actions by reductive subgroups $L$.
A recent work by Tojo~\cite{Tojo2019Classification}  asserts that this list is exhaustive,
 as mentioned in Remark~\ref{rem:tab_kobayashi_yoshino} below.

\begin{table}[!hbt]
\begin{tabular}{c|c|c|c}
Case & $G/H$ & $L_{ss}$ & $L_{max}$ \\
\hline
 1 & $SU(2n,2)/Sp(n,1)$ & $SU(2n,1)$
& $U(2n,1)$
\\
\hline
 2 & $SU(2n,2)/U(2n,1)$ & $Sp(n,1)$ & ---
\\
\hline
 3 & 
$SO(2n,2)/U(n,1)$ & $SO(2n,1)$ & --- \\
\hline
 4 &
$SO(2n,2)/SO(2n,1)$ & 
$SU(n,1)$  & $U(n,1)$ \\
\hline
 5 &
$SO(4n,4)/SO(4n,3)$ & 
$Sp(n,1)$ & $Sp(n,1)\times Sp(1)$\\
\hline
 6 &
$SO(8,8)/SO(8,7)$ & 
$Spin(8,1)$ & --- \\
\hline
 7 &
$SO(4,4)/(SO(4,1)\times SO(3))$ &
$Spin(4,3)$ & --- \\
\hline
 8 &
$SO(8,\mathbb{C})/SO(7,\mathbb{C})$ &
$Spin(7,1)$ & --- \\
\hline
 9 &
$SO(8,\mathbb{C})/SO(7,1)$ &
$Spin(7,\mathbb{C})$ & --- \\
\hline
 10 &
$SO^*(8)/U(3,1)$ &
$Spin(6,1)$ & --- \\
\hline
 11 &
$SO^*(8)/
(SO^*(6)\times SO^*(2)) $ & $Spin(6,1)$ & --- \\
\hline
 12 &
$SO(4,3)/(SO(4,1)\times SO(2))$ 
& $G_{2(2)}$ & ---  \\
\end{tabular}
\caption{Symmetric spaces $G/H$ with $G$ simple that admit 
a proper and cocompact action by a reductive subgroup $L$ of $G$
with $L_{ss} \subset L \subset L_{max}$.}
\label{tab:kobayashi-yoshino}
\end{table}

We have listed the triples $(G,H,L)$ satisfying the condition (\ref{eq:LGH_cocompact_proper}) in Table~\ref{tab:kobayashi-yoshino} where $G/H$ is a symmetric space and $G$ is simple.
Since the condition (\ref{eq:LGH_cocompact_proper}) is symmetric with respect to $H$ and $L$, additional triples can be obtained by switching $H$ and $L$.
The following table is derived in this way, with cases  where both $G/H$ and $G/L$ are symmetric spaces omitted.

\begin{table}[!htb]
\begin{tabular}{c|c|c|c}
Case & $G/H$ & $L_{ss}$ & $L_{max}$ \\
\hline
1' & $SU(2n,2)/SU(2n,1)$ & $Sp(n,1)$ & --- \\
\hline
4' & 
$SO(2n,2)/SU(n,1)$ & $SO(2n,1)$ & --- \\
\hline
5' &
$SO(4n,4)/Sp(n,1)$ & 
$SO(4n,3)$ & --- \\
\hline
6' &
$SO(8,8)/Spin(8,1)$ & 
$SO(8,7)$ & --- \\
\hline
7' &
$SO(4,4)/Spin(4,3)$ &
$SO(4,1)$ & $SO(4,1)\times SO(3)$ \\
\hline
8' &
$SO(8,\mathbb{C})/Spin(7,1)$ &
$SO(7,\mathbb{C})$ & --- \\
\hline
9' &
$SO(8,\mathbb{C})/Spin(7,\mathbb{C})$ &
$SO(7,1)$ & --- \\
\hline
10' &
$SO^*(8)/Spin(6,1)$ &
$SU(3,1)$ & $U(3,1)$ \\
\hline
11' &
$SO^*(8)/Spin(6,1)$ & $SO^*(6)$ & $SO^*(6)\times SO^{*}(2)$ \\
\hline
12' &
$SO(4,3)/G_{2(2)}$ 
& $SO(4,1)$  & $SO(4,1)\times SO(2)$
\end{tabular}
\caption{Non-symmetric spaces $G/H$ with $G$ simple that admit 
a proper and cocompact action by a reductive subgroup $L$ of $G$
with $L_{ss} \subset L \subset L_{max}$.}
\label{tab:kobayashi-yoshino-dual}
\end{table}

The numbering $i \ (1\leq i\leq 12)$ of  Case~$i$' in Table~\ref{tab:kobayashi-yoshino-dual}
corresponds to the numbering of Case~$i$ in
Table~\ref{tab:kobayashi-yoshino}, with $H$ and $L$ switched.

Regarding the subgroup $L$, there are several possible choices. 
While we will see eventually that the answers to Question~\ref{question:deform-ck}
remain unaffected by these choices, we introduce some notation for clarity.
\begin{notation}
\label{notation:compact-factor}
 \begin{itemize}
     \item Let    
     $L_{ss}$ be as given in Tables~\ref{tab:kobayashi-yoshino} or ~\ref{tab:kobayashi-yoshino-dual}. 
     \item 
     Let $L_{max}$ be the identity component
    of the normalizer of $L_{ss}$ in $G$
    in the sense of Definition~\ref{def:Zariski-dense}.
    
     \item
     For a subgroup $L$ satisfying
     $L_{ss}\subset L\subset L_{max}$,
     we denote by $L_{c}$ the compact factor of $L$.
 \end{itemize}
\end{notation}

\begin{remark}
 \begin{enumerate}[label=$(\arabic*)$]
      \item
        We have included $n=1$ in Cases~3 and 4 in Table~\ref{tab:kobayashi-yoshino} and in Case~4' in Table~\ref{tab:kobayashi-yoshino-dual}, although $G$ is not a simple Lie group.
\item 
    The table includes triples that are locally isomorphic in the case of low dimensions; however, we have not excluded overlaps in the table.
    See Remark~\ref{rem:isom_GH_low_dim} for such examples.
   \end{enumerate} 
\end{remark}

\begin{remark}
\label{rem:tab_kobayashi_yoshino}
Tables~\ref{tab:kobayashi-yoshino}
and \ref{tab:kobayashi-yoshino-dual}
complete the earlier list presented in
Kulkarni~\cite{Kulkarni1981proper} and Kobayashi~\cite{Kobayashi89, Kobayashi_kawaguchiko90,Kobayashi1997discontinuouspseudoRiemannian}. 
In particular, Cases~4 and 5 are from
 \cite{Kulkarni1981proper}, and Cases~1, 1', 2, 3, 4', and 5' are from \cite{Kobayashi89}.

It is plausible that Tables~\ref{tab:kobayashi-yoshino}
and \ref{tab:kobayashi-yoshino-dual} list all the homogeneous spaces $G/H$ with $G$ simple, $H$ non-compact and reductive, that admit a proper and cocompact action by a reductive subgroup $L$. 

Based on the (infinitesimal) classification of irreducible symmetric spaces  by Berger~\cite{Berger57},
Tojo~\cite{Tojo2019Classification} proved that irreducible symmetric spaces $G/H$ that admit a reductive subgroup $L$ satisfying the properness and cocompactness condition (\ref{eq:LGH_cocompact_proper}) are either listed in Table~\ref{tab:kobayashi-yoshino}, Riemannian symmetric spaces $G/K$, or group manifolds $(G\times G)/\diag G$.

A more recent preprint of  Boche\'{n}ski--Tralle~\cite{BochenskiTralle24} shows,
under the assumption that $G$ is absolutely simple, that 
Tables~\ref{tab:kobayashi-yoshino}
and \ref{tab:kobayashi-yoshino-dual} contain all the homogeneous spaces of the form $G/H$ with $H$ non-compact and reductive that admit a proper and cocompact action by a reductive subgroup $L$.
We note that our list also includes the cases, such as Cases~8 and 9, where $G$ is a complex simple Lie group.
\end{remark}

\subsection{Answers to \texorpdfstring{Question~\ref{question:deform-ck}}{Question 2.5}: Classification of the triples \texorpdfstring{$(G,H,L)$}{(G,H,L)} 
}
\label{section:answer}

In this section, we provide classification results,
as stated in Theorem~\ref{theorem:local-nonstd-zariski},
for (Q1), (Q2), and (Q3) of  Question~\ref{question:deform-ck}.
This is based on the list of triples $(G, H, L)$ that satisfy (\ref{eq:LGH_cocompact_proper}),
as presented in Tables~\ref{tab:kobayashi-yoshino} and~\ref{tab:kobayashi-yoshino-dual}.

\begin{theorem}
\label{theorem:local-nonstd-zariski}
Let $(G,H,L)$ be one of the triples in Tables~\ref{tab:kobayashi-yoshino} and~\ref{tab:kobayashi-yoshino-dual},
where $L_{ss} \subset L \subset L_{max}$. 
Then, the answer to (Q1), (Q2), and (Q3) in Question~\ref{question:deform-ck} 
is provided in Table~\ref{tab:cpt-CK-sym}. 
\end{theorem}

\begin{remark}
\label{remark:results_do_not_change_upto_compact_factors}
\begin{enumerate}[label=(\arabic*)]
    \item
Theorem~\ref{theorem:local-nonstd-zariski} includes the claim 
that the choice of $L$, where $L_{ss} \subset L \subset L_{max}$, does not affect
 the answer to any of (Q1), (Q2), or (Q3) in Question~\ref{question:deform-ck}.

\item
Theorem~\ref{theorem:local-nonstd-zariski} remain unchanged if we replace $(G, H, L)$ with locally isomorphic triples (Definition~\ref{def:locally_isomorphic_triple}), as can be seen from the proof and 
Remark~\ref{remark:local_isomorphic_trip}.
\end{enumerate}
\end{remark}

Theorem~\ref{theorem:local-nonstd-zariski} will be proved in Section~\ref{section:deform-cptCK}, using Theorem~\ref{theorem:bending}
from Section~\ref{section:(G,Gamma)-deformation}.

The triples $(G,H,L)$ listed in Table~\ref{tab:cpt-CK-sym} appear in either 
Table~\ref{tab:kobayashi-yoshino} 
or Table~\ref{tab:kobayashi-yoshino-dual}.
The interpretation of Theorem~\ref{theorem:local-nonstd-zariski} is illustrated using Case 1 as an example:

\begin{example}
In the first row of the table, Case~1 shows that for the symmetric space $G/H=SU(2n,2)/Sp(n,1)$ and $L=SU(2n,1)$, (Q1) is ``yes'', (Q2) is ``no'' (and thus (Q3) is also ``no''). This means that there exists a torsion-free cocompact discrete subgroup $\Gamma$ of $L$ such that $\Gamma$ is deformable as a discontinuous group for $G/H$
 (Q1). However, $\Gamma$ cannot be deformed into a non-standard discontinuous group (Q2), and, in particular, it cannot be deformed into a Zariski-dense subgroup while preserving the proper discontinuity of the action on $G/H$ (Q3). 
\end{example}
  
    Regarding the concept of local isomorphisms of triples, we adopt the definition from 
\cite{Kobayashi_2022_bounded_multiplicty_theorems}, which allows for individual applications of \emph{inner automorphisms} 
     at the infinitesimal level. This definition fits well with (Q1)--(Q3) of Question~\ref{question:deform-ck} and is given as follows.
     
\begin{definition}[{\cite[Def.\ 7.1]{Kobayashi_2022_bounded_multiplicty_theorems}}]
\label{def:locally_isomorphic_triple}
The triples $L_1 \subset G_1 \supset H_1$ and  $L_2 \subset G_2 \supset H_2$
    are said to be \emph{locally isomorphic} if there exists an isomorphism $\mathfrak{g}_1 \simeq \mathfrak{g}_2$ between the Lie algebras of the Lie groups $G_1$ and $G_2$, such that under this isomorphism,
    $\mathfrak{h}_1$ is conjugate to ${\mathfrak{h}}_2$ by an inner automorphism,
    and $\mathfrak{l}_1$ is conjugate to $\mathfrak{l}_2$ by another inner automorphism.

we say that the homogeneous spaces $G_1/H_1$ and  $G_2/H_2$
    are \emph{locally isomorphic} 
    if the triples $\{e \} \subset G_1 \supset H_1$ and  $\{ e\} \subset G_2 \supset H_2$
    are locally isomorphic. 
\end{definition}
\begin{remark}
\label{remark:local_isomorphic_trip}    
In Definition~\ref{def:locally_isomorphic_triple}, we do not require the existence of a morphism between $G_1$ and $G_2$.
However, if there is another triple $\tilde{L} \subset \tilde{G} \supset \tilde{H}$
with the following properties:
\begin{enumerate}
    \item[(1)] all the three triples of linear Lie groups are locally isomorphic;
    \item[(2)] there exists a morphism $\phi_j \colon \tilde{G} \to G_j$ for each $j=1,2$;
    \item[(3)] $\tilde{\Gamma}$ is a discrete subgroup of $\tilde{L}$, 
    and $\Gamma_j=\psi_j(\tilde{\Gamma})$,
\end{enumerate}
then we have the following easy consequences:
\begin{enumerate}
\item[(1)] the following three conditions are equivalent:
\begin{enumerate}
    \item[(i)]
      the $\Gamma_1$-action on $X_1=G_1/H_1$ is properly discontinuous (respectively, cocompact),
      \item[(ii)]
       the $\Gamma_2$-action on $X_2=G_2/H_2$ is properly discontinuous (respectively, cocompact),
       \item[(iii)]
        the $\tilde{\Gamma}$-action on $\tilde{X}=\tilde{G}/\tilde{H}$ is properly discontinuous (respectively, cocompact);
\end{enumerate}
    \item[(2)] if $\tilde{\Gamma}$ is deformable as a discontinuous group for $\tilde{X}$, then
        $\Gamma_j$ is deformable as a discontinuous group for $X_j$ ($j=1,2)$. 
\end{enumerate}

\end{remark}

\begin{table}[htb!]
\begin{tabular}{c|c|c|c|c|c}
Case & $G/H$ & $L_{ss}$ & Q1 & Q2 & Q3\\
\hline
1 & $SU(2n,2)/Sp(n,1)$ & $SU(2n,1)$
& yes & no & no
\\
\hline
1'-1 & $SU(2n,2)/SU(2n,1)$ $(n\geq 2)$ & $Sp(n,1)$
& no & no & no \\
1'-2 &
$SU(2,2)/SU(2,1)$ & $Spin(4,1)$
& yes & yes & yes \\
\hline
2-1 & $SU(2n,2)/U(2n,1)$ $(n\geq 2)$ & $Sp(n,1)$
& no & no & no \\
2-2 &
$SU(2,2)/U(2,1)$ & $Spin(4,1)$
& yes & yes & yes \\
\hline
3 & 
$SO(2n,2)/U(n,1)$ & $SO(2n,1)$
& yes & yes & yes \\
\hline
4-1 &
$SO(2n,2)/SO(2n,1)$ $(n\geq 2)$ & 
$SU(n,1)$
& yes & no & no \\
4-2 &
$SO(2,2)/SO(2,1)$ & 
$SU(1,1)$
& yes & yes & yes \\
\hline
4' & 
$SO(2n,2)/SU(n,1)$ & $SO(2n,1)$
& yes & yes & yes \\
\hline
5-1 &
$SO(4n,4)/SO(4n,3)$ $(n\geq 2)$ & 
$Sp(n,1)$
& no & no & no \\
5-2 &
$SO(4,4)/SO(4,3)$ & 
$Spin(4,1)$
& yes & yes & yes \\
\hline
5' &
$SO(4n,4)/Sp(n,1)$ & 
$SO(4n,3)$
& no & no & no \\
\hline
6 &
$SO(8,8)/SO(8,7)$ & 
$Spin(8,1)$
& no & no & no \\
\hline
6' &
$SO(8,8)/Spin(8,1)$ & 
$SO(8,7)$
& no & no & no \\
\hline
7 &
$SO(4,4)/(SO(4,1)\times SO(3))$ &
$Spin(4,3)$
& no & no & no \\
\hline
7' &
$SO(4,4)/Spin(4,3)$ &
$SO(4,1)$
& yes & yes & yes \\
\hline
8 &
$SO(8,\mathbb{C})/SO(7,\mathbb{C})$ &
$Spin(7,1)$ 
& no & no & no \\
\hline
8' &
$SO(8,\mathbb{C})/Spin(7,1)$ &
$SO(7,\mathbb{C})$ 
& no & no & no \\
\hline
9 &
$SO(8,\mathbb{C})/SO(7,1)$ &
$Spin(7,\mathbb{C})$ 
& no & no & no \\
\hline
9' &
$SO(8,\mathbb{C})/Spin(7,\mathbb{C})$ &
$SO(7,1)$ 
& no & no & no \\
\hline
10 &
$SO^*(8)/U(3,1)$ &
$Spin(6,1)$ 
& yes & yes & yes \\
\hline
10' &
$SO^*(8)/Spin(6,1)$ &
$SU(3,1)$ 
& yes & no & no \\
\hline
11 &
$SO^*(8)/
(SO^*(6)\times SO^*(2)) $ & $Spin(6,1)$
& yes & yes & yes \\
\hline
11' &
$SO^*(8)/Spin(6,1)$ & $SO^*(6)$
& no & no & no \\
\hline
12 &
$SO(4,3)/(SO(4,1)\times SO(2))$ 
& $G_{2(2)}$ 
& no & no & no \\
\hline
12' &
$SO(4,3)/G_{2(2)}$ 
& $SO(4,1)$ 
& yes & yes & yes 
\end{tabular}
\caption{Complete answers to Question~\ref{question:deform-ck} for the triples $(G,H,L)$ with non-compact semisimple factors $L_{ss}$ in Tables~\ref{tab:kobayashi-yoshino} and~\ref{tab:kobayashi-yoshino-dual}.}
\label{tab:cpt-CK-sym}
\end{table}

\begin{remark}
    For Cases~1',~2, and~5, there is an isomorphism of Lie groups $L$, 
   $Sp(1,1)\simeq Spin(4,1)$.
   As can be seen in the proof, it is this isomorphism that causes the differences in the answers to Question~\ref{question:deform-ck}
   (for each of Cases 1', 2, and 5).   
\end{remark}

As a direct consequence of Theorem~\ref{theorem:local-nonstd-zariski},
we obtain the following list of homogeneous spaces $G/H$ that have the property that (Q2) holds (and (Q3) holds as well).
Since this property is unaffected by replacing $H$ with a connected cocompact subgroup, the following classification result in (iii) is stated with this consideration in mind.

\begin{theorem}
\label{thm:zariskidense-up-to-local-isom}
Let $(G,H,L)$ be one of the triples in Tables~\ref{tab:kobayashi-yoshino} and~\ref{tab:kobayashi-yoshino-dual},
or let $H$ be replaced by a connected cocompact subgroup of $H$. 
Then the following are equivalent: 
\begin{enumerate}[label=(\roman*)]
   \item 
   \label{item:thm:zariskidense-up-to-local-isom-nonstd}
   ((Q2) in Question~\ref{question:deform-ck} for $G/H$)
    There exists a torsion-free, cocompact discrete subgroup $\Gamma$ of $L$ such that the quotient
    $\Gamma\backslash G/H$ can be deformed into a non-standard quotient.

    \item 
    \label{item:thm:zariskidense-up-to-local-isom-zariski-dense}
    ((Q3) in Question~\ref{question:deform-ck} for $G/H$)
    There exists a torsion-free, cocompact discrete subgroup $\Gamma$ of $L$ such that $\Gamma$ can be deformed into a Zariski-dense subgroup in $G$, keeping the proper discontinuity of the action on $G/H$.
    \item 
    \label{item:theorem:zariskidense-up-to-local-isom}
    (Classification).
    $G/H$ is one of the following homogeneous spaces, modulo compact factors of the subgroup $H$: 
    \begin{itemize}
    \item 
    $SU(2,2)/U(2,1)$;
    \item 
    $SO(2,2n)/U(1,n)$ $(n\geq 2)$;
    \item 
    $SO(4,4)/SO(4,3)$;
    \item 
    $SO^{*}(8)/U(3,1)$;
    \item 
    $SO^{*}(8)/(SO^*(6)\times SO^*(2))$;
    \item
    $SU(2,2)/SU(2,1)$; 
    \item 
    $SO(2,2n)/SU(1,n)$ $(n\geq 2)$;      
    \item 
    $SO(4,4)/Spin(4,3)$;
    \item 
    $SO(4,3)/G_{2(2)}$;
    \item 
    $SO^{*}(8)/SU(3,1)$;
    \item 
    $SO^{*}(8)/SO^*(6)$.
    \end{itemize}       
\end{enumerate}

\end{theorem}

Contrary to the equivalence (i) $\Leftrightarrow$ (ii)  in Theorem~\ref{thm:zariskidense-up-to-local-isom} for $\Gamma \backslash G/H$, 
such an equivalence does not hold for the exotic quotients $\Gamma_H\backslash G/\Gamma_L$, 
as stated in Theorem~\ref{theorem:group-manifold-case}~(2).

    \begin{remark}
    \label{rem:isom_GH_low_dim}    

Kassel proved in \cite{Kassel12} that the answer to (Q3) is affirmative when $G/H=SO(2,2n)/U(1,n)$.
To clarify the relationship between this space and the other homogeneous spaces in the list in 
Theorem~\ref{thm:zariskidense-up-to-local-isom}~~\ref{item:theorem:zariskidense-up-to-local-isom}, 
we introduce the following symbols: $\dashrightarrow$,  $\approx$, and $\xrightarrow{S^{1}}$.
\begin{enumerate}
    \item 
$G'/H' \approx G/H$: two homogeneous spaces 
$G'/H'$ and $G/H$ are locally isomorphic to each other;
\item
$G/H' \xrightarrow{S^{1}} G/H$: a $G$-equivariant $S^1$ fiber bundle when $H$ is a subgroup of $H'$ such that $H'/H$ is isomorphic to $S^1$;
\item
$G'/H' \dashrightarrow G/H$: an injective homomorphism $\psi\colon G' \to G$ induces the diffeomorphism $G'/H'$ onto $G/H$.
  \end{enumerate}

Obviously, the relation $\xrightarrow{S^{1}}$
does not affect the answer to (Q3).
A typical example of this relation is:
\begin{equation}
\label{eqn:S1_bundle}
SO(2,2n)/SU(1,n)     \xrightarrow{S^{1}} 
SO(2,2n)/U(1,n).
\end{equation}

Using the binary relation symbols $\dashrightarrow$, $\approx$, and $\xrightarrow{S^{1}}$,
the homogeneous spaces of lower dimensions in the list in 
Theorem~\ref{thm:zariskidense-up-to-local-isom}~~\ref{item:theorem:zariskidense-up-to-local-isom} can be grouped into two categories, which are connected to (\ref{eqn:S1_bundle}) as follows:
    
    \begin{enumerate}[label=(\arabic*)]
    \item \ ($n=2$ in (\ref{eqn:S1_bundle}))
\[
\begin{array}{ccc}
    SU(2,2)/SU(1,2) & \xrightarrow{S^{1}} & SU(2,2)/U(1,2)
    \\
    \rotatebox[origin=c]{90}{$\approx$} & & \rotatebox[origin=c]{90}{$\approx$}
    \\
    SO(2,4)/SU(1,2) & \xrightarrow{S^{1}} & SO(2,4)/U(1,2)
    \\
    \rotatebox[origin=c]{90}{$\dashleftarrow$} & & 
    \\
    SO(4,3)/G_{2(2)} &&
    \\
    \rotatebox[origin=c]{90}{$\dashleftarrow$}& & 
    \\
    SO(4,4)/Spin(4,3) &&
     \\
    \rotatebox[origin=c]{90}{$\approx$} & & 
    \\
    SO(4,4)/SO(4,3) &&
\end{array}
\]
\item \
 ($n=3$ in (\ref{eqn:S1_bundle}))
\[
\begin{array}{ccc}
    SO(2,6)/SU(1,3) & \xrightarrow{S^{1}} & SO(2,6)/U(1,3)
    \\
    \rotatebox[origin=c]{90}{$\approx$} & & \rotatebox[origin=c]{90}{$\approx$}
    \\
    SO^{*}(8)/SO^{*}(6) & \xrightarrow{S^{1}} & SO^{*}(8)/(SO^{*}(6)\times SO^{*}(2))
    \\
    \rotatebox[origin=c]{90}{$\approx$} & & \rotatebox[origin=c]{90}{$\approx$}
    \\
    SO^{*}(8)/SU(1,3) & \xrightarrow{S^{1}} & SO^{*}(8)/U(1,3)
    \\
\end{array}
\]
\end{enumerate}

To be precise, we recall that $G_{2(2)}=\Aut(\mathbb{O}')$ (where $\mathbb{O}'$ is the 8-dimensional nonassociative algebra of 
  split-octonions) has a maximal compact subgroup $(Sp(1)\times Sp(1))/\diag \{\pm 1\}$.

  \medskip
Second, We examine whether the  binary relation $\approx$  affects the answer to (Q3).

As a typical example of the relation $G'/H' \approx G/H$, we consider the case where $\psi \colon G' \to G$ is a covering map.
In this case, given a torsion-free cocompact discrete subgroup $\Gamma$ of $G$, it may happen that there does not exist a torsion-free discrete subgroup $\Gamma'$ of $G'$ such that $\Gamma = \psi(\Gamma')$, as discussed in Remark~\ref{rem:lifting_to_spin}~\ref{item:lift-counterexample} for $\psi \colon Spin(n,1) \to SO_0(n,1)$ with $n \geq 4$. 
Thus, the relation $\approx$ may, a priori, affect the answer to (Q3).
This is why we formulate the deformation theory in $Spin(n,1)$ in Section~
\ref{section:(G,Gamma)-deformation},
rather than $SO(n,1)$. From this formulation, we get the conclusion that our results remain unaffected by the binary relation $\approx$.

Third, regarding the relation $G'/H'\dashrightarrow G/H$, where the dimension of $G$ is larger than that of $G'$,
  the answers to (Q1)--(Q3) generally differ between the two (see Remark~\ref{rem:Question_is_non-trivial_for_diffeo} below). 
However, somewhat surprisingly, it does not affect the results concerning the series of homogeneous spaces discussed above.
\end{remark}

\begin{remark}
\label{rem:Question_is_non-trivial_for_diffeo}
Consider the relation $G'/H' \dashrightarrow G/H$ in Remark~\ref{rem:isom_GH_low_dim},
which is given by a group homomorphism
$\psi \colon G' \to G$ that induces a diffeomorphism 
\[ G'/H' \overset{\sim}{\rightarrow} G/H
\]
between the two homogeneous spaces $G'/H'$ and $G/H$.
For a discrete subgroup $\Gamma'$ of $G'$,
we set
\begin{equation}
\label{eqn:Gamma_image}
    \Gamma := \psi(\Gamma'), 
    \ L:=\psi(L')
\end{equation}
Regarding (Q1)--(Q3) in  Question~\ref{question:deform-ck},
we may compare the $\Gamma'$-action on $G'/H'$ with
the $\Gamma$-action on $G/H$.
Since the bijection
$G'/H' \overset{\sim}{\rightarrow} G/H$ is a homeomorphism, any topological properties like proper discontinuity or cocompactness are the same for $G'/H'$ and $G/H$,
and some of obvious relationships follow: for instance, the existence of a cocompact discontinuous group for $G'/H'$ implies that there exists a cocompact discontinuous group for $G/H$.
However, there are delicate differences regarding deformations as follows:
\begin{enumerate}
    \item[On (Q1)] Even if $\Gamma'$ is locally rigid as a discontinuous group for $G'/H'$, 
     it is possible that $\Gamma=\psi(\Gamma')$ is deformable as a discontinuous group for $G/H$.

   \item[On (Q2)] Even if $\Gamma'$ cannot be deformed into a non-standard discontinuous group for $G'/H'$, it is possible that $\Gamma=\psi(\Gamma')$ is deformed into a non-standard discontinuous group for $G/H$.
  
\item[On (Q3)]  Even if $\Gamma'$ is deformable into a Zariski-dense discrete subgroup of $G'$ preserving the proper discontinuity of the action on $G'/H'$, 
     $\Gamma$ is not necessarily deformable into a Zariski-dense discrete subgroup of $G$.
  
\end{enumerate}
\end{remark}

\subsection{Group manifold case}

In contrast to the Selberg--Weil rigidity theorem for the Riemannian symmetric space $G/K$, an irreducible pseudo-Riemannian symmetric space may admit a cocompact discontinuous group that is not locally rigid, even in higher dimensions. This was first observed in the early 90s (see \cite{Kobayashi93}) for the group manifold $G$, viewed as a homogeneous space $(G \times G)/\diag G$.
The classification for such simple Lie groups $G$ is obtained by Kobayashi as follows:
\begin{fact}
[{\cite[Thm.~A]{Kobayashi98}}]
\label{fact:Kobayshi98_deform_G}
   Let $G$ be a non-compact linear semisimple Lie group.
Then the following three conditions on $G$ are equivalent:
    \begin{enumerate}
        \item 
        
        the Lie algebra $\mathfrak{g}$ of $G$ is isomorphic to either $\mathfrak{so}(n,1)$ or $\mathfrak{su}(n,1)$; 
        \item (affirmative answer to (Q1) in Question~\ref{question:deform-ck})
        there exists a cocompact discrete subgroup $\Gamma$ of $G$ such that  
        $\Gamma\times \{1\}$ is not locally rigid as a discontinuous group for $(G\times G)/\diag G$; 
        \item (affirmative answer to (Q2) in Question~\ref{question:deform-ck})
        there exists a cocompact discrete subgroup $\Gamma$ of $G$ such that the compact manifold $ \Gamma\backslash G$ admits a non-standard deformation of the quotient of 
     $(G\times G)/\diag G \ (\simeq G)$.
    \end{enumerate}
    
\end{fact}

The implication (i) $\Rightarrow$ (iii) in Fact~\ref{fact:Kobayshi98_deform_G}, in the case where $\mathfrak{g} = \mathfrak{so}(2,1) \simeq \mathfrak{su}(1,1)$, is due to Goldman~\cite{Goldmannonstandard}. 
 
A natural extension of this problem is to strengthen the \lq\lq non-standard deformation\rq\rq\ in (iii) of Fact~\ref{fact:Kobayshi98_deform_G} and ask (Q3): Is it possible to deform it into a Zariski-dense discrete subgroup while preserving the proper discontinuity of the action on the group manifold $(G \times G)/\diag G$? The answer is affirmative for $\mathfrak{g}=\mathfrak{so}(n,1)$, as follows.

\begin{theorem}
\label{thm:SOn1-Zariski-dense}
For any Zariski-connected real algebraic group $G$ with Lie algebra $\mathfrak{so}(n,1)$,
there exist a torsion-free, cocompact discrete subgroup $\Gamma$ of G and a small deformation $\Gamma'$ of $\Gamma \times \{e\}$ such that $\Gamma'$ is Zariski-dense in the direct product group $G \times G$, while preserving the proper discontinuity of the action on the group manifold $(G \times G)/\diag G$. 
\end{theorem}
Theorem~\ref{thm:SOn1-Zariski-dense}
 is more or less well-known to experts, at least for $G=SO(n,1)$ or its quotient groups.
Nevertheless, we provide a proof for
$G=Spin(n,1)$ in Section~\ref{section:proof-son1-group-manifold1},
as an immediate application of Theorem~\ref{theorem:Zariski-dense-discontinuous-spin-lattice}.
\begin{remark}
In the case where $G$ is a linear Lie group locally isomorphic to $SU(n,1)$ with $n \ge 3$, the authors are unsure whether
 the implication (i) $\Rightarrow$ (iii) in Fact~\ref{fact:Kobayshi98_deform_G} can be strengthened, as in Theorem~\ref{thm:SOn1-Zariski-dense} for $\mathfrak g=\mathfrak{so}(n,1)$.
However, the answer is affirmative when $n=2$, as noted in Tholozan~\cite[Sect.~2.3.2]{Tholozan_Habilitation}.
\end{remark}

\begin{theorem}
    \label{theorem:group-manifold-case}
    Let $(G,H,L)$ be a triple of reductive groups in Table~\ref{tab:cpt-CK-sym}.
    \begin{enumerate}[label=$(\arabic*)$]
    \item 
    \label{item:theorem:group-manifold-case:local-rigidity}
    The following three conditions on the triple $(G, H, L)$ are equivalent: 
    \begin{enumerate}[label=$(\roman*)$]
        \item 
        \label{item:G/H-deformable}
           ((Q1) in Question~\ref{question:deform-ck} for $G/H$).
     Up to switching $H$ and $L$ if necessary,
      there exists a torsion-free, cocompact discrete subgroup $\Gamma$ of $L$ such that $\Gamma$ is deformable as a discontinuous group for $G/H$.
        \item 
        \label{item:group-case-deformable}
     ((Q1) in Question~\ref{question:deform-ck} for $(G\times G)/\diag G$). 
     There exist torsion-free, cocompact discrete subgroups
        $\Gamma_{H}$ and $\Gamma_{L}$ of $H$ and $L$, respectively, such that the discrete subgroup
        $\Gamma_{H}\times \Gamma_{L}$ is deformable as a discontinuous group for the group manifold $(G\times G)/\diag G$
     \item (Classification)
     \label{item:group-deformable}
     Up to switching $H$ and $L$, and up to compact factors, 
    the triple $(G,H,L)$ does not belong to one of the following lists: 
    \begin{itemize}
        \item (Case~5, $n\geq 2$) $(SO(4n,4), SO(4n,3), Sp(n,1))$;
        \item (Case~6) $(SO(8,8), SO(8,7), Spin(8,1))$;
        \item (Case~8) $(SO(8,\C), SO(7,\C), Spin(7,1))$;
        \item (Case~9) $(SO(8,\C), SO(7,1), Spin(7,\C))$.
    \end{itemize}
    \end{enumerate}
    \item
    \label{item:theorem:group-manifold-case:equiv}
The following three conditions on the triple $(G, H, L)$ are equivalent: 
    \begin{enumerate}[label=$(\roman*)$]
        \item 
        \label{item:G/H-nonstd}
           ((Q2) in Question~\ref{question:deform-ck} for $G/H$).
     Up to switching $H$ and $L$ if necessary,
      there exists a torsion-free, cocompact discrete subgroup $\Gamma$ of $L$ such that $\Gamma$ can be deformed into a non-standard, cocompact, disconitnuous group for $G/H$.
        \item 
        \label{item:group-case-nonstd}
         ((Q2) in Question~\ref{question:deform-ck} for $(G\times G)/\diag G$). 
        There exist torsion-free, cocompact discrete subgroups
        $\Gamma_{H}$ and $\Gamma_{L}$ of $H$ and $L$, respectively, such that the discrete subgroup
        $\Gamma_{H}\times \Gamma_{L}$ can be deformed into a non-standard, cocompact, discontinuous group for $(G\times G)/\diag G$.
      
        \item 
        \label{item:group-case-examples}
        (Classification).
        Up to switching $H$ and $L$ and up to compact factors, the triple $(G,H,L)$ belongs to one of the following lists: 
        \begin{itemize}
        \item (Cases~1 and 2, $n=1$) $(SU(2,2),Sp(1,1), SU(2,1))$;
        \item (Cases~3 and 4) $(SO(2n,2),SO(2n,1), SU(n,1))$;
        \item (Case~5, $n=1$)  $(SO(4,4),SO(4,3),Spin(4,1))$;
        \item (Case~7)  $(SO(4,4),SO(4,1), Spin(4,3))$;
        \item (Case~10) $(SO^{*}(8),SU(3,1),Spin(6,1))$;
        \item (Case~11) $(SO^{*}(8),SO^*(6),Spin(6,1))$;
        \item (Case~12) $(SO(4,3),SO(4,1), G_{2(2)})$.
        \end{itemize}
    \end{enumerate}
    \item 
    \label{item:theorem:group-manifold-case:non-zariski}
    ((Q3) in Question~\ref{question:deform-ck} for $(G\times G)/\diag G$). 
    Assume $G$ is simple. Then 
  there does not exist a cocompact, direct product discrete subgroup $\Gamma_{H} \times\Gamma_{L}$ in $H \times L$, which can be deformed into a Zariski-dense subgroup of $G\times G$ while keeping the proper discontinuity of the action on $(G\times G)/\diag G$ for any triple $(G,H,L)$ in Table~\ref{tab:cpt-CK-sym}.
    \end{enumerate}
\end{theorem}

As stated in Remark~\ref{remark:results_do_not_change_upto_compact_factors} for Theorem~\ref{theorem:local-nonstd-zariski}, Theorem~\ref{theorem:group-manifold-case} also includes the claim that the choice of $H$ and $L$, where $H_{ss} \subset H \subset H_{max}$ and $L_{ss} \subset L \subset L_{max}$, respectively,
does not affect the answer to any of (Q1), (Q2), or (Q3) in Question~\ref{question:deform-ck}, and that the answer remains unchanged if we replace $(G, H, L)$ with locally isomorphic triples (Definition~\ref{def:locally_isomorphic_triple}).

Theorem~\ref{theorem:group-manifold-case}  will be proved in Section~\ref{section:proof-group-manifold}.

\medskip
In the last statement \ref{item:theorem:group-manifold-case:non-zariski} of Theorem~\ref{theorem:group-manifold-case}, we assumed that $G$ is simple. An analogous statement fails when $G$ is not simple, as observed in the following proposition, where $G = SO(2,2)$ in Cases 3 and 4 with $n = 1$.

\begin{proposition}
\label{prop:so(2,2)-exotic}
The six-dimensional group manifold $G = SO(2,2)$ has an exotic compact quotient of the form $\Gamma_1 \backslash G / \Gamma_2$, where both $\Gamma_1$ and $\Gamma_2$ are torsion-free, Zariski-dense discrete subgroups of $G$.
\end{proposition}

 This proposition will be proven also in 
    Section~\ref{section:proof-group-manifold}.

\section{Deformation of discrete subgroups}
\label{section:(G,Gamma)-deformation}
Let $\Gamma$ be a discrete subgroup of $G$.
In this section, in order to study the deformation of $\Gamma$ as a \emph{discontinuous group} for a $G$-manifold $X$, 
we will momentarily set aside the space $X$ and focus on the deformation of $\Gamma$ within $G$.
To put it succinctly, the focus will be solely on the pair $(\Gamma, G)$ rather than the triplet $(\Gamma,G,X)$, which simplifies the argument.

In Section~\ref{section:(G,Gamma)-locally-rigid},
we recall results related to the local rigidity of representations of discrete groups, such as Raghunathan's vanishing theorem and a theorem of Klingler~\cite{Klingler11}.
 
Owing to these local rigidity theorems, to examine the extent to which the standard quotient can be deformed (Problems~\ref{mainquestion} and~\ref{mainquestion'}), it suffices to focus on the case where
$\Gamma$ is a cocompact discrete subgroup of $Spin(n,1)$.
This will be the topic of the latter part of this section.

In Section~\ref{section:(G,Gamma)-deform-spin}, we will consider a general setup where
$G$ is a Zariski-connected real algebraic group and $\varphi\colon Spin(n,1) \to G$ is a homomorphism, and
introduce a real algebraic subgroup $G^{\varphi}$ (Definition~\ref{def:g-phi}). 
We prove in Theorem~\ref{theorem:bending}~\ref{item:maximal-zariski-closure},
for any cocompact discrete subgroup $\Gamma$ of $Spin(n,1)$ and 
any small deformation $\varphi'$ of $\varphi|_{\Gamma}$ in $G$,
that $\varphi'(\Gamma)$ is contained in $G^{\varphi}$ up to $G$-conjugacy.
This upper bound is optimal: we shall prove in Theorem~\ref{theorem:bending}~\ref{item:bending-realize-Gphi} that  
there exist $\Gamma$ and $\varphi'$
such that the Zariski-closure of $\varphi'(\Gamma)$ coincides with $G^{\varphi}$. 

In Section~\ref{section:(G,Gamma)-bending-application},
we will discuss in more detail the special cases where $n=2$ or $\varphi$ is a spin representation.
In the case where $n=2$, we explain how Theorem~\ref{theorem:bending}  
connects the Zariski-closure of deformations of specific surface group representations
to the combinatorics of nilpotent orbits in complex reductive Lie algebras.
In the case where $\varphi$ is a spin representation, we compute $G^{\varphi}$ in terms of the Clifford algebra associated with the indefinite quadratic form of signature $(p,q)$ and provide sufficient conditions on the triple $(p,q,n)$ for the classical group $G(p,q)$ to admit 
a Zariski-dense subgroup as a deformation of a cocompact discrete subgroup of $Spin(n,1)$ (see Theorem~\ref{thm:Clifford-z.d}).
For notation related to $G(p,q)$ and Clifford algebras, see Appendix~\ref{section:clifford-spin}.

\subsection{Local rigidity theorems after Weil, Matsushima--Murakami, Raghunathan, Goldman--Millson, and Klingler}
\label{section:(G,Gamma)-locally-rigid}
A finitely-generated discrete subgroup $\Gamma$ of $G$ is locally rigid if it is \emph{infinitesimally rigid},
that is, the first cohomology
$H^{1}(\Gamma,\mathfrak{g})$ vanishes (Weil~\cite{Weil_remark}).

We begin with the classical vanishing theorems for the first cohomology, notably those by Matsushima--Murakami~\cite{Matsushima-Murakami-63}
and Raghunathan:
\begin{fact}[{Raghunathan~\cite{Raghunathan65}}] \label{fact:raghunathan} Let $L$ be a non-compact, simple Lie group, $\pi$ be an irreducible finite-dimensional representation of $L$ on a real vector space $V$, and $\Gamma$ be a cocompact discrete subgroup without torsion.
If $H^{1}(\Gamma,V) \neq 0$, then one of the following holds: 
\begin{enumerate}[label=(\arabic*)] 
\item
The Lie algebra $\mathfrak{l}$ is isomorphic to $\mathfrak{so}(n,1)$ with $n \ge 3$, and $V$ contains a non-zero $\mathfrak{so}(n-1,1)$-invariant vector. 
\item 
The Lie algebra $\mathfrak{l}$ is isomorphic to $\mathfrak{su}(n,1)$ with $n \ge 1$, and $(\pi, V)$ is isomorphic to the symmetric tensor of the standard representation or its dual. \end{enumerate} \end{fact}

\begin{remark} 
\label{remark:spherical_harmonics}
\begin{enumerate}[label=(\arabic*)]
    \item For an irreducible representation $\pi$ of $\mathfrak{so}(n,1)$ on a finite-dimensional real vector space $V$, the following two conditions are equivalent:
    \begin{enumerate}
        \item[(i)]  $V$ contains a non-zero $\mathfrak{so}(n-1,1)$-invariant vector;
        \item[(ii)] $V$ contains a non-zero $\mathfrak{so}(n)$-invariant vector.
    \end{enumerate}
    For $n\geq 2$, the complexification of such a representation $\pi$
    can be realized in
    the space of spherical harmonics of degree $k$, defined by 
    \[
    \{f\in C^{\infty}(S^{n})\mid 
    \Delta_{S^{n}}f=-k(k+n-1)f\},
    \]
    where $\Delta_{S^{n}}$ is 
    the Laplacian of the unit sphere $S^{n}$.
    We refer to $\pi$ as the \emph{spherical harmonics representation} of degree $k$.
    
    \item 
    There is an isomorphism 
$\mathfrak{so}(2,1) \simeq \mathfrak{su}(1,1)$ of Lie algebras. For this Lie algebra, we refer to the second statement for $\mathfrak{su}(1,1)$.
\end{enumerate}
\end{remark}

Thanks to Fact~\ref{fact:raghunathan},
 a cocompact discrete subgroup of $L$ can potentially be deformed within a larger Lie group $G$ only if it is locally isomorphic to $SO(n,1)$ or $SU(n,1)$.

Infinitesimal rigidity is a necessary condition for local rigidity, it is not sufficient.
For complex hyperbolic lattices, Goldman--Millson~\cite{goldman-millson88} 
discovered a striking example where local rigidity holds and infinitesimal rigidity fails. This example has been further generalized by a theorem of Klingler (\cite[Thm.~1.3.7]{Klingler11}), whose special cases we recall here:

\begin{fact}[{Klingler~\cite[Thm.~1.3.8]{Klingler11}}]
    \label{fact:klingler}
Let $\Gamma$ be a cocompact discrete subgroup of $SU(n,1)$, and
$\varphi \colon SU(n,1)\hookrightarrow SU(p,q)$ $(p\geq n,q\geq 1)$
be the standard embedding,
Then any morphism $\varphi'$ suffiently close to $\varphi|_{\Gamma}$ 
is conjugate to a representation of the form $\varphi \cdot \chi$,
where $\chi\colon \Gamma \to SU(p,q)$ is a deformation
of the trivial representation in the centralizer
of $SU(n,1)$ in $SU(p,q)$.
\end{fact}

\begin{remark}
\label{remark:klingler}
The above results hold for locally isomorphic groups.
That is, let $\phi\colon L \to G$ be a Lie group homomorphism, and $\Gamma$ be a cocompact discrete subgroup of $L$.
Suppose the differentials of $\phi$ and $\varphi$ (from the theorem above) agree. In this case, any homomorphism sufficiently close to $\phi|_{\Gamma}$ maps into $\phi(L) \cdot Z_G(L)$ up to conjugation by $G$, where $Z_G(L)$ denotes the centralizer of $\phi(L)$ in $G$.

In fact, the proof of Fact~\ref{fact:klingler} is based on the obstruction of the integrability of infinitesimal deformation,
which can be traced back to Goldman--Millson~\cite{goldman-millson88}:
the conclusion of Fact~\ref{fact:klingler} holds if 
\begin{equation}
\label{eqn:cup_product_of_first_cohom}    
\text{for any } c\in H^{1}(\Gamma,\mathfrak{g})\smallsetminus H^{1}(\Gamma,\mathfrak{z}_{\mathfrak{g}}(\mathfrak{l})),\ [c,c]\neq 0
\text{ in }H^{2}(\Gamma,\mathfrak{g}),
\end{equation}
and this statement holds under the same assumption at the Lie algebra level.
\end{remark}

We will use Fact~\ref{fact:klingler} along with Remark~\ref{remark:klingler} in the proof (Step 4) of Theorem~\ref{theorem:local-nonstd-zariski} in Section~\ref{section:deform-cptCK}.

In the sequel, we discuss cocompact discrete subgroup in a group which is locally isomorphic to $SO(n,1)$.

\subsection{Deformation of cocompact discrete subgroups of 
\texorpdfstring{$Spin(n,1)$}{Spin(n,1)}
}
\label{section:(G,Gamma)-deform-spin}

Let $\mathbf{G}$ be a Zariski-connected real algebraic group, 
$G$ the Lie group of real points $\mathbf{G}(\R)$,  
and $\varphi\colon Spin(n,1)\rightarrow G$ a homomorphism $(n\geq 2)$ in the sense of 
Lemma~\ref{lemma:morphism}.
 For basic notions of algebraic groups, we refer to Appendix~\ref{section:algebraic_group}.
In this section, we discuss to what extent the cocompact discrete subgroup $\Gamma$ of $Spin(n,1)$ can be deformed in $G$ through $\varphi$.

\begin{definition}[Lie algebra $\mathfrak{g}^{\varphi}$]
    \label{definition:lie-g-phi}
    We regard the Lie algebra $\mathfrak{g}$ of $G$ as a 
    $\mathfrak{spin}(n,1)$-module via $\varphi$. 
Let $\mathfrak{g}^{\varphi}$ denote the Lie subalgebra of $\mathfrak{g}$, generated by 
    $d\varphi(\mathfrak{spin}(n,1))$ and all  the 
    $\mathfrak{spin}(n,1)$-submodules that are isomorphic to irreducible
    spherical harmonics modules of $\mathfrak{spin}(n,1)$ (see Remark~\ref{remark:spherical_harmonics}).      
In other words, $\mathfrak{g}^{\varphi}$ is the smallest Lie subalgebra of $\mathfrak g$
that contains
 \begin{equation}
 \label{eq:g-phi}
 d\varphi(\mathfrak{spin}(n,1)) + 
 \mathfrak{z}_{\mathfrak g}({d\varphi(\mathfrak{spin}(n-1,1))}),
 \end{equation}
 where
 $\mathfrak{z}_{\mathfrak g}({d\varphi(\mathfrak{spin}(n-1,1))})$ denotes the centralizer of 
$d\varphi(\mathfrak{spin}(n-1,1))$ in $\mathfrak{g}$.
\end{definition}

We define a Zariski-connected real algebraic group $\mathbf{G}^{\varphi}$ corresponding to the Lie algebra $\mathfrak{g}^{\varphi}$ (we refer to Definition~\ref{def:Zariski-connected} in Appendix~\ref{section:algebraic_group}
for the definition of Zariski-connectedness of real algebraic groups).  
For this purpose, we use the following lemma, the proof of which will be given in Section~\ref{section:proof:g-phi-alebraic}.
\begin{lemma}
    \label{lemma:g-phi-algebraic}
    Let $G_{\C}$ be the complex Lie group of complex points $\mathbf{G}(\C)$,
    and $G^{\varphi}_{\C}$ the analytic subgroup of $G_{\C}$
    corresponding to $\mathfrak{g}^{\varphi}\otimes_{\R}\C$.
    Then $G^{\varphi}_{\C}$ is a Zariski-closed subset defined over $\R$. 
\end{lemma}

Since $G^{\varphi}_{\C}$ is connected in the usual topology, 
it is also connected in the Zariski-topology.

\begin{definition}[Real algebraic group $G^{\varphi}$]
    \label{def:g-phi}
    Let $\mathbf{G}^{\varphi}$ be the Zariski-connected real algebraic group whose set of $\mathbb{C}$-points is the Zariski-closed subset $G^{\varphi}_{\mathbb{C}}$ given in Lemma~\ref{lemma:g-phi-algebraic}.  
    Furthermore, define $G^{\varphi} := \mathbf{G}^{\varphi}(\mathbb{R})$.
\end{definition}

\begin{remark}
   If $G$ is reductive, then $G^\varphi$ is also reductive. 
\end{remark}

Here is the main result of this section.
 
\begin{theorem}
\label{theorem:bending}
Let $G$ be a Zariski-connected real algebraic group,  
and $\varphi\colon Spin(n,1)\rightarrow G$ a homomorphism, where $n\geq 2$.
\begin{enumerate}[label=(\arabic*)]
    \item 
    \label{item:bending-realize-Gphi}
    There exist a torsion-free cocompact discrete subgroup $\Gamma$ of $Spin(n,1)$
and a small deformation $\varphi'$ of 
$\varphi|_{\Gamma}$ in $\Hom(\Gamma,G)$ such that the Zariski-closure of 
$\varphi'(\Gamma)$ coincides with $G^{\varphi}$, 
where 
$G^{\varphi}$ is the real algebraic group, as given in Definition~\ref{def:g-phi}.

\item 
\label{item:discrete-and-faithful}
If $G$ is reductive and if $\varphi\colon Spin(n,1)\rightarrow G$ is non-trivial,
then we can choose $\varphi'\in \Hom(\Gamma,G)$ as in \ref{item:bending-realize-Gphi} such that
$\varphi'$ is discrete and faithful.

\item 
\label{item:maximal-zariski-closure}
Assume $n\geq 3$.
For any torsion-free cocompact discrete subgroup $\Gamma$ of $Spin(n,1)$, 
there exists a neighborhood $U$ of $\varphi|_{\Gamma}$ 
in $\Hom(\Gamma,G)$ 
such that  $\varphi'(\Gamma) \subset G^{\varphi}$ up to $G$-conjugacy 
for any $\varphi'\in U$. 
\end{enumerate}
\end{theorem}

\begin{remark}
Analogous results to Theorem~\ref{theorem:bending} hold as well
if we replace $Spin(n,1)$ with $SO(n,1)$.
The results are likewise new even in that context, and can be directly deduced
from Theorem~\ref{theorem:bending}
via the double covering map
$Spin(n,1)\rightarrow SO_{0}(n,1)$. 
\end{remark}

In general,  we have
\[
d\varphi(\mathfrak{spin}(n,1))+
\mathfrak{z}_{\mathfrak g}({d\varphi(\mathfrak{spin}(n,1))})\subset
\mathfrak{g}^\varphi\subset \mathfrak{g}.
\]
As special cases of Theorem~\ref{theorem:bending}, corresponding to the minimal case $\mathfrak{g}^\varphi = d\varphi\left(\mathfrak{spin}(n,1) + \mathfrak{z}_{\mathfrak{g}}(d\varphi(\mathfrak{spin}(n,1)))\right)$ and 
the maximal case $\mathfrak{g}^\varphi = \mathfrak{g}$, we obtain Corollary~\ref{cor:deform_to_centralizer} and Corollary~\ref{cor:deform_Zariski_dense},
respectively.

We begin with the particular case where $\mathfrak{g}^\varphi$ is minimal.
\begin{corollary}
    \label{cor:deform_to_centralizer}
Let $G$ be a Zariski-connected real algebraic group and let $n \ge 3$. 
 Suppose that $\varphi\colon Spin(n,1)\rightarrow G$ is a non-trivial homomorphism. Then the following two conditions are equivalent:
\begin{enumerate}[label=(\roman*)]
\item $\mathfrak{z}_{\mathfrak g}({d\varphi(\mathfrak{spin}(n,1))})=
\mathfrak{z}_{\mathfrak g}({d\varphi(\mathfrak{spin}(n-1,1))})$.
\item
For any cocompact discrete subgroup $\Gamma$ of $Spin(n,1)$, any
morphism $\varphi' \colon \Gamma \to G$ sufficiently close to $\varphi|_{\Gamma}$ 
is conjugate to a representation of the form $\varphi \cdot \chi$,
where $\chi\colon \Gamma \to G$ is a deformation
of the trivial representation with image in the centralizer of $\varphi(Spin(n,1))$ in $G$.
\end{enumerate}
\end{corollary}
\begin{remark}
Suppose that $\rho\colon L \to G$ is a homomorphism.
In the case where $L=SU(n,1)$,    Klingler~\cite[Thm.1.3.7]{Klingler11} established a sufficient condition under which every torsion-free cocompact discrete subgroup $\Gamma \subset L$ admits no nontrivial deformations beyond the centralizer of $L$. In contrast, Corollary~\ref{cor:deform_to_centralizer} considers the analogous problem for $L=Spin(n,1)$, and provides a necessary and sufficient condition formulated in terms of finite-dimensional representations of Lie algebras.
\end{remark}
\begin{proof}[Proof of Corollary~\ref{cor:deform_to_centralizer}]
    
By Theorem~\ref{theorem:bending}~(3), it suffices to show that the condition (i) is equivalent to $\mathfrak{g}^\varphi=d\varphi(\mathfrak{spin}(n,1)) + \mathfrak{z}_{\mathfrak{g}}(d\varphi(\mathfrak{spin}(n,1)))$,
but this is clear from the alternative definition of $\mathfrak{g}^\varphi$, 
as stated in Remark~
\ref{definition:lie-g-phi}.
\end{proof}

By contrast, the following is an immediate  consequence of Theorem~\ref{theorem:bending} ~\ref{item:bending-realize-Gphi} in the particular case where $\mathfrak{g}^\varphi$ is maximal.
\begin{corollary}
\label{cor:deform_Zariski_dense}
    Let $G$ be a Zariski-connected real algebraic group. 
    If $\mathfrak{g}^{\varphi}=\mathfrak{g}$ holds for 
    a homomorphism $\varphi\colon Spin(n,1)\rightarrow G$, then $\varphi(\Gamma)$
    can be deformed into a Zariski-dense subgroup. 
    The converse also
    holds, provided that $n\geq 3$. 
\end{corollary}

We end this section with the proof of Theorem~\ref{theorem:bending}~\ref{item:discrete-and-faithful} and~\ref{item:maximal-zariski-closure}. 
The proof of~\ref{item:bending-realize-Gphi} is differed to Section~\ref{section:deform-spin-lattice}.
\begin{proof}[Proof of Theorem~\ref{theorem:bending}~\ref{item:discrete-and-faithful}]
If $\varphi\colon Spin(n,1)\rightarrow G$ is non-trivial, then 
$\varphi|_{\Gamma}$ is discrete.
Since $\Gamma$ is torsion-free, $\varphi|_{\Gamma}$ is faithful.
Therefore, the claim that
a small deformation of $\varphi|_{\Gamma}$ preserves faithfulness and discreteness
is a special case of the stability of discontinuous groups,
as will be explained in Section~\ref{section:preliminary-deform-discontinuous-group}.
(We can apply Fact~\ref{fact:stability-for-properness} by Kassel to $X=G/\{e\}$.) 
\end{proof}

\begin{remark}
 The proof of the stability of the discontinuous group for homogeneous spaces $X=G/H$ of reductive groups $G$, using a quantitative estimate, may be traced back to Kobayashi~\cite[Thm.~2.6]{Kobayashi98}.
 In the above proof of \ref{item:discrete-and-faithful} with $H=\{e\}$,
 one can apply Guichard~\cite[Thm.~2]{Guichard04} if $G$ is semisimple. 
\end{remark}

\begin{proof}[Proof of Theorem~\ref{theorem:bending}
~\ref{item:maximal-zariski-closure}]
We apply a general result concerning an upper bound for the local deformation of discrete subgroups, which will be presented in Proposition~\ref{prop:local-rigid-g/l} in Appendix~\ref{section:deform-upper-bound}.

Let $n \geq 3$, and we consider the $\mathfrak{spin}(n,1)$-module $\mathfrak{g}/\mathfrak{g}^{\varphi}$.  
By the definition of $\mathfrak{g}^{\varphi}$, this module does not contain any spherical harmonics representation as an irreducible component.  
Furthermore, since $n \geq 3$, the Lie group $Spin(n,1)$ is not locally isomorphic to $SU(m,1)$ for any $m \in \mathbb{N}_+$.  
Then, it follows from Raghunathan's vanishing theorem (Fact~\ref{fact:raghunathan}) that $H^{1}(\Gamma, \mathfrak{g}/\mathfrak{g}^{\varphi}) = 0$.
By applying Proposition~\ref{prop:local-rigid-g/l} to $L = G^{\varphi}$, we obtain the desired conclusion.
\end{proof}

\begin{remark}
\label{remark:Kim-Pansu}
The assumption $n\ge 3$ in Theorem~\ref{theorem:bending}~\ref{item:maximal-zariski-closure} is necessary and cannot be dropped for this upper bound result.
In fact, an analogous statement fails when $n=2$.
For instance, let $G = SL(3, \mathbb{R})$, and let $L=SL(2,\mathbb{R})$ be the standard subgroup.
Since $Spin(2,1)$ is isomorphic to $SL(2,\mathbb{R})$, we may consider the embedding  
    \[
    \varphi \colon Spin(2,1)\simeq SL(2,\R) \hookrightarrow SL(3, \mathbb{R}).
    \]
    Applying Kim--Pansu~\cite[Thm.~1]{KIMPAN15} to this setting,  
    one sees that a discrete surface subgroup of $SL(2, \mathbb{R})$ 
    of genus $\geq 128$ can be deformed into a Zariski-dense subgroup of $SL(3, \mathbb{R})$. 
    On the other hand, $G^{\varphi}$ is given by 
    \[
    G^{\varphi} = S(GL(2, \mathbb{R}) \times GL(1, \mathbb{R})),
    \]
    which is a proper subgroup of $G$.
\end{remark}

\subsection{Examples of \texorpdfstring{Theorem~\ref{theorem:bending}}{Theorem 3.8}}
\label{section:(G,Gamma)-bending-application}
In this section, we illustrate Theorem~\ref{theorem:bending}
with some simple examples.

Theorem~\ref{theorem:bending}, applied to the deformation of discontinuous groups for $X=G/H$, will be discussed in Theorem~\ref{theorem:Zariski-dense-discontinuous-spin-lattice}, 
which is used in the proof of the main theorems in Section~\ref{section:deformation_compact_CK} for the compact standard quotients $X_\Gamma$ as well as in the deformation theory of non-compact standard quotients $X_\Gamma$ in Section~\ref{section:deform-noncptCK}.

We begin with the case $n=2$. Theorem~\ref{theorem:bending} in the $n=2$ case and Corollary~\ref{cor:KOT} were recently proved in \cite{KannakaOkudaTojo24}, where it is noted that $Spin(2,1)$ is isomorphic to $SL(2,\mathbb{R})$.

Let $G$ be a Zariski-connected real reductive algebraic group and 
$\varphi\colon SL(2,\R)\rightarrow G$ a homomorphism. 
We define
\[
\sigma^{\varphi} := 
\varphi(\begin{pmatrix}
    -1 & 0 \\
    0 & -1
\end{pmatrix})
=
\exp(\pi\sqrt{-1}d\varphi(\begin{pmatrix}
    1 & 0 \\
    0 & -1
\end{pmatrix}))\in G.
    \]
Then it follows from \cite[Lem.~3.3]{KannakaOkudaTojo24} that  
$\mathfrak{g}^{\varphi}$ is the centralizer of $\sigma^{\varphi}$ in $\mathfrak{g}$ and that 
$G^{\varphi}$ is the identity component
of the centralizer of $\sigma^{\varphi}$ in $G$ in the Zariski-topology.

The following equivalence makes the condition $G=G^\varphi$ transparent:
\begin{deflem}
    \label{definition-lemma:even}
    Let $G$ be a Zariski-connected real reductive algebraic group and 
    $\varphi\colon SL(2,\R)\rightarrow G$ a homomorphism.
    The following four conditions on a homomorphism $\varphi\colon SL(2,\mathbb{R}) \rightarrow G$ are equivalent:
    \begin{enumerate}[label=(\roman*)]
            \item
        \label{item:grp-g-phi}
        $G^{\varphi}=G$; 
        
        \item 
        \label{item:alg-g-phi}    $\mathfrak{g}^{\varphi}=\mathfrak{g}$; 
        \item 
        \label{item:nilpotent-orbit-even}
        The complex nilpotent orbit in $\mathfrak{g}_{\C}:=\mathfrak{g}\otimes_{\R}\C$
        given by
        \begin{equation*}
        \Int(\mathfrak{g}_{\C})\cdot d\varphi(\begin{pmatrix}
        0 & 1 \\
        0 & 0
        \end{pmatrix})       
        \end{equation*}
         is even in the sense 
         that the weights in the weighted Dynkin diagram associated to this nilpotent orbit are even, see
        \cite[Chap.~3.8]{CoMc93}, for example;
        \item 
        \label{item:sl2-even}
        We regard $\mathfrak{g}$ as an $SL(2,\mathbb{R})$-module via $\operatorname{Ad} \circ \varphi$.  
        Then, every irreducible component of $\mathfrak{g}$ 
        has an odd dimension, or equivalently, its highest weight is even.
    \end{enumerate}
    We say that the homomorphism $\varphi$ is \emph{even} if it satisfies one of these equivalent conditions. 
\end{deflem}

\begin{proof}
    \ref{item:grp-g-phi}$\Leftrightarrow$\ref{item:alg-g-phi}: The implication \ref{item:grp-g-phi} $\Rightarrow$ \ref{item:alg-g-phi} is clear. 
    Conversely, assume $\mathfrak{g}^{\varphi}=\mathfrak{g}$.  
    The complex algebraic group $G_{\C}$ is Zariski-connected and,  
    by Lemma~\ref{lem:Zariski-connected}, also connected in the usual topology.  Now, recall from Definition~\ref{def:g-phi} that the analytic subgroup of $G_{\mathbb{C}}$ corresponding to $\mathfrak{g}^{\varphi}\otimes_{\R} \C$ is $G^{\varphi}_{\mathbb{C}}$.  
    Thus, we conclude that $G_{\mathbb{C}} = G^{\varphi}_{\mathbb{C}}$.  
    By taking real points of these two algebraic groups, it follows that $G = G^{\varphi}$.

    \ref{item:alg-g-phi}$\Leftrightarrow$\ref{item:sl2-even}: 
    For any irreducible representation $(\tau,V)$ of $SL(2,\mathbb{R})$,  
    the element  
    \[
\tau(\begin{pmatrix} -1 & 0 \\ 0 & -1 \end{pmatrix})
    \]  
    acts trivially on $V$ if and only if $V$ has odd dimension.  
    The equivalence \ref{item:alg-g-phi} $\Leftrightarrow$ \ref{item:sl2-even} follows immediately from this fact.

    \ref{item:nilpotent-orbit-even}$\Leftrightarrow$\ref{item:sl2-even}: This is an elementary representation theory of $\mathfrak{sl}(2, \R)$, see \cite[Lem.~3.8.7]{CoMc93}, for instance. 
\end{proof}
As a corollary to Theorem~\ref{theorem:bending}~\ref{item:bending-realize-Gphi}
applied to the $n=2$ case, 
we obtain the following: 

\begin{corollary}[{\cite[Cor.~1.7]{KannakaOkudaTojo24}}]
\label{cor:KOT}
Let $G$ be a Zariski-connected real reductive algebraic group and 
$\varphi\colon SL(2,\R)\rightarrow G$ an even homomorphism.
Then there exists a discrete surface subgroup $\Gamma$ of $SL(2,\R)$ such that
$\varphi(\Gamma)$ can be deformed into a Zariski-dense subgroup of $G$.
\end{corollary}

\medskip

Next, we consider the case where $\varphi$ is the spin representation of $Spin(n,1)$ for general $n$. 
As an example of Theorem~\ref{theorem:bending}~\ref{item:bending-realize-Gphi}, we present classical groups that contain a Zariski-dense discrete subgroup isomorphic to a cocompact discrete subgroup of $Spin(n,1)$.

To be more precise, 
let $C(p,q)$ be the Clifford algebra 
associated with a real quadratic form of signature $(p,q)$, and let $G$ be
the group $G(p,q)$, introduced in \cite[Def.~4.3.1]{KobayashiYoshino05}.
We refer to Appendix~\ref{section:clifford-spin} for the notation of $G(p,q)$, and to Proposition~\ref{proposition:Kobayashi-Yoshino-G(p,q)} for the identifications of $G(p,q)$ with classical groups that exhibit certain periodicities of $p+q$ and $p-q$.

When $p\geq n$ and $q\geq 1$, there exists a natural injective homomorphism $\varphi\colon Spin(n,1)\rightarrow G(p,q)$. 
    \begin{theorem}
    \label{thm:Clifford-z.d}
    Let $p\geq n \ge 2$ and $q\geq 1$. We set $m=p+q$.  
    Assume that one of the following conditions holds:
 \begin{enumerate}[label=$( \arabic* )$] \item \label{item:Clifford-z.d-general} When $m \equiv i \mod 4$ (where $i = 1, 2, 3, 4$), it holds that $m - n \geq i + 2$.

\item \label{item:Clifford-z.d-exceptional} $m < 10$ and $n \leq 6$. \end{enumerate} 
Then, there exist a torsion-free, cocompact discrete subgroup $\Gamma$ of $Spin(n,1)$ and a small deformation $\varphi' \in \mathcal{R}(\Gamma, G(p,q))$ of the map $\varphi|_{\Gamma} \colon \Gamma \to G(p,q)$ such that the Zariski-closure of $\varphi'(\Gamma)$ coincides with the identity component of the group $G(p,q)$ in the Zariski-topology.
    \end{theorem}

Owing to Theorem~\ref{theorem:bending}~\ref{item:bending-realize-Gphi},  
    the proof of Theorem~\ref{thm:Clifford-z.d} reduces to the following proposition on the finite-dimensional representation theory of Lie algebras.
    \begin{proposition}
    \label{prop:Clifford-z.d}
    In the setting of Theorem~\ref{thm:Clifford-z.d}, 
    we have $\mathfrak{g}^{\varphi} = \mathfrak{g}$.
    \end{proposition}

    Before entering the proof, let us provide some examples of Theorem~\ref{thm:Clifford-z.d} in the special case $n=2$.
    \begin{example}
    \label{example:Clifford-z.d}
    Let us consider the following inclusion maps: 
    \begin{enumerate}[label=$(\arabic*)$]
        \item \label{item:example:Clifford-z.d:O(4,4)}
        $Spin(4,1)\rightarrow G(4,3)\simeq O(4,4)$;
        \item \label{item:example:Clifford-z.d:O^{*}(8)}
        $Spin(6,1)\rightarrow G(6,1) \simeq O^{*}(8)$;
        \item 
        $Spin(6,1)\rightarrow G(7,1) \simeq O(8,\C)$;
        \item 
        $Spin(6,1)\rightarrow G(7,2) \simeq O^{*}(16)$;
        \item 
        $Spin(6,1)\rightarrow G(7,3) \simeq GL(8,\HA)$;
        \item 
        $Spin(6,1)\rightarrow G(8,1) \simeq O(8,8)$;
        \item 
        $Spin(6,1)\rightarrow G(8,2) \simeq U(8,8)$;
        \item 
        $Spin(6,1)\rightarrow G(8,3) \simeq Sp(8,8)$;
        \item 
        $Spin(6,1)\rightarrow G(9,1) \simeq GL(16,\R)$;
        \item 
        $Spin(6,1)\rightarrow G(10,1) \simeq Sp(16,\R)$;
        \item 
        $Spin(6,1)\rightarrow G(11,1) \simeq Sp(16,\C)$.
    \end{enumerate}
    Then, we have $\mathfrak{g}^{\varphi}=\mathfrak{g}$. 
    In particular, there
    exists a cocompact discrete subgroup $\Gamma$ of $Spin(6,1)$ that can be deformed into a Zariski-dense subgroup  
        of the identity component of these classical groups in the Zariski-topology.
    \end{example}

    \begin{proof}[Proof of Example~\ref{example:Clifford-z.d}]
        The fact that each $ G(p,q) $ is isomorphic to the classical groups described above follows from Proposition~\ref{proposition:Kobayashi-Yoshino-G(p,q)}. 
    \end{proof}

    To prove Proposition~\ref{prop:Clifford-z.d}, 
    let us introduce some notation and lemmas. 
    Let $C(p,q)$ denote the Clifford algebra associated with the standard
    quadratic form $x_1^2+\dotsb +x_p^2-x_{p+1}^2-\dotsb -x_{p+q}^2$.
Let $\{e_{1},\ldots, e_{p+q}\}$ be
    the standard basis of $\R^{p+q}$.
We define the basis  of $C(p,q)$ by $\{e_I\}$, where
    \[
    e_I := e_{i_1} \cdots e_{i_k} \in C(p,q),
    \]
    for $I=\{i_1,\ldots,i_k\}$ with $1\leq i_1<\dotsb<i_k\leq p+q$.
In the following lemma, we regard  
    $C(p,q)$ 
    as a Lie algebra over $\R$ with the Lie bracket defined by $[x,y] = xy - yx$.

    \begin{lemma}
        \label{lemma:sym-diff}
        Given two subsets $I,J\subset \{1,\ldots,p+q\}$, where $\# I$ and $\# J$ are even, 
        we have $[e_{I},e_{J}]\neq 0$ if and only if $\#(I \cap J)$ is odd. Furthermore, $[e_{I},e_{J}]$ is a scalar multiple of $e_{I\triangle J}$, where 
        $I\triangle J=(I\smallsetminus J) \cup (J\smallsetminus I)$ denotes the symmetric difference.
    \end{lemma}

    \begin{proof}
    For each $j\in J$, 
    \[
    e_{j}e_{I} = e_{I}e_{j} \times 
    \begin{cases}
        -1 & (j\in I), \\
        1 & (j\not\in I).
    \end{cases}
    \]
    Hence, we have 
    \begin{equation*}
        e_{J}e_{I}= (-1)^{\#(I\cap J)} \ e_{I}e_{J}.
    \end{equation*}
    Thus, the assertions follow immediately.
    \end{proof}

    For each $k=0,\ldots,m=p+q$,
    we write $\wedge^{k}$ for the 
    $\R$-vector subspace of $C(p,q)$ spanned by the elements $e_{I}$ for all subsets $I\subset \{1,\ldots,p+q\}$ with $\#I=k$.
    
With this notation, the module structure of the Lie algebras
$\mathfrak{g}(p,q)$ and $\mathfrak{spin}(p,q)$ is given as follows
\begin{equation}
\label{eqn:g(p,q)-structure}
\mathfrak{g}(p,q) = \bigoplus_{\substack{0\leq k\leq m \\
k\equiv 2\bmod 4}} \wedge^{k}
\quad \text{and} \quad \mathfrak{spin}(p,q) = \wedge^{2},
\end{equation}
see Lemma~\ref{lemma:g(p,q)-structure} in Appendix~\ref{section:clifford-spin}.
Furthermore,  each $k$, the subspace $\wedge^k$ is an $\mathfrak{spin}(p,q)$-submodule of $\mathfrak{g}(p,q)$ and is isomorphic to the $k$-th exterior tensor representation $\wedge^k V$, where $V$ is the standard representation of $\mathfrak{spin}(p,q)$.

    \begin{lemma}
        \label{lemma:g(p,q)-generator}
        Let $\ell$ be an even number.
        If $\min(\ell,m-\ell)\geq 3$, then 
        the Lie algebra  $\mathfrak{g}(p,q)$ is generated 
        by $\mathfrak{spin}(p,q)$ and 
        $\wedge^{\ell}$. 
    \end{lemma}

    \begin{proof}
    The key step in the proof is to show the following claim:
    $\wedge^{6}\subset [\wedge^{\ell},\wedge^{\ell}]$ if $\min(\ell,m-\ell)\geq 3$.
To prove this claim, let $I'$ be an arbitrary subset of $\{1,\ldots, m\}$
    with $6$ elements. Since $\min(\ell,m-\ell)\geq 3$ by assumption, we can find subsets $I,J \subset \{1,\ldots, m\}$, each containing $\ell$ elements such that $I\triangle J = I'$. 
Since $\ell$ is even,  $\#(I\cap J)$ is odd, and consequently, by Lemma~\ref{lemma:sym-diff},
    we have $e_{I'}\in [\wedge^{\ell},\wedge^{\ell}]$.
    This completes the proof of the claim. 

A similar argument shows that for an even number $k$, we have
\begin{align*}
&\wedge^{k+4}\subset [\wedge^{6},\wedge^{k}]
 &\textrm{ if } \ & k\leq m-4,
\\
&\wedge^{k-4}\subset [\wedge^{6},\wedge^{k}] &\textrm{ if } \ & k\geq 4.
\end{align*}
By (\ref{eqn:g(p,q)-structure}), 
    our assertion follows inductively.
    \end{proof}

  Let $V = \R^{p,q}$ and $W = \R^{n,1}$ be the standard representations of $\mathfrak{spin}(p,q)$ and $\mathfrak{spin}(n,1)$, respectively.  
        Then we have the following branching law of 
        the $\mathfrak{spin}(p,q)$-module $V$ 
        when restricted to $\mathfrak{spin}(n,1)$ via $\varphi$:
        \[
        V \simeq W \oplus \mathbf{1}^{\oplus (m-(n+1))},
        \]
        where $\mathbf{1}$ denotes the trivial representation and $m=p+q$.
\begin{lemma}
\label{lemma:wedge_l_contains_spherical_harmonics}
Suppose $m \ge n+1 \ge 3$.
\begin{enumerate}[label=(\arabic*)]
\item
\label{item:wedge^lcontained-in-spin}
 If $m-n \ge \ell$ or if $\ell \ge n$, then
    $\wedge^\ell$ contains a non-zero spherical harmonics representation
    of  $\mathfrak{spin}(n,1)$. 
\item 
\label{item:g(p,q)-contain-spin}
$ \mathfrak{g}^{\varphi} \supset \mathfrak{spin}(p,q) $.
\end{enumerate}
\end{lemma}
\begin{proof}
\ref{item:wedge^lcontained-in-spin}:
The assertion is clear if $\ell = 0$ or $m$.
Suppose that $1 \le \ell \le m-1$.
We have the following decomposition as $\mathfrak{spin}(n,1)$-modules:
\begin{equation}
\label{eqn:wedge_ell_V_decomposition}
        \wedge^{\ell}   
        \simeq
        \bigoplus_{k=0}^{\ell} 
        \wedge^k W \otimes \wedge^{\ell -k} ({\C}^{m-(n+1)}),
\end{equation}
where $\mathfrak{spin}(n,1)$ acts on $\wedge^k W$ as the $k$-th exterior power
of the standard representation $W$, and trivially on $\wedge^{\ell -k} ({\C}^{m-(n+1)})$.
Therefore, $\wedge^{\ell}$ contains a non-zero spherical harmonics if $W \otimes \wedge^{\ell -1} ({\C}^{m-(n+1)}) \neq 0$, which occurs if $m-n \ge \ell$.

Furthermore, we have $\wedge^k V \simeq \wedge^{m-k} V$ for all $0 \le k \le m$ as $\mathfrak{spin}(p,q)$-modules when $m \ge 3$. Hence, the same conclusion holds if $m-n \ge m-\ell$, i.e., if $\ell \ge n$. Thus, the first statement is proved.
\newline
\ref{item:g(p,q)-contain-spin}: The $\ell = 2$ case in (\ref{eqn:wedge_ell_V_decomposition}) gives
an isomorphism of $\mathfrak{spin}(n,1)$-modules:
 \begin{equation}
\label{eqn:wedge_2_V_decomposition}
        \wedge^2 V \simeq \wedge^2 W \oplus W^{\oplus(m-(n+1))} \oplus  \mathbf{1}^{\oplus\binom{m-(n+1)}{2}}.
\end{equation}      
As we have seen in (\ref{eqn:g(p,q)-structure}), the Lie algebra $\mathfrak g=\mathfrak{g}(p,q)$ contains $\wedge^2=\wedge^2 V$ as a Lie subalgebra,  isomorphic to $\mathfrak{spin}(p,q)$.
Furthermore,  the first component $\wedge^2 W$ in (\ref{eqn:wedge_2_V_decomposition}) corresponds to the Lie subalgebra $d\varphi(\mathfrak{spin}(n,1))$,
while all other irreducible components are spherical harmonics representations of 
 $\mathfrak{spin}(n,1)$.  
Thus, by Definition~\ref{definition:lie-g-phi}, we conclude that $ \mathfrak{g}^{\varphi} \supset \mathfrak{spin}(p,q) $.
\end{proof}

     Before proving Proposition~\ref{prop:Clifford-z.d}~\ref{item:Clifford-z.d-general} and \ref{item:Clifford-z.d-exceptional},  
        we state a general principle used in the proof.  
        Since $\mathfrak{g}^{\varphi}$ is a Lie algebra,  
        it is also a $\mathfrak{spin}(p,q)$-submodule of $\mathfrak{g}=\mathfrak{g}(p,q)$
        by Lemma~\ref{lemma:wedge_l_contains_spherical_harmonics}.
        Thus, to show that an irreducible
        $\mathfrak{spin}(p,q)$-submodule $U$ of $\mathfrak{g}$  
        is contained in $\mathfrak{g}^{\varphi}$,  
        it suffices to verify that $\mathfrak{g}^{\varphi} \cap U \neq 0$.

    We are ready to show Proposition~\ref{prop:Clifford-z.d}.  
\begin{proof}[Proof of Proposition~\ref{prop:Clifford-z.d}]
 \ref{item:Clifford-z.d-general}:
        First, we consider the case where $m \geq 10$ and $m \neq 12$. 
        In this case, the $\mathfrak{spin}(p,q)$-module $\wedge^{m-2-i}$ is irreducible
        because $m - 2 - i > m/2$ and $i \in \{1,2,3,4\}$.
        Moreover, since $m\equiv i \mod 4$,   
        it follows from (\ref{eqn:g(p,q)-structure}) that $\wedge^{m-2-i} \subset \mathfrak{g}=\mathfrak{g}(p,q)$. 
Since $m -2 -i \ge n$ by the assumption of \ref{item:Clifford-z.d-general},
 $\mathfrak{g}^{\varphi} \cap \wedge^{m-2-i} \neq 0$ by Lemma~\ref{lemma:wedge_l_contains_spherical_harmonics}~\ref{item:wedge^lcontained-in-spin}.
        By the principle above,  
        it follows that $ \mathfrak{g}^{\varphi} \supset \wedge^{m-2-i} $.  
        Furthermore, since $m \geq 10$, we have $\min(m-2-i, i+2) \geq 3$,  
        it follows from Lemma~\ref{lemma:g(p,q)-generator} that  
        $\mathfrak{g}^{\varphi} = \mathfrak{g}(p,q)$.  

        Next, we consider the case $m = 12$.  
        In this case, the assumption~\ref{item:Clifford-z.d-general} implies that $m = 12$, $i = 4$, and $n \leq 6$.  
        Then we have  
        \[
        \mathfrak{g}(p,q) = \wedge^2 + \wedge^6 + \wedge^{10}.
        \]  
        By Lemma~\ref{lemma:g(p,q)-generator},  
        it suffices to show that $ \wedge^6 $ is contained in $ \mathfrak{g}^{\varphi} $.  

        In contrast to the above cases, we need to be careful because the module $ \wedge^6 $ is not irreducible as a $ \mathfrak{spin}(p,q) $-module  
        and splits into the direct sum of two distinct irreducible submodules, which we denote by $ V_1 $ and $ V_2 $.  
        Since $ n+1 < m $, the modules $ V_1 $ and $ V_2 $  
        are isomorphic to each other as $ \mathfrak{spin}(n,1) $-modules.  
Moreover, since $ n \leq 6 $, the module $ \wedge^6 \simeq \wedge^6 V$ contains at least two irreducible components, each of which is isomorphic to the standard representation $ W \simeq \wedge^{n}W $ as a $ \mathfrak{spin}(n,1) $-module.
        Hence, both $ V_1 $ and $ V_2 $ contain $ W $  
        as an irreducible component when regarded as $ \mathfrak{spin}(n,1) $-modules.  
        Thus, for each $ i = 1,2 $, we obtain $ \mathfrak{g}^{\varphi} \cap V_i \neq 0 $,  
        which, by the principle above, implies $ V_i \subset \mathfrak{g}^{\varphi} $.  
        Consequently, we conclude that $\wedge^6 \subset \mathfrak{g}^{\varphi}$,  
        which completes the proof in this case.

\ref{item:Clifford-z.d-exceptional}:
       When $m \leq 5$, we have 
        $\mathfrak{g}(p,q) = \wedge^2 = \mathfrak{spin}(p,q)$,  
        and thus it follows that $\mathfrak{g}^{\varphi} = \mathfrak{g}(p,q)$.
        Hence, the assertion is obvious. 

        Next, when $6\leq m< 10$, we have 
        $\mathfrak{g}(p,q)=\wedge^2 + \wedge^6$.     
Since $n \ge 6$ by the assumption of \ref{item:Clifford-z.d-exceptional},
 $\mathfrak{g}^{\varphi} \cap \wedge^{6} \neq 0$ by Lemma~\ref{lemma:wedge_l_contains_spherical_harmonics}~\ref{item:wedge^lcontained-in-spin}.
Since $\wedge^6$ is an irreducible $\mathfrak{spin}(p,q)$-module,  based on the above principle, we conclude that $ \mathfrak{g}^{\varphi} = \mathfrak{g}$.
Thus, the proof of \ref{item:Clifford-z.d-exceptional} is complete.
    \end{proof}

As mentioned, the proof of Theorem~\ref{thm:Clifford-z.d} is derived from Proposition~\ref{prop:Clifford-z.d} by Theorem~\ref{theorem:bending}~\ref{item:bending-realize-Gphi}.
    The proof of Theorem~\ref{theorem:bending}~\ref{item:bending-realize-Gphi} is the main task of the next section.

\section{Deformations of the representations of spin hyperbolic lattices that maximize the Zariski-closure}
\label{section:realize-zariski}
Let $\mathbf{G}$ be a Zariski-connected real algebraic group,  
$G = \mathbf{G}(\mathbb{R})$, $G_{\mathbb{C}} = \mathbf{G}(\mathbb{C})$,  
and let $\varphi \colon Spin(n,1) \to G$ be a homomorphism (we refer to Lemma~\ref{lemma:morphism} in Appendix~\ref{section:algebraic_group} for equivalent definitions in different categories).  Suppose that $\Gamma$ is a torsion-free, cocompact discrete subgroup of $Spin(n,1)$.

In this section, as in the previous one,
we set aside the space $X = G/H$ and focus on the deformation of $\varphi|_{\Gamma} \in \Hom(\Gamma, G)$.  
We construct a pair $(\Gamma, \varphi')$,
where $\Gamma$ is a torsion-free, cocompact discrete subgroup and $\varphi'$ is a small deformation of $\varphi|_{\Gamma}$, such that the pair achieves the maximal Zariski-closure of $\varphi'(\Gamma)$ up to $G$-conjugacy, 
when $n \geq 3$. 
The argument in this section completes the proofs of the statements postponed in  
Section~\ref{section:(G,Gamma)-deform-spin},  including
Lemma~\ref{lemma:g-phi-algebraic}  
and Theorem~\ref{theorem:bending}~\ref{item:bending-realize-Gphi}.

Lemma~\ref{lemma:g-phi-algebraic} was necessary for the definition of the maximal Zariski-closure $G^{\varphi}$
in Definition~\ref{def:g-phi}
and is proven in Section~\ref{section:proof:g-phi-alebraic}.   

The proof of Theorem~\ref{theorem:bending}~\ref{item:bending-realize-Gphi} consists of three steps.  
In Section~\ref{section:outline-bending}, we explain the role of each step.  
The details of each step are provided in
Sections~\ref{section:proof-step1}, \ref{section:step2}, and \ref{section:step3}, where we also
 complete the proof of the theorem.

\label{section:deform-spin-lattice}
\subsection{Proof of \texorpdfstring{Lemma~\ref{lemma:g-phi-algebraic}}{Lemma 3.6}}
    \label{section:proof:g-phi-alebraic}
    In this section, we prove Lemma~\ref{lemma:g-phi-algebraic}. 

    Let $\mathbf{G}=\mathbf{S}\cdot \mathbf{U}$ be a Levi decomposition, 
    where $S=\mathbf{S}(\R)$ is a maximal real reductive algebraic subgroup containing  $\varphi(Spin(n,1))$ 
    and $\mathbf{U}$ is 
    the unipotent radical of 
    $\mathbf{G}$. Then the real Lie algebra $\mathfrak{s}^{\varphi}$ is defined in a manner similar to $\mathfrak{g}^{\varphi}$ in Definition~\ref{definition:lie-g-phi}, using $\varphi\colon Spin(n,1)\to S$.
    Lemma~\ref{lemma:g-phi-algebraic} asserts that the analytic subgroup $G_{\mathbb{C}}^{\varphi}$ of $G_{\mathbb{C}}$ corresponding to $\mathfrak{g}^{\varphi} \otimes_{\mathbb{R}} \mathbb{C}$ is Zariski-closed in $G_{\mathbb{C}}$
    and is defined over $\R$.
    
    First, let us show that the analytic subgroup $S^{\varphi}_{\C}$ of $S_{\C}=\mathbf{S}(\C)$ corresponding to 
    $\mathfrak{s}^{\varphi} \otimes_{\R} \C$ is Zariski-closed.
    Take a Cartan involution $\theta$ of $\mathbf{S}(\R)$ 
    which preserves $\varphi(Spin(n,1))$. 
    By definition, $\theta$ also preserves $\mathfrak{s}^{\varphi}$, 
    and thus we see that the Lie algebra $\mathfrak{s}^{\varphi}$ is reductive. 
    Let $(S_{\C}^{\varphi})_{ss}$ be the analytic subgroup of $\mathbf{S}(\C)$ corresponding to the semisimple part of $\mathfrak{s}^{\varphi} \otimes_{\R} \C$. Then, 
    by Lemma~\ref{lemma:analytic->algebraic}~\ref{item:semisimple->Zariski-closed} in Appendix~\ref{section:algebraic_group}, we see that $(S_{\C}^{\varphi})_{ss}$ is Zariski-closed.
    Furthermore, let us consider the identity component $\mathbf{Z}$ of the following 
    real algebraic group in the Zariski-topology:
    \begin{align*}
    \{g \in \mathbf{S} &\mid 
    g\text{ centralizes } d\varphi(\mathfrak{spin}(n,1)) \text{ and all the submodules }\\
    &\text{of $\mathfrak{s}$ isomorphic to some spherical harmonics of $\mathfrak{spin}(n,1)$}
    \}.
    \end{align*} 
    By the definition of $\mathfrak{s}^{\varphi}$, 
    we see that the Lie algebra of $\mathbf{Z}(\C)$ coincides with the center of $\mathfrak{s}^{\varphi}\otimes_{\R}\C$.
    Since $\mathbf{Z}(\C)$ is also connected in the usual topology by Lemma~\ref{lem:Zariski-connected}, 
    $\mathbf{Z}(\C)$ is the analytic subgroup corresponding to the center of $\mathfrak{s}^{\varphi}\otimes_{\R}\C$. Hence, 
    $S_{\C}^{\varphi}=(S_{\C}^{\varphi})_{ss} \cdot \mathbf{Z}(\C)$, and thus $S_{\C}^{\varphi}$ is Zariski-closed.

    Next, let us show our assertion.
    Denote by $\mathfrak{u}$ the nilpotent radical of $\mathfrak{g}$.
    Since $\mathfrak{u}$ is stable under the adjoint action of  $d\varphi(\mathfrak{spin}(n,1))$, 
    we get a decomposition 
    \[
    \mathfrak{g}^{\varphi} = \mathfrak{s}^{\varphi} + \mathfrak{u}^{\varphi},
    \]
    where $\mathfrak{u}^{\varphi}$ is a Lie subalgebra of $\mathfrak{u}$.
    Let $U^{\varphi}_{\C}$ be the analytic subgroup of $\mathbf{U}(\C)$
    corresponding to $\mathfrak{u}^{\varphi} \otimes_{\R} \C$.
    Here we note that the exponential map $\exp\colon \mathfrak{u} \otimes_{\R} \C \rightarrow \mathbf{U}(\C)$ 
    gives an isomorphism of algebraic varieties. 
    Hence, $U^{\varphi}_{\C}=\exp(\mathfrak{u}^{\varphi} \otimes_{\R} \C)$ 
    is Zariski-closed, and thus 
    $G^{\varphi}_{\C}=S^{\varphi}_{\C}\cdot U^{\varphi}_{\C}$ 
    is also Zariski-closed. 
    Since $G^{\varphi}_{\C}$ is connected in the usual topology and 
    its Lie algebra is defined over $\R$, 
    $G^{\varphi}_{\C}$ is stable under the complex conjugation of $G_{\C}$. Thus 
    the Zariski-closed subset $G^{\varphi}_{\C}$ is defined over $\R$. Thus Lemma~\ref{lemma:g-phi-algebraic} is proved.

\subsection{
Outline of the proof of 
\texorpdfstring{
Theorem~\ref{theorem:bending}~\ref{item:bending-realize-Gphi}}{Theorem 3.8 (1)}}
\label{section:outline-bending}
This section provides an outline of Theorem~\ref{theorem:bending}~\ref{item:bending-realize-Gphi}.
We recall the setting: $G$ is a Zariski-connected real algebraic group, $L=Spin(n,1)$, 
$\varphi\colon L\rightarrow G$ is a homomorphism, 
and $G^{\varphi}$ is the algebraic group introduced in Definition~\ref{def:g-phi}. 
What we need to do is to find 
a torsion-free cocompact discrete subgroup 
$\Gamma$ and a small deformation of 
$\varphi|_{\Gamma}$ such that 
the Zariski-closure of $\varphi'(\Gamma)$ coincides with $G^{\varphi}$. 

The proof of Theorem~\ref{theorem:bending}~\ref{item:bending-realize-Gphi} consists of the following three steps: 
\begin{description}
    \item[Step~1] Construction of a cocompact discrete subgroup $\Gamma$ of $Spin(n,1)$;
    \item[Step~2] An overview of Johnson--Millson's bending deformation; 
    \item[Step~3] 
    Finding the Zariski-closure of a specific deformation of $\varphi|_{\Gamma}$.
\end{description}
First, we will outline the summary of each step, 
and then proceed with the proofs of each step in the following sections.

\subsection*{Step~1 (Construction of 
\texorpdfstring{$\Gamma$}{Gamma})}

Let $L=Spin(n,1)$ ($n\geq 2$), $L'=Spin(n-1,1)$ a subgroup of $L$, $L_{K}=Spin(n)$
a maximal compact subgroup of $L$, and 
$L'_{K}=L'\cap L_{K}\simeq Spin(n-1)$. 
Then $X'=L'/L'_{K}$ is a totally geodesic hypersurface of 
the $n$-dimensional hyperbolic space $X=L/L_{K}$.
For a torsion-free cocompact discrete subgroup $\Gamma$ of $L$, the quotient space $\Gamma\backslash X$
is an orientable connected compact hyperbolic $n$-manifold.

We consider the following condition for $\Gamma$: 
\begin{condition}
    \label{condition-ck}
    Let $k$ be a positive integer. 
    There exist $k$ orientable connected totally geodesic closed hypersurfaces $N_{1},\ldots, N_{k}$  of $M:=\Gamma\backslash X$ such that 
\begin{itemize}
    \item $N_{i}\cap N_{j} = \emptyset$ for any $i\neq j$.
    \item $M\smallsetminus (N_{1}\cup\cdots \cup N_{k})$ is connected.
\end{itemize}
\end{condition}
Step~1 is to prove the following:
\begin{theorem}
\label{theorem:Millson-spin}
For any positive integer $k$, 
there exists a torsion-free cocompact arithmetic subgroup $\Gamma\equiv\Gamma_{k}$ of $L=Spin(n,1)$ satisfying 
Condition~\ref{condition-ck}.
\end{theorem}

\begin{example}
    \label{example:millson-n=2}
    When $n=2$, Theorem~\ref{theorem:Millson-spin} clearly holds.
In fact, let $M=\Sigma_k$ be an orientable 
    compact hyperbolic Riemann surface of genus $k \ge 2$,
    and  we express  $\pi_{1}(\Sigma_{k})$ in terms of generators and relations as follows: 
    \[
    \pi_{1}(\Sigma_{k}) = \langle a_{1},b_{1},\ldots, a_{k},b_{k} \mid [a_{1},b_{1}]\cdots[a_{k},b_{k}]= 1\rangle.
    \]
    Let $N_i$ be a simple closed geodesic representing the free homotopy class of $a_i$ for each $i=1,\ldots,k$.
    Since the holonomy representation 
    $\pi_{1}(\Sigma_{k})\rightarrow PSL(2,\R)$
    can be lifted to 
    $SL(2,\R)\simeq Spin(2,1)$, Theorem~\ref{theorem:Millson-spin} 
    holds when $n=2$.
\end{example}

We shall prove Theorem~\ref{theorem:Millson-spin}
in Section~\ref{section:proof-step1}. 
When $n\geq 3$, 
we will actually construct such a group
(see Example~\ref{example:desired-Gamma}).

Theorem~\ref{theorem:Millson-spin}
was originally proved by Millson~\cite{Millson76} where he treated the case where $L=SO(n,1)$.

\begin{remark}
    \label{remark:lift}
    \begin{enumerate}[label=(\arabic*)]
        \item 
        Theorem~\ref{theorem:Millson-spin} in the case $L=Spin(n,1)$ implies 
        an analogous result in the case where $L=SO(n,1)$.
        \item 
        As seen in Remark~\ref{rem:lifting_to_spin}~\ref{item:lift-counterexample},  
        the converse implication holds for $n=2,3$, but fails for $n \geq 4$.
    \end{enumerate}
\end{remark}

We shall need a torsion-free cocompact discrete subgroup $\Gamma$ in $L=Spin(n,1)$
rather than in $L=SO(n,1)$, as stated in Theorem~\ref{theorem:Millson-spin}.
As we will see in Section~\ref{section:deform-noncptCK}, for a certain class of homogeneous spaces $X=G/H$, any non-abelian and standard discontinuous groups are virtually contained in $Spin(n,1)$ but not in $SO(n,1)$ for some $n$. 
For example, deformations of compact standard quotients of $SO(4,4)/SO(3,4)$ are obtained via $Spin(4,1)$  (see Section~\ref{section:deform-cptCK}), and those of non-compact standard quotients of $SO(8,8)/SO(7,8)$ via $Spin(6,1)$ (see Section~\ref{section:deform-noncptCK}). 

\subsection*{Step~2 (An overview of bending construction)}
In Step~2, based on the geometric idea of 
bending construction by 
Johnson--Millson~\cite{JoMi84}, we
reformulate a general principle of 
the construction of small deformation, 
that we shall use (Lemma~\ref{lemma:bending}).

    From now on, fix $k \in \mathbb{N}$. 
    By Theorem~\ref{theorem:Millson-spin}, 
    we take  a torsion-free cocompact discrete subgroup $\Gamma\equiv \Gamma_{k}$ of $Spin(n,1)$  
    satisfying Condition~\ref{condition-ck} for $k \in \mathbb{N}$.
    Let $X=Spin(n,1)/Spin(n)$, $M=\Gamma\backslash X$ and $N_{1},\ldots, N_{k}$ as in Condition~\ref{condition-ck}. 

    For each $i = 1, \dots, k$,  
    by the tubular neighborhood theorem,  
    we choose an open neighborhood $\tilde{N}_{i}^{(4)}$ of $N_{i}$ in $M$  
    and a diffeomorphism  
    \[
    f_{i}\colon N_{i} \times (-4,4) \simeq \tilde{N}_{i}^{(4)}
    \]
    such that $f_{i}(N_{i}\times \{0\})$ coincides with $N_i$.
    We define
    \[
    \tilde{N}_{i} := f_{i}(N \times (-2,2)),
    \]
    fix $y_{i} \in N_{i}$, and put 
    \[ 
    y_{i,+} := f_{i}(y_{i},1) \text{ and } y_{i,-} := f_{i}(y_{i},-1).
    \]  
    Since $N_{1},\ldots,N_{k}$ are disjoint and since 
    \[
    S := M \smallsetminus (N_{1} \cup \dots \cup N_{k})
    \]
    is path-connected,  
    we may and do assume 
    \begin{itemize}
        \item $\tilde{N}_{1},\ldots,\tilde{N}_{k}$ are disjoint;
        \item $M \smallsetminus (\tilde{N}^{(4)}_{1} \cup \dots \cup \tilde{N}^{(4)}_{k})$ is path-connected.
    \end{itemize}

    \begin{definition}
        \label{def:nu_i}
        Fix a base point $x_0 \in M \smallsetminus (\tilde{N}^{(4)}_{1} \cup \dots \cup \tilde{N}^{(4)}_{k})$.  
        For each $i = 1, \dots, k$, 
        we define the oriented loop $\nu_{i}$ in $M$, starting at $x_0$,  
        as the composition of the following paths:      
        \begin{itemize}
            \item a path from $x_0$ to $f_{i}(y_{i},3)$  
            inside $M \smallsetminus f_{i}(N_{i} \times (-3,3))$;
            \item the path from $f_{i}(y_{i},3)$ to $f_{i}(y_{i},-3)$,  
            given by $f_{i}(\{y_{i}\} \times [-3,3])$;
            \item a path from $f_{i}(y_{i},-3)$ to $x_0$  
            inside $M \smallsetminus f_{i}(N_{i} \times (-3,3))$.
        \end{itemize}
Furthermore, we can take the loop $\nu_{i}$ to be a closed \emph{submanifold} of $M$.  
Let $\nu_{i,+}$ denote the segment of the loop $\nu$ from $x_0$ to $y_{i,+}$,  
and let $\nu_{i,-}$ denote the segment of $\nu$ from $y_{i,-}$ back to $x_0$.  

    \end{definition}
    
    Figure~\ref{fig:hnn2} summarizes some of the notation introduced so far.
    \begin{figure}
        \centering
        \includegraphics[width=5cm]{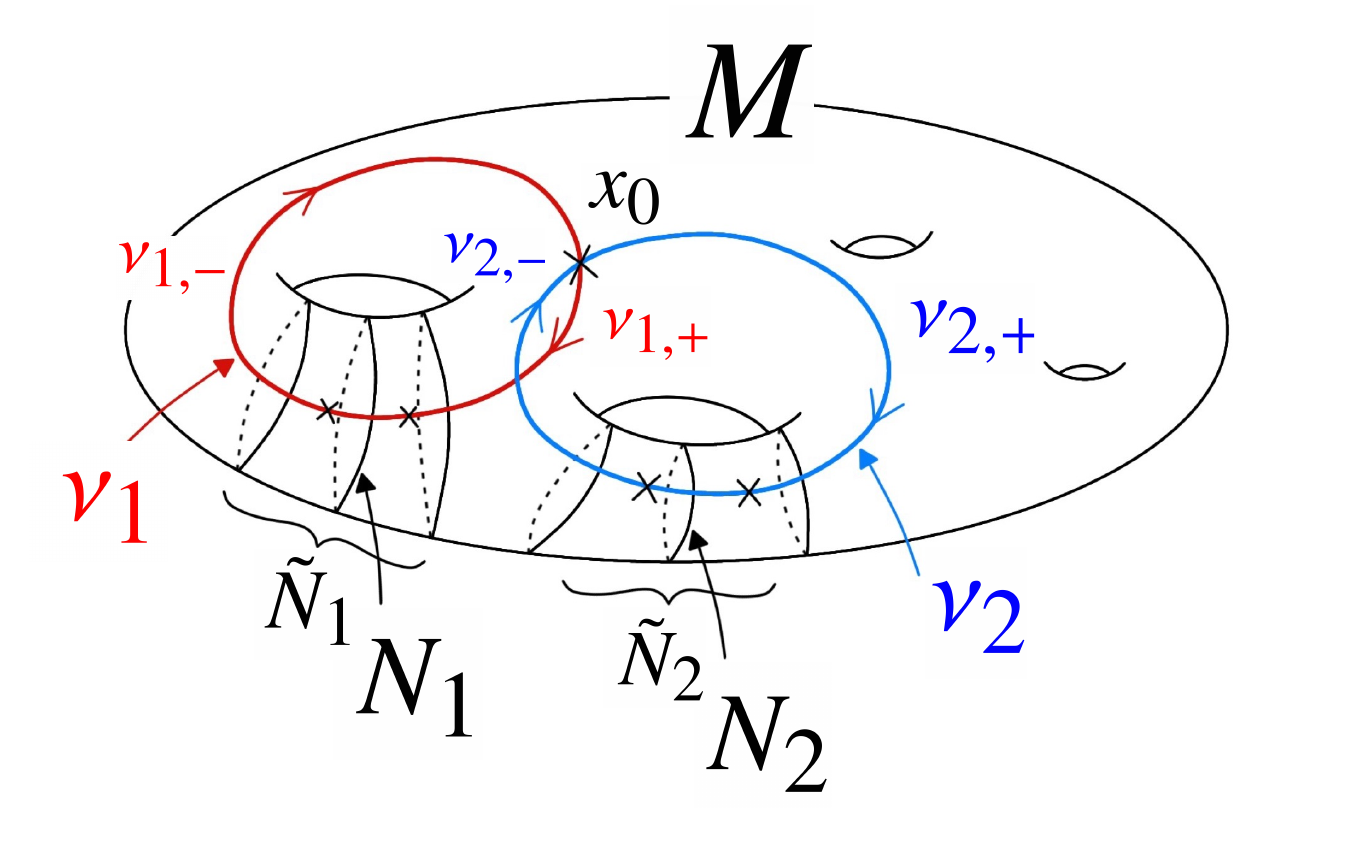}
        \caption{The loops $\nu_{1}$ and $\nu_{2}$ in the case $k=2$.}
        \label{fig:hnn2}
    \end{figure}

    Let $i=1,\ldots,k$. 
    For an oriented loop $\ell$ starting at $y_{i}$ in $N_{i}$,  
    $f_{i}(\ell, 1)$ and $f_{i}(\ell, -1)$ form oriented loops in $M$ starting at $y_{i,+}$ and $y_{i,-}$, respectively.
    With the above paths $\nu_{i,+}$ and $\nu_{i,-}$, 
    we define $j_{i,+},j_{i,-}\colon \pi_{1}(N_{i},y_{i})\rightarrow \pi_{1}(S,x_{0})$ 
    as two group homomorphisms 
    \begin{align}
    \label{def:j+-}
    j_{i,+}([\ell])&:= [\nu_{i,+}^{-1} \circ f_{i}(\ell,1) \circ \nu_{i,+}], \\
    j_{i,-}([\ell])&:= [\nu_{i,-}\circ f_{i}(\ell,-1)\circ\nu_{i,-}^
    {-1}]. \nonumber
    \end{align}

    Here, $b^{-1}$ denotes the path obtained by reversing the orientation of the path $b$, and $[c]$ denotes the homotopy class defined by the loop $c$.

For the deformation of a representation of $\Gamma$, we require the following two lemmas.  
As discussed later, via deck transformations,  
the discontinuous group $\Gamma$ for the hyperbolic space $X$ can be identified with 
the fundamental group $\pi_1(M)$ of $M=\Gamma\backslash X$.  
Lemma~\ref{lemma:hnn-repeat} states that the group structure of $\pi_1(M)$  
is obtained as an iterated HNN extension of $\pi_1(S)$. 
Lemma~\ref{lemma:bending} utilizes this group structure
to explicitly construct a small deformation of a representation of $\Gamma\simeq \pi_1(M, x_0)$ into a Lie group. For the proofs, see Section~\ref{section:step2}.  

\begin{lemma}
\label{lemma:hnn-repeat}
Let $F_k$ be the free group generated by the words $a_1, \dots, a_k$.  
We define a homomorphism $\Psi$
from the free product 
$\pi_1(S,x_0) * F_k$
to $\pi_1(M, x_0)$,
\begin{equation}
\label{eqn:Psi_pi_1}
\Psi \colon \pi_1(S,x_0) * F_k \to \pi_1(M, x_0),  
\end{equation}
as the homomorphism induced by the natural map
$\pi_1(S,x_0) \to \pi_1(M,x_0)$
and by $\Psi(a_i) = [\nu_i]$  for each $i = 1, \dots, k$.
Then $\Psi$ is surjective
and its kernel is the normal subgroup $\mathcal{N}$ generated by
\[
a_i j_{i,+}([\ell]) a_i^{-1} j_{i,-}([\ell])^{-1} \ \ \   
\text{ for } [\ell] \in \pi_1(N_i,y_{i}) \text{ and } i = 1, \dots, k.
\]

\end{lemma}

To state the second lemma, we introduce some notation.  
Let $L' := Spin(n-1,1)$ be regarded as a subgroup of $L := Spin(n,1)$, 
and let $X' := Spin(n-1,1)/Spin(n-1)$ be viewed
as a totally geodesic hypersurface of the hyperbolic space $X = Spin(n,1)/Spin(n)$.  
Recall that $N_1, \dots, N_k$ are orientable, connected, totally geodesic hypersurfaces of the hyperbolic manifold  
$M = \Gamma \backslash X$.  
Hence, for each $i = 1, \dots, k$, there exist an element $\alpha_i \in L$  
and a diffeomorphism  
$N_i \simeq (\alpha_i L' \alpha_i^{-1} \cap \Gamma) \backslash \alpha_i X'$  
such that the following diagram commutes:
\begin{equation}
\label{eq:N_{i}-isom}
\xymatrix{
N_{i} \ar@{^{(}->}[r]^{\text{inclusion}} \ar[d]^{\simeq} & M \ar[d]^{=}  \\ 
(\alpha_{i}L'\alpha_{i}^{-1} \cap \Gamma)\backslash \alpha_{i}X'
\ar[r]^(.7){\text{natural}} & 
\Gamma\backslash X. \ar@{}[lu]|{\circlearrowright}
}
\end{equation}

Now we recall how the discontinuous group $\Gamma$ for $X$ is identified with the fundamental group $\pi_1(M, x_0)$ of $M = \Gamma \backslash X$.
Since the hyperbolic space $X$ is simply-connected, 
the quotient map $\pi_{\Gamma}\colon X\rightarrow \Gamma\backslash X=M$ is a universal covering. 
We fix a point $\tilde{x}_{0}\in X$ in the discrete fiber $X_{x_{0}} :=\pi_\Gamma^{-1}(x_0)$ of $x_{0}\in M$. 
Then, we get a group isomorphism $D_{\tilde{x}_{0}}^{M}\colon \pi_{1}(M,x_{0})\rightarrow \Gamma$, by the following relation of the deck transformation:
\begin{equation}
\label{def:deck-transformation}
[\ell]\cdot \tilde{x}_{0}
=D_{\tilde{x}_{0}}^{M}([\ell])
\cdot \tilde{x}_{0}
\quad\text{for $[\ell]\in \pi_{1}(M,x_{0})$}.
\end{equation}

\begin{lemma}
    \label{lemma:bending}
    Let $G$ be a Lie group,
    $\varphi\colon \Gamma\rightarrow G$ a group homomorphism, $L'=Spin(n-1,1)$,
    and $\alpha_{1},\ldots, 
    \alpha_{k}\in L=Spin(n,1)$ as above. 
    For each $i=1,\ldots,k$, we take a possibly zero element $v_{i}$ of $ \mathfrak{g}$ such that
    \begin{equation}
    \label{assumption:lemma:bending}
    v_{i}\text{ is }\varphi(\alpha_{i}L'\alpha_{i}^{-1} \cap \Gamma)\text{-invariant}.
    \end{equation}
    Retain the notation as in Lemma~\ref{lemma:hnn-repeat}.
    For $t\in \R$, we put
    \begin{enumerate}[label=(\alph*)]
        \item 
        \label{item:lemma:bending-S}
$\varphi_{t} \equiv \varphi$
 on $D_{\tilde{x}_{0}}^{M} \circ \Psi (\pi_{1}(S,x_{0}))$;
        \item 
        $\varphi_{t}(D_{\tilde{x}_{0}}^{M}([\nu_{i}])) := \varphi(D_{\tilde{x}_{0}}^{M}([\nu_{i}]))\exp(tv_{i})$
        for each $i=1,\ldots,k$. 
    \end{enumerate}
    Then $\varphi_{t}$ induces a group homomorphism 
    $\varphi_{t}\colon \Gamma\rightarrow G$. 
\end{lemma}

\subsection*{Step~3 (Zariski-closure of specific deformation)}
Let $G$ be a Zariski-connected real algebraic group, $L:=Spin(n,1)$
and $\varphi\colon L\rightarrow G$ a homomorphism. 
In Step~3, we construct a small deformation $\varphi'$ of $\varphi|_{\Gamma}$
such that the Zariski-closure of 
$\varphi'(\Gamma)$ 
coincides with $G^{\varphi}$ (Proposition~\ref{prop:zariski-closure-gphi}),
where $\Gamma$ is a certain cocompact discrete subgroup of $L$. 
This is achieved by finding appropriate vectors $v_{i}\in \mathfrak{g}$ ($1 \le i \le k$) in the setting of Lemma~\ref{lemma:bending},
considering the full generality of $G$ and  $\varphi\colon L\rightarrow G$.
We shall also optimize the number $k$, which corresponds to the number of totally geodesic hypersurfaces in Condition~\ref{condition-ck} that we use for bending constructions.

\begin{remark}
\label{remark:difference-Kassel}
When ${\mathfrak g}^{\varphi}=\mathfrak g$, 
the conclusion of Proposition~\ref{prop:zariski-closure-gphi}
below implies that $\Gamma$ can be deformed into a Zariski-dense subgroup in $G$. 
The previous results for the case $L=SO(n,1)$, such as
Johnson--Millson~\cite{JoMi84} ($\varphi$ maps into $G=SO(n+1,1)$ or $PSL(n+1, \mathbb{R})$), 
and Kassel~\cite{Kassel12}
($\varphi$ maps into $SO(n,2)$), 
are (implicitly) based on
the following structure in their proofs:
\textit{the $\mathfrak{l}$-module $\mathfrak{g}/\mathfrak{l}$ is an irreducible spherical harmonics module.}
Along the same lines of argument,
this assumption can be relaxed to the existence of an increasing sequence of reductive Lie algebras 
\[
    \mathfrak{l}=\mathfrak{l}^{0}
    \subset \mathfrak{l}^{1}\subset 
    \dots\subset \mathfrak{l}^{m}=\mathfrak{g}
\]
   such that $\mathfrak{l}^{k+1}/\mathfrak{l}^{k}$ is an irreducible $\mathfrak{l}^{k}$-module and that it contains a non-zero spherical harmonics representation of $\mathfrak{l}$.
   For example, this applies to the case
$G=SO(p,q)$ $(p\geq n,\ q\geq 1)$, as discussed in a recent preprint \cite[Appendix~A]{Beyrer-Kassel23} by Beyrer and Kassel.
  
However, in our general setting of $\varphi$ and $G$, we cannot rely on such a restrictive structure for $\mathfrak{l}$ and $\mathfrak{g}$.
    For instance, in the context of Theorem~\ref{thm:Clifford-z.d}, 
    this is not the case for
    the natural homomorphism $\varphi \colon Spin(n,1) \to G(p,q)$ when $p$ and $q$ take arbitrary values.
\end{remark}

To construct a desired small deformation, we first introduce some notation. 

Let $V_{i}$ denote the irreducible spherical harmonics representation of $L:=Spin(n,1)$ of degree $i$ (see Remark~\ref{remark:spherical_harmonics}).
Given a homomorphism $\varphi\colon L\rightarrow G$,
we denote by $Z_G(\varphi(L))$ the identity component (in the Zariski topology) of the centralizer of $\varphi(L)$ in $G$.
For $i \in \N$, let $\mathfrak{g}(V_i)$ denote the isotypic component of $V_i$ in $\mathfrak{g}$.
Then $\mathfrak{g}(V_0)= \mathfrak{z}_{\mathfrak{g}}(d\varphi(\mathfrak{l}))$, the Lie algebra of $Z_G(\varphi(L))$.

We consider a decomposition of $\mathfrak{g}$  
as an $L$-module via $\varphi \circ \operatorname{Ad}$:
\begin{align}
\label{eq:decomposition-g-even}
\mathfrak{g}= d\varphi(\mathfrak{l})
\oplus \mathfrak{z}_{\mathfrak{g}}(d\varphi(\mathfrak{l}))
\oplus \bigoplus_{i=1}^{\infty} 
\mathfrak{g}(V_i)
\oplus (\text{other representations}).
\end{align}
By Definition~\ref{definition:lie-g-phi}, $\mathfrak{g}^\varphi$ is the Lie algebra generated by
\[
d\varphi(\mathfrak{l})
\oplus \mathfrak{z}_{\mathfrak{g}}(d\varphi(\mathfrak{l}))
\oplus \bigoplus_{i=1}^{\infty} 
\mathfrak{g}(V_i).
\]
Needless to say, the summation is finite because $\dim \mathfrak{g} < \infty$.
The multiplicity of the irreducible $L$-module $V_i$ in $\mathfrak{g}$ is denoted by
\[
[\mathfrak g:V_i] :=\dim_{\R} \Hom_{L}(V_i, \mathfrak g).
\]
Then, we may decompose
\[
\mathfrak{g}(V_i)
=
\bigoplus_{j=1}^{[\mathfrak{g}:V_{i}]} V_{i}^{(j)},
\]
where each $V_{i}^{(j)}$ is isomorphic to $V_{i}$.

We set $m := \max_{i\in \N_+}[\mathfrak{g}:V_{i}]$.
We define $k$, for example, by
\begin{equation}
    \label{def:k}
    k:=\max(\eta(Z_G(\varphi(L))), m)  \in \N,
\end{equation}
with the notation as defined in Definition~\ref{def:eta} of Appendix~\ref{section:appendix-Zariski-dense-subgroups}.
We note that $\eta(Z_G(L)) \le 2$ if $G$ is reductive by Theorem~\ref{thm:two_generates_reductive}.

We take $\Gamma$ satisfying Condition~\ref{condition-ck} for $k$ by Theorem~\ref{theorem:Millson-spin}.
Let $\alpha_{1},\ldots,\alpha_{k}$ be elements of $L=Spin(n,1)$ as in \eqref{eq:N_{i}-isom}.

We recall $L'=Spin(n-1,1)$.
For each $i \in \N_+$, we take $v_i^{(1)}, \dots, v_i^{(k)}$ as follows:
$v_i^{(j)}$ is a non-zero 
 $\varphi(\alpha_{i}L'\alpha_{i}^{-1})$-invariant element in $V_{i}^{(j)}$
 for $1 \le j \le [\mathfrak g:V_i]$
 and $v_i^{(j)}:=0$ for $[\mathfrak g:V_i]+1\le j \le k$.
 We note that $v^{(j)}_{i}=0$ for $i\gg 0$.

We take $u_{1}, \dots, u_{k}
\in \mathfrak{z}_{\mathfrak{g}}(\mathfrak{l})$ as follows:
With the notation as defined in 
Definition~\ref{def:GLtX_Zariski_closure},
$G(t\{u_1, \dots, u_{\eta(Z_G(\varphi(L)))}\}) = Z_G(L)$ and $u_j:=0$ if $\eta(Z_G(L)) < j \le k$.

We now define $v_j\in \mathfrak{g}$, 
for  $1 \le j\le k$, by a finite sum: 
\begin{equation}
\label{eqn:optimal_direction_for_bending}
   v_j:=u_j+\sum_{i=1}^\infty v_i^{(j)}. 
\end{equation}
\begin{proposition}
    \label{prop:zariski-closure-gphi}
Let $G$ be a Zariski-connected, real algebraic group, and let $\varphi\colon Spin(n,1)\rightarrow G$ a homomorphism.
Let $\varphi_t\colon \Gamma \rightarrow G$ be the homomorphism associated with $v_{1},\ldots, v_k$, as defined in Lemma~\ref{lemma:bending}.
Then, the Zariski-closure of $\varphi_{t}(\Gamma)$ coincides with $G^{\varphi}$ for any sufficiently small real number
$t \neq 0$.
\end{proposition}

This proposition will be proved in Section~\ref{section:step3}, 
and thus the proof of Theorem~\ref{theorem:bending}~\ref{item:bending-realize-Gphi} 
is complete.

\subsection{Proof of \texorpdfstring{Theorem~\ref{theorem:Millson-spin} (Step~1)}{Theorem 4.2 (Step 1)}}
\label{section:proof-step1}
In this section, we prove a $Spin(n,1)$ analogue  
of Millson's theorem \cite{Millson76} for $SO(n,1)$  
by reformulating his argument \cite{Millson76} in terms of Clifford algebras (Theorem~\ref{theorem:Millson-spin}).   
The part of the proof in  \cite{Millson76} 
 that relies on the strong approximation theorem  becomes slightly simpler when working with $Spin(n,1)$ instead of $SO(n,1)$ since 
 the algebraic group $\mathbf{Spin}_{n,1}$ is simply-connected
(Proposition~\ref{prop:Gamma/GammaI}).

To prove Theorem~\ref{theorem:Millson-spin}, 
we introduce some notation.
Let $\F$ be a totally real number field of degree $r>1$, 
$\{\sigma_{1},\ldots,\sigma_{r}\}$ the set of embeddings of the field $\F$ into $\R$,
$\mathcal{O}$ the ring of integers of $\F$, 
$V$ an $(n+1)$-dimensional $\F$-vector space equipped with
a quadratic form $Q=Q(v)$ on $V$, 
and $C_{\even}(V)$ the even Clifford algebra associated to $(V,Q)$.
We refer to Appendix~\ref{section:clifford-spin} for 
the notation related to Clifford algebras and spin groups. 

Take an orthogonal $\F$-basis $\{e_{0},\ldots, e_{n}\}$ of $V$ with respect to $Q$,
such that $Q(e_{i})\in \mathcal{O}$ for any $i=0,\ldots,n$.
Let $\Lambda$ be the $\mathcal{O}$-submodule of $V$ generated by
$e_{0},\ldots,e_{n}$, and let $C_{\even}(\Lambda)$ be the $\mathcal{O}$-subalgebra 
of $C_{\even}(V)$ generated by 
$e_{i_{1}}\cdots e_{i_{\ell}}$ for any even number $\ell$
and any $0\leq i_{1}<\cdots <i_{\ell}\leq n$. 
We set
\begin{equation}
\label{eq:gamma-lambda}
    \Gamma_{\Lambda}:=C_{\even}(\Lambda)\cap Spin(V). 
\end{equation}
For an ideal $I$ of $\mathcal{O}$, we also consider its congruence subgroup
\begin{equation}
\label{eq:gamma-lambda-I}
\Gamma_{\Lambda}(I)
    :=\{\gamma\in \Gamma_{\Lambda}\mid \gamma-1\in IC_{\even}(\Lambda)\}.
\end{equation}
This is a finite-index subgroup of $\Gamma_{\Lambda}$.
We will prove that $\Gamma_{\Lambda}(I)$ satisfies Condition~\ref{condition-ck} under certain assumptions (Proposition~\ref{prop:Ck}). 

First, we verify that $\Gamma_{\Lambda}$ satisfies the following properties:
\begin{itemize}
    \item (Lemma~\ref{lemma:arithmetic}) $\Gamma_{\Lambda}$
    can be regarded as 
    an $\mathcal{O}$-arithmetic group of 
    the $\F$-group $\mathbf{Spin}_{V}$ (see e.g., \cite[p.227]{PlRa94});
    \item (Proposition~\ref{proposition:discrete-cocompact}) 
    Under certain conditions, \( \Gamma_{\Lambda} \) can be regarded as a cocompact discrete subgroup of \( Spin(n,1) \).
\end{itemize}
See \eqref{def:spin_V} in Appendix~\ref{section:clifford-spin} for the definition of the $\F$-group $\mathbf{Spin}_{V}$. 

Using the $\F$-basis of $C_{\even}(V)$
\begin{equation}
\label{eq:Fbasis}
    \{e_{i_{1}}\cdots e_{i_{\ell}}
    \mid \text{$\ell$ is even and }0\leq i_{1}<\ldots <i_{\ell}\leq n\},
\end{equation}
we identify the $\F$-vector space
$C_{\even}(V)$ with $\F^{2^{n}}$. 
For any extension field $\E$ of $\F$,  
we identify the $\E$-algebra $C_{\even}(V) \otimes_{\F} \E$ naturally
with the Clifford algebra $C_{\even}(V \otimes_{\F} \E)$.
For each $x \in C_{\even}(V) \otimes_{\F} \E$,  
consider the left multiplication map by $x$.  
This yields an injective $\E$-algebra homomorphism  
\[
\iota_{\E} \colon C_{\even}(V) \otimes_{\F} \E  
\to \End_{\E}(C_{\even}(V) \otimes_{\F} \E)  
\simeq M(2^{n}, \E).
\]  
Using this embedding, we regard $\mathbf{Spin}_V$  
as an $\F$-algebraic group and define 
the group of $\mathcal{O}$-points by 
\[
\mathbf{Spin}_V(\mathcal{O})  
:= \mathbf{Spin}_V(\E) \cap \iota_{\E}^{-1}(GL(2^n, \mathcal{O})).
\]  
This definition does not depend on the choice of the extension field $\E$.
 
\begin{lemma}
\label{lemma:arithmetic}
    The group $\Gamma_{\Lambda}$ coincides with
    $\mathbf{Spin}_{V}(\mathcal{O})$. 
    In particular, $\Gamma_{\Lambda}$ is an $\mathcal{O}$-arithmetic subgroup of $\mathbf{Spin}_{V}$.
\end{lemma}

\begin{proof}
Note that \eqref{eq:Fbasis} forms an $\mathcal{O}$-basis of  
$C_{\even}(\Lambda)$. 
We have
    \begin{align*}
        \mathbf{Spin}_{V}(\mathcal{O}) &= \mathbf{Spin}_{V}(\F)\cap 
        \iota_{\F}^{-1}(GL(2^{n},\mathcal{O})) \\
        &= Spin(V) \cap \iota_{\F}^{-1}(\{x\in GL(2^{n},\F)\mid x\mathcal{O}^{2^{n}}=\mathcal{O}^{2^{n}}\}) \\
        &= \{x\in Spin(V)\mid 
        xC_{\even}(\Lambda) = C_{\even}(\Lambda)\} \\
        &= Spin(V)\cap C_{\even}(\Lambda) \\
        &= \Gamma_{\Lambda}.
    \end{align*}
    Hence, the assertion is proved.
\end{proof}

\begin{proposition}
    \label{proposition:discrete-cocompact}
    For each $i = 1, \dots, r$,  
    let $\F_{\sigma_i}$ denote the field $\R$ regarded as an extension field of $\F$ via the embedding $\sigma_i \colon \F \to \R$.  
    Assume the following:  
    \begin{itemize}
    \item $\sigma_{1}(Q(e_{i}))>0$ for any $i=0,\ldots,n-1$
    and $\sigma_{1}(Q(e_{n}))<0$;
    \item $\sigma_{j}(Q(e_{i}))>0$ for any $i=0,\ldots,n$ and $j=2,\ldots,r$;
    \end{itemize}
    Then we have 
    \[
    \Spin_{V}(\F_{\sigma_{i}})\simeq 
    \begin{cases}
    Spin(n,1) & (i=1), \\
    Spin(n+1) & (i>1).
    \end{cases}
    \]
    Furthermore, the inclusion map
    \[
    \Spin_{V}(\mathcal{O})=\Gamma_{\Lambda} 
    \rightarrow \Spin_{V}(\F_{\sigma_{1}})
    \simeq Spin(n,1)
    \]
    is discrete and cocompact.
\end{proposition}

\begin{proof}
The first statement follows immediately from the assumption on the quadratic form \( Q \).
Thus we focus on proving the second statement.  
Since \( r > 1 \), there exists an inclusion  
\[
\Spin_V(\mathbb{F}) = Spin(V) \to \Spin_V(\mathbb{F}_{\sigma_2}) \simeq Spin(n+1),
\]  
which implies that \( Spin(V) \) has no non-trivial unipotent elements.  
Hence, \( \Spin_V \) is anisotropic over \( \mathbb{F} \).  
By a well-known criterion (e.g., \cite[Thm.~4.17~(3)]{PlRa94}),  
the diagonal embedding  
\[
\Spin_V(\mathcal{O}) = \Gamma_{\Lambda}  
\to \prod_{i=1}^{r} \Spin_V(\mathbb{F}_{\sigma_i})  
\simeq Spin(n,1) \times Spin(n+1)^{r-1}
\]  
is discrete and cocompact.  
Since \( Spin(n+1)^{r-1} \) is compact, the second statement follows.
\end{proof}

Next, let $I$ be an ideal of $\mathcal{O}$.
We examine properties of $\Gamma_{\Lambda}(I)$.  
By an argument similar to the proof of Lemma~\ref{lemma:arithmetic},  
we obtain the identity  
\begin{equation}
\label{eq:congruence}
\Gamma_{\Lambda}(I) = \Gamma_{\Lambda} \cap \iota_{\F}^{-1}(\{g \in GL(2^n, \mathcal{O}) \mid g \equiv 1 \mod I\}).
\end{equation}
Using Lemma~\ref{lemma:torsion-free-general} in Appendix, we have:
\begin{lemma}
    \label{lemma:torsion-free}
    Let $P,Q$ be prime ideals of $\mathcal{O}$ 
    such that $P\cap \Z\neq Q\cap \Z$. 
    Then $\Gamma_{\Lambda}(PQ)$ is torsion-free. 
\end{lemma}

To prove Theorem~\ref{theorem:Millson-spin},  
we analyze the finite group $\Gamma_{\Lambda}/\Gamma_{\Lambda}(I)$  
for some ideal $I$ (Proposition~\ref{prop:Gamma/GammaI}).  
To state the proposition, we first review some basic concepts of algebraic groups over number fields needed in our discussion.  
For further details, see, for example, \cite{PlRa94}.  

For a maximal ideal $P$ of the ring of integers $\mathcal{O}$,  
the completion of the global field $\F$ with respect to the $P$-adic topology  
defines a non-Archimedean local field $\F_P$.  
Let $\mathcal{O}_P$ denote the ring of integers of $\F_P$.  

In the same manner as the definition of $\mathbf{Spin}_V(\mathcal{O})$,  
we define the $\mathcal{O}_P$-points $\mathbf{Spin}_V(\mathcal{O}_P)$ 
as  
\[
\mathbf{Spin}_V(\mathcal{O}_P) := \mathbf{Spin}_V(\F_P) \cap \iota_{\F_P}^{-1}(GL(2^n, \mathcal{O}_P)),
\] 
via the left regular representation  
$\iota_{\F_P} \colon C_{\even}(V \otimes_{\F} \F_P) \to M(2^n, \F_P)$
using the $\F$-basis from \eqref{eq:Fbasis}.

Furthermore, assume that the maximal ideal $P$ does not contain $2Q(e_0) \cdots Q(e_n)$.  
In this case, the residue field  $k_P := \mathcal{O}/P$
has characteristic greater than $2$.  
We define a $k_P$-vector space by  
\[
\overline{\Lambda}_P := \Lambda / P\Lambda,
\]  
and denote by $\overline{Q}_P$ the quadratic form on $\overline{\Lambda}_P$  
induced by $Q$.  
Thus the Clifford algebra $C_{\even}(\overline{\Lambda}_P)$  
and the spin group $Spin(\overline{\Lambda}_P)$ are defined.  
By definition, it follows immediately that  
$Spin(\overline{\Lambda}_P)$ provides a \emph{reduction} of $\mathbf{Spin}_V$ modulo $P$  
in the sense of \cite[p.~142]{PlRa94}.  
Hence, by \cite[Prop.~3.20]{PlRa94}, we obtain the following:

\begin{lemma}
    \label{lemma:exceptional-set-maximal-ideals}
    There exists a finite set of maximal ideals $E_V$ of $\mathcal{O}$,  
    depending only on $(V, Q, \Lambda)$ such that every maximal ideal $P \not\in E_V$ does not contain $2Q(e_0) \cdots Q(e_n)$ and that 
    the reduction map  
    \[
    \mathbf{Spin}_{V}(\mathcal{O}_{P}) \to Spin(\overline{\Lambda}_{P})
    \]  
    is surjective.  
\end{lemma}

\begin{remark}
For the orthogonal group $\mathbf{SO}_{V}$, the reduction map  
\[
\mathbf{SO}_{V}(\mathcal{O}_{P}) \to SO(\overline{\Lambda}_{P})
\]
is surjective for every maximal ideal $P$ that does not contain $2$ or the discriminant $Q(e_0) \cdots Q(e_n)$.  
Millson~\cite{Millson76} used this fact for the orthogonal group $\mathbf{SO}_{V}$.  
However, the authors do not know whether a similar statement holds for the spin group $\mathbf{Spin}_{V}$,
as it is unclear whether it is possible to take $E_{V}$ to be the set of all maximal ideals $P$  that do not contain $2Q(e_0) \cdots Q(e_n)$  in Lemma~\ref{lemma:exceptional-set-maximal-ideals}.
This statement would strengthen Example~\ref{example:desired-Gamma} below by removing the assumption of ``sufficiently large'' from it.
\end{remark}

We are now ready to state the proposition on $\Gamma_{\Lambda}/\Gamma_{\Lambda}(I)$: 
\begin{proposition}
    \label{prop:Gamma/GammaI}
    Suppose that $\dim_{\F} V \geq 3$ (recall that $\dim_{\F} V = n+1$)
    and that all the assumptions in Proposition~\ref{proposition:discrete-cocompact} hold.
    Let $E_{V}$ be the finite set of maximal ideals in Lemma~\ref{lemma:exceptional-set-maximal-ideals},  
    and let $P_1, \dots, P_m$ be maximal ideals of $\mathcal{O}$ that do not belong to $E_{V}$. Then, the reduction map induces an isomorphism  
    \[
    \Gamma_{\Lambda}/\Gamma_{\Lambda}(P_1 \cdots P_m) \simeq \prod_{i=1}^{m} Spin(\overline{\Lambda}_{P_i}).
    \]
\end{proposition}
\begin{proof}
    Since $\dim_{\F} V \geq 3$, the $\F$-group $\mathbf{Spin}_V$ 
    is simply-connected. 
    Under the assumptions in Proposition~\ref{proposition:discrete-cocompact}, 
    $\mathbf{Spin}_V(\F_{\sigma_{1}})\simeq Spin(n,1)$, 
    which is a non-compact simple Lie group for every 
    $n\geq 2$.
    Thus, the \emph{strong approximation theorem}  
    (Kneser~\cite{Kne66}, see \cite[Thm.~7.12]{PlRa94}) applies to $\mathbf{Spin}_V$. Although the following proof relies on the standard argument using the strong   approximation theorem, we provide a proof for readers who may be unfamiliar with it.  

    Let us recall the statement of the strong approximation theorem  
    in the form required for our setting.  
    We define the finite adelization of $\mathbf{Spin}_{V}$ as  
    \[
    \mathbf{Spin}_{V}(\mathbb{A}^{f}_{\F})
    :=\bigcup_{S} \left( \prod_{P\in S} \mathbf{Spin}_{V}(\F_{P})  
    \times \prod_{P\not \in S} \mathbf{Spin}_{V}(\mathcal{O}_{P}) \right),
    \]  
    where $S$ ranges over all finite sets of maximal ideals of $\mathcal{O}$.  
    Furthermore, we endow $\mathbf{Spin}_{V}(\mathbb{A}^{f}_{\mathbb{F}})$  
    with the inductive limit topology arising from the direct system  
    of locally compact totally disconnected groups
    \[
    \prod_{P \in S} \mathbf{Spin}_{V}(\mathbb{F}_{P})  
    \times \prod_{P \notin S} \mathbf{Spin}_{V}(\mathcal{O}_{P}). 
    \]
    With this topology, $\mathbf{Spin}_{V}(\mathbb{A}^{f}_{\F})$  
    becomes a locally compact Hausdorff topological group.  
    Moreover, the diagonal morphism  
    $\mathbf{Spin}_{V}(\F) \to \mathbf{Spin}_{V}(\mathbb{A}^{f}_{\F})$  
    is well-defined, and thus we can regard $\mathbf{Spin}_{V}(\F)$  
    as a subgroup of $\mathbf{Spin}_{V}(\mathbb{A}^{f}_{\F})$.  
    The strong approximation theorem asserts that  
    $\mathbf{Spin}_{V}(\F)=Spin(V)$ is dense in $\mathbf{Spin}_{V}(\mathbb{A}^{f}_{\F})$.    
    
    We now prove our assertion.  
    Since the kernel of the reduction map  
    \[
    \Gamma_{\Lambda} \to \prod_{i=1}^{m} Spin(\overline{\Lambda}_{P_{i}})
    \]  
    is clearly $\Gamma_{\Lambda}(P_{1} \cdots P_{m})$,  
    it suffices to prove that the reduction map is surjective.  

    For each $i = 1, \dots, m$, take an element $\bar{g}_{i} \in Spin(\overline{\Lambda}_{P_{i}})$,  
    and consider the open subset of $\mathbf{Spin}_{V}(\mathbb{A}^{f}_{\F})$ given by  
    \[
    U = \prod_{i=1}^{m}
    \{ g \in \mathbf{Spin}_{V}(\mathcal{O}_{P_{i}}) \mid g \bmod P_{i} = \bar{g}_{i} \}  
    \times  
    \prod_{P\neq P_{i}} \mathbf{Spin}_{V}(\mathcal{O}_{P}).
    \]  
    By Lemma~\ref{lemma:exceptional-set-maximal-ideals}, we see that $U$ is non-empty.  
    Since $Spin(V)$ is dense in $\mathbf{Spin}_{V}(\mathbb{A}^{f}_{\F})$, 
    we obtain $U \cap Spin(V) \neq \emptyset$.  

    Here note 
    \[
    \mathbf{Spin}_{V}(\mathcal{O}) = \prod_{P} \mathbf{Spin}_{V}(\mathcal{O}_{P}) \cap \mathbf{Spin}_{V}(\F)
    \quad
    \text{in}\quad \mathbf{Spin}_{V}(\mathbb{A}^{f}_{\F}).
    \]  
    Thus it follows that any element in $U \cap Spin(V)$ must belong to $\Gamma_{\Lambda} = \mathbf{Spin}_{V}(\mathcal{O})$.  
    Hence, we conclude that $U \cap \Gamma_{\Lambda} \neq \emptyset$,  
    which implies that the reduction map is surjective.  
    This completes the proof.  
\end{proof}

Finally, we prove that for a suitable ideal $I$, the $\mathcal{O}$-arithmetic subgroup $\Gamma_{\Lambda}(I)$ satisfies Condition~\ref{condition-ck} (Proposition~\ref{prop:Ck}).  
To state the proposition, we introduce some notation.  

Recall that $e_0, \dots, e_n$ form an $\F$-basis of $V$.
Let $V'$ be the $\F$-subspace of $V$ generated by $e_1, \dots, e_n$, and define $\Lambda' := V' \cap \Lambda$.  
As a result, the following are defined:  
the Clifford algebra $C_{\even}(V') \equiv C_{\even}(V', Q|_{V'})$,  
its $\mathcal{O}$-subalgebra $C_{\even}(\Lambda')$,  
the spin group $\Spin_{V'}$,  
the $\mathcal{O}$-arithmetic subgroup $\Gamma_{\Lambda'}$,  
and its congruence subgroup $\Gamma_{\Lambda'}(I)$ for an ideal $I$.  

The following proposition is the main goal of this section.  
The proof is achieved by reformulating the argument in Millson~\cite[p.~245]{Millson76}  
regarding the construction of a family of compact hyperbolic manifolds  
with arbitrarily large first Betti numbers in terms of Clifford algebras.

\begin{proposition}
    \label{prop:Ck}
    Assume that $\dim_{\F} V \geq 4$ (recall that $\dim_{\F} V = n+1$).  
    Take $E_V$ and $E_{V'}$ be  finite sets of maximal ideals of $\mathcal{O}$ as given by applying  
    Lemma~\ref{lemma:exceptional-set-maximal-ideals}  
    to $(V, Q, \Lambda)$ and $(V', Q|_{V'}, \Lambda')$, respectively.  

    For $k \in \mathbb{N}$, choose an integer $m \geq 2$ such that $2^m - 1 \geq k$.  
    Suppose that there exist $m+2$ distinct prime ideals $P_{-1}, P_{0}, \dots, P_{m}$ satisfying the following conditions:  
    \begin{itemize}
        \item $P_i \cap \mathbb{Z} \neq P_j \cap \mathbb{Z}$ for any $i \neq j$;
        \item Each $P_i$ belongs to neither $E_V$ nor $E_{V'}$.
    \end{itemize}  
    Moreover, in addition to the assumptions in Proposition~\ref{proposition:discrete-cocompact},  
    suppose that $Q(e_{0}) \in \mathcal{O}^{\times}$.  
    Then, the torsion-free cocompact discrete subgroup $\Gamma_{\Lambda}(P_{-1}\cdots P_{m})$
    of $Spin(n,1)$ satisfies Condition~\ref{condition-ck} for $k$.  
\end{proposition}

\begin{example}
    \label{example:desired-Gamma}
    Let $n\geq 3$ and $\mathbb{F} = \mathbb{Q}(\sqrt{2})$. 
    We consider the following quadratic form with coefficients in $\mathcal{O} = \mathbb{Z}[\sqrt{2}]$:  
    \[
    Q(x) = x_0^2 + \cdots + x_{n-1}^2 - \sqrt{2} x_n^2.
    \]  
    We regard this as a quadratic form on $V = \mathbb{F}^{n+1}$
    with standard basis $e_0, \dots, e_n$,  
    and let $\Lambda = \mathcal{O}^{n+1}$.  
    The triple $(V,Q,\Lambda)$ satisfies all the assumptions in Proposition~\ref{proposition:discrete-cocompact} and 
    $Q(e_{0})=1$ is invertible in $\mathcal{O}$.
    
    Primes of the form $8a\pm 3$ exist infinitely and remain primes in $\mathcal{O}$. 
    For $k \in \mathbb{N}$, take a natural number $m \geq 2$ such that $2^m - 1 \geq k$, and choose $m+2$ such sufficiently large primes $p_{-1}, p_{0}, \ldots, p_m$. 
    By Proposition~\ref{prop:Ck},
    $\Gamma_{\Lambda}(p_{-1}\cdots p_m\mathcal{O})$ 
    is a torsion-free cocompact discrete subgroup of 
    $Spin(n,1)$ satisfying Condition~\ref{condition-ck} for $k\in \N$.
\end{example}

Thus, once Proposition~\ref{prop:Ck} is established,  
the proof of Theorem~\ref{theorem:Millson-spin} for $n\geq 3$ is complete (see Example~\ref{example:millson-n=2} for the proof in the case $n=2$).  
The rest of this section is devoted to the proof of Proposition~\ref{prop:Ck}.  
For this purpose, we use two geometric lemmas from Millson~\cite{Millson76}  
(Lemmas~\ref{lem:jaffe} and~\ref{lemma:millson-main}).  

The proof of the following lemma is essentially the same as that of  
\cite[Lem.~2.1]{Millson76} (referred to as Jaffe's lemma).  
In \cite{Millson76}, the statement was proved under the assumptions that  
$X$ is a hyperbolic manifold and $T$ is a group of order $2$.  
Since the results hold in a more general setting, we provide a proof in this broader context for future reference.

\begin{lemma}
\label{lem:jaffe}
    Let $X:=G/H$ be a homogeneous space,
    $\Gamma$ a discontinuous group for $X$,
    and 
    $T\subset \Aut(G)$ a finite subgroup such that $H$ and $\Gamma$ are
    both stable under $T$.

    Put $G':=\{g\in G\mid t(g) = g \text{ for any }t\in T\}$,
    $H':=G'\cap H$, $\Gamma'=G'\cap \Gamma$, and $X':=G'/H'$. 
    Then 
    the map $\pi\colon \Gamma'\backslash X'\rightarrow \Gamma\backslash X$ defined by
    $\Gamma'g' H' \mapsto \Gamma g' H$ 
    is a diffeomorphism into a regular submanifold of 
    $\Gamma\backslash X$.
\end{lemma}

\begin{proof}
    Define the action of $t \in T$ on $gH \in X$ by $t(gH) = t(g)H$. Under this action, $T$ fixes every element of $X'$.

Take any point $x \in X'$. Then we choose a neighborhood $U$ of $x$ in $X$ such that, if $\gamma \in \Gamma$ satisfies $\gamma U \cap U \neq \emptyset$, then $\gamma = 1$. By replacing $U$ with $\bigcap_{t \in T} t(U)$, we may assume that $U$ is stable under the $T$-action.

We claim that if $\gamma \in \Gamma$ satisfies $\gamma U \cap X' \neq \emptyset$, then $\gamma \in \Gamma'$. Indeed, suppose there exists $y \in U$ such that $\gamma y \in X'$. Then, for any $t \in T$, we have $t(\gamma)t(y) = \gamma y$. Since $t(y) \in U$, it follows that $t(\gamma) = \gamma$. Hence, $\gamma \in \Gamma'$. This proves the claim. In particular, we also see that the map $\pi\colon \Gamma'\backslash X'\rightarrow \Gamma\backslash X$ is injective.

Let $\pi_{\Gamma} \colon X \to \Gamma \backslash X$
and $\pi_{\Gamma'} \colon X' \to \Gamma \backslash X'$ be 
the quotient maps. Then the image of $\pi$ is $\pi_{\Gamma}(X')$.
It follows from the previous claim that $\pi_{\Gamma}(U) \cap \pi_{\Gamma}(X') = \pi_{\Gamma}(U \cap X')$.
Noting that $\pi_{\Gamma}|_U \colon U \to \pi_{\Gamma}(U)$ is a diffeomorphism by the choice of $U$, we see that
$\pi_{\Gamma}(X')$ is a regular submanifold of $\Gamma \backslash X$. 
Since $\pi_{\Gamma}|_{X'}=\pi\circ\pi_{\Gamma'}$,  
the map $\pi$ is locally diffeomorphic to the regular submanifold $\pi_{\Gamma}(X')$.
From the above, our assertion is proved. 
\end{proof}

The following lemma is proved in Millson~\cite[Sect.~4, p.245]{Millson76}, where $\tilde{M}\supset F$, $\gamma_{1},\ldots,\gamma_{r}$,
and $\sigma$ therein correspond to our 
$M\supset N$, $\alpha_{1},\ldots,\alpha_{k}$,
and $\tau$, respectively. 
\begin{lemma}
\label{lemma:millson-main}
    Let $M$ be an orientable connected Riemannian manifold,
    $\tau$ an involutive isometry on $M$, 
    and $N$ an orientable connected hypersurface of $M$
    such that $\tau$ is the identity on $N$.
    Suppose that we are given
isometries $\alpha_{1},\ldots,\alpha_{k}$ on $M$ satisfying the following conditions:
    \begin{enumerate}[label=$(\alph*)$]
        \item 
        \label{item:disjoint}
        $\alpha_{i}(N)\cap N=\emptyset$ and 
        $\alpha_{i}(N)\cap \alpha_{j}(N)=\emptyset$ for 
        any $i\neq j$;
        \item 
        \label{item:stable}
        $\tau(\alpha_{i}(N))= \alpha_{i}(N)$ for any $i$; 
        \item 
        \label{itsm:inverse}
        $\tau$ preserves the orientation of
        $\alpha_{i}(N)$ for any $i$.
    \end{enumerate}
    Then $M\smallsetminus (\alpha_{1}(N)\cup\cdots\cup \alpha_{k}(N))$ is connected.
\end{lemma}

We are ready to show Proposition~\ref{prop:Ck}. 
\begin{proof}[Proof of Proposition~\ref{prop:Ck}]
Recall Proposition~\ref{proposition:discrete-cocompact}.  
We shall write $\F_{\sigma_{1}}$ simply as $\R$ and let us define 
\[
L=\Spin_{V}(\R)\simeq Spin(n,1).
\]
 
Now we fix the embeddings of the subgroups  
\[
L' = Spin(n-1,1), \quad L_K = Spin(n), \quad L'_K = L' \cap L_K = Spin(n-1)
\]  
into $L$.  
Recall that $V'$ is the $\F$-subspace of $V$ generated by $e_1, \dots, e_n$,
where $\{e_{0},\ldots,e_{n}\}$ forms an $\F$-basis of $V$.  
Hence, the Clifford algebra $C_{\even}(V') \equiv C_{\even}(V', Q|_{V'})$ is naturally embedded in $C_{\even}(V)$.  
This allows us to regard $C_{\even}(\Lambda')$ 
as an $\mathcal{O}$-subalgebra of $C_{\even}(\Lambda)$  
and $\Spin_{V'}$ as an $\F$-subgroup of $\Spin_{V}$.  
We then define 
\[
L' := \Spin_{V'}(\R)\simeq Spin(n-1,1).
\]
Similarly, we define $L_K$ and $L_K'$ as the Lie groups of $\R$-points of  
the spin groups associated with the $\F$-subspaces of $V$  
generated by $\{e_0, \dots, e_{n-1}\}$ and $\{e_1, \dots, e_{n-1}\}$, respectively.  
Then we have $L_K\simeq Spin(n)$ and $L_K'\simeq Spin(n-1)$.  

We consider the ideal of $\mathcal{O}$:
\[
I=P_{-1}P_{0}\cdots P_{m}.
\]
By Proposition~\ref{proposition:discrete-cocompact} and Lemma~\ref{lemma:torsion-free}, 
the quotient spaces 
\[
X_{\Lambda}(I):= \Gamma_{\Lambda}(I)\backslash X \text{ and } 
X'_{\Lambda'}(I) := \Gamma_{\Lambda'}(I)\backslash X'
\]
are both orientable connected compact hyperbolic manifolds,
where $X=L/L_{K}$ and $X'=L'/L'_{K}$.

To show that $X'_{\Lambda'}(I)$ is a hypersurface of $X_{\Lambda}(I)$,  
we apply Lemma~\ref{lem:jaffe} to $(G,H,\Gamma) = (L,L_{K},\Gamma_{\Lambda}(I))$.  
For this purpose, now we verify that the assumptions of Lemma~\ref{lem:jaffe} are satisfied.  

Recall that elements of $V$ are regarded as elements of the Clifford algebra $C(V) = C(V,Q)$.  
Let $\tau\colon C_{\even}(V) \to C_{\even}(V)$ be the involutive $\F$-automorphism defined by  
\[
\tau(x) = e_0 x e_0^{-1}.
\]  
It is straightforward to verify that  
\[
C_{\even}(V') = \{ x \in C_{\even}(V) \mid \tau(x) = x \}.
\]  
In particular, we have 
\begin{equation}
    \label{eq:L'-centralizer}
    L'=\{x\in L\mid \tau(x)=x\}.
\end{equation}
Since $Q(e_{0})\in \mathcal{O}^{\times}$, 
the map $\tau$ preserves $C_{\even}(\Lambda)$. Hence, we also obtain
$C_{\even}(\Lambda') = \{x\in C_{\even}(\Lambda)\mid \tau(x)=x\}$, 
and thus
\begin{equation}
    \label{eq:Gamma'-centralizer}
    \Gamma_{\Lambda'}(I) = \Gamma_{\Lambda}(I)\cap L'. 
\end{equation}
Similarly, we have $L_{K}=\{x\in L\mid e_{n}xe_{n}^{-1}=x\}$
and $L'_{K}=L_{K}\cap L'$.

The map $\tau\colon C_{\even}(V) \to C_{\even}(V)$ 
defines an involutive automorphism of $L$,  
which we also denote by $\tau \in \Aut(L)$.  
It is clear that $\tau$ preserves both $\Gamma_{\Lambda}(I)$ and $L_K$.  
Thus, by \eqref{eq:L'-centralizer} and \eqref{eq:Gamma'-centralizer},  
Lemma~\ref{lem:jaffe} implies that $X'_{\Lambda'}(I)$  
can be naturally identified with a totally geodesic hypersurface of $X_{\Lambda}(I)$.  

Next, to complete the proof of our assertion, 
we apply Lemma~\ref{lemma:millson-main} to $(M,N) = (X_{\Lambda}(I), X'_{\Lambda'}(I))$. 
For this purpose, we find isometries $\alpha_1, \dots, \alpha_k$ of $X_{\Lambda}(I)$  
that satisfy conditions \ref{item:disjoint}--\ref{itsm:inverse} in Lemma~\ref{lemma:millson-main}.  

Before proceeding, we make the following observations.  
By Lemma~\ref{lemma:torsion-free}, 
the group $\Gamma_{\Lambda}(P_{-1}P_{0})$ is torsion-free.  
Moreover, by an argument similar, we have $\Gamma_{\Lambda'}(P_{-1}P_{0})=\Gamma_{\Lambda}(P_{-1}P_{0})\cap L'$ (see \eqref{eq:Gamma'-centralizer}) and 
Lemma~\ref{lem:jaffe} can also be applied to 
$(G,H,\Gamma) = (L,L_K,\Gamma_{\Lambda}(P_{-1}P_{0}))$.  
Thus, by the injectivity result of Lemma~\ref{lem:jaffe},  
for each $\alpha \in \Gamma_{\Lambda}(P_{-1}P_{0})$, we see that
\begin{equation}
\label{eq:jaffe'}
\alpha(X') \cap X' \neq \emptyset  
\text{ if and only if } \alpha \in \Gamma_{\Lambda'}(P_{-1}P_{0}).
\end{equation}  

We define finite groups $\Phi$ and $\Phi'$ as  
\[
\Phi := \Gamma_{\Lambda}(P_{-1}P_{0}) / \Gamma_{\Lambda}(I)\text{ and }\Phi' := \Gamma_{\Lambda'}(P_{-1}P_{0}) / \Gamma_{\Lambda'}(I).
\]  
The natural left actions of the finite group $\Phi$  
and the involution $\tau \in \Aut(L)$ on  
the Riemannian manifold $X_{\Lambda}(I)$  
are isometric.  
Moreover, the action of $\Phi$ preserves the orientation of $X_{\Lambda}(I)$,  
while the involution $\tau \in \Aut(L)$ reverses it.  

In what follows, we seek isometries $\alpha_1, \dots, \alpha_k$  
satisfying conditions \ref{item:disjoint}--\ref{itsm:inverse} in Lemma~\ref{lemma:millson-main}
among the elements of $\Phi$.  
By \eqref{eq:jaffe'}, we observe that for each $\alpha \in \Gamma_{\Lambda}(P_{-1}P_{0})$,
\begin{equation}
\label{eq:jaffe}
\alpha(X'_{\Lambda'}(I)) \cap X'_{\Lambda'}(I) \neq \emptyset  
\text{ if and only if }\alpha \in \Gamma_{\Lambda'}(P_{-1}P_{0}).
\end{equation}  

On condition~\ref{item:disjoint}:
Let $\alpha,\beta\in \Gamma_{\Lambda}(P_{-1}P_{0})$.
By \eqref{eq:jaffe}, we see that 
$\alpha(X'_{\Lambda'}(I))\cap \beta(X'_{\Lambda'}(I))\neq \emptyset$
if and only if $\alpha\equiv\beta\mod \Gamma_{\Lambda'}(P_{-1}P_{0})$. 

On condition~\ref{item:stable}: 
Let $\alpha \in \Gamma_{\Lambda}(P_{-1}P_{0})$.
For $x'\in X'_{\Lambda'}(I)$, we have 
\begin{equation}
    \label{eq:b}
    \tau\cdot(\alpha\cdot x')=\tau(\alpha)\cdot(\tau\cdot x')
    = (e_{0}\alpha e_{0}^{-1})\cdot x'.
\end{equation}
By \eqref{eq:jaffe}, we see that 
$\tau(\alpha(X'_{\Lambda'}(I)))=\alpha(X'_{\Lambda'}(I))$
if and only if $e_{0}\alpha e_{0}^{-1}\equiv \alpha \mod \Gamma_{\Lambda'}(P_{-1}P_{0})$.

On condition~\ref{itsm:inverse}: 
Let $\alpha \in \Gamma_{\Lambda}(P_{-1}P_{0})$.
In our setting, 
condition \ref{itsm:inverse} holds automatically under condition \ref{item:stable}.  
That is, we show that 
if $\alpha(X'_{\Lambda'}(I))$ is $\tau$-stable,
then $\tau$ preserves the orientation of $\alpha(X'_{\Lambda'}(I))$. 
By \eqref{eq:b}, the action of $\tau$ on $\alpha(X'_{\Lambda'}(I))$ coincides with the action of $e_{0}\alpha e_{0}^{-1}\alpha^{-1}\in \Gamma_{\Lambda'}(P_{-1}P_{0})$. This preserves the orientation of $\alpha(X'_{\Lambda'}(I))$.

Therefore, in order to apply Lemma~\ref{lemma:millson-main}, 
it suffices to find
$\alpha_{1},\ldots,\alpha_{k}\in \Gamma_{\Lambda}(P_{-1}P_{0})$
satisfying the following conditions:
\begin{enumerate}[label=(\roman*)]
\item 
\label{item:a}
$[\alpha_{i}]\not\equiv 1 \mod \Phi'$ and 
$[\alpha_{i}]\not\equiv[\alpha_{j}]\mod \Phi'$ for any $i\neq j$;
\item 
\label{item:b}
    $[e_{0}\alpha_{i}e_{0}^{-1}]\equiv [\alpha_{i}]\mod \Phi'$,
\end{enumerate}
where $[\alpha]$ means the element of $\Phi=\Gamma_{\Lambda}(P_{-1}P_{0})/\Gamma_{\Lambda}(I)$
represented by $\alpha\in \Gamma_{\Lambda}(P_{-1}P_{0})$. 

To this end, we analyze the finite set $\Phi / \Phi'$. 
Since $\dim_{\F}V'\geq 3$, 
Proposition~\ref{prop:Gamma/GammaI} can be applied to both
$\mathbf{Spin}_{V}$ and $\mathbf{Spin}_{V'}$. Hence, we have
\[
\Phi \simeq \prod_{i=1}^{m} Spin(\overline{\Lambda}_{P_{i}})
\text{ and }
\Phi' \simeq \prod_{i=1}^{m} Spin(\overline{\Lambda'}_{P_{i}}),
\]  
and thus obtain 
\begin{equation*}
    \Phi / \Phi' \simeq  
    \prod_{i=1}^{m} \big(Spin(\overline{\Lambda}_{P_{i}}) / Spin(\overline{\Lambda'}_{P_{i}})\big).
\end{equation*}  

Let us consider the natural action of each \( Spin(\overline{\Lambda}_{P_{i}}) \)  
on the vector space \( \overline{\Lambda}_{P_{i}} \) over the finite field \( k_{P_{i}} \),  
which is given by  
\[
\bar{v} \in \overline{\Lambda}_{P_{i}}\mapsto \bar{g} \bar{v} \bar{g}^{-1} \in \overline{\Lambda}_{P_{i}}
\]
for \( \bar{g} \in Spin(\overline{\Lambda}_{P_{i}}) \).  
Recall that elements of \( \overline{\Lambda}_{P_{i}} \) are regarded as elements of  
the Clifford algebra \( C_{\even}(\overline{\Lambda}_{P_{i}}) \).  
This action factors through the adjoint representation  
$\Ad\colon Spin(\overline{\Lambda}_{P_{i}}) \to SO(\overline{\Lambda}_{P_{i}})$. 

Here, note that the image of $\Ad$ is the subgroup $O'$ of  
$SO(\overline{\Lambda}_{P_{i}})$ consisting of elements with \emph{spinor norm} $1$  
(see, e.g., \cite[Sect.~24,  the identity~(24.7a)]{Shimura10}).  
By Millson~\cite[Lem.~3.1]{Millson76},  the $O'$-orbit of $e_0 \bmod P_i$, 
equivalently the $Spin(\overline{\Lambda}_{P_{i}})$-orbit,  
coincides with the sphere through $e_0 \bmod P_i$ in $\overline{\Lambda}_{P_{i}}$.  
Moreover, the isotropy group of $e_0 \bmod P_i$ is $Spin(\overline{\Lambda'}_{P_{i}})$.  
Thus, we obtain the bijection  
\begin{equation}
    \label{eq:strong-approximation}
    \Phi/\Phi' \simeq \prod_{i=1}^{m}\{\bar{x}\in \overline{\Lambda}_{P_{i}} \mid Q(x)\equiv 
Q(e_{0}) \mod P_{i}\},
\end{equation}  
where $\bar{x}$ means the element of $\overline{\Lambda}_{P_{i}}$ represented by $x\in \Lambda$. 
This map assigns to each $[g] \in \Phi$ (where $g \in \Gamma_{\Lambda}(P_{-1}P_{0})$)  
the element on the right-hand side whose $i$-th component is given by  
$ge_{0}g^{-1} \bmod P_{i}$.  

In the right-hand side of \eqref{eq:strong-approximation},  
choose $k$ distinct elements from the set  
\[
\{(\pm \overline{e_{0}},\dots,\pm \overline{e_{0}})\} \smallsetminus \{(\overline{e_{0}},\dots,\overline{e_{0}})\},
\]
which consists of $2^m - 1$ elements (recall $2^{m}-1>k$).  
Let $\alpha_1, \dots, \alpha_k$ be elements of $\Phi$  
that represent the corresponding elements of $\Phi / \Phi'$. 
It is immediate that these satisfy conditions \ref{item:a} and \ref{item:b}. 

Now we define $k$ orientable connected totally geodesic closed hypersurfaces  
$N_1, \ldots, N_k$ of $M = X_{\Lambda}(I)$ by  
\[
N_i := \alpha_i (X'_{\Lambda'}(I))\ \ \  (i=1, \ldots, k).
\]  
Applying Lemma~\ref{lemma:millson-main} to  
$(M, N) = (X_{\Lambda}(I), X'_{\Lambda'}(I))$ and $\alpha_1, \dots, \alpha_k$,  
we conclude that \( N_i \cap N_j = \emptyset \) for all \( i \neq j \)  
and that $M \smallsetminus (N_1 \cup \cdots \cup N_k)$ is connected.  
Thus, $\Gamma_{\Lambda}(I)$ satisfies Condition~\ref{condition-ck} for $k$,  
which completes the proof.  
\end{proof}

\subsection{Proof of Step 2 (An overview of bending construction)}
\label{section:step2}
In Step 2, we provide a proof of general properties (Lemmas~\ref{lemma:hnn-repeat}, ~\ref{lemma:bending}, and~\ref{lemma:deform-of-hyperbolic-lattice-property}) concerning iterated HNN extensions and bending construction.  
In Step 3, we will apply these results to construct the desired small deformation. 

Fix $k\in \N_+$.
Let $\Gamma$ be a torsion-free cocompact discrete subgroup of $L=Spin(n,1)$ satisfying Condition~\ref{condition-ck}.   
Lemma~\ref{lemma:hnn-repeat} asserts that the fundamental group  
$\Gamma$ of the hyperbolic manifold $M = \Gamma \backslash X$  
can be expressed as an iterated HNN extension of length $k$.  
Lemma~\ref{lemma:bending} provides a method for constructing  
a small deformation of a representation of $\Gamma$
through $k$ iterations of the bending construction.  
Furthermore, Lemma~\ref{lemma:deform-of-hyperbolic-lattice-property}  
provides a lower bound of the Zariski-closure of such a small deformation.

\begin{proof}[Proof of Lemma~\ref{lemma:hnn-repeat}]
The proof follows by repeatedly applying Proposition~\ref{prop:hnn}.  

We proceed by induction on $k$.  
For $k = 1$, the claim follows immediately by applying Proposition~\ref{prop:hnn} to our setting.  

Suppose $k \geq 2$, and assume that the assertion holds for $k - 1$.  
We set $S' := M \smallsetminus (N_2 \cup \dots \cup N_k)$. 
For each $i=2,\ldots,k$,
the two homomorphisms  
$j_{i,+}, j_{i,-} \colon \pi_{1}(N_{i}, y_{i}) \to \pi_{1}(S, x_{0})$  
can be factored as  
$\pi_{1}(N_{i}, y_{i}) \to \pi_{1}(S', x_{0}) \to \pi_{1}(S, x_{0})$.  
We also denote the first maps 
$\pi_{1}(N_{i}, y_{i}) \to \pi_{1}(S', x_{0})$  
by the same symbols $j_{i,+}$ and $j_{i,-}$.

Let $F_{k-1}$ be the free group generated by $a_2, \ldots, a_k$, and let $\mathcal{N}'$ be the normal subgroup of $\pi_1(S', x_0) * F_{k-1}$  
generated by  
\[
a_i j_{i,+}([\ell]) a_i^{-1} j_{i,-}([\ell])^{-1}, \quad ([\ell] \in \pi_1(N_i,y_{i}),\ i = 2, \dots, k).
\] 
By the induction hypothesis, there is a natural isomorphism  
\[
\pi_1(M, x_0) \simeq  
 (\pi_1(S', x_0) * F_{k-1}) / \mathcal{N}'.
\]

Next, let $S:=S'\smallsetminus N_1$, and let $\mathcal{N}_1$ be the normal subgroup of $\pi_1(S, x_0) * a_1^\mathbb{Z}$  
generated by  
\[
a_1 j_{1,+}([\ell]) a_1^{-1} j_{1,-}([\ell])^{-1}, \quad ([\ell] \in \pi_1(N_1,y_{1})).
\]  
Applying Proposition~\ref{prop:hnn} to the orientable connected manifold $S'$,  
we obtain the isomorphism  
\[
\pi_1(S', x_0) \simeq (\pi_1(S, x_0) * a_1^\mathbb{Z}) / \mathcal{N}_1.
\]

Combining these isomorphisms, we have  
\[
\pi_1(M, x_0) 
\simeq (( (\pi_1(S, x_0) * a_1^\mathbb{Z}) / \mathcal{N}_1 ) * F_{k-1}) / \mathcal{N}'
\simeq (\pi_1(S, x_0) * F_k) / \mathcal{N}.
\]
The second isomorphism is derived from the universal property of the free product.  
It is clear that the above isomorphism is induced by the surjective homomorphism $\Psi$ given by (\ref{eqn:Psi_pi_1}),
which completes the proof.  
\end{proof}

Before proving Lemma~\ref{lemma:bending}, we fix some notation.  
We recall that $\pi_\Gamma \colon X \rightarrow M = \Gamma\backslash X$ is the covering map, and we fix $\tilde{x}_{0} \in X$ such that $\pi_\Gamma(\tilde{x}_{0})=x_0$, as in the paragraph preceding Lemma~\ref{lemma:bending}.
For every $i = 1, \dots, k$, recall that the loop $\nu_{i}$ (see Definition~\ref{def:nu_i})  
intersects $N_{i}$ at exactly one point, namely, at $y_{i}$.  
Let $\nu'_{i,+}$ denote the segment of the loop $\nu_{i}$ from $x_{0}$ to $y_{i}$.
Via the covering map $\pi_\Gamma$,
we lift the path $\nu'_{i,+}$ to a path in $X$ that starts at $\tilde{x}_{0} \in X$,
and we denote its endpoint by $\tilde{y}_{i} \in X$. Then, $\pi_\Gamma(\tilde{y}_{i})=y_{i} \in N_{i}$. It follows from the commutative diagram~\eqref{eq:N_{i}-isom}  
that $\tilde{y}_{i}$ lies in $\alpha_{i} X'$.  
Thus, just as the isomorphism  
$D^{M}_{\tilde{x}_{0}} \colon \pi_{1}(M, x_{0}) \simeq \Gamma$  
(see \eqref{def:deck-transformation}) is defined,  
we also define the isomorphism  
$D^{N}_{\tilde{y}_{i}} \colon \pi_{1}(N, y_{i}) \simeq \alpha_{i} L' \alpha_{i}^{-1} \cap \Gamma$.

We are ready to prove Lemma~\ref{lemma:bending}. 
\begin{proof}[Proof of Lemma~\ref{lemma:bending}]
    First, we show that for each $i=1,\ldots, k$
    and any 
    \begin{equation}
        \label{eq:image-j-i+}
        D^{M}_{\tilde{x}_{0}}(\Psi(j_{i,+}([\ell])))\in \alpha_{i}L'\alpha_{i}^{-1}\cap \Gamma
        \ \ \text{ for any $[\ell]\in \pi_{1}(N_{i},y_{i})$}. 
    \end{equation}

    To verify \eqref{eq:image-j-i+}, we consider the following diagram.
    \[
    \xymatrix{
    (\pi_{1}(S,x_{0})* F_{k})/\mathcal{N} \ar[r]^(.6){\Psi}
     & \pi_{1}(M,x_{0}) \ar[rd]^{D^{M}_{\tilde{x}_{0}}}
    &  \\ 
    & \pi_{1}(M,y_{i}) \ar[r]_{D^{M}_{\tilde{y}_{i}}}
    \ar[u]^{[\nu'_{i,+}]^{-1}(-)[\nu'_{i,+}]}
    & \Gamma \\
    \pi_{1}(S,x_{0}) \ar[uu]^{\text{natural}} & 
    \pi_{1}(N,y_{i}) \ar[l]_{j_{i,+}} \ar[r]^(.45){D^{N_{i}}_{\tilde{y}_{i}}} \ar[u]^{\text{push}} &
    \alpha_{i}L'\alpha_{i}^{-1}\cap \Gamma 
    \ar@{^{(}->}[u]_{\text{inclusion}}
    }
    \]
    The commutativity of this diagram follows from the unique lifting property of covering maps with respect to paths.  
    For the definition of $j_{i,+}$, see \eqref{def:j+-}.  
    Recall that $\nu'_{i,+}$ is the path from $x_{0}$ to $y_{i}$ along the oriented loop
    $\nu_{i}$.  From the commutativity of this diagram, \eqref{eq:image-j-i+} follows immediately.

    Now we prove our assertion.
    Put $\psi:=D^{M}_{\tilde{x}_{0}}\circ\Psi$. 
    By Lemma~\ref{lemma:hnn-repeat},
    it suffices to show that for each $i=1,\ldots,k$
    and for any $[\ell]\in \pi_{1}(N_{i},y_{i})$, 
    \[
    \varphi_{t}(\psi(a_{i}))
    \varphi_{t}(\psi(j_{i,+}([\ell])))
    \varphi_{t}(\psi(a_{i}))^{-1}
    \varphi_{t}(\psi(j_{i,-}([\ell]))^{-1}
     =1.
    \]
    Since $j_{i,+}([\ell]),j_{i,-}([\ell])\in \pi_{1}(S,x_{0})$, we have 
    \begin{align*}
        (\text{LHS})&=
        \varphi(\psi(a_{i}))e^{tv_{i}}
        \varphi(\psi(j_{i,+}([\ell])))
        (\varphi(\psi(a_{i}))e^{tv_{i}})^{-1}
        \varphi(\psi(j_{i,-}([\ell])))^{-1} \\
        &=\varphi(\psi(a_{i}j_{i,+}([\ell])a_{i}^{-1}j_{i,-}([\ell])^{-1}))=1 
    \end{align*}
    where the second equality follows from \eqref{eq:image-j-i+}
    and assumption~\eqref{assumption:lemma:bending}.
    Thus our assertion is proved.
\end{proof}

For the proof of Step 3,  
we give the following lemma concerning the Zariski-closure  
of small deformations of $\Gamma$ obtained via Lemma~\ref{lemma:bending}.
The following proof is in the same line with Kassel~\cite[Sect.~6]{Kassel12}
for $\varphi\colon SO(n,1)\rightarrow SO(n,2)$ using Johnson--Millson~\cite[Lem.~5.9]{JoMi84}. 
\begin{lemma}[{cf.\ the argument in \cite[Lem.~6.4]{Kassel12}}]
\label{lemma:deform-of-hyperbolic-lattice-property}
Let $G$ be a Zariski-connected real algebraic group.  
Under the assumptions and notation of Lemma~\ref{lemma:bending},  
assume that the group homomorphism  
$\varphi\colon \Gamma\rightarrow G$  
arises from a homomorphism  
$\varphi\colon Spin(n,1)\rightarrow G$.  
Then, for any $t \in \mathbb{R}$,  
the Zariski-closure of $\varphi_{t}(\Gamma)$  
contains $\varphi(Spin(n,1))$ and $e^{t v_i}$ for $i = 1, \dots, k$. 
\end{lemma}

\begin{proof}
Let $L_{t}$ denote the Zariski-closure of $\varphi_{t}(\Gamma)$ in $G$.

Let $\varpi \colon Spin(n,1) \to SO_{0}(n,1)$ be the double cover,  
and consider the group isomorphism  
\[
\psi := D^{M}_{\tilde{x}_{0}} \circ \Psi \colon  
\pi_{1}(S, x_{0}) * F_{k} \to \Gamma \ (\subset Spin(n,1)).
\]  
By \cite[Lem.~5.9]{JoMi84},  
$\varpi (\psi(\pi_{1}(S, x_{0})))$ is Zariski-dense in $SO(n,1)$.  
Hence, $\psi(\pi_{1}(S, x_{0}))$ is also Zariski-dense in $Spin(n,1)$.  

We recall that $\varphi \colon Spin(n,1) \to G$ is a homomorphism  
in the sense of Lemma~\ref{lemma:morphism}.  
Since $\varphi$ is continuous in the Zariski-topology,  
and since $\varphi_{t}(\psi(\pi_{1}(S, x_{0}))) = \varphi(\psi(\pi_{1}(S, x_{0})))$  
by the construction of $\varphi_{t}$ (see \ref{item:lemma:bending-S} in Lemma~\ref{lemma:bending}),  
the Zariski-closure of $\varphi_{t}(\psi(\pi_{1}(S, x_{0})))$  
contains $\varphi(Spin(n,1))$.  
Thus, $L_{t}$ contains $\varphi(Spin(n,1))$.
Then it follows that $e^{tv_{i}} \in L_{t}$.
Indeed, the inclusive relation $\varphi(Spin(n,1)) \subset L_{t}$ implies that
 $\varphi(D^{M}_{\tilde{x}_{0}}(\nu_{i})) \in L_{t}$ for every $i = 1, \dots, k$.  
Thus, it follows that  
\[
e^{tv_{i}} = \varphi(D^{M}_{\tilde{x}_{0}}(\nu_{i}))^{-1} \varphi_{t}(D^{M}_{\tilde{x}_{0}}(\nu_{i})) \in L_{t}.
\]

Thus, the proof is completed.
\end{proof}

\subsection{Proof of Step 3 (Zariski-closure of specific deformation)}
\label{section:step3}
We recall from Theorem~\ref{theorem:bending}~\ref{item:maximal-zariski-closure} 
that when $n\geq 3$, the algebraic subgroup $G^{\varphi}$ (see Definition~\ref{def:g-phi}) achieves the upper bound of the Zariski-closure of any small deformation of $\varphi|_{\Gamma}$, up to $G$-conjugacy.
In this section, we provide a proof of Proposition~\ref{prop:zariski-closure-gphi}, which states that repeating the bending construction $k$ times results in
a small deformation $\varphi_{t}(\Gamma)$ of $\varphi|_{\Gamma}$, 
whose Zariski-closure in $G$, to be denoted by $L_t$, 
 attains the upper bound $G^{\varphi}$. 
The proof requires a careful discussion to handle the general situation in which we cannot expect
$v_i^{(j)} \in \mathfrak l_t$, nor the existence of towers of symmetric pairs $\mathfrak{l}=\mathfrak{l}^{0}\subset \mathfrak{l}^{1}\subset \dotsb \subset 
\mathfrak{l}^{k}=\mathfrak{g}$, as mentioned in Remark~\ref{remark:difference-Kassel}.

Before proceeding with the proof of Proposition~\ref{prop:zariski-closure-gphi},
 we first establish a couple of auxiliary lemmas.
\begin{lemma}
\label{lem:span_generic}
Let $\F$ be a field, and 
$A$ be an associated $\F$-algebra,
and let $V$ be an $A$-module which 
is finite-dimensional over $\F$.
For $n\ge 1$, let $V^{\oplus n}_{gen}$ 
denote the set of $n$-tuples
$v_1, \ldots, v_n \in V$ such that
\[
  A v_1 +\dotsb +A v_n =V.
\]
Then $V^{\oplus n}_{gen}$ 
is Zariski-open in $V^{\oplus n}$.
\end{lemma}
\begin{example}
In the case $n=1$, an element in $V_{gen} (:=V^{\oplus 1}_{gen})$ is called a \emph{cyclic vector}.
We note that
$V_{gen} = V\smallsetminus \{0\}$ if and only in $V$ is irreducible.
\end{example}
\begin{proof}[Proof of Lemma~\ref{lem:span_generic}]
Suppose $V$ is $m$-dimensional. 
Then $V^{\oplus n}_{gen}$ is the union of the Zariski-open set
\[
\{ (v_1, \ldots, v_n)|
\dim_{\F}(\operatorname{span}_{\F}\{a_1 v_1, \ldots, a_n v_n\}) = m\}
\]
where $(a_1, \dots, a_n)$ runs over $A\times \dotsb \times A$.
\end{proof}

For a Lie algebra $\mathfrak{l}$,
we denote by $U(\mathfrak{l})_+$ the augmentation ideal of the universal
enveloping algebra of $\mathfrak{l}$.

We extend the adjoint action of the Lie algebra $\mathfrak{g}$
to an action on its universal enveloping algebra $U(\mathfrak{g})$, which we continue to denote by $\operatorname{ad}$.

We also use the following notation, which will be studied in Appendix~\ref{section:appendix-Zariski-dense-subgroups}. For a finite subset $\mathcal{X} \subset \mathfrak{g}$ and a subgroup $L \subset G$, we denote by $G(L; \mathcal{X})$ the identity component of the Zariski closure of the subgroup of $G$ generated by $L$ and the elements $e^X$ for $X \in \mathcal{X}$ (Definition~\ref{def:GLtX_Zariski_closure}). Its Lie algebra is denoted by $\mathfrak{g}(L;\mathcal{X})$. Furthermore, for each $t \in \mathbb{R}$, we define $t\mathcal{X}:=
\{tX|X\in \mathcal{X}\}$.

\begin{lemma}
\label{lemma:analytic_gt_through_ux}
Let $L\subset G$ be two Zariski-connected real algebraic groups.
 For any $X \in \mathfrak{l}$ and for any $u\in U(\mathfrak{l})_+$,
 there exists an analytic map
 \[
 a\equiv a_{X,u} \colon \R \to \mathfrak{g}
 \]
 such that
 $a(t) \in\mathfrak{g}(L; t\{X\})$ 
 for all $t \in \R\smallsetminus\{0\}$
 and
 $a(0)= \operatorname{ad}(u)X$.
\end{lemma}
\begin{proof}
We fix $X\in \mathfrak{l}$. 
The set of $u\in U(\mathfrak{l})_+$ for which
there exists such an analytic map $a \colon \R \to \mathfrak{g}$ with $a(0)=\operatorname{ad}(u)X$
forms a vector space.
 Thus, it suffices to show the claim in the case $u=Y_1\dotsb Y_n$ for some
 $Y_1, \ldots, Y_n \in U(\mathfrak{l})$ with $n \ge 1$. We prove this by the induction on $n$.

First, consider the case $n=1$. When $u=Y\in \mathfrak{l}$, we set
 \[
 a(t):=\frac{1}{t} (\operatorname{Ad}(e^{-tX}) Y-Y).
 \]
 Clearly, $a\colon \R \to \mathfrak{g}$ is an analytic map with $a(0) = - [X,Y] =
 \operatorname{ad}(Y) X$. Moreover, since $Y \in \mathfrak{l}$ and $e^{t X} \in G(L;t \{ X\})$, we conclude that $a(t) \in \mathfrak{g}(L; t \{X\})$ for any $t \neq 0$.

Next, suppose that $v=Y_2\dotsb Y_n\in U(\mathfrak l)_+$ and that there exists an analytic map $b \colon \R \to \mathfrak{g}$ such that $b(t) \in \mathfrak{g}(L; t \{X\})$
for any $t\neq 0$ and  $b(0)= \operatorname{ad}(v)X$.
Let $u := Y_1 v$ where $Y_1 \in \mathfrak{l}$.
We set
\[
 a(t):=\frac{1}{t} (\operatorname{Ad}(e^{tY_1}) b(t)-b(t)).
 \]
 Then $a\colon \R \to \mathfrak{g}$ is an analytic map and we have $a(0)=[Y_1, b(0)] = \operatorname{ad}(Y) (\operatorname{ad}(v)X) = \operatorname{ad}(Y_1 v) X$.
 Since $e^{tY} \in L$  and $b(t) \in  \mathfrak{g}(L;t \{X\})$,
we have $a(t) \in  \mathfrak{g}(L; t \{X\}) $ for any $t \neq 0$.

    By induction, the lemma is proved.
\end{proof}

\begin{proof}[Proof of Proposition~\ref{prop:zariski-closure-gphi}]
Let $L=Spin(n,1)$, and $L_{t}$ denote the Zariski-closure of $\varphi_{t}(\Gamma)$ in $G$.

First, we claim that $L_t \subset G^{\varphi}$ for every $t \in \R$.
It follows from the definition \eqref{eqn:optimal_direction_for_bending} of $v_j$ that $v_j \in \mathfrak{g}^\varphi$, where we recall
 Definition~\ref{definition:lie-g-phi} for $\mathfrak{g}^\varphi$.
Moreover, $\varphi(L)\subset G^{\varphi}$ by the definition of $G^{\varphi}$ in Definition~\ref{def:g-phi}.
Therefore, the construction of $\varphi_{t}$ in Lemma~\ref{lemma:bending} implies that $\varphi_{t}(\Gamma)\subset G^{\varphi}$, which concludes that $L_{t}\subset G^{\varphi}$.

Second, we claim that $\mathfrak{l}_{t}\supset \mathfrak{g}^{\varphi}$ 
if $t$ is sufficiently small, where we write $\mathfrak{l}_{t}$ for the (real) Lie algebra of $L_t$.
This claim will imply that $L_t = G^\varphi$
because $G^{\varphi}$ is Zariski-connected.

We recall from Lemma~\ref{lemma:deform-of-hyperbolic-lattice-property} the following two properties:
\begin{align}
    \varphi(L) &\subset L_{t},
    \label{eq:sl2-Lt}\\
    \exp(tv_{j}) &\in L_{t} \ \ \text{ for any }  1 \le j \le k.
    \label{eq:Xij-Lt}
\end{align}
We might expect from \eqref{eq:Xij-Lt} that 
$v_{j} \in \mathfrak{l}_{t}$ for a sufficiently small $t$.
However, such a statement is not true in our general setting.
Therefore, in what follows, we will proceed with a more careful discussion.

The properties (\ref{eq:sl2-Lt}) and (\ref{eq:Xij-Lt}) imply that
$G(\varphi(L);t\{v_1, \dots, v_k\})\subset L_t$
with the notation as in Definition~\ref{def:GLtX_Zariski_closure} of Appendix~\ref{section:appendix-Zariski-dense-subgroups}, and hence, we have
$\mathfrak{g}(\varphi(L), t \{v_1, \dots, v_m\}) \subset \mathfrak{l}_t$.

 Recall that $\mathfrak{g}(V_i) \ (\simeq \Hom_{L}(V_i, \mathfrak{g})\otimes V_i)$ denote the
 isotypic component of $V_i$ in $\mathfrak{g}$.
 We note that
$\mathfrak{g}(V_i)=
\mathfrak{g}^\varphi(V_i)$
by Definition~\ref{definition:lie-g-phi} of $\mathfrak{g}^\varphi$ for any $i \in \N$.
Let
 \[
 \operatorname{pr}_{i} \colon\mathfrak{g}\rightarrow 
 \mathfrak{g}(V_{i})
 \]
denote the projection to $\mathfrak{g}(V_i)$, 
see (\ref{eq:decomposition-g-even}).
Then, $\operatorname{pr}_i$ is an $L$-homomorphism.
Moreover, for $1 \le i$,
by the definition \eqref{eqn:optimal_direction_for_bending} of $v_j$, we have 
$\operatorname{pr}_i(v_j)=v_i^{(j)}$
for any $1 \le j \le [\mathfrak{g}:V_i]$.

We note that there are only finitely many non-zero $v_i^{(j)}$ because $\dim \mathfrak{g}<\infty$.
Since $V_i$ is not the trivial representation for any $i \ge 1$, there exists
$Y \in \mathfrak{l}$ such that
$\operatorname{ad}(Y) v_i^{(j)} \neq \{0\}$ whenever
$v_i^{(j)} \neq 0$.
In turn, we have $U(\mathfrak{l}) (\operatorname{ad}(Y) v_i^{(j)}) = V_i^{(j)}$ by the irreducibility of $V_i^{(j)} \simeq V_i$, and consequently
\[
\operatorname{pr}_i(\sum_{j=1}^m U(\mathfrak{l}) [Y,v_j])
=
\sum_{j=1}^m U(\mathfrak{l}) [Y, v_i^{(j)}] 
=\sum_{j=1}^m V_i^{(j)}
=
\mathfrak{g}(V_{i}).
\]

For each $i \ge 1$,
let $a_i \colon \R \to \mathfrak{g}$
be an analytic map, as given in Lemma~\ref{lemma:analytic_gt_through_ux} in Appendix~\ref{section:appendix-Zariski-dense-subgroups}, satisfying $a_i(0)=[Y, v_i]$ and such that 
$
a_i(t) \in \mathfrak{g}(\varphi(L); t \{v_i\}) 
$ for any  $t \neq 0$.
Then, for any sufficiently small $t \neq 0$, we have

\begin{equation*}
\operatorname{pr}_i(\sum_{j=1}^m U(\mathfrak{l}) a_j(t))
=
\sum_{j=1}^m U(\mathfrak{l})\operatorname{pr}_i( a_j(t))
=
\mathfrak{g}(V_{i}).
\end{equation*}
by Lemma~\ref{lem:span_generic}.
Since the $L$-module
$\sum_{j=1}^m U(\mathfrak{l}) a_j(t)$ is contained in
$\mathfrak{g}(\varphi(L); t \{v_1, \dots, v_m\})
\subset 
\mathfrak{l}_t$, we have shown, for any $i \in \N_+$,
\begin{align}
\label{eq:lt-Vij}
\mathfrak{l}_t\supset \mathfrak{g}(V_{i}).
\end{align}

Now we prove (\ref{eq:lt-Vij}) for $i=0$, that is,
$\mathfrak{l}_{t}\supset \mathfrak{z}_{\mathfrak{g}}(d\varphi(\mathfrak{l}))$.

Let $L'$ denote the Zariski closure
of the group generated by
$\varphi(L)$ and $\exp(\mathfrak{g}(V_i))$ for all $i \in \N_+$.
It follows from \eqref{eq:lt-Vij} that $L_t \supset L'$.
Moreover, the identity component $Z_G(\varphi(L))$ of
the centralizer normalizes $L'$.
Therefore, it follows from the Baker--Campbell--Hausdorff formula that
\[
\exp(-t u_j) \exp(t v_j)
= \exp(-t u_j) \exp(t (u_j+\sum_{i} v_i^{(j)})) 
\in L'   \ (\subset L_t).
\]
because $u_j \in \mathfrak{z}_{\mathfrak g}(d \varphi(\mathfrak{l}))$
and $\sum_{i} v_i^{(j)} \in \mathfrak{l}'$.
Thus  $\exp (t u_j) \in L_t$
because $\exp(t v_j) \in L_t$.
By the choice of $u_1, \dots, u_{\eta(Z_G(\varphi(L)))}$,
 $L_t$ contains $Z_G(\varphi(L))
=G(t\{u_1, \ldots, u_{\eta(Z_G(\varphi(L)))}\})$
for sufficiently small $t>0$.

This proves (\ref{eq:lt-Vij}) in the case $i=0$.
    
Therefore, we obtain $\mathfrak{l}_{t}\supset \mathfrak{g}^{\varphi}$,
    which we wanted to prove. 
\end{proof}

\section{Proof of main results}
\label{section:deformation-cpt-non-cpt}
In this section, we study deformations of discontinuous groups for $X = G/H$,  
and in particular, we complete the proofs of the main theorems of this article,  
Theorems~\ref{theorem:local-nonstd-zariski} and \ref{theorem:group-manifold-case}.  

This section is organized as follows:  

In Section~\ref{section:preliminary-deform-discontinuous-group},  
we review preliminaries on deformations of discontinuous groups for $X = G/H$ with non-compact $H$,  
as well as results on local rigidity and the stability of proper discontinuity and cocompactness.  
In Section~\ref{section:deform-cptCK}, we complete the proof of Theorem~\ref{theorem:local-nonstd-zariski}, which answers Question~\ref{question:deform-ck} (Q1)--(Q3).  
In Section~\ref{section:proof-son1-group-manifold1}, we prove Theorem~\ref{thm:SOn1-Zariski-dense}. 
In Section~\ref{section:proof-group-manifold},  
we complete the proof of Theorem~\ref{theorem:group-manifold-case},  
which answers Question~\ref{question:deform-ck} (Q1)--(Q3) for the exotic quotient $\Gamma_{L} \backslash G / \Gamma_{H}$.  
In Section~\ref{section:deform-noncptCK},  
we construct discontinuous groups for certain homogeneous spaces $X=G/H$ which are Zariski-dense subgroups of $G$ 
and have cohomological dimension between $2$ and the non-compact dimension $d(X)$,  
in connection with problem~\ref{mainquestion'}.

\subsection{Preliminaries on deformations of quotients 
\texorpdfstring{$\Gamma \backslash X$}{X/Gamma}
}
\label{section:preliminary-deform-discontinuous-group}

Let $X=G/H$ be a homogeneous space of a Lie group $G$ with $H$ non-compact.
In this section, we summarize some results of the local rigidity of discontinuous groups, and also discuss when local rigidity fails, along with how small deformations of discrete subgroups can preserve (or sometimes not preserve) the properties of the action, such as proper discontinuity and cocompactness.

Since proper discontinuity may change under small deformations in the general setting, the following concept introduced
by Kobayashi and Nasrin will be useful in clarifying the subsequent discussion.

\begin{definition}[{Stability of proper discontinuity, \cite{KobayashiNasrin06,Kobayashi2006-kokyuroku}}]
\label{def:deformation_space_for_X}
    Let $\varphi\in \mathcal{R}(\Gamma,G;X)$, as defined in Section~\ref{section:deformation_compact_CK-def}.
    The $\Gamma$-action on $X$ via $\varphi$
    is called \emph{stable as a discontinuous group for $X$} 
    (under any small deformation) if 
    $\mathcal{R}(\Gamma,G;X)$ is a neighborhood of $\varphi$ in
$\Hom(\Gamma,G)$, that is, 
if any small deformation $\varphi'$ of $\varphi$ in $\Hom(\Gamma,G)$
is faithful and discrete, and if the $\Gamma$-action on $X$ through $\varphi'$ 
is properly discontinuous. 
\end{definition}

Let us make a simple observation on the relationship between local rigidity for $\mathcal{R}(\Gamma,G;X)$ (see Definition~\ref{def:locally_rigid_for_X}) and
local rigidity for $\mathcal{R}(\Gamma,G)\equiv \mathcal{R}(\Gamma,G;G/\{e\}) $ (see Section~\ref{section:(G,Gamma)-locally-rigid}): 
\begin{lemma}
    \label{lemma:locally-rigid}
    Let $\Gamma$ be a discontinuous group for $G/H$, 
    and let $\iota\colon \Gamma\rightarrow G$ denote the inclusion map. 
    The following claims hold:
    \begin{enumerate}[label=$(\arabic*)$]
        \item 
        \label{item:locally-rigid->locally-rigid-discont}
        If $\iota$ is locally rigid, then 
        it is also locally rigid as a discontinuous group for $X$.
        \item 
        \label{item:stable->imverse-implication}
        If $\iota$ is stable as a discontinuous group, then the converse of \ref{item:locally-rigid->locally-rigid-discont} holds.
    \end{enumerate}
\end{lemma}
\begin{proof}
  \ref{item:locally-rigid->locally-rigid-discont} is straightforward,
  and \ref{item:stable->imverse-implication} follows directly from the definition of stability.   
\end{proof}

\begin{example}
In the group case where $X=G/\{e\}$, we have
 $\mathcal{R}(\Gamma,G;X)=\mathcal{R}(\Gamma,G)$.
For $\varphi\in \mathcal{R}(\Gamma,G;X)$ to be stable in the sense of Definition~\ref{def:deformation_space_for_X}, it is equivalent to the condition that any small deformation of $\varphi$ in $\Hom(\Gamma,G)$
is faithful and discrete.
\end{example}

\begin{remark}[{Weil~\cite{Weil_discrete_subgroups}}]
    Let $G$ be a (not necessarily reductive) Lie group, 
    and $\Gamma$ a cocompact discrete subgroup of $G$. 
    Then, the $\Gamma$-action on $G/\{e\}$ is stable as a discontinuous group for $G/\{e\}$.
\end{remark}

We now consider the case where $H$ is non-compact. 
    
\begin{example}
\label{ex:stability}
\begin{enumerate}[label=$(\arabic*)$]
    \item 
    \label{item:cocompact-unstable}
    (cocompact but unstable).
    Let $G:=\mathrm{Aff}(\R)$ be the affine transformation group of $X:=\mathbb{R}$.
    For $\varepsilon\neq 1$, 
    we define a homomorphism $\varphi_{\varepsilon}\colon  \Z \to G$ by mapping the generator
$1\in\Z$ to the affine transformation  $x\mapsto (1-\varepsilon)x + 1$.
Clearly, $\varphi_{0}(\Z)$ is a cocompact discontinuous group for $X=\R$, and $\varphi_\varepsilon \in\Hom(\Gamma, G)$ is faithful and discrete.
However, the properly discontinuous action of $\Gamma=\mathbb Z$ on $X=\mathbb R$ is not stable at $\varphi_0$, even though $\varphi_{\varepsilon}(\Z)$ is discrete in $G$. Indeed, $\varphi_{\varepsilon}(\Z)$
has the fixed point $\varepsilon^{-1} \in X$ for any $\varepsilon\neq 0,1$,
and consequently, the action of $\varphi_\varepsilon(\mathbb {Z})$ on $X$ cannot be properly discontinuous.
    \item (cocompact and stable).
In the setting where $X=G/H$ is of reductive type with non-compact $H$,
the first example of a \emph{stable} cocompact discontinuous group $\Gamma$, in arbitrarily high dimensions, was proven by Kobayashi~\cite{Kobayashi98}, as a solution to a conjecture by Goldman~\cite{Goldmannonstandard}. 
    \item (stability varies with a homomorphism $\varphi$).
 For some triples $(\Gamma,G,X)$,
 it is possible that $\mathcal{R}(\Gamma,G;X)$ contains both stable and unstable elements.
 (For an explicit description, see \cite{KobayashiNasrin06} by Kobayashi and Nasrin, particularly when $G$ is a nilpotent Lie group.) 
\end{enumerate}
\end{example}

In contrast to the failure of stability for cocompact properly discontinuous actions, as presented in Example~\ref{ex:stability}~\ref{item:cocompact-unstable} in the non-reductive case, we expect that proper discontinuity is an open condition in a reasonable setting where $X=G/H$ is of reductive type.
 After the solution to Goldman's conjecture \cite{Goldmannonstandard}  by Kobayashi~\cite{Kobayashi98} in 1990s, several sufficient conditions have been developed to ensure the stability property in the reductive setting
 (see, for example, Gu\'{e}ritaud--Kassel~\cite{GueritaudKassel2017maximally}, Gu\'{e}ritaud--Guichard--Kassel--Wienhard~\cite{GueritaudGuichardKasselWienhard2017anosov}, Kannaka~\cite{kannaka24}, and
 Kassel--Tholozan~\cite{KasselTholozan} (preprint), for recent work). In this article, we need: 
\begin{fact}[{Kassel~\cite[Thm.~1.3]{Kassel12}}]
\label{fact:stability-for-properness}
Let $G$ be a real reductive algebraic group, 
$X=G/H$ a homogeneous space of reductive type,
$L$ a real simple algebraic group of real rank $1$, and $\varphi\colon L\rightarrow G$ a Lie group homomorphism 
such that the $L$-action on $X=G/H$ via $\varphi$ is proper.
If $\Gamma$ is a torsion-free cocompact discrete subgroup  
of $L$, then the $\Gamma$-action on $X$ via $\varphi$ is stable as a discontinuous group for $X$.
\end{fact}

In the case where $\Gamma$ is a discontinuous group for a group manifold
$(G\times G)/\diag G$ of the form
$\Gamma=\Gamma_{H}\times \Gamma_{L}$,
$\Gamma$ is not an irreducible lattice, and thus the assumption of
Fact~\ref{fact:stability-for-properness} is not satisfied.
However, an argument in the same spirit can still be applied, and the stability of discontinuous groups holds in the following form: 
\begin{fact}[{Kassel and Kobayashi~\cite{KasselKobayashi16}}]
\label{fact:exotic-stability}
    Let $G$ be a real reductive algebraic group, 
    $H$ and $L$ two closed reductive subgroups of $G$ such that the action of $H\times L$ on $G\simeq (G\times G)/\diag G$ is proper, 
    and let $\Gamma_{H}$ and $\Gamma_{L}$ be torsion-free irreducible cocompact discrete subgroups of $H$ and $L$, respectively. 
    Then the action of $\Gamma_{H}\times \Gamma_{L}$ on $G$
    is stable as a discontinuous group for $(G\times G)/\diag G$. 
\end{fact}

Fact~\ref{fact:exotic-stability} is derived from \cite[Lem.\ 4.23]{KasselKobayashi16}, which provides a quantitative estimate of the sharpness constant in this setting.

\medskip

Finally, with respect to the cocompactness of the action, as shown by Kobayashi in \cite{Kobayashi89}, cocompactness is always preserved as long as the stability of the discontinuous group holds.

\begin{fact}[{\cite[Cor.~5.5]{Kobayashi89}}]
    Let $\Gamma$ be a cocompact discontinuous group for 
    $X=G/H$. Then for any $\varphi\in \mathcal{R}(\Gamma,G;X)$, 
    the quotient 
    $X_{\varphi(\Gamma)}$ is also compact.
\end{fact}

\medskip

We end this section by providing
  a key framework for proving the answer to Question~\ref{question:deform-ck}~(3), 
  which follows from the main result of Section~\ref{section:(G,Gamma)-deformation} along with the aforementioned stability theorem.

\begin{theorem}
\label{theorem:Zariski-dense-discontinuous-spin-lattice}
 Let $G$ be a Zariski-connected real reductive algebraic group,
 and $X=G/H$ a homogeneous space of reductive type.
 Let
$\varphi\colon Spin(n,1)\rightarrow G$ be a homomorphism such that 
the action of $Spin(n,1)$ on $X$, via $\varphi$, is proper.
 \begin{enumerate}
    \item [(1)]
There exists a pair $(\Gamma, \varphi')$ such that the Zariski-closure
of $\varphi'(\Gamma)$ equals $G^\varphi$ (Definition~\ref{def:g-phi}), where
$\Gamma$ is a torsion-free cocompact discrete subgroup 
of $Spin(n,1)$ and $\varphi'\in \mathcal{R}(\Gamma,G;X)$ is a small deformation of $\varphi|_{\Gamma}$ in $\Hom(\Gamma,G)$.
    \item [(2)]
If $\mathfrak{g}^{\varphi}=\mathfrak{g}$ (see Definition~\ref{definition:lie-g-phi}),
then $\varphi'(\Gamma)$ is Zariski-dense in $G$, where
 $(\Gamma, \varphi')$ is the pair in (1).
 \end{enumerate}
\end{theorem}
\begin{proof}
We have shown in
Theorem~\ref{theorem:bending}~(1) that there exists a pair $(\Gamma, \varphi')$ where $\Gamma$ is a cocompact discrete subgroup $\Gamma$ of $L$ without torsion and $\varphi'$ is a small deformation of $\varphi|_\Gamma$ such that $\varphi'(\Gamma)$ is Zariski-dense in  $G^\varphi$.
Then the theorem follows from
Fact~\ref{fact:stability-for-properness}
applied to $L=Spin(n,1)$.   
\end{proof}

\subsection{Proof of 
\texorpdfstring{Theorem~\ref{theorem:local-nonstd-zariski}}{Theorem 2.9}}
\label{section:deform-cptCK}
In this section, we complete the proof of Theorem~\ref{theorem:local-nonstd-zariski}.
In the first half of this section, we discuss the case where the discontinuous group is locally rigid, and in the second half, we deal with the case where it is deformable. 

As shown in \cite[Sect.~3.12]{Kobayashi98} using Weil~\cite{Weil_remark}, the vanishing of the first cohomology of $\Gamma$ serves as an effective first sufficient condition to the local rigidity of a cocompact discontinuous group $\Gamma$ acting on a homogeneous space $G/H$. This condition can be addressed without considering the subgroup $H$, see Steps 1 and 3 below.
We shall prove in Steps 4 and 5 that this approach covers all the local rigidity cases in Theorem~\ref{theorem:local-nonstd-zariski}.
That is, in the setting where $(G,H,L)$ is one of the triples in Table~\ref{tab:cpt-CK-sym}, 
there exists a cocompact discrete subgroup $\Gamma$ of $L$ for which $H^{1}(\Gamma,\mathfrak{g})\neq 0$ if and only if there exists such a $\Gamma$ that is deformable as a discontinuous group for $X=G/H$.

We recall that $L_{ss}$ is the semisimple factor of $L$.
The proof of Theorem~\ref{theorem:local-nonstd-zariski} is organized in the following five steps.

\begin{description}
    \item[Step~1] Proof of local rigidity in the cases where $\mathfrak{l}_{ss}$ is 
    isomorphic to neither $\mathfrak{so}(n,1)$ nor $\mathfrak{su}(n,1)$. 
    \item[Step~2] Irreducible decomposition of the $\mathfrak{l}_{ss}$-module $\mathfrak{g}/\mathfrak{l}_{ss}$ in the cases where 
    $\mathfrak{l}_{ss}$ is isomorphic to either $\mathfrak{so}(n,1)$ or $\mathfrak{su}(n,1)$.
    \item[Step~3] Proof of local rigidity in the cases where
    $\mathfrak{l}_{ss}$ is isomorphic to either $\mathfrak{so}(n,1)$ or $\mathfrak{su}(n,1)$
    and $H^{1}(\Gamma,\mathfrak{g})=0$ for any cocompact discrete subgroup $\Gamma$ of $L$. 
    \item[Step~4] We discuss $\Gamma$ which is
    deformable but cannot be deformed into a non-standard discontinuous group for $X$. 
    This happens when $\mathfrak{l}_{ss}$ is isomorphic to $\mathfrak{su}(n,1)$.
    \item[Step~5] We discuss $\Gamma$ which can be deformed into a Zariski-dense subgroup.
    This happens when $\mathfrak{l}$ is isomorphic to $\mathfrak{so}(n,1)$.
\end{description}

\medskip
The proof of the theorem is mainly carried out for $L=L_{ss}$.
For each $j=1,2,3$, if (Q$j$) is affirmative for $L_{ss}$, then (Q$j$) is also affirmative for all $L$ such that  
$L_{ss} \subset L \subset L_{max}$,
as in Notation~\ref{notation:compact-factor}.
For the reverse direction, 
we use the following lemma.
\begin{lemma}
\label{lem:L_c_does_not_affect}
    Let $L_c$ denote the compact factor of a real reductive group $L$ without split center.
    Then we have \[  L=L_{ss} L_{c}.  \] 
Let $\Gamma$ be a torsion-free discrete subgroup of $L$. 
Then, the quotient map
$L\rightarrow L/L_{c}\simeq L_{ss}/(L_{ss}\cap L_{c})$ is injective when restricted to $\Gamma$, and therefore,
we may regard $\Gamma$ as a discrete subgroup of the semisimple Lie group $L_{ss}/(L_{ss}\cap L_{c})$.
\end{lemma}

\begin{proof}
Since $\Gamma$ is torsion-free, the finite subgroup $\Gamma \cap L_c$ must be the singleton $\{e\}$. Hence, the assertion is clear. 
\end{proof}

Step~1. We prove the cases where $\mathfrak{l}_{ss}$ is not isomorphic to $\mathfrak{so}(n,1)$ or $\mathfrak{su}(n,1)$. The proof is parallel to \cite[Sect.~3.12]{Kobayashi98}.
Let $\Gamma$ be a torsion-free cocompact discrete subgroup of $L$.
It follows from Lemma~\ref{lem:L_c_does_not_affect} that we 
can deduce that $\Gamma$ is a torsion-free cocompact discrete 
subgroup of the semisimple Lie group $L_{ss}/(L_{ss} \cap L_c)$.
Since $\mathfrak{l}_{ss}$ is not isomorphic to 
 $\mathfrak{so}(n,1)$ or $\mathfrak{su}(n,1)$,
we apply Raghunathan's vanishing theorem (Fact~\ref{fact:raghunathan}),
which implies that we have $H^{1}(\Gamma, \mathfrak{g}) = 0$. 
    This implies that the inclusion map $\iota\colon \Gamma \to G$ is locally rigid by a theorem of Weil~\cite{Weil_remark} in Section~\ref{section:(G,Gamma)-locally-rigid}.
    Therefore, by Lemma~\ref{lemma:locally-rigid}~\ref{item:locally-rigid->locally-rigid-discont}, $\iota$ is locally rigid as a discontinuous group for $G/H$.
    This gives a proof of Theorem~\ref{theorem:local-nonstd-zariski} in Cases~1'-1, 2-1, 5-1, 5', 6', 7, 8', 9, 11', and 12 in Table~\ref{tab:cpt-CK-sym}.
    
The remaining cases are written out in Table~\ref{tab:irr-dec} for the next steps.

\medskip

Step~2. Suppose that $\mathfrak{l}_{ss}$ is isomorphic to $\mathfrak{so}(n,1)$ or $\mathfrak{su}(n,1)$. 
As preparation for the following steps, we will provide the irreducible decomposition of $\mathfrak{g}/\mathfrak{l}_{ss}$ as an $\mathfrak{l}_{ss}$-module in the following lemma,
for the cases listed in Table~\ref{tab:irr-dec}:
\begin{lemma} \label{lemma:irr-dec} In the cases listed in Table~\ref{tab:irr-dec}, the irreducible decomposition of the real vector space \(\mathfrak{g}/\mathfrak{l}_{ss}\) 
as an \(\mathfrak{l}_{ss}\)-module is given. 
Here, \(\mathbb{R}^{n,1}\) denotes the standard representation of
\(\mathfrak{l}_{ss} = \mathfrak{so}(n,1)\) or \(\mathfrak{spin}(n,1)\), 
and \(\mathbb{C}^{n,1}\) represents the standard representation of
\(\mathfrak{l}_{ss} = \mathfrak{su}(n,1)\). The symbol \(\mathbf{1}\) denotes the trivial representation. \end{lemma}

        \renewcommand{\arraystretch}{1.3}
    \begin{table}
    \begin{tabular}{c|c|c|c}
& $\mathfrak{g}$ & $\mathfrak{l}_{ss}$ & $\mathfrak{g}/\mathfrak{l}_{ss}$\\
\hline
Case 1 & $\mathfrak{su}(2n,2)$ & $\mathfrak{su}(2n,1)$ & $\C^{2n,1}+\mathbf{1}$ \\
\hline
\begin{tabular}{c}
    Case 1'-2   \\
    Case 2-2  
\end{tabular}
 & $\mathfrak{su}(2,2)$ & $\mathfrak{spin}(4,1)$ & $\R^{4,1}$ 
\\
\hline
Case 3 & 
$\mathfrak{so}(2n,2)$ & $\mathfrak{so}(2n,1)$ & $\R^{2n,1}$ \\
\hline
Case 4-1-a &
$\mathfrak{so}(2n,2)$ $(n\geq 3)$ & $\mathfrak{su}(n,1)$ & $\bigwedge^{2}\C^{n,1} + \mathbf{1}$ \\
Case 4-1-b &
$\mathfrak{so}(4,2)$ & $\mathfrak{su}(2,1)$ & $\C^{2,1} + \mathbf{1}$ \\
Case 4-2 &
$\mathfrak{so}(2,2)$ & $\mathfrak{so}(2,1)$ & $\mathbf{1}^{\oplus 3}$ \\
\hline
Case 4' &
$\mathfrak{so}(2n,2)$ & $\mathfrak{so}(2n,1)$ & $\R^{2n,1}$ \\
\hline
Case 5-2 &
$\mathfrak{so}(4,4)$ & $\mathfrak{spin}(4,1)$ & $(\R^{4,1})^{\oplus 3} + \mathbf{1}^{\oplus 3}$ \\
\hline
Case 6 &
$\mathfrak{so}(8,8)$ & $\mathfrak{spin}(8,1)$ & $\bigwedge^{3}\R^{8,1}$ \\
\hline
Case 7' &
$\mathfrak{so}(4,4)$ & $\mathfrak{so}(4,1)$ & $(\R^{4,1})^{\oplus 3} + \mathbf{1}^{\oplus 3}$ \\
\hline
Case 8 &
$\mathfrak{so}(8,\C)$ & $\mathfrak{spin}(7,1)$ & $\bigwedge^{2}\R^{7,1}$ \\
\hline
Case 9' &
$\mathfrak{so}(8,\C)$ & $\mathfrak{so}(7,1)$ & $\bigwedge^{2}\R^{7,1}$ \\
\hline
Case 10 &
$\mathfrak{so}^*(8)$ & $\mathfrak{spin}(6,1)$ & $\R^{6,1}$ \\
\hline
Case 10' &
$\mathfrak{so}^*(8)$ & $\mathfrak{su}(3,1)$ & $\bigwedge^2\C^{3,1} + \mathbf{1}$ \\
\hline
Case 11 &
$\mathfrak{so}^*(8)$ & $\mathfrak{spin}(6,1)$ & $\R^{6,1}$ \\
\hline
Case 12' &
$\mathfrak{so}(4,3)$ & $\mathfrak{so}(4,1)$ & $\R^{4,1} + \mathbf{1}^{\oplus 2}$ \\
\end{tabular}
    \caption{Irreducible decomposition of $\mathfrak{g}/\mathfrak{l}_{ss}$ as an $\mathfrak{l}_{ss}$-module.}
        \label{tab:irr-dec}
    \end{table}
    \renewcommand{\arraystretch}{1}

\medskip
Step~3. We prove Cases~6, 8, and 9' in Table~\ref{tab:irr-dec}.

 In these cases, $\mathfrak{l}_{ss}$ is isomorphic to $\mathfrak{so}(n,1)$. In light of the irreducible decomposition of $\mathfrak{g}/\mathfrak{l}_{ss}$ in Table~\ref{tab:irr-dec}, as stated in Lemma~\ref{lemma:irr-dec}, we again apply Raghunathan's vanishing theorem (Fact~\ref{fact:raghunathan}), and we have $H^{1}(\Gamma, \mathfrak{g}) = 0$ for any torsion-free cocompact discrete subgroup $\Gamma$ of $L$.
In these cases, as in Step~1, the inclusion map $\Gamma \to G$ is locally rigid as a discontinuous group for $G/H$.

\medskip
Step~4. We consider Cases~1, 4-1, and 10'.
In these cases, the following two conditions hold:
\begin{itemize}
    \item $\mathfrak{l}_{ss}$ is isomorphic to $\mathfrak{su}(n,1)$ $(n\geq 2)$;
    \item there exists a torsion-free cocompact discrete subgroup $\Gamma$ of $L$ such that $H^{1}(\Gamma,\mathfrak{g})\neq 0$.
\end{itemize}
    \begin{proposition}
    \label{prop:so*8/su(3,1)}
    Suppose that $(G,H,L)$ is one of Cases~1, 4-1, and 10'.
    \begin{enumerate}[label=(\arabic*)]
        \item 
        \label{item:su(n,1)-Q1-yes}
        ((Q1) in Question~\ref{question:deform-ck} is yes.)
   There exists a torsion-free cocompact discrete subgroup $\Gamma$ of $L$ such that $\Gamma$ is deformable as a discontinuous group for $G/H$.
   \item 
   \label{item:su(n,1)-Q2-no}
   ((Q2) in Question~\ref{question:deform-ck} is no.)
  No torsion-free cocompact discrete subgroup $\Gamma$ can be deformed into a non-standard discontinuous group for $G/H$.
    \end{enumerate}
    \end{proposition}
\begin{proof}[Proof of Proposition~\ref{prop:so*8/su(3,1)}]
Before entering the proof, 
we observe that Case~4-1-b is locally isomorphic to Case~1 with $n=1$, that is,
\[
(G,H,L_{ss})\approx (SU(2,2),Sp(1,1), SU(2,1)),
\]
whereas Case 10' is locally isomorphic to Case~4-1-a with $n=3$, that is,
\[
(G,H,L_{ss}) \approx (SO(6,2),SO(6,1), SU(3,1)).
\]

\ref{item:su(n,1)-Q1-yes}. This statement was proved in \cite[Thm.\ B]{Kobayashi98}.

\ref{item:su(n,1)-Q2-no}. 
Using the notation introduced in Notation~\ref{notation:compact-factor}, in all cases under consideration, $L_{max}$ is isomorphic to $U(n,1)$ up to finite covering. Accordingly, $\Gamma$ is cocompact in $L_{max}$, and for the purpose of  our proof, we may assume without loss of generality that $\mathfrak{l}$ is isomorphic to $\mathfrak{u}(n,1)$.

We begin with the Cases 4-1-a and 10', in which the first cohomology group $H^1(\Gamma, \mathfrak{g}/\mathfrak{l})=0$ vanishes.

In Case~4-1-a, $\mathfrak{g}/\mathfrak{l}$ is isomorphic to $\bigwedge^2 \mathbb{C}^{n,1}$ as a $\mathfrak{u}(n,1)$-module.
Therefore, we have $H^1(\Gamma, \mathfrak{g}/\mathfrak{l})=0$
 by the vanishing theorem of Raghunathan~\cite{Raghunathan65}, see Fact~\ref{fact:raghunathan}. 
 We apply Proposition~\ref{prop:local-rigid-g/l} in Appendix~\ref{section:deform-upper-bound} to $\mathfrak{l} = \mathfrak{u}(n,1)$ and conclude that any homomorphism $\Gamma \to G$ which is sufficiently close to the initial inclusion $\Gamma \to U(n,1)$ remains inside $U(n,1)$, up to conjugation by elements of $G$.

 In Case~10', $\mathfrak{g}/\mathfrak{l}$ is isomorphic to $\bigwedge^2 \mathbb{C}^{3,1}$ as a $\mathfrak{u}(3,1)$-module.
 Hence, the proof is similar to Case~4-1-a. 

\medskip
We next consider Cases 1 and 4-1-b. In these case, the first cohomology group does not necessarily vanish, so the analysis is more involved.
In these cases, the pair $(L, G)$ is isomorphic finite covering of the pair
\[
  (U(n,1), SU(n, 2)).
\]

We prove the claim in the  spirit of the original work of Goldman--Millson~\cite{goldman-millson87}, which treated the pair $(SU(n, 1), SU(n+1, 1))$.
Applying their argument to the $c$-dual symmetric pair
 of $(U(n, 1), SU(n+1, 1))$, namely $(L, G)=(U(n, 1), SU(n, 2))$, we confirm that the argument extends to pairs that are only locally isomorphic. 
This is carried out in the following two steps.

The first step consists of proving Proposition~\ref{prop:Goldman-Millson}. For future reference, we give a proof in Appendix~\ref{section:deform-upper-bound} under more general assumptions, 
where we not only drop the requirement that $L \subset G$ form a reductive symmetric pair, but also eliminate the assumption that $L$ and $G$ are reductive altogether.

The second step is to show that the three assumptions of Proposition~\ref{prop:Goldman-Millson} are satisfied for the specific pair $(L, G) = (U(n,1), SU(n,2))$.
In the cases $(L, G) = (U(n,1), SU(n+1,1))$ and $(U(n,1), SU(n,2))$, the only non-trivial assumption in Proposition~\ref{prop:Goldman-Millson} is condition~\ref{item:cup-product} therein.
Condition~\ref{item:cup-product}, while formulated in terms of the cohomology of a $\Gamma$-module, can be reinterpreted—using the result of Matsushima--Murakami~\cite{Matsushima-Murakami-63}—as a statement about harmonic forms with values in a local system over the compact complex hyperbolic manifold $\Gamma \backslash \mathbb{H}_{\mathbb{C}}^n$, which facilitates its verification.
Moreover, the same argument remains valid in each step, provided the groups are locally isomorphic.

Since $U(n,1)$ acts properly on $G/H$,
this deformation remains standard.

Thus the proposition is shown.
    \end{proof}

\begin{remark}
Alternatively, the second statement of
 Proposition~\ref{prop:so*8/su(3,1)} may be verified in Cases 1 and 4-1-b when $L$ is locally isomorphic to $SU(n,1)$, by applying Klingler's local rigidity (Fact~\ref{fact:klingler}) together with Remark~\ref{remark:klingler}, which asserts in this setting that any homomorphism $\Gamma \to G$ sufficiently close to the initial inclusion $\Gamma \hookrightarrow L$ is, up to conjugation, contained in
$L \cdot Z_G(L) = U(n, 1)$.
\end{remark}

\medskip
Step~5. The remaining cases are Cases~1'-2, 2-2, 3, 4-2, 4', 5-2, 7', 10, 11, and 12'.
In these cases, the following two conditions hold:
\begin{itemize}
    \item $\mathfrak{l}_{ss}$ is isomorphic to $\mathfrak{so}(n,1)$ $(n\geq 2)$;
    \item there exists a torsion-free cocompact discrete subgroup $\Gamma$ of $L$ such that $H^{1}(\Gamma,\mathfrak{g})\neq 0$.
\end{itemize}
We take a closer look at the Lie algebra $\mathfrak{g}^\varphi$ in Definition~\ref{definition:lie-g-phi},
or the corresponding algebraic group $G^\varphi$ in Definition~\ref{def:g-phi}:
 
    \begin{proposition}
    \label{lem:l=so-Zd}
Let $(G,H,L)$ be one of Cases~1'-2, 2-2, 3, 4-2, 4', 5-2, 7', 10, 11, or 12'.
We denote by $\varphi$ the natural morphism $L_{ss}\rightarrow G$.
Then we have $\mathfrak{g}^{\varphi} =\mathfrak{g}$.
\end{proposition}

Note that Case~5-2 of Proposition~\ref{lem:l=so-Zd}  
has already been proven in Example~\ref{example:Clifford-z.d}~(1) using $G(p,q)$, 
and Cases~10 and~11 have been shown in Example~\ref{example:Clifford-z.d}~(2) using $G(p,q)$.

\begin{proof}
We recall from Definition~\ref{definition:lie-g-phi}
that $\mathfrak g^{\varphi}$ is 
the Lie algebra containing $\mathfrak{l}_{ss}$ and all the subrepresentations in the complementary subspace that contain non-zero $\mathfrak{l}_{ss}$-spherical harmonics.
It then follows immediately from Lemma~\ref{lemma:irr-dec} that $\mathfrak{g}^\varphi$ coincides with $\mathfrak{g}$
in each of the above cases.
   \end{proof}

Thus, from Theorem~\ref{theorem:Zariski-dense-discontinuous-spin-lattice}, we conclude that the answer to (Q3) is ``yes'' in these cases.

\subsection{Proof of \texorpdfstring{Theorem~\ref{thm:SOn1-Zariski-dense}}{Theorem 2.19}}
\label{section:proof-son1-group-manifold1}

In this section, we provide a proof of
Theorem~\ref{thm:SOn1-Zariski-dense}. 

For any Zariski-connected real algebraic group $G$ that is locally isomorphic to $SO(n,1)$, there is a covering homomorphism $Spin(n,1) \to G$.
Hence, it suffices to prove the case where $G=Spin(n,1)$.
We apply Theorem~\ref{theorem:Zariski-dense-discontinuous-spin-lattice} to the setting where $(G \times \{e\}, G\times G, \diag G)$ plays the role of $(L, G, H)$.

Via the natural embedding $\varphi\colon G \times \{e\} \hookrightarrow G\times G$,
the Lie algebra
$\mathfrak g \oplus \mathfrak g$
is decomposed into irreducible representations $\mathfrak{g} + \mathbf{1}^{\oplus\dim\mathfrak{g}}$ as a $\mathfrak g \oplus \{0\}$-module.

Then it follows from 
Definition~\ref{definition:lie-g-phi} that
$(\mathfrak{g}\oplus\mathfrak{g})^\varphi$ coincides with $\mathfrak{g}\oplus\mathfrak{g}$.
Thus, from Theorem~\ref{theorem:Zariski-dense-discontinuous-spin-lattice}, we conclude that there exist a torsion-free, cocompact discrete subgroup $\Gamma$ of $G$ and a small deformation $\Gamma'$ of $\Gamma \times \{e\}$ such that $\Gamma'$ is Zariski-dense in the direct product group $G \times G$, while preserving the proper discontinuity of the action on the group manifold $(G \times G)/\diag G$.   

\subsection{Proof of 
\texorpdfstring{Theorem~\ref{theorem:group-manifold-case}}{Theorem 2.21}}
\label{section:proof-group-manifold}

The majority of the proof of Theorem~\ref{theorem:group-manifold-case} in the $\Gamma_L\backslash G/\Gamma_H$ case reduces to the proof of Theorem~\ref{theorem:local-nonstd-zariski}
in the $\Gamma \backslash G/H$ case,
given in Section~\ref{section:deform-cptCK}. In addition, let us prepare the necessary lemmas and propositions in advance.

In the sequel, we assume that $G$ is a Zariski-connected real reductive algebraic group, $H$ and $L$ are real reductive algebraic subgroups of $G$, $\Gamma_H$ and $\Gamma_L$ are torsion-free cocompact discrete subgroups of $H$ and $L$, respectively, and that $\iota_H \colon \Gamma_H \hookrightarrow G$, $\iota_L \colon \Gamma_L \hookrightarrow G$ denote the natural embeddings.

The proof of the following lemma is straightforward and we omit it.
\begin{lemma}
\label{lemma:group-manifold-deformation}
Suppose that $\varphi \in \Hom(\Gamma_H \times \Gamma_L, G \times G)$
 is sufficiently close to the original embedding $\iota_H \times \iota_L \colon \Gamma_H \times \Gamma_L \to G \times G$. Then, for each group $J = H$ or $L$, there exist group homomorphisms 
\begin{itemize} \item $j_J \colon \Gamma_J \hookrightarrow G$, sufficiently close to the original $\iota_J$, \item $\rho_J \colon \Gamma_J \to G$, sufficiently close to the trivial homomorphism, 
\end{itemize} such that \begin{equation} \label{eq:lemma:group-manifold-deformation} \varphi(\gamma_H, \gamma_L) = (j_H(\gamma_H) \rho_L(\gamma_L), j_L(\gamma_L) \rho_H(\gamma_H)) \end{equation} for all $\gamma_H \in \Gamma_H$ and $\gamma_L \in \Gamma_L$.

 Moreover, we have
\begin{gather*}
\rho_L(\Gamma_L)\subset Z_G(H'),
\\
\operatorname{pr}_{1}\circ\varphi(\Gamma_H \times \Gamma_L)
\subset H'\cdot Z_{G}(H'),
\end{gather*}
where $\operatorname{pr}_1 \colon G \times G \to G$ denotes the first projection, $H'$ is the Zariski-closure of $j_H(\Gamma_H)$, and $Z_G(H')$ denotes the centralizer of $H'$ in $G$.
\end{lemma}

\begin{proposition}
    \label{proposition:exotic-non-standard}
Suppose that the direct product group
    $H\times L$ acts properly on the group manifold $(G\times G)/\diag(G) \ (\simeq G)$.
    Furthermore, we assume the following three conditions:
\begin{enumerate}[label=(\alph*)] \item \label{item:H-locally-rigid} 
Both the inclusion map $\iota_{H} \colon \Gamma_{H} \hookrightarrow G$ and the trivial representation $\mathbf{1}_{H} \colon \Gamma_{H} \to G$ are locally rigid; 
\item \label{item:L-locally-rigid-up-to-compact} 
The image of any small deformation of $\iota_{L} \colon \Gamma_{L} \hookrightarrow G$ remains inside $L$, up to conjugation by elements of $G$; 
\item \label{item:centralizer-compact} 
The centralizer $Z_{G}(H)$ of $H$ in $G$ is a compact group. \end{enumerate}

Under these assumptions, $\Gamma_{H} \times \Gamma_{L}$ 
cannot be deformed into a non-standard discontinuous group for the group manifold $(G \times G) / \operatorname{diag}(G) \simeq G$.
\end{proposition}

\begin{proof}
Suppose that $\varphi\in \mathcal{R}(\Gamma_{H}\times \Gamma_{L},G\times G;G)$ is a small deformation of $\iota_{H}\times \iota_{L}$.
It follows from Lemma~\ref{lemma:group-manifold-deformation} that $\varphi$ is
of the form
\eqref{eq:lemma:group-manifold-deformation}
for some group homomorphisms $j_{H},\rho_{H}\colon \Gamma_{H}\rightarrow G$, 
    $j_{L},\rho_{L}\colon \Gamma_{L}\rightarrow G$.

By the assumptions~\ref{item:H-locally-rigid} and \ref{item:L-locally-rigid-up-to-compact}, if necessary, we can replace $\varphi$ by an appropriate conjugation in $G \times G$, and assume that 
$j_H = \iota_H$, $\rho_H = \mathbf{1}_H$, and 
$j_L(\Gamma_L) \subset L$. Furthermore, by Borel's density theorem (e.g., \cite[Cor.~5.16~(ii)]{Raghunathan-discrete}), $\Gamma_H$ is Zariski-dense in $H$.
Applying again Lemma~\ref{lemma:group-manifold-deformation}, we obtain
    \[
    \varphi(\Gamma_H \times \Gamma_L)\subset (H\cdot Z_{G}(H))\times L.
    \]
Since $Z_G(H)$ is compact by
the assumption \ref{item:centralizer-compact}, the group $(H \cdot Z_G(H)) \times L$ acts properly on the group manifold $G$. Hence, $\varphi(\Gamma_H \times \Gamma_L)$ is a standard discontinuous group for $G$.
\end{proof}

\begin{proposition}
    \label{proposition:exotic-locally-rigid}
Retain the setting and the assumption~\ref{item:H-locally-rigid} in Proposition~\ref{proposition:exotic-non-standard}. Furthermore, suppose that the assumptions~\ref{item:L-locally-rigid-up-to-compact} and \ref{item:centralizer-compact} are strengthened by the following stronger assumptions: 

\begin{enumerate}[label=(\alph*')] 
\setcounter{enumi}{1}
\item \label{item:L-locally-rigid-up-to-compact-prime} 
$\iota_{L} \colon \Gamma_{L} \hookrightarrow G$ is locally rigid;
\item \label{item:centralizer-compact-prime} 
The centralizer $Z_{G}(H)$ of $H$ in $G$ is a finite group. \end{enumerate}

Then, $\iota_H \times \iota_L$ is not deformable as a discontinuous group for the group manifold $(G \times G) / \operatorname{diag}(G) \ (\simeq G)$

\end{proposition}
\begin{proof}
As we already know from the proof of Proposition~\ref{proposition:exotic-non-standard}, $\rho_L(\Gamma_L)$ is contained in $Z_G(H)$, which is a finite group in our setting. Hence, the small deformation $\rho_L$ must be the trivial representation. Thus, the local rigidity follows.
\end{proof}

\begin{proposition}
    \label{proposition:exotic-not-zariski-dense}
 Assume the following: 
 \begin{enumerate}[label=(\alph*)] 
  \item 
  \label{item:G-simple} 
 $G$ is simple and Zariski-connected; 
  \item 
  \label{item:H-not-Zariski-dense} 
  The inclusion map $\iota_{H} \colon \Gamma_{H} \to G$ cannot be deformed into a Zariski-dense discrete subgroup of $G$.
  \end{enumerate} 
  Under these assumptions, the direct product group $\Gamma_{H} \times \Gamma_{L}$ cannot be deformed into a Zariski-dense discrete subgroup of $G \times G$.
 \end{proposition}

\begin{proof}
Take any  $\varphi \in \mathcal{R}(\Gamma_{H} \times \Gamma_{L}, G \times G)$ that is suffiently close to the original embedding $\iota_{H} \times \iota_{L}$. Then, by Lemma~\ref{lemma:group-manifold-deformation}, there exist group homomorphisms $j_{H}, \rho_{H} \colon \Gamma_{H} \to G$ and $j_{L}, \rho_{L} \colon \Gamma_{L} \to G$ satisfying \eqref{eq:lemma:group-manifold-deformation}.
Let $H'$ be the Zariski-closure of $j_{H}(\Gamma_H)$ in $G$. 
Since $j_{H}(\Gamma_{H}) \simeq \Gamma_{H}$ is an infinite group, $H'$ has positive dimension.
On the other hand, by the assumption \ref{item:H-not-Zariski-dense}, $H'$ is a proper algebraic subgroup of $G$. 
Therefore, by the assumption \ref{item:G-simple}, we can conclude that $H' \cdot Z_{G}(H')$ is a proper algebraic subgroup of the simple algebraic group $G$. 
On the other hand, from Lemma~\ref{lemma:group-manifold-deformation}, we have
 \[
    \operatorname{pr}_{1}\circ\varphi(\Gamma_H \times \Gamma_L)
    \subset H'\cdot Z_{G'}(H').
\]
 Therefore, the image of $\varphi$ cannot be Zariski-dense.
\end{proof}

We are ready to prove Theorem~\ref{theorem:group-manifold-case}. 
\begin{proof}[Proof of Theorem~\ref{theorem:group-manifold-case}]

\ref{item:theorem:group-manifold-case:local-rigidity}. 
    The equivalence of \ref{item:G/H-deformable} and \ref{item:group-deformable}:
    It follows from Theorem~\ref{thm:zariskidense-up-to-local-isom} that collecting 
    the triples $(G, H, L)$ for which the answer to (Q1) is ``yes'' for at least one of Case~i or Case~i' in Table~\ref{tab:cpt-CK-sym} yields the classification in \ref{item:group-deformable}.

The implication from \ref{item:G/H-deformable} to \ref{item:group-case-deformable} is straightforward. In fact, if $\Gamma_L$ is deformable as a discontinuous group for $G/H$, then it is also deformable for $G/\Gamma_H$. Consequently, the direct product group $\Gamma_L \times \Gamma_H$ is deformable as a discontinuous group for the group manifold $G \times G / \operatorname{diag}(G)$.

Let us prove the implication from \ref{item:group-case-deformable} to \ref{item:group-deformable} by contraposition.
Suppose that the triple $(G,H,L)$ belongs to one of the following cases, up to switching $H$ and $L$, and up to compact factors:
    \begin{itemize}
        \item Case~5. $(SO(4n,4),SO(4n,3),Sp(n,1))$ ($n\geq 2$);
        \item Case~6. $(SO(8,8),SO(8,7),Spin(8,1))$;
        \item Case~8. $(SO(8,\C),SO(7,\C),Spin(7,1))$;
        \item Case~9. $(SO(8,\C),Spin(7,\C),SO(7,1))$.
    \end{itemize}

Let $\Gamma_H$ and $\Gamma_L$ be torsion-free cocompact discrete subgroups of $H$ and $L$, respectively. Then the triples $(G, H, L)$ satisfy all the assumptions of Propositions~\ref{proposition:exotic-non-standard} and \ref{proposition:exotic-locally-rigid}. In fact, the assumptions~\ref{item:H-locally-rigid} in Proposition~\ref{proposition:exotic-non-standard} and~\ref{item:L-locally-rigid-up-to-compact-prime} in Proposition~\ref{proposition:exotic-locally-rigid} follow from the vanishing of the first cohomologies, $H^1(\Gamma_H, \mathfrak{g}) = H^1(\Gamma_L, \mathfrak{g}) = 0$, which is verified using Lemma~\ref{lemma:irr-dec} and Raghunathan's vanishing theorem (Fact~\ref{fact:raghunathan}). Furthermore, it is clear that $Z_G(H)$ is a finite group in each case. Hence, by Proposition~\ref{proposition:exotic-locally-rigid}, $\Gamma_H \times \Gamma_L$ is locally rigid as a discontinuous group for the group manifold $G$.

\ref{item:theorem:group-manifold-case:equiv}. 
    The equivalence of
    \ref{item:G/H-nonstd} and \ref{item:group-case-examples} has been proven in Theorem~\ref{thm:zariskidense-up-to-local-isom}.
    Let us verify the implications
     \ref{item:G/H-nonstd}$\Rightarrow$\ref{item:group-case-nonstd} and
    \ref{item:group-case-nonstd}$\Rightarrow$\ref{item:group-case-examples}.

\ref{item:G/H-nonstd}$\Rightarrow$\ref{item:group-case-nonstd}: Suppose that \ref{item:G/H-nonstd} holds. Then, by Theorem~\ref{thm:zariskidense-up-to-local-isom}~\ref{item:thm:zariskidense-up-to-local-isom-nonstd}$\Rightarrow$\ref{item:thm:zariskidense-up-to-local-isom-zariski-dense}, one can find a small deformation $j_{L} \in \mathcal{R}(\Gamma, G; G/H)$ of the original embedding $\iota_{L} \colon \Gamma_{L} \rightarrow G$ such that $j_{L}(\Gamma)$ is Zariski-dense in $G$, while preserving the proper discontinuity of the action on $G/H$. We take any cocompact discrete subgroup $\Gamma_H$ of $H$, and denote by $\iota_{H} \colon \Gamma_H \hookrightarrow G$ the natural embedding map. Then, $\iota_H \times j_L$ gives an element of $\mathcal{R}(\Gamma_H \times \Gamma_L, G \times G; G)$, since the Zariski-closure of $\Gamma_H \times j_L(\Gamma_L)$ is equal to $H \times G$, which cannot act properly on the group manifold because $H$ is non-compact.
   
Thus, $\Gamma_{H}\times j_{L}(\Gamma_{L})$ is a non-standard discontinuous group for the group manifold $G\simeq (G\times G)/\diag G$.

    \ref{item:group-case-nonstd}$\Rightarrow$\ref{item:group-case-examples}:

We prove the implication from \ref{item:group-case-nonstd} to \ref{item:group-case-examples} by contraposition. The only triples that are not in the list of \ref{item:group-case-examples}, yet satisfy the condition of being deformable (i.e., fulfilling the classification in Theorem~\ref{theorem:group-manifold-case} \ref{item:theorem:group-manifold-case:local-rigidity}), are the following: \newline Cases 1 and 2: $(SU(2n,2), Sp(n,1), U(2n,1))$ ($n \geq 2$).

This triple $(G, H, L)$ satisfies all the assumptions of Proposition~\ref{proposition:exotic-non-standard}. In fact, the assumption \ref{item:H-locally-rigid} follows from Raghunathan's vanishing theorem (Fact~\ref{fact:raghunathan}), and the assumption \ref{item:centralizer-compact} can be easily verified. Moreover, the assumption \ref{item:L-locally-rigid-up-to-compact} is verified using Klingler's local rigidity theorem, which holds for $\Gamma$ as a torsion-free cocompact discrete subgroup of $SU(n,1)$, and also holds for $U(n,1)$.
Thus, we conclude from Proposition~\ref{proposition:exotic-non-standard} that \ref{item:group-case-nonstd} is not true. 
Therefore, the implication \ref{item:group-case-nonstd}$\Rightarrow$\ref{item:group-case-examples} is established.
   
\ref{item:theorem:group-manifold-case:non-zariski}. We observe that the answer to (Q3) is \lq\lq no\rq\rq\ for at least one of Case i or Case i' in Table~\ref{tab:cpt-CK-sym}, except for Case 4.2, where $G = SO(2,2)$ is not simple. Hence, since $G$ is a simple Lie group, the direct product group $\Gamma_H \times \Gamma_L$ cannot be deformed into a Zariski-dense subgroup of $G \times G$, as stated in Proposition~\ref{proposition:exotic-not-zariski-dense}.

Thus, the proof of Theorem~\ref{theorem:group-manifold-case} is complete.
\end{proof}

We have excluded the case $G=SO(2,2)$ in
Theorem~\ref{theorem:group-manifold-case}~\ref{item:theorem:group-manifold-case:non-zariski}, which is not a simple Lie group,
and the conclusion differs as stated in Proposition~\ref{prop:so(2,2)-exotic}.
We now provide a proof of Proposition~\ref{prop:so(2,2)-exotic}).
Specifically, this occurs in Cases~3 and 4 $(n=1)$, i.e., in the following case where
\[
(G,H,L)=(SO(2,2),SO(2,1), SU(1,1))
\]
up to the switching of $H$ and $L$, and up to compact factors of $H$ and $L$.
 
\begin{proof}[Proof of Proposition~\ref{prop:so(2,2)-exotic}]
Let $\iota_H\colon H \hookrightarrow G$ and $\iota_L\colon L \hookrightarrow G$ denote the natural embeddings. 
The Lie algebra $\mathfrak{g}=\mathfrak{so}(2,2)$ is decomposed into the sum of irreducible representations of $H=SO(2,1)$ and $L=SU(1,1)$ via the adjoint actions as follows:
 \[
    \mathfrak{g} \simeq \mathfrak{h}+ \R^{2,1},\ \ 
    \mathfrak{g} \simeq \mathfrak{l}+ \mathbf{1}^{\oplus 3}. 
\]
Then, we have $\mathfrak{g}^{\iota_H} = \mathfrak{g}^{\iota_L}= \mathfrak{g}$.
Applying Theorem~\ref{theorem:bending} to $H=SO(2,1)$ and $L=SU(1,1)\simeq Spin(2,1)$,
there exist torsion-free cocompact discrete subgroups $\Gamma_{H}$ and $\Gamma_{L}$
of $H$ and $L$, respectively, that can be deformed into Zariski-dense subgroups of $G$.
Therefore, the direct product group $\Gamma_{H} \times \Gamma_{L}$ also has a small deformation $\phi \colon \Gamma_{H} \times \Gamma_{L} \to G \times G$ that is Zariski-dense in $G \times G$.
By the stability of proper discontinuity (Fact~\ref{fact:exotic-stability}), the action of $\phi(\Gamma_{H} \times \Gamma_{L})$ on $G$ is properly discontinuous. Hence, $\Gamma_{H} \times \Gamma_{L}$ can be deformed into a Zariski-dense subgroup of $G \times G$, preserving the proper discontinuity of the action on the group manifold $G = (G \times G)/\operatorname{diag}(G)$.   
\end{proof}

\subsection{Some results for \texorpdfstring{Problem~\ref{mainquestion'}}{Problem 1.2}}
\label{section:deform-noncptCK}

In this section, we present several results regarding Problem~\ref{mainquestion'}, namely, the existence problem of Zariski-dense discontinuous groups $\Gamma$ for $X=G/H$, when the cohomological dimension $\operatorname{cd}_{\mathbb R}(\Gamma)$ is between 2 and $d(X)$.

First, we consider the case where $\Gamma$ is isomorphic to the surface group $\pi_{1}(\Sigma_{g})$,  
for which $\operatorname{cd}_{\mathbb{R}}(\Gamma) = 2$.  
In this case, 
by combining Theorem~\ref{theorem:Zariski-dense-discontinuous-spin-lattice} (for $n=2$)  
with Okuda's results, as discussed in Remark~\ref{remark:okuda} below,   
we obtain the following theorem.  
For the proof, see \cite{KannakaOkudaTojo24}.
\begin{corollary}[{\cite[Thm.~5.4 (v)$\Rightarrow$(iii)]{KannakaOkudaTojo24}}]
\label{cor:surface-zariski}
    Let $X=G/H$ be a symmetric space, 
    where $G$ is a Zariski-connected real reductive algebraic group.
    Assume that there exists a discontinuous group for $X$ which is isomophic to 
    a surface group $\pi_{1}(\Sigma_{g})$.
    Then, among
    such discontinuous groups,
    there exists a subgroup that is Zariski-dense in $G$.
\end{corollary}

\begin{remark}
    \label{remark:okuda}
    Several equivalent statements are known for the assumption in Corollary~\ref{cor:surface-zariski}, which we recall from Okuda~\cite{Oku13,Okuda17}. For  a semisimple symmetric space $X = G/H$,  the following six conditions are equivalent:
    \begin{enumerate}
        \item $X$ admits a proper action of $SL(2,\R)$ via a homomorphism $SL(2,\R)\rightarrow G$;
        \item 
        $X$ admits a proper action of $SL(2,\R)$ 
        via an \emph{even} homomorphism $SL(2,\R)\rightarrow G$
        (Definition-Lemma~\ref{definition-lemma:even});
        \item 
        \label{item:okuda-some-rank}
        $X$ admits a properly discontinuous action of the free group $F_{k}$ 
        of rank $k$ for \emph{some} $k \geq 2$ via a group homomorphism $F_{k}\rightarrow G$;
        \item 
        \label{item:okuda-any-rank}
        $X$ admits a properly discontinuous action of the free group $F_{k}$ 
        of rank $k$ for \emph{any} $k\geq 2$ via a group homomorphism $F_{k}\rightarrow G$;
        \item 
        \label{item:okuda-some-genus}
        $X$ admits a properly discontinuous action of the surface group $\pi_{1}(\Sigma_{g})$ of genus $g$ for \emph{some} $g\geq 2$ via a group homomorphism $\pi_{1}(\Sigma_{g})\rightarrow G$;
        \item 
        \label{item:okuda-any-genus}
        $X$ admits a properly discontinuous action of the surface group $\pi_{1}(\Sigma_{g})$ of genus $g$ for \emph{any} $g\geq 2$ via a group homomorphism $\pi_{1}(\Sigma_{g})\rightarrow G$;
    \end{enumerate}
\end{remark}
 
The above equivalence between \eqref{item:okuda-some-genus} and \eqref{item:okuda-any-genus} suggests the following question:
\begin{question}
       Let $\pi_{1}(\Sigma_{g})$ be as in the assumption of Corollary~\ref{cor:surface-zariski}. 
       Does there exist a Zariski-dense subgroup in $G$ 
    isomorphic to $\pi_{1}(\Sigma_{g})$?
\end{question}

We also do not know the answer to the following question:
\begin{question}
Can the assumption that $G/H$ is symmetric be removed in  
Corollary~\ref{cor:surface-zariski}?
\end{question}

The following result provides yet another perspective on
deformations with discontinuous groups of lower cohomological dimension
for reductive symmetric spaces.
\begin{corollary}
\label{cor:symmetric_space_cohomological_dim_Zariski_dense}
    Let $X=G/H$ be a symmetric space, 
    where $G$ is a Zarski-connected real semisimple algebraic group.
    If $X$ admits a finitely-generated discontinuous group which is Zariski-dense in $G$, then 
    there exists a discrete subgroup $\Gamma$ of $G$ with the following three properties: the comological dimension of $\Gamma$ is greater than 1,
    $\Gamma$ is Zariski-dense in $G$, and $\Gamma$ acts properly discontinuously on $X$.
\end{corollary}

\begin{lemma}
\label{lem:finite-index-Zariski-dense}
    Let $G$ be a Zariski-connected 
    real algebraic group, 
    $\Gamma$ a Zariski-dense subgroup of $G$, and
    $\Gamma'$ a finite-index normal subgroup of $\Gamma$. 
    Then $\Gamma'$ is also Zariski-dense in $G$. 
\end{lemma}
\begin{proof}
    Let $G'_{\C}$ be the
    Zariski-closure of $\Gamma'$ in 
    the set $G_{\C}$ of complex points of $G$. Then $G'_{\C}$ is a  
    Zariski-closed normal subgroup of $G_{\C}$. Hence, $G_{\C}/G'_{\C}$ is a complex algebraic group. 
    The image of the finite group
    $\Gamma/\Gamma'$ into $G_{\C}/G'_{\C}$ is Zariski-dense. Since a finite subgroup 
    of an algebraic group is Zariski-closed, we see that $G_{\C}/G'_{\C}$ is a finite group. 
    Hence, we have $\dim G'_{\C}=\dim G_{\C}$. 
    Since $G_{\C}$ is Zariski-connected, $G'_{\C}$ coincides with $G_{\C}$, which means  that $\Gamma'$ is Zariski-dense in $G$. This proves our lemma.  
\end{proof}

\begin{proof}[Proof of Corollary~\ref{cor:symmetric_space_cohomological_dim_Zariski_dense}]
Suppose that $\Lambda$ is
 a finitely-generated,  Zariski-dense discrete subgroup of $G$ such that it acts properly on $X$.
By Selberg's lemma (\cite[Lem.~8]{Selberg-lemma}), there exists a torsion-free normal subgroup $\Lambda'$ of $\Lambda$ such that $\Lambda'$ is of finite index in $\Lambda$. Since $\Lambda'$ is also Zariski-dense in $G$ by Lemma~\ref{lem:finite-index-Zariski-dense}, and since the cohomological dimension $\operatorname{cd}_{\mathbb{R}}(\Lambda')$ is the same with $\operatorname{cd}_{\mathbb{R}}(\Lambda)$, we may replace $\Lambda$ with $\Lambda'$ and assume that $\Lambda$ is torsion-free.

Since $\Lambda$ is Zariski-dense in $G$, the cohomological dimension $\operatorname{cd}_{\mathbb{R}}(\Lambda)$ of $\Lambda$ must be positive.  If $\operatorname{cd}_{\mathbb{R}}(\Lambda) \geq 2$, we can take $\Gamma = \Lambda$.
Suppose $\operatorname{cd}_{\mathbb{R}}(\Lambda)=1$.

By the Stallings theorem~\cite{Stallings-free-group}, 
$\Lambda$ must be a free group $F_k$.
If $k=1$, then $\Lambda$ is an abelian group, which cannot be Zariski-dense in the non-abelian group $G$.
If $k \ge 2$, then the proper discontinuity of the action of $\Lambda$ on $X$ implies that there exists an even homomorphism $\varphi \colon SL(2, \mathbb R) \to G$ (Definition--Lemma~\ref{definition-lemma:even})
such that $SL(2, \mathbb R)$ acts properly on $X$ via $\varphi$, by Remark~\ref{remark:okuda}.
Then, there exist a discrete subgroup in $\varphi(SL(2, \mathbb R))$ that is isomorphic to a surface group, and a small deformation of it into a Zariski-dense subgroup of $G$, 
as shown in Corollary~\ref{cor:KOT}.
This small deformation preserves the proper discontinuity of the action on $X=G/H$,
as stated by the stability theorem (Fact~\ref{fact:stability-for-properness}). 
See also Theorem~\ref{theorem:Zariski-dense-discontinuous-spin-lattice} in the case $n=2$.
\end{proof}

Next, we consider the case where the cohomological dimension of the discontinuous group $\Gamma$  
for $X=G/H$ is greater than $2$.  
As mentioned in Remark~\ref{rem:lifting_to_spin}~\ref{item:proper-action},  
there exist natural families of homogeneous spaces 
on which $SO(n,1)$ does not act properly,  
while $Spin(n,1)$ does act properly.  
These families yield a somewhat surprising connection between the existence of proper actions of $Spin(n,1)$ and the Hurwitz--Radon number, as will be detailed in the forthcoming paper
\cite{KannakaTojo25} by Kannaka and Tojo, for which we now provide a brief overview
and applications of Theorem~\ref{theorem:Zariski-dense-discontinuous-spin-lattice}.

Let $\ord_{2} \colon \mathbb{Q}^{\times} \to \mathbb{Z}$ denote the $2$-adic valuation,  
that is, for $t = 2^{m} (\text{odd}) / (\text{odd})$ with $m \in \mathbb{Z}$,  
we define $\ord_{2}(t) := m$.  
Furthermore, when $\ord_{2} t = 4a + b$ with $a \in \mathbb{Z}$ and $0 \leq b \leq 3$,  
we define $\rho(t) = 8a + 2^{b}$, which is called the \emph{Hurwitz--Radon number} of $t\in \Q^{\times}$.  
Since $\rho(t)$ depends only on $\ord_{2} t$,  
we summarize its values for powers of $2$ in Table~\ref{tab:value-HR-number}.
\begin{table}
\begin{tabular}{c||c|c|c|c|c|c|c|c|c|c|c}
    $t$ & $\ldots$ & $1/16$ & $1/8$ & $1/4$ & $1/2$ & $1$ & $2$ & $4$ & $8$ & $16$ & $\ldots$  \\
    \hline
    $\rho(t)$ & $\ldots$ & $-7$ & $-6$ & $-4$ & $0$ 
    & $1$ & $2$ & $4$ & $8$ & 9 & $\ldots$
\end{tabular}
\caption{Table of $\rho(t)$}
\label{tab:value-HR-number}
\end{table}

In \cite{KannakaTojo25}, for a pair $(\mathfrak{g}, \iota)$ of a reductive Lie algebra $\mathfrak{g}$  
and its completely reducible representation $\iota \colon \mathfrak{g} \to \mathfrak{gl}(N, \mathbb{C})$,  
a certain integer $\rho(\mathfrak{g}, \iota)$ is introduced.  
It is shown that for a classical Lie algebra $\mathfrak{g}$ and its natural representation $\iota$, the value of $\rho(\mathfrak{g}, \iota)$ can be computed using  
either the Hurwitz--Radon number $\rho(N)$ or the $2$-adic valuation $\ord_{2}(N)$.  
In particular,  for the standard representation $\iota\colon \mathfrak{so}(N,N)\rightarrow \mathfrak{gl}(2N,\C)$, 
\begin{equation*}
    \label{eq:Hurwitz-radon-so(N,N)}
    \rho(\mathfrak{so}(N,N), \iota) = \rho(N).
\end{equation*}
Furthermore, for certain families of homogeneous spaces $G/H$,  
the semisimple Lie groups that can act properly on $G/H$  
can be classified using $\rho(\mathfrak{g}, \iota)$ as follows:
\begin{theorem}[\cite{KannakaTojo25}]
\label{theorem:Kannaka_Tojo_Spin_very-even}
    Let $G/H$ be one of the classical 
    homogeneous spaces in Table~\ref{tab:very-even-symmetric},
    $\iota$ the standard representation of the classical Lie algebra $\mathfrak{g}$,
    and $L$ a connected semisimple 
    Lie group with no compact factors.
    Then $G/H$ admits a proper 
    $L$-action via a homomorphism from $L$ to $G$ if and only if 
    $L$ is isomorphic to $Spin(n,1)$
    for some $2\leq n\leq \rho(\mathfrak{g},\iota)$.  
\end{theorem}
 
\begin{table}
    \centering
    \begin{tabular}{|c|c|c|}
        \hline
        $G$ & $H$ & $\rho(\mathfrak{g},\iota)$ \\
        \hline
        $SL(2N,\R)$ & $SO(N+1,N-1)$ & $\rho(N)+1$ \\
        $SL(2N,\R)$ & $SL(2p+1,\R)\times SL(2N-(2p+1),\R)$ & $\rho(N)+1$ \\
        $SL(2N,\C)$ & $SU(N+1,N-1)$ & $2\ord_{2}(N)+3$ \\
        $SL(2N,\C)$ & 
        $SL(2p+1,\C)\times SL(2N-(2p+1),\C)$ & $2\ord_{2}(N)+3$ \\
        $SL(2N,\HA)$ & $Sp(N+1,N-1)$ & $\rho(N/2)+5$ \\
        $SL(2N,\HA)$ & 
        $SL(2p+1,\HA)\times SL(2N-(2p+1),\HA)$ & $\rho(N/2)+5$ \\
        $SO(4N,\C)$ & $SO(2p+1,\C)\times SO(4N-2p-1,\C)$ & $\rho(N/4)+7$ \\
        $SO(4N,\C)$ & $SO(2N+1,2N-1)$ & $\rho(N/4)+7$ \\
        $SU(N,N)$ & $S(U(p,p+1)\times U(N-p,N-p-1))$ & $2\ord_{2}(N)+2$ \\
        $SO(N,N)$ & $SO(p,p+1)\times SO(N-q,N-q-1)$ & $\rho(N)$ \\
        $SO^{*}(4N)$ & $SO^{*}(4p+2)\times SO^{*}(4N-4p-2)$ & $\rho(N/4)+6$ \\
        $SO^{*}(4N)$ & $U(N+1,N-1)$ & $\rho(N/4)+6$  \\
        $Sp(N,N)$ & $Sp(p,p+1)\times Sp(N-p,N-p-1)$ & $\rho(N/2)+4$  \\
        $Sp(N,\R)$ & $Sp(N-1,\R)$ & $\rho(N/2)+2$ \\
        $Sp(N,\C)$ & $Sp(N-1,\C)$ & $\rho(N/2)+3$ \\
        \hline
    \end{tabular}
    \caption{The parameters $p$ and $q$ range over all integers satisfying  
$0 \leq p < N/2$ and $0 \leq q < N$, respectively.
    }
    \label{tab:very-even-symmetric}
\end{table}

For example, the space form $SO(N,N)/SO(N-1,N)$ admits a proper action of a connected semisimple Lie group $L$ with no compact factors if and only if
$L$ is isomorphic to $Spin(n,1)$ for some $2 \leq n \leq \rho(N)$.  

Furthermore, the following lemma provides a criterion to determine which actions of $Spin(n,1)$ are proper on the homogeneous spaces $G/H$ listed in Table~\ref{tab:very-even-symmetric}.  
This lemma is proved based on the properness criterion \cite[Thm.\ 4.1]{Kobayashi89} established by Kobayashi:
\begin{lemma}[\cite{KannakaTojo25}]
\label{lem:2_equiv_n}
    Let $G/H$ be one of the classical homogeneous spaces 
    in Table~\ref{tab:very-even-symmetric}, and
    $\iota$ the standard representation of the classical group $G$. 
    Then, for any non-trivial homomorphism $\varphi\colon Spin(n,1)\to G$ $(n\geq 2)$, the $Spin(n,1)$-action on $G/H$ via $\varphi$ is proper if and only if $\iota(\varphi(-1))=-I$, where $1$ is the identity element in the Clifford algebra $C_{\even}(n,1)$ and 
    $I$ is the identity matrix. 
    In particular, 
    any subgroup of $G$ isomorphic to
    $SO_{0}(n,1)$ cannot acts properly 
    on $G/H$.
\end{lemma}

We now return to one of our main themes, Problem~\ref{mainquestion'}, and provide examples of 
 discontinuous groups $\Gamma$ for certain family homogeneous spaces $X=G/H$
that are Zariski-dense subgroups of $G$ with cohomological dimension $6$.
To be more precise, we shall take $G$ to be the identity component in the Zariski-topology, $G(p,q)_{0}$, of the group $G(p,q)$ (see Appendix~\ref{section:clifford-spin} for the definition).
We take $\Gamma$ to be the discrete subgroups of $G(p,q)_{0}$  
constructed by combining Theorem~\ref{theorem:bending}~\ref{item:bending-realize-Gphi} with Example~\ref{example:Clifford-z.d},  
which are of cohomological dimension $6$ and Zariski-dense.

As a corollary of Lemma~\ref{lem:2_equiv_n}, the following result follows immediately:
\begin{proposition}
    \label{proposition:proper-g(p,q)}
    Among the classical homogeneous spaces $G/H$ listed in Table~\ref{tab:very-even-symmetric},  
we consider those for which there exist some $p, q \in \mathbb{N}_{+}$ such that  
$G$ is isomorphic to $G(p,q)_{0}$.  
For any natural number $n \leq p$,  
the action of $Spin(n,1)$ on $G/H$,  
induced via the inclusion $Spin(n,1) \to G(p,q)_{0} \simeq G$,  
is proper.
\end{proposition}

By combining Proposition~\ref{proposition:proper-g(p,q)}, 
Example~\ref{example:Clifford-z.d},  
and Theorem~\ref{theorem:Zariski-dense-discontinuous-spin-lattice}, 
we obtain the following:
\begin{theorem}
    \label{theorem:zariski-dense-noncpt-CK}
    Every homogeneous space $X=G/H$ in Table~\ref{tab:very-even-symmetric-coh-6} admits a discontinuous group
    for $X$ which is a 
    Zariski-dense subgroup of $G$ with cohomological dimension 6.
\end{theorem}

\begin{table}
    \centering
    \begin{tabular}{|c|c|c|c|}
        \hline
        $G$ & $H$ & $p$ & $d(X)$ \\
        \hline
        $SL(16,\R)$ & $SO(7,9)$ &  & $72$ \\
        $SL(16,\R)$ & $SL(2p+1,\R)\times SL(15-2p,\R)$
        & $0\leq p \leq 3$ & $-4p^2+28p+16$ \\
        $SL(8,\HA)$ & $SL(2p+1,\HA)\times SL(7-2p,\HA)$ & $p=0,1$ & $-16p^2 + 48p +29$ \\
        $Sp(16,\R)$ & $Sp(p,\R)\times Sp(15-p,\R)$ & $0\leq p\leq 7$ & $2(p+1)(16-p)$ \\
        $Sp(16,\C)$ & $Sp(p,\C)\times Sp(15-p,\C)$ & $0\leq p\leq 7$ & $-4p^2+60p+63$ \\
        $SO(8,8)$ & $SO(p,p+1)\times SO(8-p,7-p)$ & $0\leq p\leq 3$ & $-2p^2+14p+8$ \\
        $SU(8,8)$ & $S(U(p,p+1)\times U(8-p,7-p))$ & $0\leq p\leq 3$ & $-4p^2+28p+16$ \\
        $Sp(8,8)$ & $Sp(p,p+1)\times Sp(8-p,7-p)$ & $0\leq p\leq 3$ & $-8p^2+56p+32$ \\
        $SO^{*}(8)$ & $SO^{*}(6)\times SO^{*}(2)$ & & $6$ \\
        $SO^{*}(8)$ & $U(3,1)$ & & $6$ \\
        $SO^{*}(16)$ & $SO^{*}(4p+2)\times SO^{*}(14-4p)$ & $p=0,1$ & $2(2p+1)(7-2p)$ \\
        $SO^{*}(16)$ & $U(3,5)$ & & $26$ \\
        $SO(8,\C)$ & $SO(2p+1,\C)\times SO(7-2p,\C)$ & $p=0,1$ & $(2p+1)(8-2p)$ \\
        \hline
    \end{tabular}
    \caption{Some homogeneous spaces admitting Zariski-dense discontinuous groups of cohomological dimension 6.}
    \label{tab:very-even-symmetric-coh-6}
\end{table}

As seen in the proof of \cite[Cor.~5.5]{Kobayashi89}, the quantity $d(X)-\operatorname{cd}_{\R}(\Gamma)$
represents the degree of deviation from compactness of $X_\Gamma$.

We extract examples of homogeneous spaces $X=G/H$ for which $d(X)\leq 16$ from
Table~\ref{tab:very-even-symmetric-coh-6}.
\[
X=
\left\{
\begin{array}{ll}
SO^{*}(8)/(SO^{*}(6)\times SO^{*}(2))  & (d(X)=6), \\
SO^{*}(8)/U(3,1) & (d(X)=6), \\
SO(8,\C)/SO(7,\C) & (d(X)=7), \\
SO(8,8)/SO(7,8) & (d(X)=8), \\
SO^{*}(16)/(SO^{*}(14)\times SO^{*}(2)) & (d(X)=14), \\
SL(16,\R)/SL(15,\R) & (d(X)=16), \\
SU(8,8)/U(7,8) & (d(X)=16).
\end{array}
\right.
\]
In the first two cases where 
$d(X)=6$, the quotient space  $X_\Gamma$ is compact, and these cases have already been discussed in
Theorem~\ref{thm:zariskidense-up-to-local-isom}.
In contrast,  $X_\Gamma$ is non-compact in the other cases where $d(X)>6$.

\medskip

To analyze the results obtained here, we introduce two numerical invariants related to discontinuous groups for homogeneous spaces and their deformations.
\begin{definition} 
\label{def:max_cohomological_dim_for_discontinuous_group}
Let $X=G/H$ be a homogeneous space of reductive type.
\begin{enumerate}[label=(\arabic*)]
    \item 
    Let $\delta(X)$ denote the maximum cohomological dimension $\operatorname{cd}_{\mathbb R}(\Gamma)$ among all finitely-generated discontinuous groups $\Gamma$ for $X$.

    \item 
    Let $\delta_Z(X)$ denote the maximum cohomological dimension $\operatorname{cd}_{\mathbb R}(\Gamma)$ among all finitely-generated discontinuous groups $\Gamma$ for $X$ such that $\Gamma$ is Zariski-dense in $G$.
We set $\delta_Z(X):=0$ if there is no such discontinuous group $\Gamma$.
\end{enumerate}
\end{definition}
By definition and from \cite[Cor.~5.5]{Kobayashi89}, we have the following inequalities:
\[
\delta_{Z}(X)\leq \delta(X)\leq d(X),
\]
where we recall from (\ref{eqn:dX}) that $d(X)$ is the \lq\lq non-compact dimension\rq\rq\ of $X$,
which can be explicitly computed.
If the answer to Question (Q3) in Question~\ref{question:deform-ck} is affirmative, then we have the following equality:
\[ \delta_{Z}(X)=\delta(X)=d(X).
\]

We summarize some established results and new findings in terms of these numerical invariants:
    \begin{enumerate}[label=(\arabic*)]
     \item  (\cite[Cor.~4.4]{Kobayashi89}) $\delta(X)=0$ if and only if $\rank_{\R}(G)=\rank_{\R}(H)$ (the criterion for the Calabi-Markus phenomenon).
    \item If a reductive subgroup $L$ of $G$ acts properly on $X$, then $d(L) \leq \delta(X)$.
    \item (\cite[Cor.~5.5]{Kobayashi89}) $\delta(X) = d(X)$ if and only if $X$ admits a cocompact discontinuous group.
    \item (cf.\ Benoist~\cite[Thm.~1.1]{Benoist96}) A criterion for $\delta_{Z}(X)=0$ is explicitly given in terms of the action of the Weyl group.
    \item (Theorem~\ref{thm:zariskidense-up-to-local-isom})
    $\delta_{Z}(X)=\delta(X)=d(X)$ holds when the homogeneous spaces $G/H$ are among those listed in (iii) of Theorem~\ref{thm:zariskidense-up-to-local-isom}. 
    \item (Theorem~\ref{thm:SOn1-Zariski-dense})
     When $X=(G\times G)/\diag G$ with $\mathfrak g=\mathfrak{so}(n,1)$, then $\delta_Z(X)=  d(G) = n$.
     \item (Corollary~\ref{cor:symmetric_space_cohomological_dim_Zariski_dense})
     For any semisimple symmetric space $X=G/H$, $\delta_Z(X) \neq 1$.
    \item (Theorem~\ref{theorem:zariski-dense-noncpt-CK})
      $6 \le \delta_{Z}(X)$ for all $X=G/H$ in Table~\ref{tab:very-even-symmetric-coh-6}
    \end{enumerate}

\medskip

We do not know the explicit values of $\delta_Z(X)$ and $\delta(X)$ in the general setting. It may be of interest to
provide effective lower or upper bounds for these numerical invariants.

\section{Spectral analysis of \texorpdfstring{$X_\Gamma$}{X/Gamma} 
under small deformation of discontinuous groups 
\texorpdfstring{$\Gamma$}{Gamma}
}
\label{section:spectral-analysis}

In this section, we discuss analytic perspectives of the quotient space $X_\Gamma$. 
Spectral analysis on locally symmetric spaces, beyond the classical Riemannian setting, is a rapidly developing area of research
\cite{KasselKobayashi16,KasselKobayashi2019Invariant,KasselKobayashi2020Spectral,KasselKobayashi2019standard, Kobayashi16intrinsic}, 
with many open problems remaining to be addressed. 
We focus on the aspects of this analysis from the viewpoint of deformations of the discontinuous group $\Gamma$, which we have discussed in Sections~\ref{section:deformation_compact_CK}-\ref{section:deformation-cpt-non-cpt}. 
At the end of this section, we will present a set of questions for further exploration.

\subsection{Discrete spectra on compact quotients} 
We begin by outlining some basic setup and notation. For more details, we refer to the forthcoming book \cite{KasselKobayashi2019standard}, as well as the article introducing open problems and providing a brief survey \cite{Kobayashi22conjecture}.

 Let $X=G/H$ be a homogeneous space of reductive type. 
The manifold $X$ carries a pseudo-Riemannian structure with a transitive isometry group; namely, the reductive group $G$ acts transitively and isometrically on $X$. Such a pseudo-Riemannian structure can be defined using the Killing form $B$ if $G$ is semisimple.

Let $\mathbb{D}_G(X)$ denote the algebra of $G$-invariant differential operators on $X$. Suppose that $\Gamma$ is a discontinuous group for $X$. Then, any $D \in \mathbb{D}_G(X)$ induces a differential operator $D_{\Gamma}$ on the quotient $X_{\Gamma}=\Gamma\backslash X$ via the covering map $\varpi \colon X\rightarrow X_{\Gamma}$. 
We consider the set 
$$
 \mathbb{D}(X_{\Gamma}):=\{D_{\Gamma}\mid 
D\text{ is a $G$-invariant differential operator on $X$}\}
$$ as the algebra of \emph{intrinsic differential operators} on the $(G,X)$-manifold $X_{\Gamma}$.
Clearly, the correspondence $D \mapsto D_\Gamma$ defines an isomorphism between the algebras $\mathbb{D}_G(X)$ and $\mathbb{D}(X_\Gamma)$.
For example, $X_{\Gamma}$ inherits a pseudo-Riemannian structure from $X$ via $\varpi$,
and the Laplacian $\Delta_{X_\Gamma}$ on $X_{\Gamma}$ is locally the same operator $\Delta_X$ on $X$ via the diffeomorphism $\varpi$; hence, $\Delta_{X_\Gamma}$  belongs to $\mathbb{D}(X_{\Gamma})$.
For global analysis on $X_\Gamma$, it is important to note that the Laplacian is no longer an elliptic differential operator in our setting, where the metric is not positive-definite.

From now on, we assume that $X=G/H$ is a symmetric space, where $G$ is a real reductive algebraic group.
Let $\mathfrak{q}$ be the orthogonal complementary subspace of $\mathfrak{h}$ in $\mathfrak{g}$
with respect to an $\operatorname{Ad}(H)$-invariant non-degenerate symmetric bilinear form (e.g., the Killing form if $\mathfrak g$ is semisimple) of the Lie algebra $\mathfrak{g}$. 
We fix a maximal semisimple abelian subspace $\mathfrak{j}$ of $\sqrt{-1}\mathfrak{q}$
and denote by $W$ the Weyl group of 
the restricted root system of $(\mathfrak{g}_{\C},\mathfrak{j}_{\C})$.
The Harish-Chandra isomorphism $\mathbb{D}_G(X)\simeq S(\mathfrak{j}_{\C})^{W}$ 
gives the bijection
\[
\Hom_{\C\text{-alg}}(\mathbb{D}_G(X),\C) \simeq \mathfrak{j}_{\C}^{\vee}/W\ \ 
(\chi_{\lambda}\leftrightarrow \lambda).
\] 

In this section, we explore the following object related to locally symmetric spaces $X_{\Gamma}=\Gamma\backslash G/H$, which goes beyond the classical Riemannian case:
\begin{align*}
L^{2}(X_{\Gamma},\mathcal{M}_{\lambda})
&:=\{f\in L^{2}(X_{\Gamma})\mid 
D_{\Gamma}f=\chi_{\lambda}(D)f \ \textrm{for any}\ D\in \mathbb{D}_G(X)
\}, \\
\mathrm{Spec}_{d}(X_{\Gamma})
&:= \{\lambda\in \mathfrak{j}_{\C}^{\vee}/W \mid 
L^{2}(X_{\Gamma},\mathcal{M}_{\lambda})\neq 0
\},
\end{align*}
where $\chi_\lambda\colon \mathbb{D}_G(X)\to \mathbb{C}$ is the algebra homomorphism associated to $\lambda \in \mathfrak{j}_{\C}^{\vee}/W$. 
The differential equation 
\[
D_{\Gamma}f=\chi_{\lambda}(D)f
\] is interpreted in the weak sense,
 or in the sense of distributions.
We refer to an element in
$\mathrm{Spec}_{d}(X_{\Gamma})$ as a \emph{discrete spectrum} of $X_{\Gamma}$.

\begin{example}
If $G/H$ is a symmetric space of rank one, such as $SO(p,q+1)/SO(p,q)$, then the $\C$-algebra
$\mathbb{D}_G(X)$ is generated by the 
Laplacian of the pseudo-Riemannian manifold $X$.
In this case, $\mathrm{Spec}_{d}(X_{\Gamma})$ is identified with the set of discrete spectrum of the Laplacian.
\end{example}

\subsection{Constructing stable spectra
by Poincar\'e series \texorpdfstring{\cite{KasselKobayashi16}}{[29]}
}

We discuss the discrete spectrum, which is stable under small deformations of $X_\Gamma$, following its construction in Kassel and Kobayashi~\cite{KasselKobayashi16}.

A general construction of the discrete spectrum for reductive symmetric spaces $X=G/H$ was established by 
Flensted-Jensen~\cite{FlenstedJensen1980Discrete} under the rank condition \begin{equation} \label{eq:FJrank} \rank(G/H) = \rank(K/(H \cap K)). \end{equation} 
This is a generalization of the celebrated rank condition due to Harish-Chandra, and the necessity was proved by Matsuki--Oshima~\cite{OshimaMatsuki1982description}. 
In our notation, Flensted-Jensen's theorem says that there exist countably many elements 
in $\mathrm{Spec}_{d}(X)$ in the case $\Gamma=\{ e \}$.

Let us now consider a setting where $\Gamma$ is a discontinuous group for a reductive symmetric space $X=G/H$.
A general construction of the discrete spectrum for the locally symmetric space $X_\Gamma$, when $X$ satisfies the rank condition~(\ref{eq:FJrank}) and $\Gamma$ is a \lq\lq sharp\rq\rq\ discontinuous group for $X$, was provided in Kassel and Kobayashi~\cite{KasselKobayashi16} through the introduction of a generalized Poincar\'{e} series.
It is noted that infinitely many discrete spectra
remain stable under any small deformation of a discontinuous group $\Gamma$, as formulated in the following definition:

\begin{definition}[{\cite[Def.~1.6]{KasselKobayashi16}}]
\label{def:stable-spectrum}
    Let $\Gamma$ be a discontinuous group for $G/H$ via the inclusion map $\iota\colon \Gamma\rightarrow G$.
    We say $\lambda$ is a \emph{stable (discrete) spectrum}
if there exists a neighborhood $U$ of 
    $\iota \in \mathcal{R}(\Gamma,G;X)$ such that 
    \[
\lambda \in
 \bigcap_{\varphi\in U}
    \mathrm{Spec}_{d}(X_{\varphi(\Gamma)}).
    \]
\end{definition}

We recall from Definition~\ref{def:locally_rigid_for_X}
that
$X_{\varphi(\Gamma)}$ is locally rigid if the $G$-orbit through $\varphi$ is open in $\mathcal{R}(\Gamma, G;X)$.
Since 
$\mathrm{Spec}_{d}(X_{\varphi(\Gamma)})$
remains unchanged when $\varphi$ is replaced with its $G$-conjugate, all the discrete spectra are \emph{stable} in the sense of Definition~\ref{def:stable-spectrum} if $X_{\varphi(\Gamma)}$
is locally rigid. Thus, we are interested in the case where $X_{\varphi(\Gamma)}$ is deformable.  This is how affirmative answers to (Q1)--(Q3) in Question~\ref{question:deform-ck} regarding deformability relate to the theme of this section. 

The following theorem for the existence of stable spectra is formulated
for a subgroup $L$ when it is of real-rank one, such as $Spin(n,1)$.

\begin{fact}[{\cite[Thm.~1.7]{KasselKobayashi16}}]
    \label{fact:stable-spectrum}
   Let $X=G/H$ be a reductive symmetric space satisfying the rank condition~(\ref{eq:FJrank}).
Let $L$ be a real rank one semisimple Lie subgroup of $G$, which acts properly on $X$, and
 $\Gamma$ be a torsion-free, cocompact discrete subgroup of $L$.
 Then there exist infinite stable discrete spectra for $X_\Gamma$.
\end{fact} 

The idea of the proof in \cite{KasselKobayashi16} is to construct periodic joint eigenfunctions of intrinsic differential operators by averaging eigenfunctions that decay rapidly at infinity over the $\Gamma$-orbits. 
To demonstrate the convergence of this infinite sum (the \emph{generalized Poincaré series}) and, furthermore, to prove that the total sum is non-zero, we require both analytic estimates for the eigenfunctions and geometric estimates related to the action of the discontinuous group $\Gamma$. 
The geometric estimates include counting $\Gamma$-orbits $\Gamma \cdot x$ in pseudo-balls, which is in a different setting compared to the counting estimate by Eskin--McMullen~\cite{EsMc93}, as well as a Kazhdan--Margulis type estimate for discontinuous groups.
The existence of the stable discrete spectrum follows from the fact that these estimates can be provided \emph{uniformly} with respect to the variables that appear in each of the estimates.

\medskip
Combining  Fact~\ref{fact:stable-spectrum} with
Theorem~\ref{theorem:Zariski-dense-discontinuous-spin-lattice},
we obtain the following:
\begin{theorem}
Let $G$ be a Zariski-connected real reductive algebraic group,  $X=G/H$ be a symmetric space satisfying the rank condition (\ref{eq:FJrank}), and $\varphi \colon Spin(n,1) \to G$ is a homomorphism such that the action of $Spin(n,1)$ on $X$ is proper via $\varphi$. Let $G^\varphi$ denote the algebraic subgroup of $G$, as introduced in Definition~\ref{def:g-phi}.
Then there exists a pair $(\Gamma, \varphi')$ with the following properties:
\begin{itemize}
    \item $\Gamma$ is a torsion-free cocompact arithmetic subgroup of $Spin(n,1)$,
    \item $\varphi'$ is a small deformation of $\varphi|_\Gamma$,
    \item the Zariski-closure of $\varphi'(\Gamma)$ is $G^\varphi$,
    \item  $\sharp (\mathrm{Spec}_{d}(X_{\varphi(\Gamma)}) \cap \mathrm{Spec}_{d}(X_{\varphi'(\Gamma)})) = \infty$.
\end{itemize}
\end{theorem}

As an illustration, the geometric results in Theorems~\ref{thm:7_dim_compact_space_form}
and \ref{theorem:zariski-dense-noncpt-CK}
yield the following examples, respectively:
\begin{example}
    There exists a 7-dimensional compact manifold $M$, with a pseudo-Riemannian metric of signature $(4,3)$ and constant sectional curvature $-1$, satisfying the following properties: 
\begin{itemize}
\item $\pi_1(M)$ is isomorphic to a torsion-free arithmetic cocompact discrete subgroup of $Spin(4,1)$;
\item the holonomy representation $\pi_1(M) \to SO(4,4)$ has a Zariski-dense image;
\item for a sufficiently large integer $N$,   $m(6-m)$ is a discrete spectrum
of the pseudo-Riemannian Laplacian $\Delta_{M}$ for any integer $m \ge N$.      
\end{itemize}
\end{example}
\begin{example}
    There exists a 15-dimensional non-compact manifold $M$, with a pseudo-Riemannian metric of signature $(8,7)$ and constant sectional curvature $-1$, satisfying the following properties: 
\begin{itemize}
\item $\pi_1(M)$ is isomorphic to a torsion-free arithmetic cocompact discrete subgroup of $Spin(6,1)$;
\item the holonomy representation $\pi_1(M) \to SO(8,8)$ has a Zariski-dense image;
\item for a sufficiently large integer $N$, $m(14-m)$ is a discrete spectrum
of the pseudo-Riemannian Laplacian $\Delta_{M}$ for any integer $m \ge N$.      
\end{itemize}
\end{example}

\subsection{Multiplicity of the discrete spectrum on compact quotients 
\texorpdfstring{$X_\Gamma$}{X/Gamma}
}
In our setting, which goes beyond Riemannian geometry, the Laplacian is no longer an elliptic differential operator.
As a result, the multiplicity of the discrete spectrum of the Laplacian can be infinite, 
even when the manifold $X_\Gamma$ is compact.
Thus, there is no guarantee that the multiplicity of the spectrum, 
$\dim L^{2}(X_{\Gamma},\mathcal{M}_{\lambda})$, is finite.

This section briefly reviews the latest findings regarding the multiplicity of the discrete spectrum $\lambda$, specifically,
$\dim_{\C}L^{2}(X_{\Gamma},\mathcal{M}_{\lambda})$
given by Kassel and Kobayashi. 
\begin{fact}[{\cite{KasselKobayashi_infty-mult, KasselKobayashi2019standard}}]
\label{fact:kobayashi-kassel-multiplicity}
    Let $X=G/H$ be a reductive symmetric space with $H$ non-compact.
    Assume that there exists a reductive subgroup 
    $L$ of $G$ such that $L$ acts properly on $G/H$,
    and that $L_{\C}$ acts spherically on
    the complexification 
    $X_{\C}=G_{\C}/H_{\C}$.

 Then, for any torsion-free discrete subgroup $\Gamma$ of $L$, we have 
    \[
    \dim_{\C}L^{2}(X_{\Gamma},\mathcal{M}_{\lambda})=\infty
    \]
    for every $\lambda\in \mathrm{Spec}_{d}(X)$ 
    sufficiently away from the walls.
\end{fact}

A key step in the proof is a geometric variant of the discrete decomposability theorem (\cite{Kobayashi94-dd}), which states that any
$H$-distinguished irreducible unitary representation of $G$ decomposes into irreducible unitary representations of $L$ without continuous spectrum, under the assumptions of Fact~\ref{fact:kobayashi-kassel-multiplicity}.
This is combined with  the method of generalized Poincaré series introduced in the proof of Fact~\ref{fact:stable-spectrum}.

The triples $(G,H,L)$, where $G$ is simple and $G/H$ is allowed to be non-symmetric, that satisfy the assumptions in
Fact~\ref{fact:kobayashi-kassel-multiplicity}
were classified in \cite[Table~1.1]{KasselKobayashi2019standard}.
We extract the following cases from that table:
$G/H$ satisfies the rank condition \eqref{eq:FJrank} and $L$ is maximal. 
These are summarized in 
Table~\ref{tab:spherical-d.s} below.
The homogeneous spaces $G/H$ in Cases 7' and 12' are not symmetric spaces; however, 
the $\mathbb{C}$-algebra $\mathbb{D}_G(G/H)$ 
is generated by the Laplacian, and 
Fact~\ref{fact:kobayashi-kassel-multiplicity} extends to these cases as well.

\begin{table}[htb!]
    \centering
    \begin{tabular}{c|c|c}
         Case & $G/H$ & $L$  \\
         \hline
         1 & $SU(2n,2)/Sp(n,1)$ & $U(2n,1)$ \\
        \hline
        2 & $SU(2n,2)/U(n,1)$ & $Sp(n,1)$ \\
        \hline
        3$_{\even}$ & $SO(2n,2)/U(n,1)$ & $SO(2n,1)$ \text{($n$ is even)}\\
        \hline
        4 & $SO(2n,2)/SO(2n,1)$ & $U(n,1)$ \\
        \hline
        5 & $SO(4n,4)/SO(4n,3)$ & $Sp(1)\cdot Sp(n,1)$ \\
        \hline
        6 & $SO(8,8)/SO(8,7)$ & $Spin(8,1)$ \\
        
        \hline
        7' & $SO(4,4)/Spin(4,3)$ & $Spin(4,1)\times SO(3)$ \\
        \hline
        12' & $SO(4,3)/G_{2(2)}$ & $SO(4,1)\times SO(2)$ 
    \end{tabular}
    \caption{The triples $(L, G, H)$ that apply to 
    Fact~\ref{fact:kobayashi-kassel-multiplicity}.}
    \label{tab:spherical-d.s}
\end{table}

\subsection{Multiplicity of stable discrete spectra on compact quotients}

Let $\Gamma$ be a discontinuous group for $X=G/H$ via the inclusion map $\iota\colon \Gamma\rightarrow G$.

By Definition~\ref{def:stable-spectrum},
$\lambda$ is said to belong to the stable discrete spectrum if there exists a neighborhood $U$ of $\iota$ in
$\mathcal{R}(\Gamma,G,X)$ such that 
$$\min_{\varphi\in U} \dim_{\C} L^{2}(X_{\varphi(\Gamma)},\mathcal{M}_{\lambda}) \ge 1.
$$
As the neighborhood $U$ becomes smaller, the left-hand side of this inequality tends to increase.
Thus, we consider the following definition:

\begin{definition}
Let $\mathcal{U}_{\Gamma}$ denote the set of all neighborhoods of $\iota$ in
$\mathcal{R}(\Gamma,G,X)$.
    For $\lambda\in \mathrm{Spec}_{d}(X_{\Gamma})$, we define the \emph{multiplicity of the stable discrete spectrum} by
    \[
    \widetilde{\mathcal{N}}_{X_{\Gamma}}(\lambda)
    := \sup_{U\in \mathcal{U}_{\Gamma}}\min_{\varphi\in U} \dim_{\C} L^{2}(X_{\varphi(\Gamma)},\mathcal{M}_{\lambda}). 
    \]
\end{definition}

When the inclusion map $\iota\colon \Gamma\rightarrow G$ is locally rigid, 
the multiplicity of the stable discrete spectrum, 
$\widetilde{\mathcal{N}}_{X_{\Gamma}}(\lambda)$, reduces to the dimension of $L^{2}(X_{\Gamma},\mathcal{M}_{\lambda})$.
In contrast, when $\iota$ is not locally rigid
as a discontinuous group for $G/H$,
we have
$\widetilde{\mathcal{N}}_{X_{\Gamma}}(\lambda)\neq 0$ if and only if 
$\lambda$ belongs to the stable discrete spectrum. 

In the special case where $X$ is the $3$-dimensional anti-de Sitter space $SO(2,2)/SO(1,2)$,
the $\mathbb{C}$-algebra $\mathbb{D}_G(X)$ is generated by the Laplacian. Therefore, it is sufficient to consider eigenvalues of the Laplacian, and we can regard 
$\mathrm{Spec}_{d}(X_{\Gamma})$ as a subset of
$\C$.
We note that $X_{\Gamma}$ has abundant non-standard deformations.
In this case, Kannaka found a lower bound for
the multiplicity of the stable discrete spectrum,
$\widetilde{\mathcal{N}}_{X_{\Gamma}}(\lambda)$,
as follows.
We define $\lambda_{m}:=4m(m-1)$ for $m\in\N$.
\begin{fact}[\cite{Kannaka_sigma21}]
    \label{fact:kannaka-sigma}
    Let $\Gamma$ be a cocompact discontinuous group for the $3$-dimensional anti-de Sitter space $X = SO(2,2)/SO(1,2)$.     
    Then there exists a positive constant $c_{\Gamma}$ such that 
    \[
    \widetilde{\mathcal{N}}_{X_{\Gamma}}(\lambda_{m})
    \geq \log_{3}m -c_{\Gamma}
    \]
    holds for all $m \in \N$.
\end{fact}

We recall from Theorem~\ref{theorem:local-nonstd-zariski} a list of triples $(G,H,L)$ such that $X=G/H$ admits a compact standard quotient that is deformable. 
Combining Table~\ref{tab:cpt-CK-sym} with Table~\ref{tab:spherical-d.s}, 
we obtain the following proposition.

\begin{proposition}
\label{prop:compare_two_Tables}
Let $(G,H,L)$ be one of Cases~1, 2 $(n=1)$, 3$_{\even}$, 4, 5 $(n=1)$, 7', and 12' in Table~\ref{tab:spherical-d.s}. 
Then, Facts~\ref{fact:stable-spectrum} and~\ref{fact:kobayashi-kassel-multiplicity}
apply to any such triple $(G,H,L)$. 
\end{proposition}

\begin{remark}
    \begin{enumerate}[label=(\arabic*)]
        \item 
        We have excluded Cases~2 $(n\geq 2)$, 5 ($n\geq 2$), and 6 because
        any compact standard quotient $X_{\Gamma}$ is locally rigid in these cases.
        \item By Theorem~\ref{thm:zariskidense-up-to-local-isom}, 
        there exists a cocompact standard quotient $X_\Gamma$ that admits a Zariski-dense deformation in Cases~2 ($n=1$), 3$_{\even}$, 5 ($n=1$), 7', and 12'.
    \end{enumerate}
\end{remark}

Regarding the stable discrete spectrum, including its multiplicity and distribution, 
we are particularly interested in the setting where 
a compact standard quotient $X_{\Gamma}$ of $X=G/H$ 
admits a non-standard deformation.
 We address the following questions.
 
\begin{question}
\label{question:discrete-stable-spec}
    Is the set of stable discrete spectra 
    $\{\lambda\in \mathrm{Spec}_{d}(X_{\Gamma})\mid 
        \widetilde{\mathcal{N}}_{X_{\Gamma}}(\lambda)\neq 0\}$
        discrete in $\mathfrak{j}_{\C}^{\vee}/W$?
\end{question}

\begin{question}
\label{question:mult-stable-spec}
    Does there exist a small deformation $\varphi\in\mathcal{R}(\Gamma,G;X)$ such that
    $\dim_{\C}L^{2}(X_{\varphi(\Gamma)},\mathcal{M}_{\lambda})<\infty$
    holds for any $\lambda\in \mathfrak{j}^{\vee}_{\C}$?
\end{question}
An affirmative answer to Question~\ref{question:mult-stable-spec} would imply that
$\widetilde{\mathcal{N}}_{X_{\Gamma}}(\lambda)<\infty$.
Moreover, if Question~\ref{question:discrete-stable-spec} has an affirmative answer, the following question might become meaningful.
\begin{question}
\label{question:dist-stable-spec}
Does there exist any geometric information about the global structure of the $(G,X)$-manifold $X_\Gamma$ that we can extract from the asymptotic behavior of 
the multiplicity of the stable discrete spectrum,
$\widetilde{\mathcal{N}}_{X_{\Gamma}}(\lambda)$?
\end{question}

We examine Question~\ref{question:discrete-stable-spec} in the classical Riemannian case:
\begin{remark}
    Let $X_{\Gamma}=\Gamma\backslash X$ be a compact 
    quotient of an irreducible Riemannian symmetric space $X=G/K$ of non-compact type. 
Then $X_\Gamma$ is not locally rigid if and only
if $X$ is the two-dimensional hyperbolic space 
by the local rigidity theorem (\cite{Weil_discrete_subgroups}).
Moreover, since $K$ is compact, the $G$-action on $X$ is proper, and thus there does not exist a non-standard deformation of $X_{\Gamma}$.

\begin{enumerate}[label=$(\arabic*)$]
    \item ($\dim X = 2$). 
    When $X$ is a two-dimensional hyperbolic space, one cannot obtain any information about $X_{\Gamma}$ because $\widetilde{\mathcal{N}}_{X_{\Gamma}}(\lambda) = 0$ for $\lambda \geq 1/4$ (Wolpert~\cite[Thm.\ 5.14]{Wo94}).
    
    \item ($\dim X \ge 3$).
    In this case, $X_\Gamma$ is locally rigid, and therefore, 
    $\widetilde{\mathcal{N}}_{X_{\Gamma}}(\lambda) = \dim_{\C} L^{2}(X_{\Gamma}, \mathcal{M}_{\lambda})$, which is finite because the Laplacian in the Riemannian case is an elliptic differential operator.
    Weyl's law in spectral theory tells us that one can extract geometric quantities, such as the dimension and volume, of $X_\Gamma$ from the asymptotic behavior of $\widetilde{\mathcal{N}}_{X_{\Gamma}}(\lambda)$ as $\|\lambda\| \to \infty$.
\end{enumerate} 
\end{remark}

\appendix

\section{Topologies of real algebraic groups}
\label{section:algebraic_group}
In this section, we provide a brief review of basic concepts for complex algebraic groups as well as real algebraic groups that are used throughout this article. In particular, we highlight the relationship between the two topologies  
on complex (or real) algebraic groups: the usual topology and the Zariski-topology, which is weaker than the usual topology.

We begin with the connectedness of complex algebraic groups.
\begin{lemma}
\label{lem:Zariski-connected}
Let $G_{\C}$ be a Zariski-connected complex algebraic group. 
Then it is also connected in the usual topology.   
\end{lemma} 

\begin{proof}
    By \cite[Chap.~1, Prop.~1.2]{Borel-alg-grp}, $G_{\C}$ is irreducible in the Zariski-topology. 
    Hence, by \cite[Chap.~VII.2]{Shafarevich_BasicAG},
    $G_{\C}$ is connected in the usual topology.
\end{proof}

Every connected complex semisimple Lie group is linear. 
Moreover, the categories of connected complex semisimple Lie groups  
and connected complex semisimple algebraic groups are equivalent, as described below.
\begin{lemma}
    \label{lemma:analytic->algebraic}
    \begin{enumerate}[label=(\arabic*)]
        \item 
        \label{item:semisimple->Zariski-closed}
        Let $G_{\mathbb{C}}$ be a complex algebraic group,  
        and $\mathfrak{l}_{\mathbb{C}}$ a complex semisimple Lie subalgebra of $\mathfrak{g}_{\mathbb{C}}$.  
        Then, the analytic subgroup $L_{\mathbb{C}}$ of $G_{\mathbb{C}}$  
        corresponding to $\mathfrak{l}_{\mathbb{C}}$ is Zariski-closed.  
        
        \item 
        \label{item:semisimple->hom-algebraic}
        Let $L_{\C}$ be a connected complex semisimple algebraic group, 
        $G_{\C}$ a complex algebraic group, 
        $\varphi_{\C}\colon L_{\C}\rightarrow G_{\C}$ a holomorphic 
        Lie group homomorphism. 
        Then $\varphi_{\C}$ is a morphism of algebraic groups. 
    \end{enumerate}
\end{lemma}

For the proof, we recall the following theorem from Chevalley~\cite[Chap.~2, Thm.~13]{Chevalley51}: 
\begin{fact}
\label{fact:chevalley}
Let $G_{\C}$ be a closed complex 
Lie subgroup of $GL(n,\C)$ with finitely many connected components, 
$\overline{G}_{\C}$ its Zariski-closure in $GL(n,\C)$, 
and $\overline{\mathfrak{g}_{\C}}$
the complex Lie algebra of $\overline{G}_{\C}$. 
Then, their derived algebras are identical: $[\mathfrak{g}_{\C},\mathfrak{g}_{\C}]=[\overline{\mathfrak{g}}_{\C},\overline{\mathfrak{g}}_{\C}]$. 
\end{fact}

\begin{proof}[Proof of Lemma~\ref{lemma:analytic->algebraic}]
    \ref{item:semisimple->Zariski-closed}. 
    It suffices to show that the Zariski-closure $\overline{L_{\mathbb{C}}}$ of $L_{\mathbb{C}}$ coincides with $L_{\mathbb{C}}$.  
    Since $L_{\mathbb{C}}$ is connected in the usual topology, $L_{\mathbb{C}}$ is Zariski-connected.  
    Hence, $\overline{L_{\mathbb{C}}}$ is also Zariski-connected.  
    By Lemma~\ref{lem:Zariski-connected},  
    $\overline{L_{\mathbb{C}}}$ is connected in the usual topology.  

    Since $\mathfrak{l}_{\mathbb{C}}$ is semisimple,  
    we have $[\mathfrak{l}_{\mathbb{C}}, \mathfrak{l}_{\mathbb{C}}] = \mathfrak{l}_{\mathbb{C}}$,  
    which implies that $[L_{\mathbb{C}}, L_{\mathbb{C}}] = L_{\mathbb{C}}$  
    due to the connectedness of $L_{\mathbb{C}}$.  
    Furthermore, by \cite[Chap.~I, Cor.~2.3]{Borel-alg-grp},  
    we also obtain $[\overline{L_{\mathbb{C}}}, \overline{L_{\mathbb{C}}}] = \overline{L_{\mathbb{C}}}$.  
    Thus, applying Fact~\ref{fact:chevalley}, we conclude that  
    \begin{equation*}
        \overline{\mathfrak{l}_{\mathbb{C}}} = [\overline{\mathfrak{l}_{\mathbb{C}}}, \overline{\mathfrak{l}_{\mathbb{C}}}]  
        = [\mathfrak{l}_{\mathbb{C}}, \mathfrak{l}_{\mathbb{C}}] = \mathfrak{l}_{\mathbb{C}}.
    \end{equation*}  

    Since both $L_{\mathbb{C}}$ and $\overline{L_{\mathbb{C}}}$ are connected,  
    it follows that $\overline{L_{\mathbb{C}}} = L_{\mathbb{C}}$.
    This proves the assertion.

    \ref{item:semisimple->hom-algebraic}. 
    Consider the graph of $\varphi_{\C}\colon L_{\C}\rightarrow G_{\C}$ given by  
    \[
        G(\varphi_{\C}) := \{ (\ell, \varphi_{\C}(\ell)) \mid \ell \in L_{\mathbb{C}} \}.
    \]
    The connected complex Lie group $G(\varphi_{\C})$ is the analytic subgroup of $L_{\mathbb{C}} \times G_{\mathbb{C}}$  
    corresponding to the graph of the differential $d\varphi_{\C}$,
    \[
        \mathfrak{g}(d\varphi_{\C}) := \{ (X, d\varphi_{\C}(X)) \mid X \in \mathfrak{l}_{\mathbb{C}} \}.
    \]
    Since the Lie algebra $\mathfrak{g}(d\varphi_{\C})$ is isomorphic to $\mathfrak{l}_{\mathbb{C}}$,  
    it is semisimple.  
    Thus, by \ref{item:semisimple->Zariski-closed},  
    we conclude that $G(\varphi_{\C})$ is Zariski-closed in $L_{\mathbb{C}} \times G_{\mathbb{C}}$.  
    It follows that $G(\varphi_{\C})$ is a Zariski-connected complex algebraic group.  

    Next, consider the first projection  
    $\operatorname{pr}_{1} \colon G(\varphi_{\C}) \to L_{\mathbb{C}}$, which  
    is a bijective morphism of algebraic groups. 
    It follows from Zariski's main theorem (see \cite[Chap.~AG, Thm.~18.2]{Borel-alg-grp}) that  
    this map is an isomorphism of algebraic groups.  
    Since $\varphi_{\C}\colon L_{\mathbb{C}} \to G_{\mathbb{C}}$ is the composition of  
    the isomorphism $\operatorname{pr}_{1}^{-1} \colon L_{\mathbb{C}} \to G(\varphi_{\C})$  
    and the second projection $\operatorname{pr}_{2} \colon G(\varphi_{\C}) \to G_{\mathbb{C}}$,  
    we see that $\varphi_{\C}$ is also a morphism of algebraic groups.
\end{proof}

Next, we focus on real algebraic groups.  
For a real algebraic group $\mathbf{G}$,  
we denote by $G$ the Lie group of real points $\mathbf{G}(\mathbb{R})$  
and by $G_{\mathbb{C}}$ the complex Lie group of complex points $\mathbf{G}(\mathbb{C})$.
Note that $G_{\mathbb{C}}$ provides a complexification of $G$,  
that is, the Lie algebra $\mathfrak{g}_{\mathbb{C}}$ of $G_{\mathbb{C}}$  
is naturally identified with the complexification $\mathfrak{g}\otimes_{\R}\C$ of the Lie algebra $\mathfrak{g}$ of $G$.
\begin{definition}
\label{def:Zariski-connected}
    A real algebraic group $\mathbf{G}$ is said to be \emph{Zariski-connected} if $G_{\C}$ is connected in the Zariski-topology.
\end{definition}

\begin{remark}
    In contrast to the case of complex algebraic groups,  
    even if $\mathbf{G}$ is Zariski-connected,  
    $G$ is not necessarily connected in the usual topology.  
    For instance, $SO(n,1)$ is Zariski-connected  
    in the sense of Definition~\ref{def:Zariski-connected}  
    but is not connected in the usual topology.
\end{remark}
    
\begin{definition}
\label{def:Zariski-dense}
    Let $\mathbf{G}$ be a real algebraic group. 
    The \emph{Zariski-topology} on  $G$ means the relative topology induced by the Zariski-topology on $G_{\mathbb{C}}$.  
    The \emph{Zariski-closure} of a subset $S$ of $G$ is defined as its closure with respect to the Zariski-topology.  
    We say that $S$ is \emph{Zariski-dense} in $G$ if it is dense in $G$ with respect to the Zariski-topology. 
\end{definition}
If $\mathbf{G}$ is Zariski-connected, then 
    $G$ is Zariski-dense in $G_{\C}$.
    Hence, if $\mathbf{G}$ is Zariski-connected, then a subset $S$ of $G$ is Zariski-dense in $G$
if and only if $S$ is Zariski-dense in $G_{\C}$.

In the case of semisimple real algebraic groups,  
the following lemma can be compared with Lemma~\ref{lemma:analytic->algebraic}:
\begin{lemma}
    \label{lemma:morphism}
    Let $\mathbf{L}$ be a Zariski-connected real semisimple algebraic group, $\mathbf{G}$ a  real algebraic group, and
    $\varphi\colon L\rightarrow G$ an abstract group homomorphism. 
    Assume that $L_{\C}$ is simply-connected. 
    Then the following claims about $\varphi$
    are equivalent: 
    \begin{enumerate}[label=(\roman*)]
        \item 
        \label{item:continuous}
        $\varphi$ is continuous in the usual topology; 
        \item 
        \label{item:C^infty}
        $\varphi$ is of $C^{\infty}$-class; 
        \item 
        \label{item:analytic}
        $\varphi$ is real analytic; 
        \item 
        \label{item:holomorphic}
        $\varphi$ can be extended to a holomorphic homomorphism
        $\varphi_{\C}\colon L_{\C}\rightarrow G_{\C}$ of complex Lie groups; 
        \item 
        \label{item:R-hom}
        $\varphi$ is obtained from an $\R$-morphism $\mathbf{L} \to \mathbf{G}$ of algebraic groups.
    \end{enumerate}
We simply refer to $f$ as a \emph{homomorphism}.
In particular, the image of $f_{\C}$ is Zariski-closed, and 
$f$ is continuous in the Zariski-topology. 
\end{lemma}

\begin{proof}
    It is well known that \ref{item:continuous}, \ref{item:C^infty}, and \ref{item:analytic} are equivalent,  
    and \ref{item:R-hom} $\Rightarrow$ \ref{item:continuous} is obvious.  
    Thus, we focus on proving \ref{item:analytic} $\Rightarrow$ \ref{item:holomorphic} $\Rightarrow$ \ref{item:R-hom}.

\ref{item:analytic}$\Rightarrow$\ref{item:holomorphic}.
    Since $L_{\mathbb{C}}$ is simply-connected,  
the complexification of the differential $df \colon \mathfrak{l} \to \mathfrak{g}$  
can be lifted to a holomorphic homomorphism  
$f_{\mathbb{C}} \colon L_{\mathbb{C}} \to G_{\mathbb{C}}$ of complex Lie groups.  
It is clear that $f_{\mathbb{C}}$ extends $f$.

    \ref{item:holomorphic}$\Rightarrow$\ref{item:R-hom}.
    By Lemma~\ref{lemma:analytic->algebraic}~\ref{item:semisimple->hom-algebraic},  
$f_{\mathbb{C}}$ is a morphism of algebraic groups.  
Thus, it suffices to show that $f_{\mathbb{C}}$ is defined over $\mathbb{R}$.  

Let $\sigma_{L} \colon L_{\mathbb{C}} \to L_{\mathbb{C}}$  
and $\sigma_{G} \colon G_{\mathbb{C}} \to G_{\mathbb{C}}$  
denote the complex conjugations associated with the real algebraic groups $\mathbf{L}$ and $\mathbf{G}$, respectively.  
Identifying $\mathfrak{l}_{\mathbb{C}} = \mathfrak{l} \otimes_{\mathbb{R}} \mathbb{C}$  
and $\mathfrak{g}_{\mathbb{C}} = \mathfrak{g} \otimes_{\mathbb{R}} \mathbb{C}$,  
we have $df_{\mathbb{C}} = df \otimes_{\mathbb{R}} \operatorname{id}_{\C}$, 
$d\sigma_{L}=\id_{\mathfrak{l}}\otimes \sigma$, and $d\sigma_{G}=\id_{\mathfrak{g}}\otimes \sigma$, where $\sigma$ is the complex conjugate of $\C$.  
Hence, it follows that $df_{\mathbb{C}} \circ d\sigma_{L} = d\sigma_{G} \circ df_{\mathbb{C}}$.
Since $L_{\mathbb{C}}$ is connected,  we see that
$f_{\mathbb{C}} \circ \sigma_{L} = \sigma_{G} \circ f_{\mathbb{C}}$.  
Thus, the morphism of algebraic groups $f_{\mathbb{C}}$ is defined over $\mathbb{R}$.
\end{proof}

\section{Clifford algebras and Spin groups}
\label{section:clifford-spin}
Throughout this article, we use Clifford algebras and spin groups extensively. 
In this appendix, we introduce some notation related to these concepts.

Let $\F$ be a field of characteristic $\neq 2$, and $V$ be an $n$-dimensional $\F$-vector space equipped with a non-degenerate symmetric bilinear form $Q\colon V\times V \rightarrow \F$. 
As usual, we put $Q(v):=Q(v,v)$, which defines a quadratic form on the $\F$-vector space $V$. 
Let $TV:=\bigoplus_{i\in\N}V^{\otimes i}$ denote the tensor algebra.
Let let $I(Q)$ be the two-sided ideal generated by 
$v\otimes v - Q(v)$ for all $v\in V$.
The quotient $\F$-algebra 
\[
C(V)\equiv C(V,Q):= TV/I(Q)
\]
is called the \emph{Clifford algebra} associated with the quadratic form $Q$.

When $\mathbb{E}$ is an extension field of $\F$, 
we consider the $\E$-linear extension of 
the $\F$-bilinear form $Q$ on the $\E$-vector space 
$V\otimes_{\F}\mathbb{E}$. 
Then the $\mathbb{E}$-algebra 
$C(V)\otimes_{\F}\mathbb{E}$
is naturally isomorphic to $C(V\otimes_{\F}\mathbb{E})$. 

Let $q\colon TV\rightarrow C(V)$ be the quotient homomorphism. 
If $\{e_{1},\ldots,e_{n}\}$ is an $\F$-basis of $V$,
then the elements 
$q(e_{i_{1}})\cdots q(e_{i_{k}})$
$(1\leq i_{1}<\cdots< i_{k}\leq n,\ 0\leq k\leq n)$
form an $\F$-basis of $C(V)$. 
In particular, $\dim_{\F}C(V)=2^{n}$ and 
the restriction of $q$ to the subspace $V$ of $TV$
is injective. Throughout this article, 
we think of $V$ as a subspace of $C(V)$
via $q$ and the image $q(v)$ of $v\in V$ is also denoted by the same symbol $v$.

The $\F$-linear map
$V\rightarrow C(V)$ defined by $v\mapsto -v$
is uniquely extended to an involutive $\F$-algebra automorphism $(-)'$
of $C(V)$. 
We define an $\F$-subalgebra $C_{\even}(V)$ of $C(V)$ by 
\[
C_{\even}(V)\equiv C_{\even}(V,Q):=\{x\in C(V,Q)\mid x'=x\},
\]
which is called the \emph{even Clifford algebra} of $Q$.
If $\{e_{1},\ldots,e_{n}\}$ is an $\F$-basis of $V$, then 
we have 
$(e_{i_{1}}\cdots e_{i_{k}})'=(-1)^{k}e_{i_{1}}\cdots e_{i_{k}}$
for $1\leq i_{1}<\cdots< i_{k}\leq n$
and $0\leq k\leq n$.
Hence, $e_{i_{1}}\cdots e_{i_{k}}\in C_{\even}(V)$
if and only if $k$ is even.
In particular, $\dim_{\F}C_{\even}(V) = 2^{n-1}$.
As is well-known, $C_{\even}(V)$ is a semisimple $\F$-algebra.

The $\F$-linear map $V\rightarrow C(V)$ defined by 
$v\mapsto v$
is uniquely extended to an $\F$-algebra involutive \emph{anti}-automorphism $(-)^{*}$
of $C(V)$. The involution $*$ preserves $C_{\even}(V)$. 
For an extension field $\E$ of $\F$, 
we define
\begin{align}
    \mathbf{G}_{V}(\E) &:= 
    \{x\in C_{\even}(V\otimes_{\F}\E)^{\times} \mid 
xx^{*} = 1 
\} \label{def:g_V} \\
    \mathbf{Spin}_{V}(\E) &:=\{x\in \mathbf{G}_{V}(\E) \mid x(V\otimes_{\F}\E)x^{-1}\subset V\otimes_{\F}\E
\}. \label{def:spin_V}
\end{align}
Here note $C_{\even}(V\otimes_{\F}\E)$ is naturally isomorphic to $C_{\even}(V)\otimes_{\F}\E$ as an $\E$-algebra. 
Hence, 
we can regard both $\mathbf{G}_{V}$ and $\mathbf{Spin}_{V}$
as $\F$-algebraic groups by identifying $C_{\even}(V)$ with $\F^{2^{n-1}}$. 
We put
\[
G(V):=\mathbf{G}_{V}(\F)\text{ and }
Spin(V):=\mathbf{Spin}_{V}(\F).
\]
The group $Spin(V)$ is called the \emph{spin group} of $Q$. 
Note that the elements $\pm 1 \in C_{\even}(V)$ 
belong to $Spin(V)$.
Since $C_{\even}(V)$ is a semisimple $\F$-algebra, 
$G(V)$ is a classical group defined over $\F$.

For $x\in Spin(V)$ and $v\in V$, 
we have $xvx^{-1}\in V$ in $C(V)$ by definition.
One can check that the $\F$-linear map $v\mapsto xvx^{-1}$ 
defines an element of the special orthogonal group $SO(V)$
with respect to $Q$. 
Thus we obtain an $\F$-homomorphism
\[
\Ad\colon \mathbf{Spin}_{V}\rightarrow \mathbf{SO}_{V},\ \ x\mapsto (v\mapsto xvx^{-1}).
\]
The kernel of this homomorphism is $\{\pm 1\}$. 
If $\mathbb{F}$ is algebraically closed, then the map $\Ad$ is surjective. 
However, this is not true in general.  
Indeed, when $\mathbb{F} = \mathbb{R}$, the image of $\Ad$ coincides with
the identity component of the orthogonal Lie group $SO(V)$, which has two connected components if the quadratic form $Q$ is not positive-definite.  

From now on, we assume $\F = \R$.
Then $G(V)$ and $Spin(V)$ are Lie groups.
Let us recall that their Lie algebras $\mathfrak{g}(V)$ and  
$\mathfrak{spin}(V)$ are realized as Lie subalgebras of $C_{\even}(V)$. Since $C_{\even}(V)$ is an $\F$-algebra, we note that $C_{\even}(V)$ is a Lie algebra over $\F$ with the bracket $[x,y] = xy - yx$. Further, via the exponential map $\exp\colon C_{\even}(V) \rightarrow C_{\even}(V)^\times$ ($x \mapsto \sum_{n=0}^{\infty} x^n/n!$), we think of $C_{\even}(V)$ as the Lie algebra of the Lie group $C_{\even}(V)^\times$. In particular, we have: 
\begin{align*}
\label{def:spin-lie}
\mathfrak{g}(V) &= \{x\in C_{\even}(V,Q)\mid 
x+ x^{*} = 0
\} \\
\mathfrak{spin}(V)&= 
\{x\in \mathfrak{g}(V)\mid \ad(x)(V) \subset V
\},
\end{align*}
where $\ad(x)(v) = [x,v]$. From this, we obtain the following lemma easily: 
\begin{lemma}
\label{lemma:g(p,q)-structure}
Let ${e_{1},\ldots,e_{n}}$ be an orthogonal basis of the $\F$-vector space $V$ with respect to $Q$. 
We write $\wedge^{k}$ $(0\leq k\leq n)$ for the 
$\F$-vector subspace of $C_{\even}(V)$ spanned by the elements $e_{i_{1}}\cdots e_{i_{k}}$ $(1\leq i_{1}<\cdots<i_{k} \leq n)$.
Then we have
\[
\mathfrak{g}(V) = \bigoplus_{\substack{0\leq k\leq n \\
k\equiv 2\bmod 4}} \wedge^{k}
\quad \text{and} \quad \mathfrak{spin}(V) = \wedge^{2}.
\]
\end{lemma}

Let $\R^{p,q}$ be the real vector space $\R^{p+q}$ equipped with the quadratic form $x_{1}^{2}+\dotsb+ x_{p}^2-x_{p+1}^2-\dotsb -x_{p+q}^2$. In this case, $C(\R^{p,q})$, $C_{\even}(\R^{p,q})$, 
$G(\R^{p,q})$, $\mathfrak{g}(\R^{p,q})$, 
$Spin(\R^{p,q})$, and $\mathfrak{spin}(\R^{p,q})$ are denoted by $C(p,q)$, $C_{\even}(p,q)$, $G(p,q)$, $\mathfrak{g}(p,q)$, 
$Spin(p,q)$, and $\mathfrak{spin}(p,q)$, respectively. 
For the structure of the even Clifford algebra $C_{\even}(p,q)$ as an 
$\mathbb{R}$-algebra, see, for example, \cite[Prop.~4.4.1]{KobayashiYoshino05}. 
Moreover, Kobayashi and Yoshino 
provided an explicit identification of the groups $G(p,q)$ with classical groups, which is useful for a uniform treatment of sporadic cases in Table~\ref{tab:kobayashi-yoshino}. We extract the classification table of $G(p,q)$ from \cite{KobayashiYoshino05}: 
\begin{proposition}[{\cite[Prop.~4.4.4]{KobayashiYoshino05}}]
    \label{proposition:Kobayashi-Yoshino-G(p,q)}
    For $p,q\geq 1$, the group $G(p,q)$ is isomorphic to one of the following classical Lie groups. Here, $n=2^{(p+q-\alpha)/2}$ and $\alpha\in\{3,4,5,6\}$
    are given according to $p-q\mod 8$ as in the table.
    
\[
\renewcommand{\arraystretch}{1.2}
\begin{array}{c|c||c|c}
G(p,q)  & \alpha & p-q \mod 8& p+q \mod 8\\ \hline
O(n,n)^2  &  &    & 0 \\
GL(2n,\R) & 4& 0 & \pm 2 \\ 
Sp(n,\R)^2   &  &   & 4\\ \hline
\begin{array}{c}
     O(n,n)  \\
     Sp(n,\R)
\end{array}  &3 & \pm 1 & 
\begin{array}{c}
     \pm 1  \\
     \pm 3
\end{array}\\ \hline
O(2n,\C)        &   &       & 0\\
U(n,n)  & 4 & \pm 2 &\pm 2\\
Sp(n,\C)       &    &       & 4\\ \hline
\begin{array}{c}
     O^*(4n)  \\
     Sp(n,n)
\end{array}  & 5 & \pm 3 & 
\begin{array}{c}
     \pm 1  \\
     \pm 3
\end{array} \\ \hline
O^*(4n)^2              &          &             & 0 \\ 
GL(2n,\HA) & 6      & 4 & \pm 2\\
Sp(n,n)^2 & & & 4 
\end{array}
\renewcommand{\arraystretch}{1}
\]
\end{proposition}
Lemma~\ref{lemma:g(p,q)-structure} and Proposition~\ref{proposition:Kobayashi-Yoshino-G(p,q)} are used in Section~\ref{section:(G,Gamma)-bending-application}. 

\section{Torsion freeness of congruence subgroups}
\label{section:torsion-arithmetic} 

In this appendix, we prove the following lemma concerning torsion of arithmetic groups, which is used in Section~\ref{section:proof-step1}.  
It seems to be well-known to experts, but we give a proof for 
the convenience of the reader. 
\begin{lemma}
    \label{lemma:torsion-free-general}
    Let $\F$ be a number field of finite degree and 
    $\mathcal{O}$ the ring of integers of $\F$. 
    Suppose that prime ideals $P,Q$ of $\mathcal{O}$
    satisfy $P\cap \Z\neq Q\cap\Z$.
    Then the congruence subgroup
    \[
    \{g\in GL(n,\mathcal{O})\mid g\equiv 1\mod PQ\}
    \]
    is torsion-free.
\end{lemma}

We give an elementary lemma for the proof:
\begin{lemma}
\label{lemma:root-of-unity}
Let $\F$, $\mathcal{O}$, $P$, and $Q$ be as in Lemma~\ref{lemma:torsion-free-general}.
For $m\in \N_{+}$, 
if a primitive $m$-th root $\zeta_{m}$ in $\F$
satisfies $\zeta_{m}\equiv 1\mod PQ$, then $m=1$.
\end{lemma}

\begin{proof}
    Let $k=\Q(\zeta_{m})$ and $\mathcal{O}_{k}=\Z[\zeta_{m}]$ the ring of integers of $k$.
    If $m$ has two distinct prime factors, then $1-\zeta_{m}$ is a unit in $\mathcal{O}_{k}$
    (see, e.g., by \cite[Lem.\ 17.2(2)]{Shimura10}). 
    Hence, $1-\zeta_{m}$ is not contained in any prime ideal of $\mathcal{O}_{k}$. 
    However, since $P\cap \mathcal{O}_{k}$ is a prime ideal of $\mathcal{O}_{k}$ and
    $1-\zeta_{m}\in P$ by assumption, 
    this leads to a contradiction.
    Hence, suppose $m=\ell^{r}$ for some prime number $\ell$ and non-negative integer $r$. 

    For contradiction, suppose $m>1$, equivalently $r>0$.
    Then $(1-\zeta_{m})\mathcal{O}_{k}$ is a prime ideal of $\mathcal{O}_{k}$ containing $\ell$ (see, e.g., by \cite[Thm.\ 17.5]{Shimura10}).
    Hence, the prime number $\ell$ is contained in both $P$ and $Q$, which contradicts the assumption that $P\cap \Z\neq Q\cap\Z$.  Therefore, we have $m=1$.
\end{proof}

We are ready to show Lemma~\ref{lemma:torsion-free-general}.  
\begin{proof}[Proof of Lemma~\ref{lemma:torsion-free-general}]
    Let $g$ be an element of $GL(n, \mathcal{O})$ such that $g \equiv 1 \mod PQ$, and let $f(t) \in \mathcal{O}[t]$ be the characteristic polynomial of $g$.

Since $g \equiv 1 \mod PQ$, we have $f(t) \equiv (t-1)^n \mod PQ$. Let $\E$ be a finite extension field of $\F$ containing all the eigenvalues of $g$, and let $P'$ and $Q'$ be prime ideals of the integers $\mathcal{O}_{\E}$ in $\E$ containing $P$ and $Q$, respectively. Then, it follows that $f(t) \equiv (t-1)^n \mod P'Q'$.

Now suppose that $g$ has finite order. In this case, $g$ is diagonalizable over $\E$, and its eigenvalues are all roots of unity. To prove our assertion, it suffices to show that these eigenvalues are equal to 1.

Let $\zeta_m$ be an eigenvalue of $g$, which is a primitive $m$-th root of unity. Since $f(\zeta_m) = 0$, we have $(\zeta_m - 1)^n \equiv 0 \mod P'Q'$. Since both $P'$ and $Q'$ are prime ideals of $\mathcal{O}_{\E}$, it follows that $\zeta_m - 1 \equiv 0 \mod P'Q'$. By Lemma~\ref{lemma:root-of-unity}, we conclude that $m = 1$. This completes the proof.
\end{proof}

\section{A lemma on the HNN extension}
\label{section:hnn-extension}
Let $M$ be an orientable connected manifold and 
    $N$ an orientable connected compact hypersurface
    such that $S:=M\smallsetminus N$ is connected. 
    In this appendix, we present Proposition~\ref{prop:hnn}, which explains how the fundamental group $\pi_{1}(M)$ of $M$ is described by $\pi_{1}(S)$ and $\pi_{1}(N)$.
    We see that $\pi_{1}(M)$ is an HNN extension of $\pi_{1}(S)$. 
    Although the argument using the van Kampen theorem is standard to experts, we provide a proof for the convenience of readers who may not be familiar with it.
    This proposition is used in Section~\ref{section:step2}, where 
    $M$ and $N$ are hyperbolic manifolds. 
 
    Let us introduce some notations. 
    By the tubular neighborhood theorem,
    one can and do take an open neighborhood $\tilde{N}^{(4)}$ of $N$ in $M$
    and a diffeomorphism 
    \[
    f\colon N\times (-4,4)\simeq \tilde{N}^{(4)}
    \]
    such that $f(N\times\{0\})$ coincides with $N$.  
    Fix $y\in N$ and put $\tilde{N}^{(t)}:=f(N\times (-t,t))$ and 
    $y^{(t)}:=f(y,t)$ for $-4 < t< 4$.
    Since $S=M\smallsetminus N$ is connected, 
    we may and do assume that $M\smallsetminus \tilde{N}^{(4)}$ is path-connected.
    \begin{definition}
        Fix a base point $x_0 \in M \smallsetminus \tilde{N}^{(4)}$.  
    We define the oriented loop $\nu$ in $M$, starting at $x_0$,  
    as the composition of the following paths:  

    \begin{itemize}
        \item a path from $x_0$ to $y^{(3)}$  
        inside $M \smallsetminus \tilde{N}^{(3)}$;
        \item the path from $y^{(3)}$ to $y^{(-3)}$  
        given by $f(\{y\} \times [-3,3])$;
        \item a path from $y^{(-3)}$ to $x_0$  
        inside $M \smallsetminus \tilde{N}^{(3)}$.
    \end{itemize}
Furthermore, we can take the loop $\nu$ to be a closed \emph{submanifold} of $M$.  
Additionally, let $\nu_+$ be the path along the loop $\nu$ from $x_0$ to $y^{(1)}$,  
and let $\nu_-$ be the path along $\nu$ from $y^{(-1)}$ to $x_0$.  

    \end{definition}
    \begin{figure}[h]
    \centering
    \begin{minipage}[b]{0.49\columnwidth}
    \centering{\includegraphics[
    width=5cm
    ]{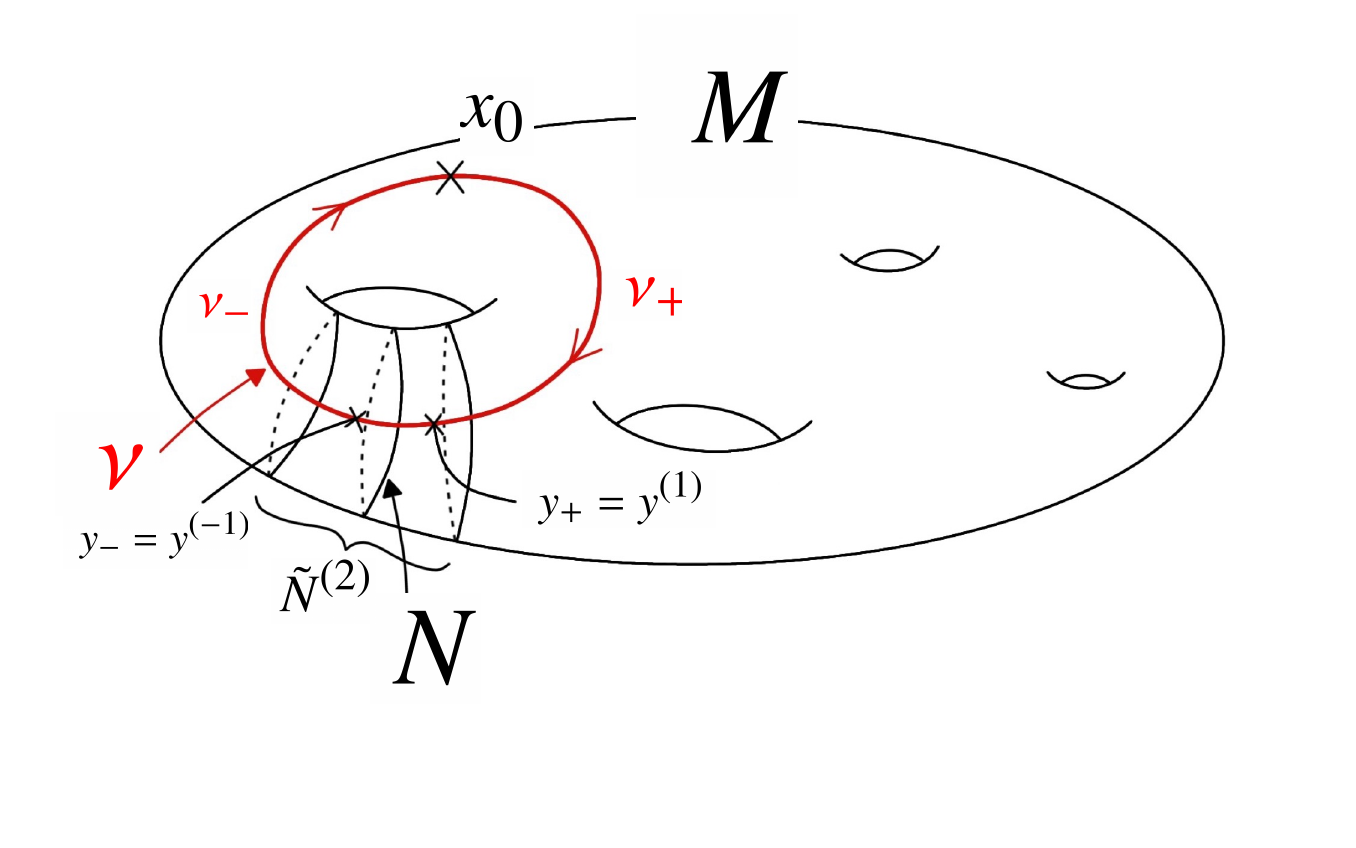}}
        \caption{}
        \label{fig:hnn1}
    \end{minipage}
    \begin{minipage}[b]{0.49\columnwidth}
    \centering{\includegraphics[width=5cm]{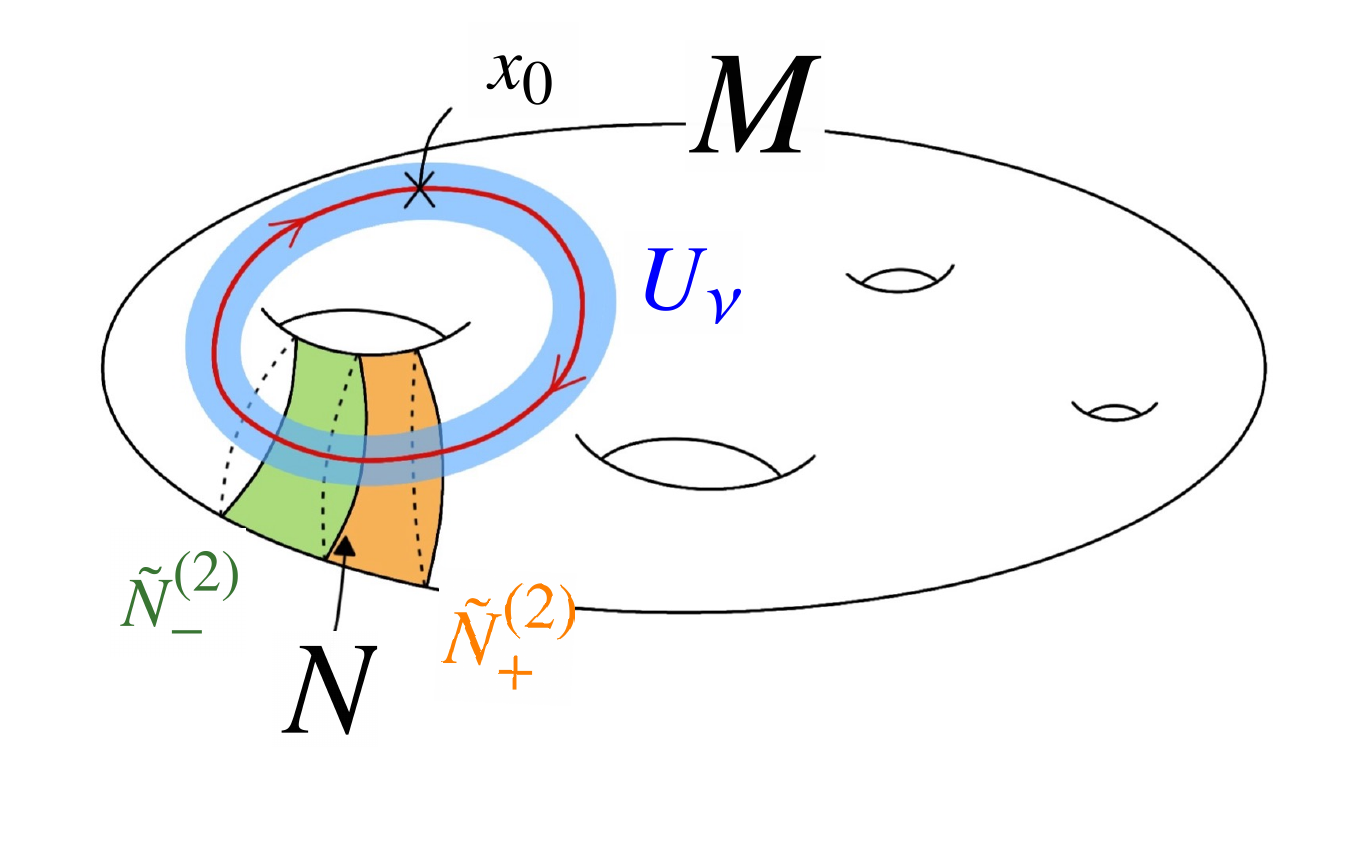}}
        \caption{}
        \label{fig:tubular_nbd}
    \end{minipage}
    \end{figure}
    \begin{figure}
        \centering
       
    \end{figure}
    Figure~\ref{fig:hnn1} summarizes some of the notation introduced so far.

    For an oriented loop $\ell$ starting at $y$ in $N$ and a fixed $-4 <t < 4$,  
    $f(\ell, t)$ forms an oriented loop in $M$ starting at $y^{(t)}$.
    With the above paths $\nu_{+}$ and $\nu_{-}$, 
    we define $j_{+},j_{-}\colon \pi_{1}(N,y)\rightarrow \pi_{1}(S,x_{0})$ as two group homomorphisms 
    \[
    j_{+}([\ell]):= [\nu_{+}^{-1}\circ f(\ell,1)\circ \nu_{+}],\ 
    j_{-}([\ell]):= [\nu_{-}\circ
    f(\ell,-1)\circ\nu_{-}^{-1}].
    \]
    Here, $b^{-1}$ denotes the path obtained by reversing the orientation of the path $b$, and $[c]$ denotes the homotopy class defined by the loop $c$.

\begin{proposition}
    \label{prop:hnn}
    Let $a^{\mathbb{Z}}$ be the infinite cyclic group generated by $a$.  
Under the above setting, we define a group homomorphism  
\[
\Psi \colon \pi_{1}(S, x_{0}) * a^{\mathbb{Z}} \to \pi_{1}(M, x_{0})
\]
by the natural map $\pi_{1}(S, x_{0}) \to \pi_{1}(M, x_{0})$, 
and by sending the generator $a$ to $[\nu]$.  
Then, $\Psi$ is surjective, and its kernel is the normal subgroup $\mathcal{N}$,  
generated by the elements  
\[
a j_{+}([\ell]) a^{-1} j_{-}([\ell])^{-1}, \ \  \text{for every } [\ell] \in \pi_{1}(N, y).
\]

\end{proposition}

\begin{proof}     
    Since $M$ is orientable, one can 
    take a tubular neighborhood $U_{\nu}$ of $\nu$ in $M$
    and a diffeomorphism $g\colon \nu\times U \rightarrow U_{\nu} $,
    where $U$ is a contractible open neighborhood $U$ of $y$ in $N$
    and $\nu=g(\nu\times \{y\})$.
    By taking a sufficiently small $U_{\nu}$, we may and do assume the following:
    \begin{itemize}
        \item 
        $g(y^{(t)},u)=f(u,t)$ holds for $-2\leq t\leq 2$ and $u\in U$. 
        
        \item 
        $U_{\nu}\smallsetminus f(N\times [-2,2])=g((\nu\smallsetminus\{y^{(t)}\mid -2\leq t\leq  2\})\times U)$. 
    \end{itemize}
    In particular, $U_{\nu}\cap \tilde{N}^{(2)}$ and 
    $U_{\nu}\smallsetminus f(N\times [-2,2])$
    are both contractible.
    
    We put 
    \[
    V:=\tilde{N}^{(2)}\cup U_{\nu}.
    \]
    Since $M=V\cup S$ and $V\cap S$ is path-connected, 
    by the van Kampen theorem,
    we obtain the following pushout diagram:
    \begin{align}
    \label{eq:pi1-M}
    \xymatrix{
    \pi_{1}(V\cap S,x_{0}) \ar[r] \ar[d] & \pi_{1}(S,x_{0}) \ar[d] \\ 
    \pi_{1}(V,x_{0}) \ar[r]  & \pi_{1}(M,x_{0}).
    \ar@{}[lu]|{\circlearrowright}
    }
    \end{align}

    Now we compute $\pi_{1}(V,x_{0})$ and 
    $\pi_{1}(V\cap S,x_{0})$ by the van Kampen theorem.
    Since $\tilde{N}^{(2)} \cap U_{\nu}$ is contractible, we have 
    \begin{align*}
        \pi_{1}(V,y^{(1)}) &\simeq \pi_{1}(\tilde{N}^{(2)},y^{(1)})
        *\pi_{1}(U_{\nu},y^{(1)}).
    \end{align*}
    Recall that $U_{\nu}$ and $\tilde{N}^{(2)}$ are homotopic to $\nu$ and $N$, respectively.
    Hence, we have $\pi_{1}(\tilde{N}^{(2)},y^{(1)}) \simeq \pi_{1}(N,y)$
    and $\pi_{1}(U_{\nu},y^{(1)})\simeq a^{\Z}$. 
    Furthermore, we use $\nu_+$ to change the base point $y^{(1)}$ to $x_0$.  
    The resulting isomorphism  
    \begin{equation}
    \label{eq:pi1-V}
    \pi_1(V, x_0) \xleftarrow[\simeq]{\lambda} \pi_1(N, y) * a^{\mathbb{Z}},
    \end{equation}  
    is obtained by $\lambda([\ell])= [\nu_{+}^{-1}\circ f(\ell,1)\circ \nu_+]$ for $[\ell]\in \pi_1(N, y)$ and $\lambda(a)=[\nu]$.

    We define  
    \[
    \tilde{N}^{(2)}_{+} := f(N \times (0,2)), \quad  
    \tilde{N}^{(2)}_{-} := f(N \times (-2,0))
    \] (see Figure~\ref{fig:tubular_nbd}).  
    Then, $V \cap S$ is the union of the two open sets  
    \[
    A_{+} := \tilde{N}^{(2)}_{+} \cup (U_{\nu} \smallsetminus f(N \times [-2,2)))
    \text{ and }
    A_{-} := \tilde{N}^{(2)}_{-} \cup (U_{\nu} \smallsetminus f(N \times (-2,2])).
    \]
    The intersection $A_+ \cap A_-$ is given by $U_{\nu} \smallsetminus f(N \times [-2,2])$,
    which is contractible.  
    Thus, by the van Kampen theorem, we obtain  
    \[
    \pi_{1}(V\cap S, x_{0})\simeq \pi_{1}(A_{+},x_{0}) * \pi_{1}(A_{-},x_{0}).
    \]
    Furthermore, by the choice of $U_{\nu}$,  
    both $A_{+}$ and $A_{-}$ are homotopy equivalent to $N$.  
    Hence, we obtain the isomorphism 
    \begin{align}
    \label{eq:pi1-VS}
    \pi_1(V \cap S, x_0)\xleftarrow[\simeq]{\mu} \pi_1(N, y) * \pi_1(N, y),
    \end{align}  
    which is induced by the two homomorphisms  
    $\pi_1(N, y) \to \pi_1(V \cap S, x_0)$  
    mapping $[\ell] \in \pi_1(N, y)$ to  
    $[\nu_{+}^{-1}\circ f(\ell,1)\circ \nu_+]$ and $[\nu_{-}\circ f(\ell,-1)\circ \nu_{-}^{-1}]$
    in $ \pi_1(V \cap S, x_0) $, respectively.

    Here we note 
    \[
    [\nu_{-}\circ f(\ell,-1)\circ \nu_{-}^{-1}]= [\nu\circ\nu_{+}^{-1}\circ f(\ell,1)\circ \nu_{+}\circ \nu^{-1}] \quad \text{in}\quad \pi_{1}(V,x_{0}).
    \]
    Hence, from \eqref{eq:pi1-V} and \eqref{eq:pi1-VS}, the following diagram commutes: 
    \begin{align*}
    \xymatrix{
    \pi_{1}(N,y)*\pi_{1}(N,y) \ar[r]^(.55){\mu}_(.55){\simeq} \ar[d]_{\id * (a(-)a^{-1})} & \pi_{1}(V\cap S ,x_{0}) \ar[d]^{\text{natural}} \\ 
    \pi_{1}(N,y)*a^{\Z} \ar[r]^{\lambda}_{\simeq}  & \pi_{1}(V,x_{0}).
    \ar@{}[lu]|{\circlearrowright}
}
    \end{align*}
    Therefore, by \eqref{eq:pi1-M}, 
    we obtain the following
    pushout diagram: 
    \begin{align}
    \label{eq:pushout-first}
    \xymatrix{
    \pi_{1}(N,y)*\pi_{1}(N,y) \ar[r]^(.6){j_{+}*j_{-}} \ar[d]_{\id * (a(-)a^{-1})} & \pi_{1}(S,x_{0}) \ar[d]^{\text{natural}} \\ 
    \pi_{1}(N,y)*a^{\Z} \ar[r]^(.55){j_{+}*(a\mapsto[\nu])}  & \pi_{1}(M,x_{0}).
    \ar@{}[lu]|{\circlearrowright}
}
    \end{align}

    On the other hand, by the definition of the normal subgroup $\mathcal{N}$, we have the following pushout diagram:
\begin{align}
\label{eq:pushout-second}
\xymatrix{
\pi_{1}(N,y)*\pi_{1}(N,y) \ar[r]^(.6){j_{+}*j_{-}} \ar[d]_{\id * (a(-)a^{-1})} & \pi_{1}(S,x_{0}) \ar[d] \\
\pi_{1}(N,y)*a^{\mathbb{Z}} \ar[r]^(.43){j_{+}*\id}  & (\pi_{1}(S,x_{0})*a^{\mathbb{Z}})/\mathcal{N}.
\ar@{}[lu]|{\circlearrowright}
}
\end{align}
By comparing the two pushout diagrams \eqref{eq:pushout-first} and \eqref{eq:pushout-second},  
we obtain the desired conclusion of the theorem.  
\end{proof}

\section{An Upper bound for local deformation of discrete subgroups}
\label{section:deform-upper-bound}

The main result of this appendix is
Proposition~\ref{prop:local-rigid-g/l}, 
which provides an upper bound for the Zariski-closure of potential deformations of discrete subgroups in the general setting.
This proposition is used in the proof of Theorem~\ref{theorem:bending}~\ref{item:maximal-zariski-closure},
as well as in part of the proof of Proposition~\ref{prop:so*8/su(3,1)}~(2).
While the results of this section are likely known to experts,  
we formulate them in a form suitable for our purposes and prove them using the \emph{curve selection lemma} in the  subanalytic geometry.
See Lemma~\ref{lemma:subanalytic-property}~\ref{item:subanalytic-curve-selection}. 
We include a proof for the benefit of readers who may not be familiar with them.

\medskip
\begin{setting}
    \label{setting:appendix-deform}
    Let $G$ be a Lie group, 
    $L$ a closed subgroup of $G$, 
    and $\Gamma$ a finitely-presented group.
    By choosing $N$ generators of $\Gamma$ satisfying a finite set of relations, we identify $\Hom(\Gamma, G)$ with a subset of $X := G^N$.
\end{setting}
We do not require $L$ or $G$ to be reductive here.
In this setting, we provide a general local rigidity theorem. In other words, we give a necessary condition for a representation of $\Gamma$ into $L$ to be deformed in $G$, beyond $L$, up to $G$-conjugacy. Recall that $G$ acts on $\Hom(\Gamma, G)$ by inner automorphisms, where $(g \cdot \varphi)(\gamma) := g \varphi(\gamma) g^{-1}$.

\begin{proposition}
    \label{prop:local-rigid-g/l}
In Setting~\ref{setting:appendix-deform}, let $\varphi \colon \Gamma \to G$ be a group homomorphism such that $\varphi(\Gamma) \subset L$. We regard the Lie algebras $\mathfrak{g}$ and $\mathfrak{l}$ as $\Gamma$-modules via $\varphi$. If $H^1(\Gamma, \mathfrak{g}/\mathfrak{l}) = 0$, then there exists a neighborhood $U$ of $\varphi$ in $\Hom(\Gamma, G)$ such that for any $\varphi' \in U$, we have $\varphi'(\Gamma) \subset L$, up to $G$-conjugacy. 
\end{proposition}

Our proof relies on properties of real analytic subsets.  
A subset $A$ of a real analytic manifold $X$ is called \emph{real analytic} in $X$  
if, for each point $x \in X$, there exist an open neighborhood $U$ of $x$ in $X$  
and analytic functions $f_1, \dots, f_n$ on $U$ such that  
\[
    U \cap A = \{ u \in U \mid f_1(u) = \cdots = f_n(u) = 0 \}.
\]  
In Setting~\ref{setting:appendix-deform},  
$\Hom(\Gamma, G)$ is a real analytic subset of $X = G^N$.  
Since $L$ is real analytic in $G$,  
it follows that $\Hom(\Gamma, L)$ is also real analytic in $X$.  

\textbf{Step~1}. 
As a first step, we prove Proposition~\ref{prop:local-rigid-g/l}, assuming that the small deformation of $\varphi$ is described by an analytic curve $\varphi_{t}$. 
The following discussion relying on Artin's approximation theorem is also seen in 
Goldman--Millson~\cite[Sect.~2]{goldman-millson87}.
To be precise, we show the following lemma: 
\begin{lemma}
    \label{lemma:local-rigid-g/l-analytic-vesrion}
    In the setting of Proposition~\ref{prop:local-rigid-g/l}, 
    for any one-parameter family of group homomorphisms 
    $\varphi_{t}\colon \Gamma\rightarrow G$, depending analytically on $t \in \R$,  with $\varphi_{0}=\varphi$,
    there exists an analytic curve $g_{t}$ in $G$ with $g_{0}=e$ such that $g_{t}\cdot \varphi_{t}\in \Hom(\Gamma,L)$ for any sufficiently small $t$.
\end{lemma}

\begin{proof}[Proof of Lemma~\ref{lemma:local-rigid-g/l-analytic-vesrion}]
For each $\gamma \in \Gamma$, the real analytic curve $\varphi_t(\gamma) \varphi(\gamma)^{-1}$ in $G$, starts at $e$ when $t=0$.
Its Taylor expansion defines a sequence of maps $u_\ell \equiv u_\ell(\varphi_t) \colon \Gamma \to \mathfrak{g}$ $(\ell=1,2,\dots)$,
such that 
\[
    \varphi_{t}(\gamma) = \exp(u_{1}(\gamma) t + u_{2}(\gamma) t^{2} + \dotsb) \varphi(\gamma),
\]
where the domain of the convergence may depend on $\gamma$.

We fix $k \geq 1$, and prove the following claim: \textit{assume that $u_{\ell}(\Gamma) \subset \mathfrak{l}$ for $\ell=1,\dots,k-1$.  
Then, there exists an analytic curve in $G$, to be denoted by $g^{(k)}_{t}$, with the following two properties:}
\begin{enumerate}
    \item[(1)] $u_\ell (\varphi_t)=u_\ell(g^{(k)}_{t}\cdot \varphi_t)$ for $\ell =1, \dots, k-1$,
    \item[(2)]
     $u_k(g^{(k)}_{t}\cdot \varphi_t)(\Gamma) \subset \mathfrak{l}$,
\end{enumerate}
where $g^{(k)}_{t}\cdot \varphi_t$ denotes a homomorphism $\Gamma \to \mathfrak g$ that is obtained by conjugating $\varphi_t$ with $g^{(k)}_{t}$.

To prove the claim, we compute
$\varphi_{t}(\gamma\eta)=\varphi_{t}(\gamma)\varphi_{t}(\eta)$ using
 the Baker--Campbell--Hausdorff formula:
\begin{align*}
    \varphi_{t}(\gamma \eta)&= 
    \exp(\sum_{n=1}^{\infty}u_{n}(\gamma)t^{n})\varphi(\gamma)\exp(\sum_{n=1}^{\infty}u_{n}(\eta)t^{n})\varphi(\eta) \nonumber \\
    &= \exp(\sum_{n=1}^{\infty}u_{n}(\gamma)t^{n})\exp(\Ad(\varphi(\gamma))(\sum_{n=1}^{\infty}u_{n}(\eta)t^{n}))\varphi(\gamma\eta) \nonumber \\
    &=\exp(A(t)+B(t)+\frac{1}{2}[A(t),B(t)]+\dotsb)\varphi(\gamma\eta), 
\end{align*}
where $A(t)=\sum_{n=1}^{\infty}u_{n}(\gamma)t^{n}$ and $B(t)=\Ad(\varphi(\gamma))(\sum_{n=1}^{\infty}u_{n}(\eta)t^{n})$. 
Hence, we have 
\begin{equation*}
    \sum_{n=1}^{\infty}u_{n}(\gamma\eta)t^{n}
    =A(t)+B(t)+\frac{1}{2}[A(t),B(t)]+\cdots.
\end{equation*} 

By assumption, the images of $u_{1},\dots, u_{k-1}$ are contained in $\mathfrak{l}$.  
Comparing the coefficients of $t^{k}$ in the above identity modulo $\mathfrak l$, we obtain  
\[
u_{k}(\gamma\eta)\equiv u_{k}(\gamma) + \Ad(\varphi(\gamma)) u_{k}(\eta)  \mod \mathfrak{l}.
\]
Thus, $u_{k} \bmod \mathfrak{l}$ defines a $1$-cocycle.  
By the assumption $H^{1}(\Gamma, \mathfrak{g}/\mathfrak{l}) = 0$,  
there exists an element $C_{k} \in \mathfrak{g}$ such that,  
\[
u_{k}(\gamma) \equiv \Ad(\varphi(\gamma)) C_{k} - C_{k} \mod \mathfrak{l},
\]
for any $\gamma\in \Gamma$. Now, we set $g_{t}^{(k)} := \exp(t^{k} C_{k})$.
Clearly, $u_{1}, \dots, u_{k-1}$ remain unchanged, when we take the conjugate $g_{t}^{(k)} \cdot \varphi_{t}$.
On the other hand, the coefficient of $t^k$ in 
\[
\log((g_{t}^{(k)} \cdot \varphi_{t}) \varphi(\gamma)^{-1})
=
\log(
\exp(t^k C_k) \exp(A(t)) \exp(- t^k \operatorname{Ad}(\varphi(\gamma)) C_k))
\]
amounts to
$C_k + u_k(\gamma) - \operatorname{Ad}(\varphi(\gamma)) C_k \in \mathfrak{l}$. 
Hence, our claim is proved.

We now use the claim inductively on $k\ge 1$.
We start with $\varphi_t$, and proceed with
$g_t^{(1)} \cdot \varphi_t$, $g_t^{(2)}g_t^{(1)}\cdot \varphi_t$, and so on.
Correspondingly, the family of maps from $\Gamma$ to $\mathfrak{g}$, 
\[
u_\ell^{(k)}:=u_\ell(g_t^{(k)} \dotsb g_t^{(1)} \cdot \varphi_t)
\]
has the following properties:
\[
u_\ell^{(k)} =   u_\ell^{(\ell)} \textrm{ and } u_\ell^{(k)}(\Gamma) \subset \mathfrak{l} \textrm{ for all } k \ge \ell.
\]

Consequently, we obtain the formal power series  
\[
g_t = \dotsb g_t^{(2)} g_t^{(1)}.
\]  
Then, $g_t$ provides a formal solution to the equation
\begin{equation}
\label{eqn:Artin_lemma_Hom}
g_t \cdot \varphi_t \in \Hom(\Gamma,L).
 \end{equation}
Since $\Hom(\Gamma,L)$ is a real analytic subset of $X = G^N$ and $\varphi_t$ is an analytic curve, the condition (\ref{eqn:Artin_lemma_Hom})
defines an analytic equation for $g_t$. 
By Artin's approximation theorem (\cite[Thm.~1.2]{artin1968solutions}),  
there exists an analytic solution $g_t \in G$ satisfying $g_0 = 1$.  
Thus, our assertion is proved.  
\end{proof}

\textbf{Step~2}. 
We complete the proof of Proposition~\ref{prop:local-rigid-g/l} 
by reducing it to the analytic curve case (Step~1).
The following discussion uses the curve selection lemma 
as in Goldman--Millson~\cite[Sect.~1]{goldman-millson87},  
where they use the fact from complex hyperbolic geometry that 
$G\cdot \Hom(\Gamma,L)$ is a locally algebraic subset of $\Hom(\Gamma,G)$
in the setting that $G=SU(n+1,1)$, $L=U(n,1)$, and $\Gamma$ is 
a cocompact discrete subgroup of $U(n,1)$. 
In the following, we simplify this discussion using the concept of subanalytic subsets, as stated in Lemma~\ref{lemma:Z-subanalytic}. 

Now we review briefly the definition and basic properties of subanalytic subsets (Definition~\ref{def:subanalytic} to Lemma~\ref{lemma:subanalytic-property}), 
as presented in Kashiwara--Schapira~\cite[Sect.~8.2]{kashiwara2013sheaves}.  
\begin{definition}
\label{def:subanalytic}
Let $X$ be a real analytic manifold. 
A subset $Z$ of $X$ is called \emph{subanalytic} if, 
for each point $x$ of $X$, 
there exists an open neighborhood $U$ of $x$ in $X$ 
such that
\begin{equation}
\label{eq:subanalytic}
Z\cap U = \bigcup_{i=1}^{p}(f_{i1}(Y_{i1})\smallsetminus f_{i2}(Y_{i2})),
\end{equation}
where, for each $i=1,\ldots,p$ and $j=1,2$, 
$Y_{ij}$ is a compact real analytic manifold 
and $f_{ij}\colon Y_{ij}\rightarrow U$ is a real analytic map. 
\end{definition}

\begin{remark}
\label{remark:subanalytic}
The above definition is equivalent to the following definition given in Bierstone--Milman~\cite[Prop.~3.13]{Bierstone-Milman-subanalytic}:
for each point $x$ of $X$, 
there exists a neighborhood $U$ of $x$ in $X$,
such that
\[
Z\cap U = \bigcup_{i=1}^{p}(f_{i1}(A_{i1})\smallsetminus f_{i2}(A_{i2})), 
\]
where, for each $i=1,\ldots,p$ and $j=1,2$, 
$A_{ij}$ is a closed real analytic subset of a real analytic manifold $Y_{ij}$,
$f_{ij}\colon Y_{ij}\rightarrow U$ is a real analytic map,
and $f_{ij}|_{A_{ij}}$ is proper.
The non-trivial implication from the definition in Remark~\ref{remark:subanalytic} to Definition~\ref{def:subanalytic} can be verified using 
the following Hironaka's uniformization theorem:
\end{remark}

\begin{fact}[{Hironaka's uniformization theorem \cite[Thm.~0.1]{Bierstone-Milman-subanalytic}}]
\label{fact:hironaka}
    Let $Z$ be a closed subanalytic subset of a real analytic manifold $X$
    in the sense of Remark~\ref{remark:subanalytic}. 
    Then there exist a real analytic manifold $Y$ and 
    a proper real analytic map $f\colon Y\rightarrow X$ such that 
    $f(Y)=Z$.
\end{fact}

We now explain immediate consequences of Definition~\ref{def:subanalytic} and Remark~\ref{remark:subanalytic}.  
By Definition~\ref{def:subanalytic}, a closed ball in Euclidean space $\mathbb{R}^n$ is subanalytic,  
as it can be expressed as the projection of the unit sphere $S^{n}$.  
Thus, every point in a real analytic manifold admits a compact subanalytic neighborhood.  
Furthermore, by Remark~\ref{remark:subanalytic}, any real analytic subset of a real analytic manifold $X$ is subanalytic.  

We also need the following:  
\begin{lemma}
    \label{lemma:subanalytic-property}
    Let $X$ be a real analytic manifold.
    \begin{enumerate}[label=(\arabic*)]       
        \item 
        \label{item:subanalytic-complement}
        If $Z_{1}$ and $Z_{2}$ are subanalytic subsets in $X$, then 
        $Z_{1}\smallsetminus Z_{2}$ is subanalytic in $X$.
        \item 
        \label{item:subanalytic-proper-image}
        Let $f\colon X\rightarrow Y$ be an analytic map 
        of real analytic manifolds. If $Z$ is a subanalytic subset of $X$ and 
        $f$ is proper on $\overline{Z}$, then $f(Z)$ is a subanalytic subset of $Y$.
        \item 
        \label{item:subanalytic-curve-selection}
        (Curve Selection Lemma)
        Let $Z$ be a subanalytic subset of $X$ and $z\in \overline{Z}$. 
Then there exists an analytic curve $c\colon (-1,1)\rightarrow X$ such that 
$c(0)=z$ and $c(t)\in Z$ for any $t\neq 0$.
    \end{enumerate}
\end{lemma}

Under the preparation above, we provide a lemma for the proof of 
Proposition~\ref{prop:local-rigid-g/l}:

\begin{lemma}
    \label{lemma:Z-subanalytic}
    Suppose that we are in Setting~\ref{setting:appendix-deform}.
    Let $W$ be a compact subanalytic neighborhood of the identity element in $G$.
    Then the subset 
    \begin{equation*}
    Z:=\Hom(\Gamma,G)\smallsetminus (W\cdot \Hom(\Gamma,L)).
    \end{equation*} 
    is subanalytic in $X=G^{N}$.
\end{lemma}

\begin{proof}
The map $G \times \Hom(\Gamma,G)\rightarrow \Hom(\Gamma,G)$ defined by $(g,\varphi)\mapsto g\cdot \varphi$
is proper on $W\times \Hom(\Gamma,L)$. 
Hence, the image 
$W\cdot \Hom(\Gamma,L)$ is a subanalytic subset of $X$ by Lemma~\ref{lemma:subanalytic-property}~\ref{item:subanalytic-proper-image}. 
By Lemma~\ref{lemma:subanalytic-property}~\ref{item:subanalytic-complement},
$Z$ is subanalytic in $X$.
\end{proof}

We are ready to prove Proposition~\ref{prop:local-rigid-g/l}.
\begin{proof}[Proof of Proposition~\ref{prop:local-rigid-g/l}]
Retain the notation from Lemma~\ref{lemma:Z-subanalytic}.  

To prove Proposition~\ref{prop:local-rigid-g/l}, 
it suffices to show that 
$G\cdot \Hom(\Gamma,L)$ is a neighborhood of $\varphi$ in $\Hom(\Gamma,G)$.
Suppose, for the sake of contradiction, that this is not the case.  
Then $W\cdot \Hom(\Gamma,L)$ is also not a neighborhood of $\varphi$ in $\Hom(\Gamma,G)$,  
which implies that $\varphi \in \overline{Z}$.  

By Lemma~\ref{lemma:Z-subanalytic}, the subset $Z$ is subanalytic in $X$.  
Applying the curve selection lemma (Lemma~\ref{lemma:subanalytic-property}~\ref{item:subanalytic-curve-selection}),  
we obtain an analytic curve $\{\varphi_t\}_{-1<t<1}$ in $X$ such that  
$\varphi_0 = \varphi$ and $\varphi_t \in Z$ for all $t \neq 0$.  
In particular, $\varphi_t \notin W \cdot \Hom(\Gamma,L)$ for all $t \neq 0$.  

On the other hand, since $\varphi_t \in \Hom(\Gamma,G)$, Lemma~\ref{lemma:local-rigid-g/l-analytic-vesrion} (Step~1) ensures the existence of an analytic curve $g_t$ in $G$  
such that $g_t \cdot \varphi_t \in \Hom(\Gamma,L)$. 
Hence, $\varphi_t \in W\cdot \Hom(\Gamma,L)$ for sufficiently small $t$,
which is a contradiction.  
Thus, $G\cdot \Hom(\Gamma,L)$ must be a neighborhood of $\varphi$ in $\Hom(\Gamma,G)$.  
This completes the proof of Proposition~\ref{prop:local-rigid-g/l}.  
\end{proof} 

Although the following result is not used in the main body of this article, we include it for future reference, as it can be derived using the same idea originally due to Goldman--Millson~\cite{goldman-millson87}. The proof outlines only the necessary modifications to Proposition~\ref{prop:local-rigid-g/l}.

In the following proposition, we recall that
a $1$-cochain $c \colon \Gamma \to \mathfrak{g}$
is a $1$-cocycle if it satisfies
\[
c(\gamma \eta)=c(\gamma)+\operatorname{Ad}(\varphi(\gamma)) c(\eta),
\]
for all $\gamma, \eta \in \Gamma$,
and that the associated $2$-cochain 
\[
[c, c] \colon \Gamma \times \Gamma \to \mathfrak{g},
\quad
[c, c](\gamma, \eta) := [c(\gamma), \Ad(\varphi(\gamma))(c(\eta))]
\]
is a $2$-cocycle whenever $c$ is a $1$-cocycle.
\begin{proposition}
    \label{prop:Goldman-Millson}
    In Setting~\ref{setting:appendix-deform}, let $\varphi \colon \Gamma \to G$ be a group homomorphism such that $\varphi(\Gamma) \subset L$. We regard the Lie algebras $\mathfrak{g}$ and $\mathfrak{l}$ as $\Gamma$-modules via $\varphi$. 
    Assume the following:
    \begin{enumerate}[label=(\arabic*)]
    \item There exists an $L$-invariant complementary subspace $\mathfrak{m}$ of $\mathfrak{l}$ in $\mathfrak{g}$;
    \item 
    $H^{1}(\Gamma,\mathfrak{l}/\mathfrak{z})=0$, where $\mathfrak{z}$ denotes the center of $\mathfrak{l}$;
    \item 
    \label{item:cup-product}
    For any non-zero element $c \in H^{1}(\Gamma,\mathfrak{m})$, we have
\[
\operatorname{pr}_{\mathfrak{l}}([c,c]) \neq 0
\quad \textrm{ in } H^{2}(\Gamma,\mathfrak{l}),
\]
where $\operatorname{pr}_\mathfrak{l}$ denotes the first projection onto $H^{2}(\Gamma,\mathfrak{l})$
    with respect to the direct sum decomposition 
    \[
    H^{2}(\Gamma,\mathfrak{g}) = H^{2}(\Gamma,\mathfrak{l}) \oplus H^{2}(\Gamma,\mathfrak{m}).
    \]
\end{enumerate}

    Then there exists a neighborhood $U$ of $\varphi$ in $\Hom(\Gamma, G)$ such that for any $\varphi' \in U$, 
     the image $\varphi'(\Gamma)$ is conjugate into $L$ by an element of $G$.
\end{proposition}

\begin{proof}
By applying a similar argument using the \emph{curve selection lemma} (see Lemma~\ref{lemma:subanalytic-property}~(3)) as in Step~2 of the proof of Proposition~\ref{prop:local-rigid-g/l}, our assertion reduces to the case of an analytic deformation.
Namely, for any family of group homomorphisms $\varphi_{t}\colon \Gamma\rightarrow G$, depending analytically on $t\in \R$, with $\varphi_{0}=\varphi$, 
it suffices to show that there exists an analytic curve $g_{t}$ in $G$ with $g_{0}=e$ such that $g_{t}\cdot \varphi_{t}\in \Hom(\Gamma,L)$ for all sufficiently small $t$.

Suppose that $\varphi_{t}\colon \Gamma\rightarrow G$ is a family of group homomorphisms that depends analytically on $t\in \R$, with $\varphi_{0}=\varphi$.
For each $\gamma \in \Gamma$, the real analytic curve $\varphi_t(\gamma) \varphi(\gamma)^{-1}$ in $G$, starts at $e$ when $t=0$.
The Taylor expansion of this curve defines a sequence of maps $u_\ell \equiv u_\ell(\varphi_t) \colon \Gamma \to \mathfrak{g}$ $(\ell=1,2,\dots)$,
such that 
\[
    \varphi_{t}(\gamma) = \exp(u_{1}(\gamma) t + u_{2}(\gamma) t^{2} + \dotsb) \varphi(\gamma),
\]
where the domain of convergence may depend on $\gamma$.

We fix $k \geq 1$, and prove the following claim: \textit{assume that $u_{\ell}(\Gamma) \subset \mathfrak{z}$ for $\ell=1,\dots,k-1$.  
Then, there exists an analytic curve in $G$, to be denoted by $g^{(k)}_{t}$, with the following two properties:}
\begin{enumerate}[label=(\arabic*)]
    \item $u_\ell (\varphi_t)=u_\ell(g^{(k)}_{t}\cdot \varphi_t)$ for all $\ell =1, \dots, k-1$,
    \item
     $u_k(g^{(k)}_{t}\cdot \varphi_t)(\Gamma) \subset \mathfrak{z}$,
\end{enumerate}
where $g^{(k)}_{t}\cdot \varphi_t$ denotes a homomorphism $\Gamma \to \mathfrak g$ that is obtained by conjugating $\varphi_t$ with $g^{(k)}_{t}$.

Once this claim has been proved, we may apply Artin's approximation theorem (\cite[Thm.~1.2]{artin1968solutions}) to the formal power series $g_t = \dotsb g_t^{(2)} g_t^{(1)}$  to conclude the proposition, just as in Step~1 of the proof of Proposition~\ref{prop:local-rigid-g/l}.

The remainder of the argument is devoted to proving the claim.
To this end, we compute
$\varphi_{t}(\gamma\eta)=\varphi_{t}(\gamma)\varphi_{t}(\eta)$ using
 the Baker--Campbell--Hausdorff formula:
\begin{align*}
    \varphi_{t}(\gamma \eta)&= 
    \exp(\sum_{n=1}^{\infty}u_{n}(\gamma)t^{n})\varphi(\gamma)\exp(\sum_{n=1}^{\infty}u_{n}(\eta)t^{n})\varphi(\eta) \nonumber \\
    &= \exp(\sum_{n=1}^{\infty}u_{n}(\gamma)t^{n})\exp(\Ad(\varphi(\gamma))(\sum_{n=1}^{\infty}u_{n}(\eta)t^{n}))\varphi(\gamma\eta) \nonumber \\
    &=\exp(A(t)+B(t)+\tfrac{1}{2}[A(t),B(t)]+\dotsb)\varphi(\gamma\eta), 
\end{align*}
where $A(t)=\sum_{n=1}^{\infty}u_{n}(\gamma)t^{n}$ and $B(t)=\Ad(\varphi(\gamma))(\sum_{n=1}^{\infty}u_{n}(\eta)t^{n})$. 
Hence, we have 
\begin{equation}
    \label{eq:BCH-formula}
    \sum_{n=1}^{\infty}u_{n}(\gamma\eta)t^{n}
    =A(t)+B(t)+\frac{1}{2}[A(t),B(t)]+\cdots.
\end{equation} 

Let us consider the decomposition 
$u_n = v_n + w_n$ $(n = 1, 2, \ldots)$ of the $1$-cochains corresponding to the decomposition $\mathfrak{g} = \mathfrak{l} + \mathfrak{m}$ as $L$-modules.
By assumption, the images of $v_{1},\dots, v_{k-1}$ are contained in $\mathfrak{z}$
and $w_{1},\dots, w_{k-1}$ vanish.  
Comparing the $\mathfrak{l}$-components of the coefficient of $t^{k}$ in the identity \eqref{eq:BCH-formula}, we see that
both $v_{k}$ and $w_{k}$ define $1$-cocycles.  
Since $H^{1}(\Gamma, \mathfrak{l}/\mathfrak{z}) = 0$ by assumption,  
there exists an element $C'_{k} \in \mathfrak{l}$ such that,  
\[
v_{k}(\gamma) \equiv \Ad(\varphi(\gamma)) C'_{k} - C'_{k} \mod \mathfrak{z},
\]
for all $\gamma\in \Gamma$. 

By comparing the $\mathfrak{l}$-component of the coefficient of $t^{2k}$ in the identity \eqref{eq:BCH-formula}, we obtain
\[
v_{2k}(\gamma\eta)= v_{2k}(\gamma) + \Ad(\varphi(\gamma)) v_{2k}(\eta) 
+\frac{1}{2}\big(\text{the $\mathfrak{l}$-component of }[w_{k},w_{k}](\gamma,\eta)\big).
\]
This shows that the $\mathfrak{l}$-component of the $2$-cocycle $[w_{k},w_{k}]$ is a $2$-coboundary.  
By assumption~\ref{item:cup-product}, $w_k\colon \Gamma \to \mathfrak{m}$ is a $1$-coboundary.
Hence,
there exists an element $C''_{k} \in \mathfrak{m}$ such that, for all $\gamma\in \Gamma$, 
\[
w_{k}(\gamma) = \Ad(\varphi(\gamma)) C''_{k} - C''_{k}.
\] 

Now, we set $C_{k}:=C'_{k}+C''_{k}\in \mathfrak{g}$ and define $g_{t}^{(k)} := \exp(t^{k} C_{k})\in G$.
Clearly, $u_{1}, \dots, u_{k-1}$ remain unchanged, when we take replace $\varphi_t$ with its conjugate $g_{t}^{(k)} \cdot \varphi_{t}\in G$.
On the other hand, the coefficient of $t^k$ in 
\[
\log((g_{t}^{(k)} \cdot \varphi_{t}) \varphi(\gamma)^{-1})
=
\log(
\exp(t^k C_k) \exp(A(t)) \exp(- t^k \operatorname{Ad}(\varphi(\gamma)) C_k))
\]
is equal to
$C_k + u_k(\gamma) - \operatorname{Ad}(\varphi(\gamma)) C_k \in \mathfrak{z}$. 
Hence, the claim is proved, and the proposition follows.
\end{proof}

\section{Continuous family of Zariski-dense subgroups}
\label{section:appendix-Zariski-dense-subgroups}

This section discusses a {\emph{continuous family}} of Zariski-dense subgroups in Zariski-connected, algebraic groups $G$.
In Theorem~\ref{thm:two_generates_reductive}, we determine
the optimal number of generators when $G$ is reductive.

Let $G$ be a Zariski-connected real algebraic group, and let $\mathfrak{g}$ be the real Lie algebra of $G$.
\begin{definition}
\label{def:GLtX_Zariski_closure}
For a finite subset $\mathcal{X}=\{X_1, \dots, X_n \}$,
    we denote by $G(\mathcal{X})$
    the identity component of the Zariski-closure of the subgroup of $G$ generated by $e^{X_{1}},\ldots,e^{X_{n}}$.
It is also useful to introduce the notation
$G(L;\mathcal{X})$, which denotes the identity component of the Zariski-closure of the subgroup of $G$ generated by
 $L$ and $e^{X_{1}},\ldots,e^{X_{n}}$,
when $L$ is a subgroup of $G$.

For each fixed $t \in \mathbb{R}$,  
we define $t\mathcal{X} := \{tX \mid X \in \mathcal{X}\}$, and accordingly define the
Zariski-connected subgroups $G(t\mathcal X)$ and $G(L;t\mathcal{X})$, separately for each $t$.
\end{definition}
\begin{example}
\label{example:unipotent-g(X)}
Let $G$ be a unipotent real algebraic group, $X \in \mathfrak{g}$, and set
$\mathcal X =\{X\}$.
Then the Lie algebra $\mathfrak{g}(\mathcal X)$ of $G(\mathcal X)$---the Zariski-closure of the subgroup generated by $\exp(X)$---is $\mathbb{R} X$,
because the exponential map $\exp \colon \mathfrak{g} \to G$ is an isomorphism of algebraic groups.  
\end{example}

In Definition~\ref{def:GLtX_Zariski_closure}, it is important to note that the subgroups $G(t\mathcal{X})$ do not necessarily vary continuously with $t$, as will already be evident in the abelian case (see Proposition~\ref{proposition:commutative-eta} below).

 \begin{definition}
 \label{def:eta}
 Let $G$ be a Zariski-connected real algebraic group.
We define $\underline{\eta}(G), \eta(G), \overline\eta(G)$
    as the minimal cardinalities 
 of finite subset $\mathcal{X}$ satisfying the following conditions, respectively:
    \begin{align*}
     \overline\eta(G)&: \quad
      G(t\mathcal{X})=G  \textrm{ for all } t>0;
        \\
        \eta(G)&: \quad
        G(t\mathcal{X})=G
        \textrm{ for any } \delta > t>0 \textrm{ with some } \delta>0; 
         \\
        \underline{\eta}(G)&: \quad
       G(\mathcal{X})=G.
    \end{align*}
\end{definition}

We observe that $\eta(G)$ and $\overline{\eta}(G)$ depend on the {\emph{family}} of Zariski-connected subgroups $G(t\mathcal{X})$, indexed by the continuous parameter $t$, whereas $\underline{\eta}(G)$ does not involve any continuous parameter.
It is clear that
\[
\underline{\eta}(G)\leq \eta(G) \le \overline{\eta}(G).
\]

Although our primary focus is on $\eta(G)$ for its application in 
 Sections~\ref{section:outline-bending} 
 and \ref{section:step3}, we frequently discuss $\underline{\eta}(G)$ and $\overline{\eta}(G)$ as well.
First of all, we show that
the value of $\underline\eta(G)$ can be arbitrarily large: 
\begin{example}
\label{example:eta-additive}
    For $G=\R^{n}$ regarded as a unipotent  commutative algebraic group,
    we have
        $\underline{\eta}(G)= \eta(G)=\overline{\eta}(G)=n$.
\end{example}

In contrast to the unipotent commutative case discussed in Example~\ref{example:eta-additive},
there exists an upper bound for $\overline{\eta}(G)$ in the reductive algebraic group.
\begin{theorem}
\label{thm:two_generates_reductive}
    Let $G$ be a Zariski-connected real reductive algebraic group. Then we have
    \[
    \underline{\eta}(G)= \eta(G)=\overline\eta(G)=2,
    \]
    except for the following cases: 
    \begin{itemize}
        \item ($\underline{\eta}, \eta, \overline{\eta})=(1,1,1)$
        when 
        $G$ is a split $\R$-torus 
        $(G\simeq (\R^{\times})^{b})$;
\item $(\underline\eta, \eta, \overline\eta)=(1,2,2)$ when 
        $G$ is a non-split $\R$-torus $(G\simeq \T^{a}\times (\R^{\times})^{b})$;
        \item $(\underline\eta, \eta, \overline\eta)=(2,2,3)$ when 
        $\mathfrak{g}$ contains $\mathfrak{su}(2)$ as an ideal.
    \end{itemize}
\end{theorem}

The upper bounds for $\underline{\eta}(G)$, $\eta(G)$, and $\overline\eta(G)$ when $G$ is \emph{non-reductive} will be given in Lemma~\ref{lemma:basic_properties_eta}.

\begin{remark}
It is known that $\underline{\eta}(G) = 2$ for semisimple $G$; see, for example, Labourie~\cite[Lem.~5.3.13]{Labourie_lecture} for a proof in the case where $G$ is the split group $SL(n, \mathbb{R})$. 
Our approach to the invariants $\eta$ and $\overline\eta$ as well as $\underline\eta$, however, offers an alternative method of proof that extends to the family $G(t \mathcal{X})$ as the parameter $t \in \R$ varies.
\end{remark}

Before turning to the proof of Theorem~\ref{thm:two_generates_reductive}, we first summarize some basic properties of  $\underline{\eta}(G)$, $\eta(G)$, and $\overline\eta(G)$ in Lemmas~\ref{lemma:hom->g(phi)} and \ref{lemma:basic_properties_eta} without assuming that $G$ is reductive.

\begin{lemma}
\label{lemma:hom->g(phi)}
    For any homomorphism of real algebraic groups $\varphi \colon G \to G'$ and any finite subset $\mathcal{X} \subset \mathfrak{g}$,  
    we have the identity  
    \[
    d\varphi(\mathfrak{g}(\mathcal{X})) =
    \mathfrak{g}'(d\varphi(\mathcal{X})).
    \]
\end{lemma}
\begin{proof}
Consider the extension of $\varphi$ to a morphism of complex algebraic groups  
$\varphi_{\mathbb{C}} \colon G_{\mathbb{C}} \to G'_{\mathbb{C}}$.  
Then $\varphi_{\mathbb{C}}(G_{\mathbb{C}}(\mathcal{X}))$  is Zariski-closed and contains $\exp(d\varphi(X))$  
for each $X \in \mathcal{X}$.  
Hence $\varphi_{\mathbb{C}}(G_{\mathbb{C}}(\mathcal{X})) \supset G'_{\mathbb{C}}(d\varphi(\mathcal{X}))$.  
The reverse inclusion follows from the Zariski-continuity of $\varphi_{\mathbb{C}}$.  
Therefore, we obtain the equality
\[\varphi_{\mathbb{C}}(G_{\mathbb{C}}(\mathcal{X})) = G'_{\mathbb{C}}(d\varphi(\mathcal{X})).\]  
Taking the corresponding real Lie algebras then yields the desired identity.
\end{proof}
\begin{definition}
\label{def:kappa}    
Suppose that $A$ is an associated $\R$-algebra, and that $V$ is a finitely generated $A$-module.
Let $\kappa(A, V)$ denote the minimal number of generators of $V$ as an $A$-module.
When $V$ is a $G$-module for a group $G$, we define 
\[ \kappa(G, V):=\kappa(\R[G], V),
\] where $\R[G]$ is the group ring.
\end{definition}

For the reader's convenience, we provide an explicit formula for $\kappa(S, \mathfrak{u})$ in terms of the multiplicities of irreducible representations.
In the following lemma, the formula applies with $A = U(\mathfrak{s})$ and $V = \mathfrak{u}$.

\begin{lemma}
    Let $A$ be an associative $\R$-algebra, 
    and $V$ an $A$-module which is finite-dimensional over $\R$. 
    Assume 
    $V$ decomposes 
    $\bigoplus_{i} V_{i}^{n_{i}}$,
    where $V_{1}, V_{2},\ldots$ 
    are mutually non-isomorphic irreducible $A$-module and $n_{i}\in \N$ is the multiplicity of $V_{i}$ in $V$. 
    Put $D_{i}:=\End_{A}(V_{i})$,
    which is isomorphic to either 
    $\R$, $\C$, or $\HA$.
    We have 
    \[
    \kappa(A,V) = 
    \max_{i}(\lceil\frac{n_{i}}{\dim_{D_{i}}V_{i}}\rceil). 
    \]
\end{lemma}
\begin{proof}
By Jacobson's density theorem,  
the homomorphism of $\R$-algebras $A \rightarrow \prod_{i} \End_{D_{i}}(V_{i})$
induced by the action of $A$ on $V_{i}$ is surjective.  
Therefore, we see that the $A$-module $\bigoplus_{i} V_{i}^{m_{i}}$ $(m_{i}\in \N)$ is cyclic if and only if  $m_{i}\leq \dim_{D_{i}} V_{i}$. 
Our conclusion follows from this fact. 
\end{proof}

\begin{lemma} 
\label{lemma:basic_properties_eta}
Let $G$ be a Zariski-connected real algebraic group. In what follows, $\xi$ denotes any one of $\underline{\eta}$, $\eta$, or $\overline\eta$. 
\begin{enumerate}[label=(\arabic*)]
    \item (Quotient).
    Let $\varphi\colon G\rightarrow G'$
    be a homomorphism of Zariski-connected real algebraic groups such that
    $d\varphi$ is surjective. 
    Then 
    \[
   \xi(G) \ge \xi(G').
    \]
    Furthermore, if $d\varphi$ is injective, then equality holds. 
\item 
\label{item:almost-direct-group}
(Almost direct group).
Let $S$ and $V$ be Zariski-connected normal subgroups such that $G=S\cdot V$ with $S \cap V$ finite, and with the Lie algebra $\mathfrak{s}$ being semisimple.
Viewed as $\mathfrak{s}$-modules, we assume that the irreducible constituents of $\mathfrak{v}$ and $\mathfrak{s}$ are distinct.
Then we have
\[
\xi(G) = \max(\xi(S), \xi(V)).
\]
\item (Levi decomposition)
     Let $G=S\cdot U$ be a Levi decomposition of 
    a Zariski-connected real  algebraic group $G$, 
    where $S$ is a maximal real reductive algebraic subgroup and $U$ is the unipotent radical of $G$.
    Then, we have
    \[
    \xi(G)\leq 
    \xi(S)+\kappa(S,\mathfrak{u}).
    \]  
\end{enumerate}
\end{lemma}
\begin{proof}
In the proof below, since the argument remains unchanged when the parameter $t$ is included, we treat $\underline{\eta}$, $\eta$, and $\overline\eta$ simultaneously, without explicitly introducing the parameter $t$. 

(1).
Let $\mathcal{X}$ be a finite subset of the Lie algebra $\mathfrak{g}$ such that $G(\mathcal{X}) = G$.  
By Lemma~\ref{lemma:hom->g(phi)}, we have
\[\mathfrak{g}'(d\varphi(\mathcal{X})) = d\varphi(\mathfrak{g}(\mathcal{X})) = \mathfrak{g}'.
\]  
Since $G'$ is Zariski-connected, it follows that $G'(d\varphi(\mathcal{X})) = G'$.  
Hence, the desired inequality follows.

Now suppose that $d\varphi$ is injective (and hence bijective).  
Let $\mathcal{X}'$ be a finite subset of the Lie algebra $\mathfrak{g}'$ such that $G'(\mathcal{X}') = G'$.  
Set $\mathcal{X}:=(d\varphi)^{-1}(\mathcal{X}')$.
Again, by Lemma~\ref{lemma:hom->g(phi)}, 
\[d\varphi(\mathfrak{g}(\mathcal{X})) = \mathfrak{g}'(\mathcal{X}') = \mathfrak{g}'.
\]  
This implies that $\mathfrak{g}(\mathcal{X}) = \mathfrak{g}$.  
Since $G$ is Zariski-connected, we obtain $G(\mathcal{X}) = G$, which proves the reverse inequality.

(2).
We know that $\xi(G)\geq \max(\xi(S),\xi(V))$ by (1). We now prove the reverse inequality.

Let $\operatorname{pr}_1 \colon \mathfrak{g} \to \mathfrak{s}$
and
$\operatorname{pr}_2 \colon \mathfrak{g} \to \mathfrak{v}$ be the projection maps
associated with the direct sum decomposition $\mathfrak{g}=\mathfrak{s}+\mathfrak{v}$.
Choose a finite subset $\mathcal{X}=\{X_1, \dots, X_m\} \subset \mathfrak{s}$ such that $S(\mathcal{X}) = S$, 
 a finite subset $\mathcal{Y}=\{Y_1, \dots, Y_n\} \subset \mathfrak{v}$ 
 such that $V(\mathcal{Y}) = V$.
Define 
\[\mathcal{Z}=\{ Z_1, \ldots, Z_{\max(m,n)}\} \text{ with } Z_i:=X_i+Y_i,
\]
where we take $X_i= 0$ for $i>m$ and $Y_i=0$ for $i>n$.
Let $\mathfrak g'$ denote the Lie algebra of $G(\mathcal{Z})$.

By Lemma~\ref{lemma:hom->g(phi)}, we have $\operatorname{pr}_1(\mathfrak{g}')=
\mathfrak{s}$ and
$\operatorname{pr}_2(\mathfrak{g}')=
\mathfrak{v}$.
Since none of the $\mathfrak{s}$-irreducible constituents of $\mathfrak{v}$ appear in $\mathfrak{s}$, it follows from
$\operatorname{pr}_1(\mathfrak{g}')=\mathfrak{s}$ that
$\mathfrak{g}'$ contains the direct sum $\mathfrak{s}\oplus 0$.
Therefore, from
$\operatorname{pr}_2(\mathfrak{g}')=
\mathfrak{v}$, we conclude that $\mathfrak{g}'=\mathfrak{g}$, and consequently, $G(\mathcal{Z})=G$. 
Hence, we have established the reverse inequality
$\xi(G)\leq \max(\xi(S),\xi(V))$. 

(3).
  Let $\mathcal{X}$ be a finite subset of $\mathfrak{s}$ such that  
$G(\mathcal{X}) = S$, and 
let $\mathcal{Y}$ be a finite generating set of the $S$-module $\mathfrak{u}$
such that $\sharp\mathcal{Y}=\kappa(S, \mathfrak{u})$.
Then, the group $G(\mathcal{X} \cup \mathcal{Y})$ contains $G(\mathcal{X}) = S$, 
and consequently, 
\[
G(\mathcal{X} \cup \mathcal{Y})=G(S; \mathcal{Y}).
\]
As seen in Example~\ref{example:unipotent-g(X)},  
the Lie algebra $\mathfrak{g}(\mathcal{X} \cup \mathcal{Y})$ contains $\mathcal{Y}$.  
It follows that
\[\mathfrak{g}(\mathcal{X} \cup \mathcal{Y}) = 
\mathfrak{s}+\mathfrak{u}
=\mathfrak{g}.
\]
Since $G$ is Zariski-connected, we conclude that 
$G(\mathcal{X} \cup \mathcal{Y}) = G$.  
This yields the desired inequality.
\end{proof}

As an immediate consequence of Lemma~\ref{lemma:basic_properties_eta}~\ref{item:almost-direct-group},
we obtain the following result in the reductive setting.
\begin{proposition}
\label{prop:reductive_eta}
Let $G$ be a Zariski-connected, reductive algebraic group,
and suppose that
\[G=G_0 \cdot G_1 \dotsb G_r\]
is the almost direct product,
where $G_0$ is commutative,
and the Lie algebras $\mathfrak{g}_i$ are of the form
$\mathfrak{s}_i^{\oplus n_i}$,
with each $\mathfrak{s}_i$ a simple Lie algebra not isomorphic to $\mathfrak{s}_j$ for $1 \le i \neq j \le r$.
In what follows, $\xi$ denotes any one of $\underline{\eta}$, $\eta$, or $\overline\eta$. 
Then we have
\[
  \xi(G)= \max_{0\le i \le r} {\xi(G_i)}. 
\]
\end{proposition}

\begin{remark}
Let $\xi$ denote any one of $\underline{\eta}$, $\eta$, or $\overline\eta$. 
When $\mathfrak{g}$ is the direct sum of $n$ of a simple Lie algebra $\mathfrak{s}$ of a fixed type,
the invariant $\xi(G)$ does not depend on $n$.
We will give an explicit formula---for example, for $\overline{\eta}(G)$--- in 
Propositions~\ref{prop:three_generators_su2_n} and \ref{prop:overline_eta_simple_n}.
\end{remark}

\medskip
We now begin the proof of Theorem~\ref{thm:two_generates_reductive}.
We first consider the elementary case in which $G$ is commutative, addressed in Proposition~\ref{proposition:commutative-eta}.
We then prove Proposition~\ref{prop:two_generates_reductive} which shows the result for $\underline{\eta}(G)$ and $\eta(G)$ in the general case where $G$ is reductive. 
Finally, the results for $\overline{\eta}(G)$, when $\mathfrak{g}$ is the direct sum of a simple Lie algebras of a fixed type are given in Propositions~\ref{prop:three_generators_su2_n} and \ref{prop:overline_eta_simple_n}.
These complete the proof of Theorem~\ref{thm:two_generates_reductive}.

\medskip

In the commutative case, the real algebraic group $G$ decomposes as a direct sum of several copies of
\[
\mathbf{SO}_{2} := \left\{ \begin{pmatrix} x & -y \\ y & x \end{pmatrix} \;\middle|\; x^{2} + y^{2} = 1 \right\}, \quad
\mathbf{G}_{m} := \left\{ \begin{pmatrix} x & 0 \\ 0 & y \end{pmatrix} \;\middle|\; xy = 1 \right\}.
\]
The groups of real points of these algebraic groups are isomorphic, as Lie groups, to 
$\T := \{ e^{it} \mid t \in \mathbb{R} \}$ and $\mathbb{R}^{\times}$, respectively.  
In what follows, we regard these Lie groups as equipped with the above real algebraic group structures, and we use the corresponding notation accordingly.

For $X=(X_1, \dots, X_n) \in \R^n$
such that $X_{1},\ldots,X_{n}$ are linearly independent over $\Q$,
we define a countable subset of $\R$ by
\begin{equation}
\label{eqn:exclude_rational}
\Lambda(X):=\bigcup_{\ell \in \N} \bigcup_{\lambda\in \Z^n\smallsetminus\{0\}} \frac{\ell}{\sum_{j=1}^n\lambda_j X_j}. 
\end{equation}

\begin{proposition}
    \label{proposition:commutative-eta}
    Let $G=\T^{a}\times (\R^{\times})^{b}$.
    Suppose that real numbers $X_{1},\ldots,  X_{a+b}$ are linearly independent over $\Q$. 
    Then 
        \[
        (e^{\sqrt{-1}t X_{1}},\ldots,e^{\sqrt{-1}t X_{a}},e^{t X_{a+1}},\ldots,e^{t X_{a+b}})\in \T^{a}\times (\R^{\times})^{b}
        \]
        generates a Zariski-dense subgroup in $G$ for any $t\in \R\smallsetminus
       \pi \Lambda(X_1, \ldots, X_a)$. 
    Furthermore, we have
    \[
    \underline{\eta}(G)=1,\quad
    \eta(G)=\overline{\eta}(G)=\begin{cases}
        2 & (a>0), \\
        1 & (a=0).
    \end{cases}
    \]
\end{proposition}

Before proving Proposition~\ref{proposition:commutative-eta}, we present two preliminary lemmas.
\begin{lemma}
    \label{lem:jordan-decomposition}
    Let $G=\T^{a}\times (\R^{\times})^{b}$,
    and let $\operatorname{pr}_{i}$ denote the projection to the $i$-th factor for $i=1,2$.  
If $G'$ is a real algebraic subgroup of $G$ such that $\operatorname{pr}_1(G')=\T^a$ and
$\operatorname{pr}_2(G')=(\R^\times)^b$.
    Then $G'=G$. 
\end{lemma}

\begin{proof}
In our setting, every element $g$ of $G$ is semisimple, and its Jordan decomposition is given by $g=g_e g_h$, where $g_e=(\operatorname{pr}_1(g),1)$
and $g_h=(1,\operatorname{pr}_2(g))$ are
the elliptic and hyperbolic parts, respectively. 
Since the Jordan decomposition in the real  algebraic subgroup $G'$ is compatible with that of $G$, we have
$G' = \operatorname{pr}_{1}(G') \times \operatorname{pr}_{2}(G')$.
Hence, we conclude $G' = G$.
\end{proof}

\begin{lemma}
    \label{lemma:eta-three-pattern}
     Suppose that real numbers $X_{1},\ldots,X_{n}$ are linearly independent over $\Q$.
    \begin{enumerate}[label=(\arabic*)]
        \item 
       When $G=(\R^{\times})^{n}$,
 $\exp(X)=(e^{X_{1}},\ldots,e^{X_{n}})\in G$ generates a Zariski-dense subgroup in $G$. 
        In particular, 
        $\underline{\eta}(G)= \eta(G)
        =\overline{\eta}(G)=1$.
        \item 
 When $G=\T^{n}$,      
        $\exp(\pi\sqrt{-1}t X)\in G$ generates a Zariski-dense subgroup of $G$ if $t\in \R\smallsetminus \Lambda(X)$, 
        where $\Lambda(X)$ is defined as in \eqref{eqn:exclude_rational} for $X=(X_{1},\ldots,X_{n})$. 
        Furthermore, we have
        $\underline{\eta}(G)=1$ and 
        $\eta(G)=\overline{\eta}(G)=2$. 
    \end{enumerate}
\end{lemma}

\begin{proof}
    (1). This is essentially
    Chevalley's lemma 
    (e.g., \cite[Lem.~2.A.1.2]{Wallach-real-reductive-group}).
    
    In fact, suppose that a polynomial  
 $f(x_{1}, \ldots, x_{n}) = \sum_{\alpha\in \mathbb{N}^{n}} a_{\alpha} x^{\alpha}$
  vanishes on $\exp(\Z X)$, 
 where $a_{\alpha} \in \mathbb{R}$ and $x^{\alpha} = x_{1}^{\alpha_{1}} \cdots x_{n}^{\alpha_{n}}$. This means that
    \[
    \sum_{\alpha \in \mathbb{N}^{n}} a_{\alpha} \exp(\langle \alpha, X \rangle \ell) = 0  
    \quad \text{for all } \ell \in \mathbb{Z},
    \]  
where $\langle \cdot, \cdot \rangle$ denotes the standard inner product on $\mathbb{R}^{n}$.  
Since the values $\langle \alpha, X \rangle$ are pairwise distinct for all $\alpha \in \mathbb{N}^{n}$ by the assumption on $X$, 
we deduce that $a_\alpha=0$ by downward induction, starting from the largest $\langle \alpha, X\rangle$, and thus $f \equiv 0$.
This shows $\underline{\eta}(G)=1$.
Since the assumption on $X$ remains unchanged if we replace $X$ with a non-zero scalar multiple $t X$, we conclude $\eta(G)=\overline{\eta}(G)=1$.

    (2). 
    If $t \in \mathbb{R} \smallsetminus  \Lambda(X)$,  
    then the set $\exp(\pi \sqrt{-1} \mathbb{Z} t X)$ is dense in $\mathbb{T}^{n}$  
    with respect to the usual topology, and hence also Zariski-dense.  
    Therefore, we conclude that $\underline{\eta}(\mathbb{T}^{n}) = 1$.
    
    Next, we show that $\eta(\mathbb{T}^{n}) \geq 2$.  
    Let $X = (X_{1}, \ldots, X_{n}) \in \mathbb{R}^{n}$ be arbitrary.  
    Then, there exists a positive number $t > 0$ as small as desired  
    such that $t X_{1} \in \mathbb{Q} \pi$.  
    For such $t$, $G(t\{X\})$ is contained
    in the $1 \times \T^{n-1}$ because
    the projection of $\exp(\sqrt{-1} \mathbb{Z} t X) \subset \mathbb{T}^{n}$  
    to the first component has finite image. 
    Hence, we have shown $\eta(\mathbb{T}^{n}) \geq 2$.

    Finally, we show that $\overline{\eta}(\mathbb{T}^{n}) \leq 2$.  
    We choose $X = (X_{1}, \ldots, X_{n})$ and $Y = (Y_{1}, \ldots, Y_{n})$ such that $X_{1}, \ldots, X_{n}, Y_{1}, \ldots, Y_{n} \in \mathbb{R}$  
    are $\mathbb{Q}$-linearly independent,
    and that $\Lambda(X)\cap \Lambda(Y)=\{0\}$.  
Then, for any positive real number $t$,
at least one of the two elements $\exp(t X)$ and $\exp(t Y)$  
    generates a Zariski-dense subgroup in $G$.  
    Therefore, we conclude that $\overline{\eta}(\mathbb{T}^{n}) \leq 2$,  
    and the proof of (2) is complete.
\end{proof}

\begin{proof}[Proof of Proposition~\ref{proposition:commutative-eta}]
Both assertions in Proposition~\ref{proposition:commutative-eta}  
follow immediately from  
Lemmas~\ref{lem:jordan-decomposition} and~\ref{lemma:eta-three-pattern}.
\end{proof}

We now consider the case where $G$ is non-commutative.
\begin{proposition}
\label{prop:two_generates_reductive}
 Let $G$ be a Zariski-connected, non-commutative, real reductive algebraic group.
 Then we have
$\underline{\eta}(G) = \eta(G) = 2$.    
\end{proposition}

To prove Proposition~\ref{prop:two_generates_reductive}, we introduce useful concepts and provide two lemmas.

Suppose that $\mathfrak{j}_\C$ is a Cartan subalgebra of a complex reductive Lie algebra $\mathfrak{g}_\C$.
Let $\Delta(\mathfrak{g}_\C, \mathfrak{j}_\C)$ denote the root system. We write
\[
\mathfrak{g}_\C
=
\mathfrak{j}_\C \oplus \bigoplus_{\alpha \in \Delta(\mathfrak{g}_\C, \mathfrak{j}_\C)} \mathfrak{g}_{\C,\alpha}
\]
for the root decomposition.
Accordingly, for $X \in \mathfrak{g}_\C$, we decompose as
$X=X_0 + \sum X_\alpha$.

Under this notation, we introduce the notion of $\mathfrak{j}$-{\emph{filled}}, which will help to clarify the structure of the proof of Proposition~\ref{prop:two_generates_reductive}.
\begin{definition} 
\label{def:spread_Cartan}
We say that $X\in \mathfrak{g}_\C$ is $\mathfrak{j}_\C$-\emph{filled} if it decomposes as $X=X_0+\sum X_\alpha$ with $X_\alpha \neq 0$ for every $\alpha \in \Delta(\mathfrak{g}_\C, \mathfrak{j}_\C)$.
Let $\mathfrak{g}_{\C}(\mathfrak{j}_\C)_{\textrm{fill}}$ denote the set of $\mathfrak{j}_\C$-filled elements. 

When $\mathfrak g$ is a real reductive Lie algebra and $\mathfrak j$ is a Cartan subalgebra, we write $\mathfrak{g}_\C$ and $\mathfrak{j}_\C$ for their complexifications, respectively.
We say that an element $X$ of $\mathfrak{g}$ is $\mathfrak{j}$-\emph{filled} if
$
X \in \mathfrak{g}(\mathfrak{j})_{\textrm{fill}}
:=\mathfrak{g}
\cap \mathfrak{g}_\C(\mathfrak{j}_\C)_{\textrm{fill}}.
$
\end{definition}

Here are basic properties of $\mathfrak{j}$-filled elements:
\begin{lemma}
\label{lem:filled_elements}
\begin{enumerate}[label=(\arabic*)]
    \item The set $\mathfrak{g}(\mathfrak{j})_{\textrm{\emph{fill}}}$ is open dense in $\mathfrak{g}$ with respect to the usual topology.
\item
If $X \in \mathfrak{g}(\mathfrak{j})_{\textrm{\emph{fill}}}$, then the Lie algebra generated by $\mathfrak{j}$ and $X$ is equal to $\mathfrak{g}$.
\item
If $X \in \mathfrak{g}(\mathfrak{j})_{\textrm{\emph{fill}}}$, then
 $X \in \mathfrak{g}(\operatorname{Ad}(g)\mathfrak{j})_{\textrm{\emph{fill}}}$
if $g\in G$ is sufficiently close to the identity element $e$.    
\end{enumerate}
\end{lemma}
\begin{proof}
Since $\mathfrak{g}_{\C}(\mathfrak{j}_\C)_{\textrm{fill}}$ is Zariski-open in $\mathfrak{g}_\C$,
 $\mathfrak{g}(\mathfrak{j})_{\textrm{fill}}$ is open dense in $\mathfrak{g}$.
Since every root space $\mathfrak{g}_{\C,\alpha}$ is one-dimensional, the complex Lie algebra generated by $\mathfrak{j}$ and a $\mathfrak{j}$-filled element $X$ is $\mathfrak{g}_\C$. 
Hence, the second statement holds. The last statement follows from the continuity of the root decomposition with respect to the choice of Cartan subalgebras.
\end{proof}

\begin{lemma}
\label{lemma:two_filled_elements}
Let $\mathfrak{g}$ be a real reductive Lie algebra.
Then there exists a pair of elements $X$ and $Y$ satisfying the following properties:
let $\mathfrak{j}_1$ be the centralizer of $X$ in $\mathfrak{g}$, and $\mathfrak{j}_2$ be that of $Y$ in $\mathfrak{g}$. Then
\begin{enumerate}
    \item[(1)]  both $\mathfrak{j}_1$ and $\mathfrak{j}_2$ are Cartan subalgebras of $\mathfrak{g}$;
    \item[(2)]  $X$ is $\mathfrak{j}_2$-filled, and $Y$ is $\mathfrak{j}_1$-filled.
\end{enumerate}   
\end{lemma}
Let $\mathfrak{g}^{\textrm{ss}}_{\textrm{reg}}$
denote the set of semisimple, regular elements in $\mathfrak{g}$.
Then $\mathfrak{g}^{\textrm{ss}}_{\textrm{reg}}$ contains an open dense subset of $\mathfrak g$.
\begin{proof}[Proof of Lemma~\ref{lemma:two_filled_elements}
]
 We fix a Cartan subalgebra $\mathfrak{j}$, and take $Y$ from  $ \mathfrak{g}^{\textrm{ss}}_\textrm{{reg}} \cap \mathfrak{g}(\mathfrak{j})_{\textrm{fill}}$.
By Lemma~\ref{lem:filled_elements}~(1), such a $Y$ exists. 
 The centralizer of $Y$ in $\mathfrak{g}$, denoted by $\mathfrak{j}_2$, is a Cartan subalgebra of $\mathfrak{g}$ because $Y \in \mathfrak{g}^{\textrm{ss}}_{\textrm{reg}}$.
 
 Since $Y \in \mathfrak{g}(\mathfrak{j})_{\textrm{fill}}$, by
 Lemma~\ref{lem:filled_elements}~(3), there exists an open neighborhood $V$ of $e$ in $G$ such that $Y \in \mathfrak{g}(\operatorname{Ad}(g)\mathfrak{j})_{\textrm{fill}}$ for any $g \in V$.

 Since $\operatorname{Ad}(V) \mathfrak{j}$ contains a non-empty open subset of $\mathfrak{g}$, there exists $g \in V$ such that
$\operatorname{Ad}(g) \mathfrak{j} \cap \mathfrak{g}^{\textrm{ss}}_{ \textrm{reg}} \cap \mathfrak{g}(\mathfrak{j}_2)_{\textrm{fill}} \neq \emptyset$.
We then choose $X$ from this set.

Let $\mathfrak{j}_1$ denote the centralizer of $X$ in $\mathfrak{g}$, which equals the Cartan subalgebra $\operatorname{Ad}(g) \mathfrak{j}$.
Then $X$ is $\mathfrak{j}_2$-filled, and $Y$ is $\mathfrak{j}_1$-filled.
Thus, the lemma is proved.
\end{proof}

We are ready to prove Proposition~\ref{prop:two_generates_reductive}.
\begin{proof}[Proof of  Proposition~\ref{prop:two_generates_reductive}]
Since $G$ is non-commutative, it is clear that
$2 \le \underline{\eta}(G)$.
Therefore, it suffices to show $\eta(G)\le 2$.

We take two Cartan subalgebras $\mathfrak{j}_i$ $(i=1,2)$,
$X \in \mathfrak{j}_1$, and $Y \in \mathfrak{j}_2$, as
given in Lemma~\ref{lemma:two_filled_elements}.
Let $J_i$ be the Cartan subgroup of $G$
with Lie algebras $\mathfrak{j}_i$ $(i=1,2)$.

By Lemma~\ref{lem:filled_elements}~(1), we may replace $X$ with a regular semisimple element $X' \in \mathfrak{j}_1$ such that $X'$ is close enough to $X$ that $X'$ remains $\mathfrak{j}_2$-filled,
while $G(t\{X'\}) = J_1$ for any $t \in \R \smallsetminus \Lambda'$ by applying Proposition~\ref{proposition:commutative-eta} to the $\R$-torus $J_{1}$, where $\Lambda'$ is a countable set of $\R$ depending on $X'$,
as given in \eqref{eqn:exclude_rational}.

Similarly, we may replace $Y$ with a sufficiently close $Y' \in \mathfrak{j}_2$ such that $Y'$ remains $\mathfrak{j}_1$-filled, while
$G(t{Y'}) = J_2$ for any $t \in \R \smallsetminus \Lambda''$,
where $\Lambda''$ is a countable set of $\R$ depending on $Y'$.
Since both $\Lambda'$ and $\Lambda''$ are countable, we can take $s \in \R$, which is sufficiently close to $1$,
such that 
 $s \Lambda' \cap \Lambda''= \{0\}$.
We set 
\[
\mathcal{X} :=\{X', s Y' \}
\]
Then, for any $t \in \R\smallsetminus\{0\}$, at least one of the following holds:
\[
G(t\{X'\})=J_1 \ \text{ or } \ G(t\{sY'\})=J_2.
\]
For instance, suppose that $G(t\{sY'\})=J_2$.
Then $G(t\{X', sY'\})=G(J_2; t\{X'\})$.
We write $X'=X'+\sum X_\alpha'$ for root decomposition.
We take $H \in \mathfrak{j}_2$ such that $\alpha(H)\neq 0$ for any $\alpha \in \Delta(\mathfrak{g}_\C, \mathfrak{j}_{2,\C})$. Then there is an analytic curve $a\colon \R \to \mathfrak{g}$ such that
$a(t) \in \mathfrak{g}(t \{X'\})$ for all $t \in \R \smallsetminus \{0\}$ and $a(0)=[H,X']
=\sum_{\alpha} \alpha(H) X'_\alpha$, as stated in Lemma~\ref{lemma:analytic_gt_through_ux}.
In particular, $a(0) \in \mathfrak{g}_{\textrm{fill}}$.
Therefore, if $t$ is sufficiently small, then
$a(t) \in \mathfrak{g}_{\textrm{fill}}$. 
It follows from Lemma~\ref{lem:filled_elements}~(2) that the Lie algebra $\mathfrak{g}(t\{X', sY'\})$ of $G(t\{X', s Y'\})$ equals $\mathfrak{g}$.

Similarly in the case where $G(t\{X', sY'\})=G(J_1; t\{sY'\})$,
the proof works. 
Therefore, we have shown that $\mathfrak{g}(t\{X', sY'\})=\mathfrak{g}$ in either case.
Since $G$ is Zariski-connected, we have  $G(t\{X', s Y'\})=G$.
Thus, the proposition is proved.
\end{proof}

We now determine $\overline\eta(G)$
for reductive groups $G$, starting
with an exceptional case:
\begin{lemma}
\label{lem:three_generators_su2}
If $\mathfrak{g}=\mathfrak{su}(2)$, then $\overline\eta (G)=3$.
\end{lemma}
\begin{proof}
First, we show the inequality $2 < \overline{\eta}(G)$.
Let $X$ and $Y \in \mathfrak{g}$.
Then there exists $t \in \R$ such that $\exp(t Y)= 1$.
Consequently,
$G(t\{X, Y\})=G(t \{X\})$ is abelian,
which cannot coincide with $G$.
Hence, we have shown that $3 \le \overline{\eta}(G)$.

Conversely, we can choose a basis $X_1, X_2, X_3 \in \mathfrak{g}$ such that, for any $t \neq 0$, at least two of $G(t\{X_i\})$ ($i=1,2,3$)  are maximal tori in $G$. Therefore,
$G(t\{X_1, X_2,X_3\})=G$, showing
$\overline\eta(G)\le 3$.
\end{proof}

\begin{proposition}
\label{prop:three_generators_su2_n}
If $\mathfrak{g}$ is the direct sum of $n$ copies of $\mathfrak{su}(2)$, then $\overline\eta (G)=3$ for any $n \in \N_+$.
\end{proposition}

\begin{proof}
By Lemmas~\ref{lemma:basic_properties_eta} (1) and \ref{lem:three_generators_su2}, we have $\overline{\eta}(G) \le 3$.
Let us prove the reverse inequality.

We write $G$ as the almost direct product of groups $G = S^{(1)} \dotsb S^{(n)}$, which corresponds to the direct sum decomposition of the Lie algebra $\mathfrak{g} \simeq \mathfrak{s}^{(1)} \oplus \dotsb \oplus \mathfrak{s}^{(n)}$, where each $ \mathfrak{s}^{(i)}$ is isomorphic to $\mathfrak{su}(2)$. 
We choose $X_1^{(i)}, X_2^{(i)}, X_3^{(i)}$ in $\mathfrak{s}^{(i)}$ for every $1 \le i \le n$, as in Lemma~\ref{lem:three_generators_su2},
and define $X_j:=\sum_{i=1}^n X_j^{(i)}$ for $j \in\{ 1,2,3\}$.
Then, $G(t\{X_{1},X_{2},X_{3}\})=G$ for all $t\neq 0$.
\end{proof}
The following lemma indicates that the case $G=SU(2)$ is the only exception.
\begin{lemma}
\label{lemma:overline_eta_simple}
Suppose that $G$ is a Zariski-connected, real simple algebraic group such that $\mathfrak{g}$ is not isomorphic to $\mathfrak{su}(2)$.
There exist regular semisimple elements $X$ and $Y \in \mathfrak g$
such that $G(t\{X, Y\})=G$ for all $t \neq 0$.
In particular, $\overline\eta(G) = 2.$
\end{lemma}

For the proof, we need some notation and lemmas.

When $G_\C$ acts on $X_\C$,
we denote by $(X_\C)^g
:=\{x\in X_\C|gx =x\}$.

\begin{lemma}
Let $G_\C$ be a Zariski-connected, complex simple algebraic group,
$Z_{G_\C}$ the center, $J_\C$ a Cartan subgroup,  $H_\C$ a Zariski-connected algebraic subgroup containing $J_\C$, and $X_\C := G_\C/H_\C$. 
 \begin{enumerate}
     \item[(1)] $(X_\C)^g = X_\C$
     if and only if $g \in Z_{G_\C}$.
 \\
 \item[(2)]
 We set
 \begin{equation}
 \label{eqn:X_universal_fixed}
(X_\C)_\textrm{fix}
:=
 \bigcup_{g\in J_\C\smallsetminus Z_{G_\C}}
(X_\C)^g
 \end{equation}
Then $(X_\C)_\textrm{fix}$
is a Zariski-closed subset
with positive codimension in $X_\C$.
 \end{enumerate}
\end{lemma}
\begin{proof}
(1). The condition $(X_\C)^g = X_\C$ implies that 
      $g$ belongs to the proper normal subgroup
      $\bigcap_{\ell \in G_\C} \ell H_\C \ell^{-1}$.
      Since $G_\C$ is simple, we conclude $g \in Z_{G_\C}$.
     The converse statement is clear.
\newline
(2).  Since $H_\C$ contains the Cartan subgroup $J_\C$, the defining equation
$\exp(T) x = x$ in $X_\C$ for $T \in \mathfrak{j}_\C$ depends solely on the
finite set 
\[
 \Delta^T:=
 \{ \alpha\in \Delta(\mathfrak{g}_\C, \mathfrak{j}_\C)| e^{\alpha(T)}=1\}.
\]
Hence, the right-hand-side of
\eqref{eqn:X_universal_fixed} is a Zariski-closed subset with positive codimention in $X_\C$.
\end{proof}
For a fixed Cartan subgroup $J_\C$ of $G_\C$, we define
\begin{equation}
\label{eqn:G_C_fix_J}
(G_\C)_\textrm{fix}
:=
\bigcup_{J_\C \subset H_\C } \pi_{H_\C}^{-1}
(G_\C/H_\C)_\textrm{fix},
\end{equation}
where $H_\C$ runs over the finite set of all Zariski-connected, maximal algebraic subgroups containing the Cartan subgroup $J_\C$, and $\pi_{H_\C}\colon G_\C \to G_\C/H_\C$ is the natural quotient map.
Then,  $(G_\C)_\textrm{fix}$ is a Zariski-closed subset of $G_\C$.

\begin{lemma}
\label{lem:J_Y_generates_G}
Suppose $Y \in \operatorname{Ad}(\ell^{-1}) \mathfrak{j}_\C$ 
where $\ell \in G_\C \smallsetminus
 (G_\C)_\textrm{fix}$.
 If $\exp(Y) \notin Z_{G_\C}$,
then $G_\C(J_\C; \{Y\})= G_\C$.
\end{lemma}
\begin{proof}
If this were not the case, there would exist a Zariski-connected maximal algebraic subgroup $H_\C$ such that $ G(J_\C; \{Y\}) \subset H_\C$, where $Y=\operatorname{Ad}(\ell^{-1}) X$ with $X \in \mathfrak{j}_\C$

Let $o \in G_\C/H_\C$ denote the origin.
We observe that
\[
 \ell \cdot o \in \ell (G_\C/H_\C)^{\exp(Y)}
=  (G_\C/H_\C)^{\exp(X)}.
\]
On the other hand, since $Y\in \mathfrak{j}_\C$ satisfies $\exp(Y) \notin Z_{G_\C}$, we have $\exp(X) \notin Z_{G_\C}$.
It follows from the definition \eqref{eqn:G_C_fix_J} that $ (G_\C/H_\C)^{\exp(X)} \subset (G_\C/H_\C)_\textrm{fix}$.
This contradicts the assumption that $\ell \notin (G_\C)_\textrm{fix}$.
\end{proof}
\begin{lemma}
\label{lem:generic_not_center}
Let $G$ be a Zariski-connected, real simple
algebraic group such that
$\mathfrak{g} \neq \mathfrak{su}(2)$,
and let $\mathfrak{j}$ be a Cartan subalgebra of $\mathfrak{g}$.
Then there exists a countable set $\Xi$
in $\mathfrak{j}$ such that 
 $\exp(t Y) \not \in Z_{G}$ for any $t\in \R \smallsetminus \{0\}$ and for any $Y \not\in \R \Xi$.
\end{lemma}
\begin{proof}
Let $J$ be the Cartan subgroup of $G$ with Lie algebra $\mathfrak j$.
We note that the dimension of $\mathfrak{j}$ is greater that $1$.
 Since the kernel of the exponential map $\exp\colon \mathfrak{j} \to J$
 is a countable set, the lemma is clear.
\end{proof}

We are ready to prove Lemma~\ref{lemma:overline_eta_simple}.
\begin{proof}[Proof of Lemma~\ref{lemma:overline_eta_simple}]
We shall find $X \in \mathfrak{j}$ and $Y \in \mathfrak{j}'$ in the order
$\mathfrak j \Rightarrow Y \Rightarrow \mathfrak{j}' \Rightarrow X$ as follows.

First, we take a Cartan subalgebra $\mathfrak{j}$ of $\mathfrak{g}$. 
The set $(G_\C)_{\textrm{fix}}$
is defined, as in \eqref{eqn:G_C_fix_J}.

Next, we choose
$Y \in \mathfrak{g}^{\textrm{ss}}_{\textrm{reg}}$
that satisfies the following three conditions, which are generic by 
Proposition~\ref{prop:overline_eta_simple_n}, 
Lemma~\ref{lem:generic_not_center},
 and Proposition~\ref{proposition:commutative-eta},
respectively: 
\begin{itemize}
    \item[(a)]
$\exp(t Y) \not \in Z_{G}$ for any $t\in \R \smallsetminus \{0\}$;
\item[(b)] $Y \in \operatorname{Ad}(\ell^{-1}) \mathfrak j$ for some $\ell \in G$ such that $\ell$ and $\ell^{-1}$ do not belong to $(G_\C)_\textrm{fix}$.
\item[(c)] $G(t\{Y\})$ is a Cartan subgroup of $G$ for any $t \in \R\smallsetminus \Lambda$, where $\Lambda$ is a countable subset.
\end{itemize}
Let $J'$ be the Cartan subgroup corresponding to $\mathfrak{j}':=\operatorname{Ad}(\ell^{-1}) \mathfrak j$.
Then, for any $t \in \R\smallsetminus \Lambda$, we have $G(t\{Y\})=J'$.  

Third, we take $X\in \mathfrak{j}$ satisfying the following two conditions: 
\begin{itemize}
    \item[(a)'] 
$\exp(t X) \not \in Z_{G}$ for any $t\in \R \smallsetminus \{0\}$;
\item[(c)']
$G(t\{X\})=J$ for any $t \in \R\smallsetminus \Lambda'$, where $\Lambda'$ is a countable subset.
\end{itemize}
We already know that
\begin{itemize}
    \item[(b)'] 
$X \in \operatorname{Ad}(\ell) \mathfrak j'$, where $\ell \in G$ such that $\ell$ and $\ell^{-1}$ do not belong to $(G_\C)_\textrm{fix}$.
\end{itemize}

Rescaling $X$ if necessary, we may and do assume that $\Lambda \cap \Lambda' =\{0\}$.
Then, for any $t\neq 0$, at least one of the following holds:
$G(t\{X\})=J$ or $G(t\{ Y\})=J'$.

In the case where $G(t\{X\})=J$,
we have $G(t\{X,Y\})=G(J; tY) =G$,
where the last identity follows from
Lemma~\ref{lem:J_Y_generates_G}.
On the other hand, in the case where
$G(t;\{ Y\})=J'$, we have $G(t\{X,Y\})=G(J'; tX) =G$,
again by Lemma~\ref{lem:J_Y_generates_G}.
Therefore, we have shown
$\overline\eta(G) \le 2$. 

Since the reverse inequality $2 \le \overline\eta(G)$ is clear,
the last assertion is also proved.
\end{proof}

\begin{proposition}
\label{prop:overline_eta_simple_n}
Suppose that $G$ is a Zariski-connected real algebraic group such that $\mathfrak{g}$ is isomorphic to $\mathfrak{s}^n$, the direct sum of $n$ copies of a simple Lie algebra $\mathfrak{s}$ which is not isomorphic to $\mathfrak{su}(2)$.
Then there exist regular semisimple elements $X$ and $Y \in \mathfrak g$
such that $G(t\{X, Y\})=G$ for all $t \neq 0$.
In particular, $\overline\eta(G) = 2$ for every $n \in \N_+$.
\end{proposition}
\begin{proof}
We write $G$ as the almost direct product of groups 
\[
G = S_1 \dotsb S_n,
\]
which corresponds to the direct sum decomposition of the Lie algebra 
\[\mathfrak{g} \simeq \mathfrak{s}_1 \oplus \dotsb \oplus \mathfrak{s}_n,
\]
where each $ \mathfrak{s}_i$ is isomorphic to the same simple Lie algebra $ \mathfrak{s}$.
For $1 \le i \le n$, we choose Cartan subgroups $J_i$ and $J'_i$ of $S_i$, and regular semisimple elements $X_i, Y_i \in  \mathfrak{s}_i$ as in Lemma~\ref{lemma:overline_eta_simple}, satisfying the following conditions:
\begin{align*}
    S(\{t X_i\})    &= J_i       & \quad & \text{for any } t \in \mathbb{R} \smallsetminus \Lambda_i, \\
    S_i(\{t Y_i\})   &= J'_i      & \quad & \text{for any } t \in \mathbb{R} \smallsetminus \Lambda'_i, \\
    S_i(J_i; \{t Y_i\}) &= S_i     & \quad & \text{for all } t \in \mathbb{R}, \\
    S_i(J'_i; \{t X_i\}) &= S_i    & \quad & \text{for all } t \in \mathbb{R},
\end{align*}
where $\Lambda_i$ and $\Lambda'_i $ are countable sets such that $ \Lambda_i \cap \Lambda'_i = \{0\}$.
We define $X := X_1 + \dotsb + X_n$ and $Y := Y_1 + \dotsb + Y_n$.
Then $G(t \{X, Y\})$ contains the subgroup $S_i$ for every $1 \leq i \leq n$ and for each $t \in \R\smallsetminus \{0\}$.
Hence, it follows that $G(t \{X, Y\}) = G$ for all $t \in \R\smallsetminus\{0\}$.
\end{proof}

\begin{proof}[Proof of Theorem~\ref{thm:two_generates_reductive}]

By Proposition~\ref{prop:reductive_eta}, it suffices to show that the Lie algebra $\mathfrak{g}$ is
either commutative or the direct sum of simple Lie algebras of the same type.

The commutative case is proved in Proposition~\ref{proposition:commutative-eta}, with two exceptional cases given in Lemma~\ref{lemma:eta-three-pattern}.

The formulas for $\underline{\eta}(G)$ and $\eta(G)$ in the non-commutative case are proved in
Proposition~\ref{prop:two_generates_reductive}.
The formula for $\overline{\eta}(G)$
in the non-commutative case is proved in 
Proposition~\ref{prop:overline_eta_simple_n}, with the exception case given in Proposition~\ref{prop:three_generators_su2_n}.

Hence, the proof of Theorem~\ref{thm:two_generates_reductive} is complete.
\end{proof}

\section*{Acknowledgments}
The first author was partially supported by the Special Postdoctoral
Researcher Program at RIKEN and JSPS under Grant-in Aid for Scientific Research 
(JP24K16929).

The second author was partially supported by the JSPS under Grant-in Aid for Scientific Research
(JP23H00084)
 and by the Institut Henri Poincar\'e (Paris) and the Institut des Hautes {\'E}tudes Scientifiques (Bures-sur-Yvette).  
He would like to express his gratitude to Gopal Prasad, Andrei Rapinchuk, B. Sury, and Aleksy Tralle for their warm hospitality during the  Conference ``Zariski Dense Subgroups, Number Theory, and Geometric Applications,'' held at the Ramanujan Lecture Hall, ICTS, Bangalore, India, from January 1 to 12, 2024, where he gave a series of three lectures on related topics.

\end{document}